\documentclass{amsart}

\usepackage{bbm}
\usepackage{color}
\usepackage{amscd}
\usepackage{nicefrac}
\usepackage{graphicx}
\usepackage{hyperref}
\usepackage{amssymb}
\usepackage{euscript}
\usepackage[all]{xy}
\usepackage[numbers]{natbib}
\usepackage{tikz-cd}
\usepackage[singlelinecheck=false]{caption}

\numberwithin{equation}{subsection}
\numberwithin{figure}{section}
\renewcommand{\section}[1]{
\vspace{4em}
\begin{center}\refstepcounter{section}{\bf \Large \thesection. #1}\end{center}\vspace{1em}\addcontentsline{toc}{section}{\thesection. #1\vspace{-1em}}
}
\renewcommand{\subsection}[1]{\vspace{2em}\begin{center}\refstepcounter{subsection}{\bf \thesubsection. #1}\end{center}\vspace{.5em}\addcontentsline{toc}{subsection}{\vspace{-1em}\hspace{1em}\thesubsection. #1}}
\renewcommand{\subsubsection}[1]{\hspace{-\parindent}\refstepcounter{subsubsection}{\bf (\thesubsubsection) #1.}}
\renewcommand{\thesubsubsection}{\arabic{section}.\arabic{subsection}\alph{subsubsection}}

\allowdisplaybreaks

\theoremstyle{plain}
\newtheorem{thm}{Theorem}[subsection]

\newtheorem{theorem}[thm]{Theorem}

\newtheorem{assumption}[thm]{Assumption}

\newtheorem{corollary}[thm]{Corollary}

\newtheorem{prop}[thm]{Proposition}
\newtheorem{definition}[thm]{Definition}

\newtheorem{remark}[thm]{Remark}

\newtheorem{proposition}[thm]{Proposition}
\newtheorem{example}[thm]{Example}
\newtheorem{non-example}[thm]{Non-example}

\newtheorem{lemma}[thm]{Lemma}

\newtheorem{conjecture}[thm]{Conjecture}

\newtheorem*{claim*}{Claim} 
\newtheorem*{lemma*}{Lemma}
\newtheorem*{theorem*}{Theorem}
\newtheorem*{conjecture*}{Conjecture}

\newtheorem{notation}[thm]{Notation}

\newcommand{\bC}{{\mathbb C}}
\newcommand{\bD}{{\mathbb D}}

\newcommand{\bF}{{\mathbb F}}

\newcommand{\bH}{{\mathbb H}}

\newcommand{\bK}{{\mathbb K}}

\newcommand{\bN}{{\mathbb N}}

\newcommand{\bR}{{\mathbb R}}

\newcommand{\bZ}{{\mathbb Z}}

\newcommand{\bfx}{\mathbf x}
\newcommand{\bfy}{\mathbf y}

\newcommand{\scrA}{\EuScript A}
\newcommand{\scrB}{\EuScript B}
\newcommand{\scrC}{\EuScript C}

\newcommand{\scrE}{\EuScript E}
\newcommand{\scrF}{\EuScript F}
\newcommand{\scrG}{\EuScript G}
\newcommand{\scrH}{\EuScript H}

\newcommand{\scrJ}{\EuScript J}
\newcommand{\scrK}{\EuScript K}
\newcommand{\scrL}{\EuScript L}
\newcommand{\scrM}{\EuScript M}

\newcommand{\scrO}{\EuScript O}

\newcommand{\fraka}{\mathfrak{a}}

\newcommand{\frako}{\mathfrak{o}}

\newcommand{\frakA}{\mathfrak{A}}

\newcommand{\frakC}{\mathfrak{C}}

\newcommand{\frakF}{\mathfrak{F}}

\newcommand{\frakH}{\mathfrak{H}}
\newcommand{\frakL}{\mathfrak{L}}
\newcommand{\frakM}{\mathfrak{M}}

\newcommand{\frakP}{\mathfrak{P}}
\newcommand{\frakR}{\mathfrak{R}}

\newcommand{\bigK}{\mathbb{K}}

\newcommand{\oSd}{\mathring{\mathfrak{R}}_d}
\newcommand{\Sd}{\mathfrak{R}_d}
\newcommand{\oRkm}{\mathring{\mathfrak{R}}_{d,m}}
\newcommand{\Rkm}{\mathfrak{R}_{d,m}}

\newcommand{\disc}{\mathbb{D}}

\newcommand{\half}{{\textstyle\frac{1}{2}}}
\newcommand{\quarter}{{\textstyle\frac{1}{4}}}

\newcommand{\smallfrac}[2]{{\textstyle\frac{#1}{#2}}}
\newcommand{\iso}{\cong}
\newcommand{\htp}{\simeq}
\newcommand{\smooth}{C^\infty}

\renewcommand{\hom}{\mathit{hom}}

\newcommand{\mymod}{\,\mathrm{mod}\,}
\newcommand{\heart}{\heartsuit}

\newcommand{\AC}{\mathfrak{AC}}
\newcommand{\AH}{\mathfrak{AH}}

\newcommand{\iotaCF}{\iota_{q}}
\newcommand{\floer}{}

\begin{document}
\title[Quantum connection]{\Large The quantum connection, Fourier-Laplace \\ transform, and families of $A_\infty$-categories}
\author{D. Pomerleano, P. Seidel\vspace{-2em}}
\maketitle

\begin{abstract} \vspace{-3em} Take a closed monotone symplectic manifold containing a smooth anticanonical divisor. The quantum connection on its cohomology has singularities at zero and infinity (in the quantum parameter). At zero it has a regular singular point, by definition. We show that the singularity at infinity is of unramified exponential type. The argument involves: realizing cohomology as a deformation of the symplectic cohomology of the divisor complement; the corresponding deformation of the wrapped Fukaya category; a new categorical interpretation of the Fourier-Laplace transform of $D$-modules; and the regularity theorem of Petrov-Vaintrob-Vologodsky in noncommutative geometry.
\end{abstract}

{\smaller\vspace{-1em}
\tableofcontents
}

\section{Introduction}

\subsection{The quantum connection\label{subsec:the-quantum-connection}}
Let $M$ be a closed $2n$-dimensional symplectic manifold which is monotone,
\begin{equation}
[\omega_M] = c_1(M) \in H^2(M;\bR).
\end{equation}
Take a formal variable $q$. The quantum connection on $H^*(M;\bC)[q^{\pm 1}]$ is the following $\bZ/2$-graded endomorphism, which differentiates with respect to $q$:
\begin{equation} \label{eq:original-quantum}
\nabla_{\partial_q} x = \partial_q x + q^{-1} ([\omega_M] \ast_q x),
\end{equation}
where $\ast_q$ is the small quantum product,
\begin{equation} \label{eq:small-quantum-product}
x \ast_q y = x \ast^{(0)} y + q \, x \ast^{(1)} y + q^2 \, x \ast^{(2)} y + \cdots
\end{equation}
The term $\ast^{(k)}$, which counts rational curves with first Chern number $k$, has degree $-2k$; for $k = 0$, it is the classical cup product. (This is the smallest, which means single-parameter, version of the quantum connection; readers interested in the wider picture may want to look at e.g.\ \cite{galkin-golyshev-iritani16, galkin-iritani19, iritani-mann-mignon16}.) Visibly, the quantum connection has a simple pole at $q = 0$, with nilpotent residue $x \mapsto [\omega_M] x$. The monodromy around that pole is conjugate to $x \mapsto \exp(-2\pi i [\omega_M] x)$. 

One can simplify the $q$-dependence in \eqref{eq:original-quantum} as follows. Take a grading operator 
\begin{equation} \label{eq:grading-operator}
\mathrm{Gr}: H^*(M;\bC) \longrightarrow H^*(M;\bC), 
\end{equation}
which multiplies the degree $d$ cohomology by some even integer $\delta_d$ such that $\delta_{d+2} = \delta_d + 2$. (If there is no odd degree cohomology, one can take $\delta_d = d$; or in general, $\delta_d = 2\lfloor d/2\rfloor$.) Then \eqref{eq:original-quantum} is gauge equivalent to
\begin{equation} \label{eq:big-q-pole-0}
\nabla_{\partial_q}^{\mathrm{Gr}}x \stackrel{\mathrm{def}}{=} (q^{\mathrm{Gr}/2\,} \nabla_{\partial_q} \,q^{-\mathrm{Gr}/2})(x) = \partial_q x - q^{-1}\frac{\mathrm{Gr}(x)}{2} + [\omega_M] \ast x,
\end{equation}
where $\ast$ without subscript means that we set $q = 1$ in \eqref{eq:small-quantum-product}. (In the literature, e.g.\ \cite{galkin-golyshev-iritani16}, odd values of $\delta_d$ are often used, leading to a gauge transformation involving $q^{1/2}$.)
We are interested in the behaviour near $q = \infty$, and therefore change variables to $Q = 1/q$. The quantum connection (multiplied by $-q^2$, to account for the difference between the vector fields $\partial_q$ and $\partial_Q$) is correspondingly written as
\begin{align} \label{eq:switch-q}
& \nabla_{\partial_Q} x = \partial_Q x - Q^{-1}([\omega_M] \ast_{Q^{-1}} x), \\
\label{eq:big-q-pole}
& \nabla_{\partial_Q}^{\mathrm{Gr}} x = \partial_Q x -
Q^{-2} ([\omega_M] \ast x) +
Q^{-1} \frac{\mathrm{Gr}(x)}{2}.
\end{align}
In \eqref{eq:big-q-pole} the pole order has been reduced to $2$, which in general is the best one can do by gauge transformations. Such quadratic singularities have a rich structure, of which we consider only the part remembered in a formal power series expansion in $Q$ (this means ignoring the Stokes phenomenon).

\begin{conjecture} \label{th:main-conjecture}
(i) $\nabla_{\partial_q}$ has a singularity of unramified exponential type at $q = \infty$.

(ii) The regularized formal monodromies at $q = \infty$ are quasi-unipotent (their eigenvalues are roots of unity). 
%
\end{conjecture}

The terminology (unramified exponential type, regularized monodromy) will be explained in Section \ref{subsec:classical}. Statement (i) occurs in several places in the literature, motivated by mirror symmetry: e.g.\ as a small piece of \cite[Conjecture 3.4]{katzarkov-kontsevich-pantev08}, or in \cite[Section 2.5]{galkin-golyshev-iritani16} (qualified by warnings: \cite[Remark 3.5(ii)]{katzarkov-kontsevich-pantev08} suggests restricting it to the algebro-geometric case of Fano varieties, while \cite{galkin-golyshev-iritani16} talks of ``a wide class of Fanos''); part (ii) is less familiar, but closely related. In both references, the unramified exponential type property plays an important role as an initial step towards a deeper understanding of the quantum connection: there are farther-reaching conjectures (the noncommutative Hodge conjecture in \cite{katzarkov-kontsevich-pantev08}, and the Gamma conjecture in \cite{galkin-golyshev-iritani16}) whose statements only make sense if one assumes that property. Those conjectures make extraordinarily precise predictions, which have been checked for certain classes of examples; but as general statements, they pose a fundamental challenge to our understanding of the enumerative geometry of Fano varieties (not to mention non-algebraic monotone symplectic manifolds, which is a very poorly understood subject).

Roughly speaking, there are three previous approaches to Conjecture \ref{th:main-conjecture}.
\begin{itemize} \itemsep.5em
\item Certain cases can be solved directly. If the endomorphism $[\omega_M] \ast$ is semisimple, (i) holds by an elementary power series argument; see Lemma \ref{th:quadratic-pole}(ii). In the more restrictive situation where the quantum cohomology ring is semisimple (a direct sum of copies of the ring $\bC$), a much stronger version of (ii) holds, by a Gromov-Witten theory argument due to Dubrovin; see Lemma \ref{th:semisimple}. From a computational perspective, for any specific manifold whose quantum product is known, it can be decided algorithmically if the properties from Conjecture \ref{th:main-conjecture} hold (see e.g.\ \cite{babbitt-varadarajan83, barkatou97}). 

\item Suppose that quantum cohomology has a mirror description in terms of a Landau-Ginzburg model (a variety with a function, called superpotential). One can apply algebro-geometric methods to the Gauss-Manin connection of that superpotential, and derive conclusions about the quantum connection by a Fourier-Laplace transform; this is the motivation mentioned above. We refer to Section \ref{subsec:classical-gauss-manin} for further explanation of that strategy, which has seen extensive use in the literature. For instance, (i) is proved for certain complete intersections in projective space in \cite[Proposition 7.4]{sanda-shamoto19}, based on results in \cite{reichelt-sevenheck17}. This method addresses cases which are not covered by the previously mentioned elementary arguments (even for Fano toric manifolds, the quantum cohomology can fail to be semisimple \cite[Remark 5.1]{ostrover-tyomkin08}). As for future outlook, one could hope to apply such arguments in the context of the ``intrinsic mirror symmetry'' of Gross-Siebert, where $M$ admits a (singular) anticanonical divisor $D$ which makes it into a maximal log Calabi-Yau pair \cite[Definition 2.7]{gross-siebert18}. 

\item For any $\lambda \in \bC$ we have a Fukaya category $\mathit{Fuk}_\lambda(M)$, which is a $\bZ/2$-graded $A_\infty$-category over $\bC$. This category is zero unless $\lambda$ is an eigenvalue of $[\omega_M] \ast$. The cyclic open-closed map \cite{ganatra19} is a map from negative cyclic homology to ordinary (co)homology, more precisely a $u$-linear map
\begin{equation} \label{eq:cyclic-oc-lambda}
\bigoplus_{\lambda} \mathit{HC}_*(\mathit{Fuk}_\lambda(M)) \longrightarrow H^{*+n}(M;\bC)[[u]],
\end{equation}
where $u$ is another formal variable. Let's adjoin $u^{-1}$. Then, on the left hand side (periodic cyclic homology) we have a connection in the $u$-variable, with a quadratic pole at $u = 0$ \cite{katzarkov-kontsevich-pantev08, shklyarov14}, which has nilpotent leading order term \cite{amorim-tu22}. Adjust that connection by adding $\lambda u^{-2}$ times the identity on each summand (this reflects the fact that $\mathit{Fuk}_{\lambda}(M)$ should really be thought of as a curved $A_\infty$-category, with a curvature term which is $\lambda$ times the identity). On the right hand side, the corresponding connection is a version of \eqref{eq:big-q-pole}, with $u$ instead of $Q^{-1}$, and our definition of grading operator replaced by $\delta_d = d-n$. Suppose that we are in the situation of a generation result such as \cite[Theorem 3]{ganatra17}, where \eqref{eq:cyclic-oc-lambda} is an isomorphism. Then one can infer properties of our connection from ones of the Fukaya category. The first use of this relation was made in \cite{hugtenburg22} (under simplifying technical assumptions). As an example, \cite[Section 6.2]{hugtenburg22} analyzes the quantum connection for the intersection of two quadrics in $\bC P^5$; this in particular shows Conjecture \ref{th:main-conjecture} for that manifold.
\end{itemize}

%


\subsection{Results and general approach}
Our approach relies on the existence of an anticanonical divisor which is smooth (unlike the situation in the Gross-Siebert program), more precisely:

\begin{assumption} \label{th:anticanonical-divisor}
$M$ contains a smooth symplectic hypersurface $D$ (integrally) Poincar{\'e} dual to $c_1(M)$, such that the completed complement of $D$, which a priori is a Liouville manifold, is actually (finite type) Weinstein.
\end{assumption}

In the algebro-geometric context, if $M$ is Fano and $D$ is an anticanonical divisor, $M \setminus D$ is affine and hence automatically Weinstein. 

\begin{remark}
Still in algebraic geometry, it is known that every Fano manifold of complex dimension $\leq 3$ has a smooth anticanonical divisor \cite{shokurov79}. In complex dimensions $\leq 5$, it is known that anticanonical divisors exist (\cite[Theorem 5.2]{kawamata00} and \cite{hoering-smiech20}), and conjecturally the same is true in all dimensions \cite[Conjecture 2.1]{kawamata00}; however, there may not be any smooth ones (\cite[Example 2.12]{hoering-voisin11} or \cite[Example 2.9]{sano14}). Whether that failure of smoothness is relevant for our purposes remains unclear (there can be smoothings in the symplectic world which are precluded algebro-geometrically; for instance, deforming the Fano does not change its symplectic structure).
\end{remark}

Here's our main result towards Conjecture \ref{th:main-conjecture}, as well as an addendum which concerns a sharpened version of part (ii) of that conjecture.

\begin{theorem} \label{th:main}
Conjecture \ref{th:main-conjecture} is true for all manifolds $M$ satisfying Assumption \ref{th:anticanonical-divisor}.
\end{theorem}

\begin{theorem} \label{th:main-2}
In the situation of Theorem \ref{th:main}, the regularized formal monodromies have the following property: any Jordan block for an eigenvalue $\neq 1$ is of size $\leq n = \mathrm{dim}_{\bC}(M)$ (meaning, the $\mathrm{dim}_{\bC}(M)$-th power of the nilpotent part is zero); and any Jordan block for the eigenvalue $1$ is of size $\leq n+1$.
\end{theorem}

We now describe the main ideas in informal terms. One point worth making is that, even though we have focused on Conjecture \ref{th:main-conjecture} as a concrete goal, our approach is based on a more comprehensive understanding of the quantum connection.
\begin{itemize} \itemsep.5em \parskip.5em \parindent0em
\item On the side of algebraic differential equations, or slightly more generally holonomic $D$-modules, the key concept is Fourier-Laplace transform, which turns $q$-differentiation into multiplication by the dual variable $t$. Classical work of Malgrange \cite{malgrange} describes how this affects singularities. 
In particular, if a $D$-module has only regular singularities in $t$, then its Fourier-Laplace transform has an irregular singularity of unramified exponential type at $q = \infty$. Moreover, the regularized formal monodromy of that singularity is related to the monodromy of the original $D$-module around the singular values of $t$; see Section \ref{subsec:classical}.

Note that the insights in \cite{malgrange} go further, relating the Stokes structure at $q = \infty$ to the global monodromy representation of the $t$-connection (there is an alternative, more topological, picture of this in \cite{dagnolo}). At present, it is not clear what implications this has for the quantum connection, because noncommutative geometry gives only limited information about the relevant monodromy representation (see the last point below).

\item Before we talk about the Fourier-Laplace transform of the quantum connection, it is appropriate to extend that connection over $q = 0$ as a holonomic $D$-module. The reason why this is sensible is that invertibility of the variable $q$ corresponds, under Fourier-Laplace transform, to invertibility of the differentiation operator $\nabla_{\partial_t}$, which is not a geometrically natural notion. We accomplish the $q = 0$ extension by replacing quantum cohomology with a relative version, defined using a $q$-deformation of ($S^1$-equivariant) symplectic cohomology of the divisor complement. The idea to recover quantum cohomology in this way was proposed by Borman and Sheridan, and has been realized in various version \cite{borman-sheridan-varolgunes21, el-alami-sheridan24b, pomerleano-seidel24}; the last is what we use here. 

From a more general perspective, deformed $S^1$-equivariant symplectic cohomology is built out of three pieces of structure. One is the $L_\infty$-algebra structure on the chain complex underlying symplectic cohomology; the second is the structure of the equivariant theory as an $L_\infty$-module; and the third is a Maurer-Cartan element in the $L_\infty$-algebra, which describes the deformation itself. The relevant Hamiltonian (closed string) Floer theory is explained in Section \ref{sec:closed-string}, building on technical preliminaries in Section \ref{sec:floer} and a summary of the algebra in Section \ref{subsec:linfty}. (There is a large amount of overlapping literature in this direction: the technical setup is similar to \cite{abouzaid10, ganatra13}; our treatment of $S^1$-equivariant Floer cohomology follows \cite{ganatra19}; and the $L_\infty$-algebra structure is considered in \cite{abouzaid-groman-varolgunes, borman-el-alami-sheridan24}.)

\item At this point, we pass to the categorical (open string) side, replacing $S^1$-equivariant deformed symplectic cohomology by the cyclic homology of a $q$-deformation of the wrapped Fukaya category of the divisor complement, with its Getzler-Gauss-Manin connection. The transition is explained in Section \ref{sec:open-string}, and the required algebraic definitions are summarized in Section \ref{subsec:ainfty}. (This also builds on previous ideas, notably \cite{ganatra13} for the $S^1$-equivariant open-closed map.)

There are two reasons for using the $q$-deformed wrapped Fukaya category of the complement, rather than (as previously suggested) the ordinary Fukaya category of the closed symplectic manifold. The first and more fundamental reason is that the latter object only ``sees'' the singularities of the Fourier-Laplace transformed connection, and not the entire connection. This is best understood in terms of mirror symmetry: the ordinary Fukaya category corresponds to the category of Landau-Ginzburg $D$-branes, which is only sensitive to the singular fibres of the superpotential (and after idempotent closure, only to the formal neighbourhood of the singular points \cite{orlov}). That limited view could still be sufficient to extract the information required for Conjecture \ref{th:main-conjecture}; but it stands in the way of applying the Fourier-Laplace transform as we want to do it, and also limits possible further developments. The second, more technical, reason is that we can treat the deformed wrapped Fukaya category as a perturbation of its $q = 0$ reduction, which is known to have good properties. In particular, thanks to \cite{gps2} we know that the open-closed map for the wrapped Fukaya category is an isomorphism, and this carries over to the deformed version.

\item We started with quantum cohomology, passed to deformed $S^1$-equivariant symplectic cohomology, and then interpreted that as cyclic homology of a deformed category. There is one more step, which is to replace those deformed theories by versions which involve only negative powers of $q$. The effect is best formulated from a Fourier-Laplace dual perspective: the two versions agree away from finitely many values of $t$ (see Section \ref{subsec:finiteness-results}).

The reason for this change is that the $q$-negative version fits in with a new algebraic development, the categorical Fourier-Laplace transform. Roughly speaking, the idea is that a formal (curved) $q$-deformation of a dg or $A_\infty$-algebra, with a parameter $q$ of degree $2$, can be dually understood as providing that algebra with a $\bC[t]$-structure, which means viewing it as the total space in a one-parameter family of ``fibre algebras'' depending on $t$. The $q$-negative version of Hochschild, or cyclic, homology of the deformation is then isomorphic to the corresponding fibrewise (taken over $\bC[t]$) homology of the family. This theory is developed in Section \ref{subsec:dga}, in a classical context (curved deformation of a dga by a central element), and then carried over to a more general situation in Section \ref{subsec:fiber}.

\item At this point, we have a family of $A_\infty$-algebras parametrized by $t$, produced from the deformed wrapped Fukaya category by the categorical Fourier-Laplace transform. This family is the noncommutative analogue of the mirror Landau-Ginzburg model. It inherits good cohomological properties: it is fibrewise proper and, away from finitely many values of $t$, fibrewise smooth (the last-mentioned property is analogous to saying, in a mirror symmetry context, that the superpotential has only finitely many critical values). The final and crucial ingredient is Petrov-Vaintrob-Vologodsky's regularity theorem for the Getzler-Gauss-Manin connection of a smooth family \cite{petrov-vaintrob-vologodsky18}, stated in Section \ref{subsubsec:monodromy}: it shows that the singularities are regular, and each has quasi-unipotent monodromy. This replaces what, in a mirror symmetry approach, would be the use of the monodromy theorem for Gauss-Manin connections. The result of \cite{petrov-vaintrob-vologodsky18} also includes a Jordan block bound (again generalizing the classical geometric theory), which leads to Theorem \ref{th:main-2}. 

The proof in \cite{petrov-vaintrob-vologodsky18} involves reduction to characteristic $p$, and the construction of a mod $p$ Fontaine-Laffaille structure. One might expect that this is only the first truncation of a richer $p$-adic structure. The analogy here is with $p$-adic Hodge theory in algebraic geometry, where that structure can be used to recover the fibrewise cohomology as a local system with $\bZ_p$-coefficients \cite{faltings}. It remains to be seen to what extent similar ideas can work in noncommutative geometry; in principle, the theory constructed here could be used to carry any resulting insights over to the quantum connection.
\end{itemize}

\subsection{A more technical description}
We need to introduce a different way of writing the quantum connection, which is closer to what happens in both noncommutative and symplectic geometry. The formal variables $q$ and $u$ have already been mentioned, but the latter only as part of the separate discussion of \eqref{eq:cyclic-oc-lambda}. Both will now be used together, and they will be given degree $2$. On $H^*(M;\bC)[q^{\pm 1},u^{\pm 1}]$ consider the $\bZ$-graded endomorphism 
\begin{equation} \label{eq:quantum-connection-with-u}
\nabla_{u\partial_q} x = u \partial_q x + q^{-1}([\omega_M] \ast_q x).
\end{equation}
The part of $H^*(M;\bC)[q^{\pm 1},u^{\pm 1}]$ in any given degree $d$ is isomorphic to $H^{d \mymod 2}(M;\bC)[q^{\pm 1}]$ (by setting $u = 1$). Under that isomorphism, \eqref{eq:quantum-connection-with-u} corresponds to \eqref{eq:original-quantum}.

\begin{remark} \label{th:explain-lattice}
It is instructive to look at \eqref{eq:quantum-connection-with-u} without inverting $u$, since that provides a more organic explanation for the gauge transformation \eqref{eq:big-q-pole-0}. Take the degree zero parameter $\bar{q} = q/u$, so that $H^*(M;\bC)[q^{\pm 1},u] \iso H^*(M;\bC)[q^{\pm 1},\bar{q}^{-1}]$. As before, setting $q = 1$ yields an identification between the degree $d$ part of $H^*(M;\bC)[q^{\pm 1},\bar{q}^{-1}]$ and $H^{d \mymod 2}(M;\bC)[\bar{q}^{-1}]$, which turns \eqref{eq:quantum-connection-with-u} into
\begin{equation} \label{eq:weird-scaling}
\nabla_{\partial_{\bar{q}}} x = \partial_{\bar{q}}x + \bar{q}^{-1} \textstyle\frac{d-j}{2}x +
[\omega_M] \ast x
\quad \text{for } x \in H^j(M;\bC)[\bar{q}^{-1}].
\end{equation}
Different $d$ correspond to different choices of the grading operator in \eqref{eq:big-q-pole-0}. Passing to $\bar{Q} = \bar{q}^{-1} = u/q$ yields the counterpart of \eqref{eq:big-q-pole}, which is 
\begin{equation} \label{eq:qq-pole}
\nabla_{\partial_{\bar{Q}}} x = \partial_{\bar{Q}}x - \bar{Q}^{-2} ([\omega_M] \ast x) + \bar{Q}^{-1} \textstyle\frac{j-d}{2} x
\quad \text{for } x \in H^j(M;\bC)[\bar{Q}].
\end{equation}
\end{remark}

The proof of Theorems \ref{th:main} and \ref{th:main-2} centers on two objects. On the closed string side, we have the deformed $S^1$-equivariant symplectic cohomology of $M \setminus D$, denoted by $H_{q,u}$. This is a graded $\bC[[q,u]]$-module. Informally speaking, $q$ counts intersections with the divisor $D$; the actual (equivalent) implementation is that holomorphic curves going through $D$ are encoded in a Maurer-Cartan element which has $q$ as a variable, and the deformation then goes via inserting that element at additional interior marked points. Actually, for degree reasons (Lemma \ref{th:boundedness}) only the first order term of the Maurer-Cartan element, which we call the Borman-Sheridan cocycle, is nontrivial. By definition, setting $q = 0$ recovers the $S^1$-equivariant symplectic cohomology of $M \setminus D$ (setting $u = 0$ recovers the non-equivariant theory). The main result of \cite{pomerleano-seidel24} says that inverting $q$ yields a space isomorphic to $H^*(M)[q^{\pm 1},u]$. Moreover, that isomorphism relates \eqref{eq:quantum-connection-with-u} to a Floer-theoretically constructed connection (originally defined in \cite{seidel18}; we give an independent account here, in Section \ref{section:connection}).

On the open string side, we have the deformed wrapped Fukaya category of $M \setminus D$, denoted by $\scrA_q$. The same observation applies here: the $q = 0$ part is the standard wrapped category, and the deformation is defined by inserting the Maurer-Cartan element (really the Borman-Sheridan cocycle) at extra marked points; in this context, that is simpler than thinking directly about holomorphic curves going through $D$. The categorical Fourier-Laplace transform (Theorem \ref{th:noncommutative-fourier-transform} for dga's, later extended using a suitable quasi-isomorphism argument) relates $\scrA_q$ to a different object $\scrA_t$, which is an $A_\infty$-category over $\bC[t]$, with the variable $t$ having degree $0$. When considered just as $\bC$-linear, this $\scrA_t$ is quasi-isomorphic to the ordinary wrapped Fukaya category (no deformation), hence smooth by general properties of Weinstein manifolds. In our specific geometric situation, $\scrA_t$ is also proper over $\bC[t]$ (Lemma \ref{th:a-proper}). We consider $\scrA_{t,1/p} = \bC[t,1/p] \otimes_{\bC[t]} \scrA_t$ for some polynomial $p = p(t)$. Unlike ordinary algebraic geometry, it's not a priori clear that in noncommutative geometry, smoothness over $\bC$ implies smoothness over $\bC[t,1/p]$ for suitable $p$; we apply a specific smoothness criterion (Proposition \ref{th:smooth-2}), which in our case can be shown to hold via a computation in deformed symplectic cohomology (Lemma \ref{th:exploit-a}). This is needed in order to apply \cite{petrov-vaintrob-vologodsky18} (the precise result we use is Corollary \ref{th:end-of-algebra}, whose statement includes the smoothness criterion we just mentioned).

Throughout the argument, completeness with respect to $u$, which is a necessary feature of cyclic homology, is the main technical problem that forces constructions to be carried out in a particular order. Notably, taking cyclic homology does not commute with inverting $p(t)$. To avoid that issue, tensoring with $\bC[t,1/p]$ has to take place much earlier in the argument, while we are still on the closed string side. We also choose that point to pass to negative powers of $q$, which is a requirement of the categorical Fourier-Laplace transform. This means that the cyclic homology of $\scrA_q$ appears only in heavily modified form. 
\begin{figure}[t!]
\begin{centering}
\xymatrix{
H^*(M;\bC)[q^{\pm 1},u]  \ar@{<->}[d]_-{\iso}
\\
\bC[q^{\pm 1}] \otimes_{\bC[q]} H_{q,u}  \ar[r]^-{\text{invert $u$}}
&
\bC[q^{\pm 1},u^{\pm 1}] \otimes_{\bC[q,u]} H_{q,u}
&
\ar@{-->}[l] \fbox{\parbox{7.5em}{\scriptsize over $\bC[(q/u)^{\pm 1}]$, this is a vector bundle with connection $\nabla_{u\partial_q}$}}
\\
H_{q,u} \ar[u]_{\text{invert $q$}} \ar[d]^-{\text{invert $p(t)$}}
\ar[r]^-{\text{invert $u$}} 
&
\bC[u^{\pm 1}] \otimes_{\bC[u]} H_{q,u} 
\ar[u]_-{\text{invert $q$}}
\ar[d]^-{\text{invert $p(t)$}}
& 
\ar@{-->}[l] \fbox{\parbox{7.5em}{\scriptsize this is a holonomic $D$-module in each degree}}
\\
\bC[t,1/p] \otimes_{\bC[t]} H_{q,u}  
\ar[r]^-{\text{invert $u$}}
&
\bC[t,1/p,u^{\pm 1}] \otimes_{\bC[t,u]} H_{q,u}
\ar@{<->}[ddd]^-{\iso}
& 
\fbox{\parbox{7.5em}{\scriptsize over $\bC[t,1/p(t)]$, $t = \nabla_{u\partial_q}$, this is a vector bundle with connection $\nabla_{\partial_t} = -q/u$}}
\ar@{-->}[l]
\\ 
q^{-1}H_{q^{-1},1/p,u}
\ar[u]_-{\parbox{7.5em}{\scriptsize acyclicity of the\\ mapping cone}}^-{\iso}
\\ 
H\big(\bC[t,1/p] \hat\otimes_{\bC[t]} q^{-1}\bC[q^{-1}] \hat\otimes_{\bC[[q]]} \mathit{CC}_*(\scrA_q)\big) 
\hspace{-6em}
\ar[u]_-{\parbox{7.5em}{\scriptsize open-closed map}}^-{\iso}
\ar[d]^-{\parbox{7.5em}{\scriptsize categorical\\ Fou\-rier-La\-place}}_-{\iso}
\\
 \mathit{HC}_*(\scrA_{t,1/p}) \ar[r]^-{\text{invert $u$}} & \mathit{HP}_*(\scrA_{t,1/p}) & \ar@{-->}[l] \fbox{\parbox{7.5em}{\scriptsize this has regular singularities with quasi-uni\-po\-tent monodromy}}
}
\caption{\label{fig:diagram}Summary of the cohomology groups appearing in our argument. 
}
\end{centering}
\end{figure}%

Figure \ref{fig:diagram} shows the various groups and isomorphisms involved. A quick walkthrough:
\begin{itemize} \itemsep.5em
\item start with the ordinary cohomology of $M$ at the top, and then write that as the $q$-inverted version of $H_{q,u}$.

\item $H_{q,u}$ becomes particularly simple to understand after inverting $u$, where we can prove that it is a holonomic $D$-module under the action of $q/u$ and $\nabla_{u\partial_q}$ (Lemma \ref{th:holonomic}).

\item The Fourier-Laplace transform for $\bC[u^{\pm 1}] \otimes_{\bC[u]} H_{q,u}$ consists of renaming variables as $t = \nabla_{u\partial_q}$ and $\nabla_{\partial_t} = -q/u$. As a general property of holonomic $D$-modules, inverting some polynomial $p(t)$ then yields a vector bundle in $t$, on which $\nabla_{\partial_t}$ is a connection. We already apply $p(t)$-inversion to $H_{q,u}$ itself. Throughout the subsequent argument, there are several point where we will make this initially chosen polynomial larger (by multiplying it with other polynomials).

\item Next, we pass to negative powers of $q$, which means tensoring with the $\bC[[q]]$-module $q^{-1}\bC[q^{-1}] = \bC((q))/\bC[[q]]$; this is required in order to apply the categorical Fourier-Mukai transform. It gives rise to the group denoted by $q^{-1}H_{q^{-1},1/p.u}$ in Figure \ref{fig:diagram}. Assuming a suitable choice of $p(t)$, the natural map will be an isomorphism.

\item We use the cyclic open-closed map for the deformed wrapped Fukaya category $\scrA_q$, which leads us to negative cyclic homology; again, in a version with negative powers of $q$ and with $p(t)$ inverted.

\item At this point, the categorical Fourier-Laplace transform kicks in and allows us to re-interpret cyclic homology as that of the $\bC[t]$-linear $A_\infty$-category $\scrA_t$. Again, we apply this in a version with $p(t)$ inverted, which means that the relevant structure is $\scrA_{t,1/p}$. 

\item In the $u$-inverted version (right column), we now have periodic cyclic homology, which is where the theory of \cite{petrov-vaintrob-vologodsky18} works.
\end{itemize}
At this point, some of the notation in Figure \ref{fig:diagram} is as yet incompletely explained. We will give another summary of the argument towards the end of the paper, in Remark \ref{th:summary}, when all the theory has been set up.

\subsection{Conventions and notation\label{subsec:conventions}}

\begin{itemize} \itemsep1em
\item[(a)] For formal variables, our convention is that they are supercommuting. In particular, if $\bK$ is a field and $\epsilon$ is an odd degree formal variable, then $\epsilon^2 = 0$, so that $\bK[\epsilon] = \bK \oplus \bK\epsilon$. Any formal variable has a corresponding derivation. In the odd case, this is the endomorphism of $\bK[\epsilon]$ defined by $\partial_\epsilon(\epsilon) = 1$, $\partial_\epsilon(1) = 0$. 

\item[(b)] Let $q$ be a formal variable of even degree. Given a graded $\bK$-vector space $V$, we use the shorthand notation
\begin{equation}
V[q^{-1}] \stackrel{\mathrm{def}}{=} V((q))/qV[[q]] = V[[q]] \otimes_{\bK[[q]]} (\bK((q))/q\bK[[q]]).
\end{equation}
One can think of the elements of $V[q^{-1}]$ as polynomials in $q^{-1}$ with coefficients in $V$, with the proviso that multiplication by $q$ acts as zero on the constant term. We will also encounter a slight modification, namely $q^{-1}V[q^{-1}] = V((q))/V[[q]]$.

\item[(c)] Throughout the discussion of algebraic structures, $|a|$ is the degree of an element, and $\|a\| = |a|-1$. The sign convention for $A_\infty$-algebras is that
\begin{equation} \label{eq:associativity}
\sum_{ij} (-1)^{\|a_1\|+\cdots+\|a_i\|} \mu_{\scrA}^{d-j+1}(a_1,\dots,a_i,\mu_{\scrA}^j(a_{i+1},\dots,a_{i+j}),\dots,a_d) = 0.
\end{equation}
For a strictly unital $A_\infty$-algebra, the unit $e_{\scrA}$ satisfies 
\begin{equation} \label{eq:strict-unit}
\begin{aligned}
& \mu^2_{\scrA}(e_{\scrA},a) = a, \quad \mu^2_{\scrA}(a,e_{\scrA}) = (-1)^{|a|} a, \\
& \mu^d_{\scrA}(\dots,e_\scrA,\dots) = 0 \text{ for all $d \neq 2$.}
\end{aligned}
\end{equation}
A differential graded algebra becomes an $A_\infty$-algebra by setting
\begin{equation} \label{eq:diff-a-infinity}
\mu_{\scrA}^1(a) = da, \;\; \mu^2_{\scrA}(a_1,a_2) = (-1)^{|a_1|} a_1a_2.
\end{equation}
For $A_\infty$-categories, we write the morphism spaces as $\scrA(X_0,X_1)$, and also use the shorthand notation 
\begin{equation} \label{eq:multi-a}
\scrA(X_0,\dots,X_d) = \begin{cases} 
\scrA(X_0,X_1) \otimes \cdots \otimes \scrA(X_{d-1},X_d) & d > 0, \\
\bK & d = 0.
\end{cases}
\end{equation}
As already indicated by this, the composition of morphisms is written in reverse order from the classical one; so the $A_\infty$-operations are
\begin{equation}
\begin{aligned}
&
\mu_{\scrA}^d: \scrA(X_0,\dots,X_d) \longrightarrow \scrA(X_0,X_d)[2-d], 
\\
&
\mu_{\scrA}^d(a_1,\dots,a_d) \in \scrA(X_0,X_d)\;\; \text{ for } a_k \in \scrA(X_{k-1},X_k).
\end{aligned}
\end{equation}
In the geometric application to Fukaya categories, the marked points and Lagrangians are ordered as in Figure \ref{fig:ordered-discs}.
\begin{figure}
\begin{centering}
\includegraphics{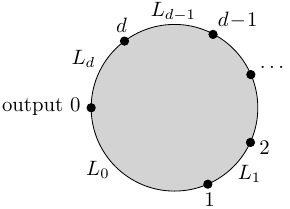}
\caption{\label{fig:ordered-discs}A disc with boundary punctures, showing the ordering convention in the definition of the Fukaya $A_\infty$-operation $\mu^d(a_1,\dots,a_d)$.}
\end{centering}
\end{figure}

\item[(d)] We write $\mathit{HH}_*(\scrA)$, $\mathit{HH}^*(\scrA)$, $\mathit{HC}_*(\scrA)$, $\mathit{HP}_*(\scrA)$ for, respectively, Hochschild homology, Hochschild cohomology, negative, and periodic cyclic homology (note that in spite of the subscript, the grading on Hochschild and cyclic homology is still cohomological). The notation for the underlying standard chain complexes is $C_*(\scrA)$, $C^*(\scrA)$, $\mathit{CC}_*(\scrA)$, and lastly $\mathit{CC}_*(\scrA) \otimes_{\bK[u]} \bK[u^{\pm 1}]$; we'll actually use several variants of those complexes, and the notation will be modified accordingly. 

\item[(e)] The formal variable $u$ which appears in $S^1$-equivariant Floer cohomology and in cyclic homology agrees with the convention in \cite{getzler95}, meaning that the sign is opposite of that in \cite{ganatra-perutz-sheridan15, sheridan20}; see \cite[Remark 3.19]{sheridan20}. This choice of sign is already visible in the definition of the quantum connection, compare \eqref{eq:quantum-connection-with-u} and \cite[Definition 3.1]{ganatra-perutz-sheridan15}.

\item[(f)] On a symplectic manifold $(M,\omega_M)$, the Hamiltonian vector field $X$ of a function $H$ satisfies $\omega_M(\cdot,X) = dH$. The Poisson bracket is $\{H_1,H_2\} = -\omega_M(X_1,X_2)$. 

\item[(g)] In the context of Floer cohomology, or more generally Cauchy-Riemann equations for maps $u:S \rightarrow M$, we use the following notation. $(S,j_S)$ is a Riemann surface; $J_S$ is a family of almost complex structures on $M$ parametrized by $S$; and $K_S$ is a one-form on $S$ with values in $\smooth(M,\bR)$, which is used to define the inhomogeneous term $Y_S$ in the Cauchy-Riemann equation, see \eqref{eq:k-y}. If $S$ has boundary, the boundary conditions $L_S$ are a family of Lagrangian submanifolds of $M$ parametrized by points of $\partial S$. For Hamiltonian Floer cohomology, which lives on a cylinder $S = \bR \times S^1$, we follow the usual convention that $S^1 = \bR/\bZ$.

\item[(h)] Operations in Floer cohomology are defined using a variety of parameter spaces (moduli spaces) of Riemann surfaces, the main ones of which are listed in Figure \ref{fig:parameter}.
\end{itemize}
\begin{figure}
\begin{centering}
{\renewcommand{\arraystretch}{2}%
\begin{tabular}{p{4em}|p{16em}|p{10em}}
notation & objects parametrized & algebraic operation
\\ \hline
$\frakF_m$ & 
\parbox{16em}{
points in the plane (Fulton-Mac\-Pherson); Section \ref{section:FMd}} 
& $\ell^m$ and $\delta_q^{\mathrm{diag}}$; Sections \ref{section:Linf}, \ref{sec:q-deformed}
\\ \hline
$\frakC_m$ & points on the cylinder; Section \ref{section:FM-cylinder} & $\ell^{m,1}$ and $\delta_q$; Sections \ref{section:linftymod}, \ref{sec:q-deformed} 
\\ \hline
$\frakA_r$ & cylinders with angles; Section \ref{section:angles} & $\delta_{S^1}$; Section \ref{subsubsec:s1-equivariant}
\\ \hline
$\AC_{m,r}$ & cylinders with angles and points; Section \ref{section:anglesFM} & $\ell^{m,1}_{S^1}$ and $\delta_{S^1,q}$;  Sections \ref{sec:s1module}, \ref{sec:q-deformed} 
\\ \hline
$\AC_{m,r,w}^{(A)}$, $\AC_{m,r,w}^{(B)}$
& same as before, but with constraints on one interior marked point; Sections \ref{section:cartan-A}--\ref{section:cartan-B}
& $KH_{(A)}$, $KH_{(B)}$; Section \ref{section:connection}
\\ \hline
$\frakR_d$ & 
points on the boundary of the disc (Fukaya-Stasheff); Section \ref{sec:pointeddiscs} & $\mu^d$; Section \ref{subsubsec:Fukaya}
\\ \hline
$\frakR_{d,m}$ & discs with boundary and interior marked points; Section \ref{subsubsec:with-interior}
& $\mu_q^d$; Section \ref{section:deform-fuk}
\\ \hline
$\frakR_{d,m}^{\pm}$, $\frakR_{d,m}^0$ &
same as before, but with constraints on one of the interior points; Section \ref{subsubsec:half-strip} & $GM_{\pm}$, $GM_0$; Section \ref{subsubsec:half-strip}
\\ \hline
$\AH_{d,m,r}$ & half-cylinders with angles and points; Section \ref{subsubsection:angledecoratedhalf} & plays an auxiliary role, to define the spaces below
\\ \hline
$\AH_{d,m,r}^{(1)}$, $\AH_{d,m,r}^{(2)}$ & subsets of $\AH_{d,m,r}$; Sections \ref{subsubsection:ocs1}--\ref{subsubsection:ocs2} &
$\mathit{OC}_{S^1,q,(1)}$, $\mathit{OC}_{S^1,q,(2)}$; Section \ref{section:deformed-oc}
\\ \hline
$\AH_{d,m,r,w}^{(A1)}$ etc.\ &
same as before, but with constraints on one interior marked point; Sections \ref{sec:A1}--\ref{sec:B2}
&
$\mathit{IT}_{(A1)}$ etc.\
Sections \ref{sec:A1}--\ref{sec:B2}
\end{tabular}
}
\caption{Notation for the most relevant moduli spaces of Riemann surfaces, and the Floer-theoretic operations they give rise to.\label{fig:parameter}}
\end{centering}
\end{figure}%

{\em Acknowledgements.} The idea of Fourier-Laplace transform, as applied to $S^1$-equivariant symplectic cohomology, was first mentioned to the second author by Nicholas Sheridan. We thank Kai Hugtenburg and Claude Sabbah for very useful explanations of their work. Both authors were partially funded by the Simons Collaboration in Homological Mirror Symmetry (Simons Foundation award 652299). The first author additionally received partial funding from NSF grant DMS-2306204. The second author additionally received partial funding as a Simons Investigator (Simons Foundation award 256290); from NSF grant DMS-1904997; and during a visit to the Simons-Laufer Mathematical Sciences Institute, from NSF grant DMS-1928930.

\section{Algebraic differential equations}

The first aim of this section is to situate our statements in the context of the classical algebraic theory of linear differential equations, of which a rather selective account is given in Section \ref{subsec:classical} (see \cite{malgrange}; the authors have found \cite[Ch.~II and V]{sabbah-isomonodromic} to be a helpful introduction to the subject). This is followed by a short digression on Gauss-Manin connections in algebraic geometry, Section \ref{subsec:classical-gauss-manin}, which is not necessary for our purpose, but may provide some helpful intuition. Finally, in Section \ref{subsec:u-theory}, we do some preliminary work that's required for the noncommutative geometry version of the Gauss-Manin connection.

\subsection{Classical theory\label{subsec:classical}}

\subsubsection{Formal classification of singularities\label{subsec:connections}}
The simplest aspect of the theory of algebraic connections is the formal (in the sense of Laurent series) one. One considers 
\begin{equation} \label{eq:formal-connection}
\nabla_{\partial_q} = \partial_q + A_q, \;\; \text{ where }\;\; 
A_q = \sum_{k=m}^{\infty} A_k q^k \in \mathit{Mat}_r(\bC((q))).
\end{equation}
Such connections are acted on by formal gauge transformations $G_q \in \mathit{GL}_r(\bC((q)))$:
\begin{equation}
\tilde\nabla_{\partial_q} = G_q^{-1} \nabla_{\partial_q} G_q = \partial_q + \tilde{A}_q, \;\;\text{ where }
\tilde{A}_q = G_q^{-1}A_qG_q + G_q^{-1}(\partial_q G_q).
\end{equation}
Occasionally, we will also use the subgroup
\begin{equation} \label{eq:small-g}
I + q\mathit{Mat}_r(\bC[[q]]) \subset \mathit{GL}_r(\bC((q)))
\end{equation}
of gauge transformations which have no poles and constant term equal to the identity. The formal classification of connections is completely understood (for expositions see e.g.\ \cite[Sections II.5 and III.1]{malgrange}, \cite[Sections II.2 and II.5]{sabbah-isomonodromic}, or \cite{babbitt-varadarajan83}). We will only use a small part of that theory, covering the simplest three classes: nonsingular connections; ones with a regular singular point; and singularities of unramified exponential type.

\begin{definition}
One says that $\nabla_{\partial_q}$ is nonsingular if, by a formal gauge transformation, it can be brought into a form where $\tilde{A}_q \in \mathit{Mat}_r(\bC[[q]])$. 
\end{definition}

This means that any apparent pole can be transformed away. After that, one can formally integrate to trivialize the connection:
\begin{lemma}
Every connection without a pole, meaning with $m = 0$ in \eqref{eq:formal-connection}, is equivalent by a gauge transformation in \eqref{eq:small-g} to the trivial one, $\tilde{A}_q = 0$.
\end{lemma} 

\begin{definition}
$\nabla_{\partial_q}$ has a regular singular point if, by a formal gauge transformation, it can be transformed into $\tilde{A}_q \in q^{-1}\mathit{Mat}_r(\bC[[q]])$.
\end{definition}

This means that while the apparent pole order may be higher, it can be reduced to $\leq 1$. (In our terminology, ``regular singular point'' includes nonsingular connections.) One can then further simplify the situation using the following classical fact:

\begin{lemma} \label{th:regular-singular}
(i) Every connection with a simple pole, meaning with $m=-1$ in \eqref{eq:formal-connection}, is formally gauge equivalent to one of the form $\tilde{A}_q = q^{-1}\tilde{A}_{-1}$, $\tilde{A}_{-1} \in \mathit{Mat}_r(\bC)$. For such connections $\tilde{\nabla}_{\partial_q}$, the formal gauge equivalence class is completely determined by the conjugacy class of the monodromy 
\begin{equation} \label{eq:monodromy}
\exp(-2\pi i \tilde{A}_{-1}) \in \mathit{GL}_r(\bC).
\end{equation}
This monodromy has the same eigenvalues as $\exp(-2\pi i A_{-1})$, where $A_{-1}$ is the residue of the original connection. (More precisely, $\exp(-2 \pi i A_{-1})$ is contained in the closure of the conjugacy class of the monodromy.)

(ii) In the ``nonresonant case'' where no two eigenvalues of $A_{-1}$ differ by a nonzero integer, one can achieve the same outcome with the sharper condition that $\tilde{A}_{-1} = A_{-1}$, by using a gauge transformation in \eqref{eq:small-g}. In particular, the monodromy is conjugate to $\exp(-2\pi i A_{-1})$.
\end{lemma}

For higher order poles, one has the elementary splitting lemma (the first part is \cite[Chapter IV, Theorem 11.1]{wasow65}; the second part is proved in \cite[Lemma 2.13]{hugtenburg22}):

\begin{lemma} \label{th:splitting}
(i) Take a connection \eqref{eq:formal-connection}, with $m \leq - 2$. This is always equivalent, by a gauge transformation in \eqref{eq:small-g}, to some $\tilde{A}_q = \tilde{A}_m q^m + \tilde{A}_{m+1} q^{m+1} + \cdots$, where $\tilde{A}_m = A_m$, and the higher order terms preserve the splitting of $\bC^r$ into generalized $A_m$-eigenspaces.

(ii) 
For any eigenvalue $\lambda$ of $A_m$, the associated piece of $\tilde{A}_{m+1}$ is just the block diagonal part of the original $A_{m+1}$, with respect to the decomposition into generalized eigenspaces. In formulae, if $P_\lambda$ is the projection to the generalized $\lambda$-eigenspace, then $P_{\lambda} \tilde{A}_{m+1} P_\lambda = P_\lambda A_{m+1} P_\lambda$.
\end{lemma}

\begin{definition}
$\nabla_{\partial_q}$ has a singularity of unramified exponential type (see \cite[Section 2.c]{sabbah11} or \cite[Lecture 1]{sabbah16}) if it can be formally gauge transformed into a direct sum
\begin{equation} \label{eq:exponential-type}
\tilde{\nabla}_{\partial_q} = \bigoplus_{\lambda} \tilde{\nabla}_{\partial_q,\lambda}, \qquad
\tilde{\nabla}_{\partial_q,\lambda} = \tilde{\nabla}_{\partial_q,\lambda}^{\mathrm{reg}} + \lambda q^{-2} I,
\end{equation}
where the $\lambda$ are a finite set of complex numbers, and each $\tilde{\nabla}_{\partial_q,\lambda}^{\mathrm{reg}}$ has a regular singular point. The monodromies of the $\tilde{\nabla}_{\partial_q,\lambda}^{\mathrm{reg}}$ will be called the regularized formal monodromies of $\nabla_{\partial_q}$.
\end{definition}

(In our terminology, ``unramified exponential type'' includes connections with a regular singular point as the special case where there is only one summand, with $\lambda = 0$.) One can think of $\tilde{\nabla}_{\partial_q,\lambda}$ as the tensor product of the scalar (rank $1$) connection $\partial_q + \lambda q^{-2}$ and of $\tilde{\nabla}_{\partial_q,\lambda}^{\mathrm{reg}}$; the latter part can then be further simplified by applying Lemma \ref{th:regular-singular}. The decomposition \eqref{eq:exponential-type} is essentially unique, which means that the $\lambda$ and the formal gauge equivalence class of each $\tilde{\nabla}_{\partial_q,\lambda}^{\mathrm{reg}}$ are invariants of the original connection; this is a consequence of the Hukuhara-Turrittin-Levelt theorem (see e.g. \cite[Th\'eor\`eme 2.1]{malgrange83}). As a consequence, the conjugacy classes of the regularized formal monodromies are gauge invariants of the original connection $\nabla_{\partial_q}$.

\begin{lemma} \label{th:quadratic-pole}
(i) Suppose that $A_q$ has a quadratic pole, meaning that $m = -2$ in \eqref{eq:formal-connection}, and has unramified exponential type. Then the numbers $\lambda$ that appear in \eqref{eq:exponential-type} are the eigenvalues of $A_{-2}$, and the dimension of each summand is the multiplicity of that eigenvalue.

(ii) Suppose that $A_q$ has a quadratic pole, and that $A_{-2}$ is semisimple. Then the connection has unramified exponential type, and can be brought into the form \eqref{eq:exponential-type} by a gauge transformation without poles. Moreover, the regularized formal monodromy of each summand has the same eigenvalues as $\exp(-2\pi i A_{-1,\lambda})$, where $A_{-1,\lambda}$ are the block diagonal terms one gets when decomposing $A_{-1}$ according to the eigenspaces of $A_{-2}$.
\end{lemma}

Both statements are well-known (for (i) see e.g.\ see \cite[Corollary 3.2]{sanda-shamoto19}, and for the main part of (ii) see \cite[Example 2.6]{sabbah11}). The proof essentially uses only Lemma \ref{th:splitting}.

\begin{remark} \label{th:dual-connection}
The dual of a connection $\nabla_{\partial_q} = \partial_q + A_q$ is $\nabla^*_{\partial_q} = d-A_q^{\mathrm{tr}}$ (note we are not using complex conjugation here). By dualizing gauge transformations, one sees that $\nabla_{\partial_q}^*$ has unramified exponential type if and only if $\nabla_{\partial_q}$ does; the corresponding numbers are related by $\lambda^* = -\lambda$; and the regularized monodromies of $\nabla^*_{\partial_q}$ are conjugate to the inverse transposes of those of $\nabla_{\partial_q}$.
\end{remark}

The global picture is that we consider rational connections on the affine line. Take a nonzero $p \in \bC[q]$, and write $\bC[q,1/p] \subset \bC(q)$ for the subring of rational functions generated by $q$ and $1/p$. A rational connection is of the form
\begin{equation} \label{eq:rational-connection}
\nabla_{\partial_q} = \partial_q + A_q, \;\; \text{ where $A_q \in \mathit{Mat}_r(\bC[q,1/p])$.}
\end{equation}
By formally expanding in $(q - \sigma)$ for some $\sigma$, one can define $\sigma$ being a nonsingular point, or a regular singular point, and so on, of the connection (obviously, all $\sigma \in \bC$ where $p(\sigma) \neq 0$ will be nonsingular points). That also extends to $\sigma = \infty$, by expanding in $Q = 1/q$.

\subsubsection{Application to quantum connections}
Let's see how the general theory works out for quantum connections, starting with simple examples.

\begin{example}
(This is an entirely fictitious consideration, as there are no known monotone symplectic manifolds with that property.) Suppose that there are no Gromov-Witten contributions to the quantum connection, meaning that $[\omega_M] \ast^{(k)} x = 0$ for all $k>0$, in the notation from \eqref{eq:small-quantum-product}. Then, from \eqref{eq:original-quantum} or \eqref{eq:switch-q}, it's clear that $q = \infty$ is also a regular singular point. The monodromy around that point is the inverse of that around $q = 0$, hence has a unipotent Jordan block of size $\mathrm{dim}_{\bC}(M)+1$, which would saturate the bound of Theorem \ref{th:main-2}.
\end{example}

\begin{example} \label{th:p1}
The quantum connection on $M = \bC P^1$ is
\begin{equation} \label{eq:p1-connection-0}
\nabla_{\partial_q} = \partial_q + \begin{pmatrix} 0 & 2q \\ 2q^{-1} & 0 \end{pmatrix},
\end{equation}
where the factors of $2$ come from $[\omega_M] = c_1(M) = 2 \mathit{PD}([\mathit{point}])$. Written as in \eqref{eq:big-q-pole}, it becomes
\begin{equation} \label{eq:p1-connection}
\nabla_{\partial_Q}^{\mathrm{Gr}} = \partial_Q - Q^{-2} \begin{pmatrix} 0 & 2 \\ 2 & 0 \end{pmatrix} + Q^{-1} \begin{pmatrix} 0 & 0 \\ 0 & 1 \end{pmatrix}.
\end{equation}
Applying the algorithm underlying Lemma \ref{th:splitting}, one finds that it is gauge equivalent to
\begin{equation}
\tilde{\nabla}_{\partial_Q}^{\operatorname{gr}} = \partial_Q + Q^{-2} \begin{pmatrix} -2 & 0 \\ 0 & 2 \end{pmatrix}
+ Q^{-1} \begin{pmatrix} \half & 0 \\ 0 & \half \end{pmatrix}.
\end{equation}
In words, it is of unramified exponential type, and the regularized formal monodromies of both summands are equal to $-1$. This illustrates the fact that (as a consequence of the Stokes phenomenon) the regularized formal monodromies usually don't agree with the monodromy of the connection in the classical sense. 
\end{example}

\begin{example} \label{th:cubic-surface}
Take the quantum connection on the cubic surface \cite{difrancesco-itzykson94, crauder-miranda94}, restricted to the invariant subspace spanned by ($1$, $c_1$, $\mathit{PD}([point])$). One gets
\begin{equation}
\nabla_{\partial_Q}^{\mathrm{Gr}} = \partial_Q - Q^{-2} \begin{pmatrix} 
0 &108 & 252 \\
1 & 9 & 36 \\
0 & 3 & 0 \end{pmatrix} + Q^{-1} \begin{pmatrix} 0&0&0 \\ 0 & 1 & 0 \\ 0 & 0 & 2 \end{pmatrix}.
\end{equation}
It turns out (code can be found at \cite{seidel-code}) that this is gauge equivalent to
\begin{equation} \label{eq:cubic-transformed}
\partial_Q - Q^{-2} 
\begin{pmatrix}
-6 & 0 & 0 \\ 0 & -6 & 0 \\ 0 & 0 & 21
\end{pmatrix}
+ Q^{-1} \begin{pmatrix} 5/3 & -40/729 & 0 \\ 0 & 1/3 & 0 \\ 0 & 0 & 1 \end{pmatrix} + O(1).
\end{equation}
By Lemma \ref{th:quadratic-pole}, this connection is of unramified exponential type. The regularized formal monodromy has eigenvalues $1$ (for the $\lambda = 21$ part), respectively the nontrivial third roots of unity (for the $\lambda = -6$ part). To complete the discussion, note that there is a complementary invariant subspace, which is the orthogonal complement of $[\omega_M]$ inside $H^2(M;\bC)$. On that subspace, $[\omega_M] \ast$ acts by $-6$ times the identity; therefore, the corresponding part of $\nabla_{\partial_Q}^{\mathit{Gr}}$ has unramified exponential type and trivial regularized monodromy.
\end{example}

\begin{remark}
A twistor construction (\cite{reznikov93}, see also \cite{fine-panov10}) associates to each closed hyperbolic $6$-manifold a $12$-dimensional monotone symplectic (but not K\"{a}hler) manifold. An  argument of Hugtenburg (\cite{hugtenburg24}), based on quantum cohomology computations in \cite{evans11}, shows that the quantum connection for those manifolds has unramified exponential type. (The authors apologize for having wrongly stated the opposite in conference talks.)
\end{remark}

We want to record a few observations about the quantum connection in general, which shed light on the computations above. All of them are familiar, and they share a common ingredient, namely Poincar{\'e} duality.

\begin{lemma} \label{th:duality}
The quantum connection is gauge equivalent to its dual (see Remark \ref{th:dual-connection}) up to a parameter change $q \mapsto -q$.
\end{lemma}

\begin{proof}
Using Poincar{\'e} duality, one can write the dual of \eqref{eq:original-quantum} as $
\nabla_{\partial_q}^*x = \partial_q x - q^{-1}([\omega_M] \ast_q x)$. Pulling this back by $q \mapsto -q$ yields the connection 
\begin{equation} \label{eq:minus-connection}
\partial_q x - q^{-1}([\omega_M] \ast_{-q} x). 
\end{equation}
Now take an automorphism $\Phi$ of $H^*(M;\bC)$ which in degree $d$ is $\phi_d$ times the identity, where $\phi_{d+2} = -\phi_d$. This satisfies $\Phi([\omega_M] \ast_{-q} x) = -\Phi([\omega_M] \ast_q x)$, hence turns \eqref{eq:minus-connection} back into the quantum connection.
\end{proof}

\begin{corollary}
Suppose that $\nabla_{\partial_Q}^{\mathrm{Gr}}$ has unramified exponential type. Then, each of the regularized formal monodromies is conjugate to its inverse transpose. Hence, the spectrum of the monodromy must be invariant under $\lambda \leftrightarrow 1/\lambda$. 
\end{corollary}

\begin{lemma} \label{th:one-dimensional}
Suppose that $\lambda$ is an eigenvalue of $[\omega_M] \ast$, which is simple (the associated generalized eigenspace is one-dimensional). Then, the corresponding piece of the quantum connection, in the sense of Lemma \ref{th:splitting}, is of unramified exponential type, and has regularized monodromy $(-1)^{\mathrm{dim}_\bC(M)}$.
\end{lemma}

\begin{proof}
The first statement (unramified exponential type) is obvious from Lemma \ref{th:splitting}(i).
As for the second one, write $H_\lambda$ for the eigenspace in question. The subspace $H_\lambda$ is itself a $\bZ/2$-graded commutative with unit $e_\lambda$. If $H_\lambda$ is one-dimensional, we must have that $H_\lambda$ is generated by $e_\lambda$ as a vector space and in particular is contained in $H^{\operatorname{even}}(M;\bC)$. To simplify computations, assume that $\mathrm{Gr}|H^d(M;\bC)$ is $d$ times the identity for even $d$. Write $H_\lambda^\perp$ for the direct sum of all other generalized eigenspaces in $H^{\operatorname{even}}(M;\bC)$. The notation is explained by the fact that these two spaces are orthogonal for the intersection pairing. As a consequence, the intersection pairing is nonzero when restricted to $H_\lambda$. Next, note that $\mathrm{Gr} - \mathrm{dim}_{\bC}(M)\,I$ is skewadjoint for the intersection pairing. In particular,
\begin{equation} \label{eq:grxx}
\int_M \mathrm{Gr}(x)x = \mathrm{dim}_{\bC}(M) \int_M x^2.
\end{equation}
When applied to $x \in H_\lambda$, this tells us that the first block diagonal entry of $\mathrm{Gr}$ with respect to the decomposition $H_{\operatorname{even}}(M;\bC) = H_\lambda \oplus H_\lambda^\perp$ is equal to $\mathrm{dim}_{\bC}(M)$. One now applies Lemma \ref{th:splitting}(ii). 
\end{proof}

The following generalization of Lemma \ref{th:one-dimensional} comes from 
Dubrovin's work \cite[Lecture 3]{dubrovin99} (see \cite[Section 2.4]{galkin-golyshev-iritani16} for an exposition in a language close to the one here; ours is a simplified version of their statement).

\begin{lemma} \label{th:semisimple}
Suppose that the quantum cohomology ring $(H^*(M),\ast)$ is semisimple (the direct sum of copies of $\bC$). Then the quantum connection has unramified exponential type, and the regularized monodromy of each piece is $(-1)^{\mathrm{dim}_\bC(M)}$ times the identity.
\end{lemma}

\begin{proof}
Unramified exponential type is obvious from Lemma \ref{th:quadratic-pole}. In this case there can be no odd degree cohomology, so we can assume that $\mathrm{Gr}$ is as in the proof of Lemma \ref{th:one-dimensional}. Write $\langle x_1,\dots,x_n \rangle$ for $n$-pointed genus zero Gromov-Witten invariants, so that
\begin{equation}
\int_M (x_1 \ast x_2) x_3 = \langle x_1,x_2,x_3 \rangle.
\end{equation}
As a special case of the divisor axiom,
\begin{equation} \label{eq:divisor-axiom}
 2 \langle c_1(M), x_1,x_2,x_3 \rangle +
\int_M \mathrm{Gr}(x_1 \ast x_2) x_3 = \int_M (\mathrm{Gr}(x_1) \ast x_2 + x_1 \ast \mathrm{Gr}(x_2)) x_3.
\end{equation}
The five-point WDVV relation is
\begin{equation} \label{eq:5-wdvv}
\langle x_1 \ast x_2, x_3,x_4, x_5 \rangle + \langle x_1,x_2, x_3 \ast x_4, x_5 \rangle =
\langle x_1 \ast x_3, x_2,x_4,x_5 \rangle + \langle x_1, x_3, x_2 \ast x_4, x_5 \rangle.
\end{equation}
In our situation, there is a basis $\{e_i\}$ of idempotents for the quantum product. These satisfy $e_i \ast e_j = 0$ and hence also $\int_M e_ie_j = 0$ for $i \neq j$. Moreover, $c_1(M) \ast e_i = \lambda_i e_i$ for some $\lambda_i \in \bC$. Applying \eqref{eq:divisor-axiom} yields
\begin{equation} \label{eq:special-divisor-axiom}
\begin{aligned}
& \text{for $e_i \neq e_j$, } \;\;
2\langle c_1(M),e_i,e_j,e_j\rangle = 2 \langle c_1(M),e_i,e_j,e_j \rangle + \int_M \mathrm{Gr}(e_i \ast e_j) e_j = 
\\ & \quad =
\int_M (\mathrm{Gr}(e_i) \ast e_j + e_i \ast \mathrm{Gr}(e_j)) e_j = \int_M (\mathrm{Gr}(e_i) \ast e_j + e_i \ast \mathrm{Gr}(e_j) ) \ast e_j \\
& \quad = \int_M \mathrm{Gr}(e_i) \ast e_j = \int_M \mathrm{Gr}(e_i) e_j.
\end{aligned}
\end{equation}
From \eqref{eq:5-wdvv} for $(x_1,\dots,x_5) = (c_1(M), e_i,e_j,e_j,e_j)$ one gets
\begin{equation} \label{eq:special-wdvv}
\text{for $e_i \neq e_j$, } \;\;
\langle c_1(M), e_i, e_j, e_j \rangle = (\lambda_j - \lambda_i) \langle e_i,e_j,e_j,e_j \rangle.
\end{equation}
As a consequence of \eqref{eq:grxx}, \eqref{eq:special-divisor-axiom} and \eqref{eq:special-wdvv} one sees that
\begin{equation}
\text{if $\lambda_i = \lambda_j$, then} \;\; 
\int_M \mathrm{Gr}(e_i) e_j = \begin{cases} \mathrm{dim}_{\bC}(M) \int_M e_i^2 & i = j, \\
0 & i \neq j.
\end{cases}
\end{equation}
In other words, each block diagonal part of $\mathrm{Gr}$ with respect to the decomposition into eigenspaces is $\mathrm{dim}_{\bC}(M)$ times the identity. One applies Lemma \ref{th:quadratic-pole}(ii) to obtain the desired conclusion.
\end{proof}

\subsubsection{Differential operators} 
An alternative viewpoint on the formal classification of connections is provided by the cyclic vector lemma, which says that there is some $v \in \bC((q))^r$ such that $v,\nabla_{\partial_q} v,\dots,\nabla_{\partial_q}^{r-1} v$ form a basis. The connection gauge transformed to that basis has the form
\begin{equation}
\partial_q + \begin{pmatrix} &&&& -a_0 \\ 1 &&&& -a_1 \\ & 1 &&& -a_2 \\ && \cdots && \cdots \\ &&& 1 & -a_{r-1} \end{pmatrix}
\end{equation}
for $a_0,\dots,a_{r-1} \in \bC((q))$. In other words, setting $a_r = 1$, we have a relation $\sum_{i=0}^r a_i \nabla_{\partial_q}^i v = 0$. We correspondingly define an order $r$ scalar differential operator
\begin{equation} \label{eq:higher-order-operator}
P_q = \sum_{i=0}^r a_i \partial_q^i.
\end{equation}
%
Suppose that $w \in \bC((q))^r$ is a covariantly constant section of the dual connection, $(\partial_q - A_q^{\mathrm{tr}})w = 0$. Then, the function $f = v \cdot w \in \bC((q))$ solves $P_q f = 0$.

\begin{remark} 
For the quantum connection, each class $y \in H^*(M;\bC)$ gives rise to a covariantly constant section $\Psi_y$ of the (Poincar\'e) dual connection \cite[Section 28.1]{hori-katz-klemm-pandharipande-thomas-vafa-vakil-zaslow}. These fundamental solutions are generally multivalued, meaning that they have coefficients in $\bC[\log(q)]((q))$ (which makes sense since there's an obvious action of $\partial_q$ on that ring). Explicitly,
\begin{equation}
\int_M x \Psi_y = \int_M x e^{\log(q)[\omega_M]} y +
\sum_{\substack{d>0 \\ m \geq 0}} q^d \big\langle x, \psi^m(e^{\log(q) [\omega_M]} y) \big\rangle_d.
\end{equation}
Here, we use standard notation $\langle x_1, \psi^m(x_2) \rangle_d$ for two-pointed Gromov-Witten invariants with gravitational descendants, counting rational curves with first Chern number $d$. The covariant constancy property can be expressed as 
\begin{equation}
\partial_q \int_M x \Psi_y = \int_M (\nabla_{\partial_q} x) \Psi_y.
\end{equation}
As discussed above, if $v$ is a cyclic vector and $P_q$ the resulting operator, then
\begin{equation} \label{eq:p-constraint}
P_q\big(\textstyle\int_M v \Psi_y\big) = 0 \;\; \text{for all $y$.}
\end{equation}
\end{remark}

\begin{remark} \label{th:givental}
Instead of considering the entire quantum connection, let's just look at the $\bC((q))$-linear subspace spanned by $1 \in H^0(M)$ and its $\nabla_{\partial_q}$-derivatives, so that $v = 1$ is tautologically a cyclic vector. Let $P_q$ be the associated differential operator. The analogue of \eqref{eq:p-constraint} is 
\begin{equation} \label{eq:p-j}
P_q(J_y) = 0,
\end{equation}
where we write (using the string equation)
\begin{equation} \label{eq:j-function}
J_y = \int_M \Psi_y = \int_M e^{\log(q)[\omega_M]}y +
\sum_{\substack{d>0 \\ m > 0}} q^d \langle \psi^{m-1}(e^{\log(q) [\omega_M]} y)\rangle.
\end{equation}
In Givental's terminology (see e.g.\ \cite[Section 10.3]{cox-katz} for an exposition in the proper context, which is more general than the one here), a differential operator $\tilde{P}_q = \sum_i \tilde{a}_i \partial_q^i$, $\tilde{a}_i \in \bC((q))$, would be called a quantum differential operator if it satisfies the analogue of \eqref{eq:p-j}, meaning that
\begin{equation} \label{eq:tilde-p}
\tilde{P}_q(J_y) = 0 \text{ for all $y$} \;\; \Leftrightarrow \int_M (\sum_i \tilde{a}_i \nabla_{\partial_q}^i(1)) \cdot \Psi_y = 0 \text{ for all $y$} \;\; \Leftrightarrow \;\; \sum_i \tilde{a}_i \nabla_{\partial_q}^i(1) = 0.
\end{equation}
Because the definition of $P_q$ involves the lowest degree relation between the $\nabla_{\partial_q}^i(1)$, 
we then have $\tilde{P}_q = (\text{\it some differential operator})P_q$; in terms of the Weyl algebra to be introduced in a little while, the quantum differential operators are the left ideal generated by $P_q$. By an easy degree argument, the rightmost equation in \eqref{eq:tilde-p} implies the following: if the coefficients of $\tilde{a}_i$ only contain nonpositive powers of $q$, the relation
\begin{equation} \label{eq:quantum-relation}
\sum_{i=0}^r \text{($q^0$-term of $\tilde{a}_i$)} [\omega_M]^{\ast i} = 0
\end{equation}
holds in quantum cohomology, where as usual $\ast$ stands for the $q = 1$ specialization of the quantum product (compare \cite[Theorem 10.3.1]{cox-katz}).
\end{remark}

The Newton polygon of $P_q$ is constructed as follows. For each $0 \leq i \leq r$ such that $a_i \neq 0$, take the point $(x_i,y_i) = (i,\mathrm{val}_q(a_i)-i)$, where $\mathrm{val}_q(a_i)$ is the lowest power of $q$ which occurs. Consider the subsets 
\begin{equation}
\{0 \leq x \leq x_i,\; y \geq y_i\}. 
\end{equation}
The Newton polygon is the convex hull of the union of those subsets. The slopes of $P_q$ are the finite (meaning, excluding vertical sides) slopes of the sides of the Newton polygon. These slopes are nonnegative rational numbers, and are invariants of the original connection (which means, they are independent of the choice of cyclic vector $v$; see e.g.\ \cite[Theorem III.1.5]{malgrange}). The classical Fuchs regularity criterion is:

\begin{lemma}
$\nabla_{\partial_q}$ has a regular singular point if and only if $0$ is the only slope of $P_q$.
\end{lemma}

More generally, the slopes describe the pole orders of the pieces of the Hukuhara-Turrittin-Levelt decomposition of a connection \cite[Remark 3.55]{vanderput-singer}. As a special case, one has:

\begin{lemma} \label{th:unramified-slopes}
If $\nabla_{\partial_q}$ has a singularity of unramified exponential type, the slopes of $P_q$ are a subset of $\{0,1\}$.
\end{lemma}

\begin{example} 
Consider \cite[Remark 2.13]{katzarkov-kontsevich-pantev08}
\begin{equation}
A_q = \begin{pmatrix} 0 & -q^{-2} \\ -q^{-1} & -\half q^{-1} \end{pmatrix}.
\end{equation}
The cyclic vector $v = (1,0)$ yields a differential operator with slope $1/2$,
\begin{equation}
P_q = \partial_q^2 + {\textstyle\frac32} q^{-1}\partial_q - q^{-3}.
\end{equation}
Hence, this is not of unramified exponential type. One can also see this in a different way (following the argument in \cite{katzarkov-kontsevich-pantev08}). Namely, introducing a square root $q^{1/2}$ allows the gauge transformation 
\begin{equation} \label{eq:half-power}
G_q = \begin{pmatrix} 1 & 1 \\ -q^{1/2} & q^{1/2} \end{pmatrix}\;\; \Longrightarrow \;\;
\tilde{A}_q = q^{-3/2} \begin{pmatrix} 1 & 0 \\ 0 & -1 \end{pmatrix}.
\end{equation}
By the uniqueness part of the Hukuhara-Turrittin-Levelt theorem, this rules out unramified exponential type. Note that the theorem classifies connections up to multivalued gauge transformations, which means over $\bC((q,q^{1/2},q^{1/3},\dots))$; so the occurrence of $q^{1/2}$ is not a problem.
\end{example}

\begin{example}
Take the previous example, and add $q^{-2}$ times the identity matrix to $A_q$. For the same cyclic vector, one now gets
\begin{equation}
\textstyle P_q = \partial_q^2 + (\frac32 q^{-1} -2q^{-2} )\partial_q + (q^{-4} - \half q^{-3}),
\end{equation}
whose only slope is $1$. By \eqref{eq:half-power} this is gauge equivalent to $\tilde{A}_q = (q^{-2} + q^{-3/2})I$. Again by the uniqueness part of Hukuhara-Turrittin-Levelt, this means that the original connection cannot be of unramified exponential type.
%
%
This example illustrates that the converse to Lemma \ref{th:unramified-slopes} is false.
\end{example}

\subsubsection{Fourier-Laplace transform\label{subsubsec:fl}}
Let $W_q$ be the Weyl algebra of differential operators in one variable $q$, over $\bC$. This is generated by $q$ and $\partial_q$, with the relation
\begin{equation}
[\partial_q, q] = 1.
\end{equation}
Left $W_q$-modules are called $D$-modules. If $N_q$ is a $D$-module, and $p \in \bC[q]$ is nonzero,
\begin{equation} \label{eq:localised-h}
N_{q,1/p} \stackrel{\mathrm{def}}{=} \bC[q,1/p] \otimes_{\bC[q]} N_q
\end{equation}
inherits the structure of a $D$-module (by the obvious differentiation rule). A $D$-module $N_q$ is called holonomic if it is finitely generated and, for every $x \in N_q$, there is a nonzero $w \in W_q$ such that $wx = 0$. If $N_q$ is holonomic, then so are its localisations \eqref{eq:localised-h}.
%
Given a rational connection $\nabla_{\partial_q}$ as in \eqref{eq:rational-connection}, the space $ \bC[q,1/p]^r$, with $\partial_q$ acting by $\nabla_{\partial_q}$, becomes a holonomic $D$-module. The $W_q$-modules obtained in this way are precisely those on which $p$ acts invertibly, and which are finitely generated over $\bC[q,1/p]$ (freeness over $\bC[q,1/p]$ is then an automatic consequence). In the converse direction, one has (see e.g.\ \cite[p.~171]{sabbah-isomonodromic}):

\begin{lemma} \label{th:localize-d-module}
Let $N_q$ be a holonomic $D$-module. Then there is a nonzero $p \in \bC[q]$, such that $N_{q,1/p}$ is isomorphic to the $D$-module coming from a connection \eqref{eq:rational-connection}.
\end{lemma}

Take a minimal such $p$. One calls $\Sigma_q = p^{-1}(0)$ the set of singularities of the $D$-module $N_q$, and speaks of the associated rational connection $\nabla_{\partial_q}$ (which, in our formulation, is unique up to gauge equivalence over $\bC \setminus \Sigma_q$).
%

\begin{definition}
Let $t$ be another formal variable. We identify $W_t \iso W_q$ by setting
\begin{equation} \label{eq:fourier}
t = \partial_q, \;\; \partial_t = -q.
\end{equation}
Given a $W_t$-module $N_t$, the Fourier-Laplace transform $N_q$ is simply the same space considered as a module over $W_q$ via \eqref{eq:fourier}. This clearly preserves holonomicity.
\end{definition}

\begin{example} \label{th:uv} 
Take linear maps $U: \bC^r \rightarrow \bC^s$, $V: \bC^s \rightarrow \bC^r$, as well as some $\sigma \in \bC$. Define a holonomic $D$-module 
\begin{equation} \label{eq:explicit-n}
N_t = \bC[t]^r \oplus \bC[\partial_t]^s, 
\end{equation}
with the $W_t$-action
\begin{equation} \label{eq:model-action}
\begin{aligned} &
\partial_t( t^m x, \partial_t^n y) = (m t^{m-1} x, \partial_t^{n+1} y) + \begin{cases} 
(t^{m-1} VU x, 0) & m > 0, \\
(0, Ux) & m = 0;
\end{cases}
\\ &
t (t^m x, \partial_t^n y) = (t^{m+1}x, \sigma \partial_t^n y - n\partial_t^{n-1}y) + \begin{cases} (0,\partial_t^{n-1} UV y) & n>0, \\
(Vy,0) & n = 0.
\end{cases}
\end{aligned}
\end{equation}
Any element of the second summand of \eqref{eq:explicit-n} is mapped to the first summand by a sufficiently high power of $(t-\sigma)$; and $(t-\sigma)\partial_t (x,0) = (VUx,0)$. From that, one sees that $N_{t,1/(t-\sigma)}$ is the module associated to the connection
\begin{equation} \label{eq:model-action-4}
\nabla_{\partial_t}= \partial_t + (t-\sigma)^{-1}VU. 
\end{equation}
Apply the Fourier-Laplace transform, and write the summands in the opposite order, which means as $N_q = \bC[q]^s \oplus \bC[\partial_q]^r$. Then, $N_{q,1/q}$ is the module associated to $\nabla_{\partial_q} = \partial_q - q^{-1}UV + \sigma$. Changing coordinates to $Q = 1/q$ yields 
\begin{equation}
\nabla_{\partial_Q} = \partial_Q - Q^{-2}\sigma + Q^{-1}UV.
\end{equation}
\end{example}

\begin{remark}
Take $P_q \in W_q$ and a formal power series solution, meaning some $\Pi = \sum_{m=0}^\infty a_m q^m \in \bC[[q]]$ such that $P_q \Pi = 0$. Then
\begin{equation}
\hat{\Pi} \stackrel{\mathrm{def}}{=} \sum_{m=0}^\infty m!a_m t^{-m-1} \in t^{-1}\bC[[t^{-1}]]
\end{equation}
is a formal solution of the dual equation $P_t \hat{\Pi} = 0$, where $P_t$ corresponds to $P_q$ under \eqref{eq:fourier}. In the case of the quantum connection and the setting from Remark \ref{th:givental}, the quantum period is defined to be \eqref{eq:j-function} specialized to the class $y = p$ Poincar{\'e} dual to a point,
\begin{equation}
\Pi = J_p = 1 + \sum_{d \geq 2} q^d \langle \psi^{d-2}(p) \rangle_d.
\end{equation}
Then $\hat{\Pi}$ is essentially the regularized quantum period (see e.g.\ \cite{coates-corti-galkin-golyshev-kasprzyk13} for the terminology).
\end{remark}

The general relation between a holonomic $D$-module and its Fourier-Laplace transform was studied extensively in \cite{malgrange}. We will need only a special case of the results from \cite[Ch.~IX-XI]{malgrange}:

\begin{proposition} \label{th:fourier}
Let $N_t$ be a holonomic $W_t$-module, with singularities at $\Sigma_t \subset \bC$. Suppose that the associated rational connection $\nabla_{\partial_t}$, in the sense of Lemma \ref{th:localize-d-module}, has only regular singular points, including at $t = \infty$. Then, 

(i) The Fourier-Laplace tranform $N_q$ is nonsingular on $\bC^*$. If we look at the associated connection $\nabla_{\partial_q}$, then that has a regular singular point at $q = 0$, and a singularity of unramified exponential type at $q = \infty$.

(ii) To describe the latter singularity more precisely, let's change coordinates to $Q = 1/q$, and look at the normal form $\tilde{\nabla}_{\partial_Q}$  from \eqref{eq:exponential-type}. The numbers that appear there are precisely $\lambda = -\sigma$ for $\sigma \in \Sigma_t$. Moreover, for each $\sigma$ there are there are matrices $U_\sigma$, $V_\sigma$ such that the monodromy of $\nabla_{\partial_t}$ around $\sigma$ is $\exp(-2\pi i V_\sigma U_\sigma)$, and the corresponding regularized formal monodromy around $q = \infty$ (anticlockwise in $Q$, which means clockwise in $q$) is $\exp(-2\pi i U_\sigma V_\sigma)$.
\end{proposition}

Part (i) is stated for instance in \cite[p.~91]{sabbah-stokes} or \cite[Lemma 1.5]{sabbah10}. Part (ii) follows from the fact that the formal structure of $N_q$ at $q = \infty$ depends only on the local structure of $N_t$ near its singularities \cite[Section V.3]{sabbah-isomonodromic}. More precisely, every singular point of $N_t$ contributes a direct summand to the formal connection $\nabla_{\partial_q}$. The local structure of a holonomic $D$-module with a regular singularity at $t = \sigma$ is isomorphic to one of those from Example \ref{th:uv}, see \cite[p.~28]{malgrange}; and therefore, the computation carried out in that example actually proves the general result. 

Given Proposition \ref{th:fourier}(ii), some linear algebra \cite{flanders51} yields the following:

\begin{corollary} \label{th:quasiunipotent}
Let $M_{\sigma,\alpha}$ be the monodromy of $\nabla_{\partial_t}$ around some singular point $\sigma$, restricted to the generalized $\alpha$-eigenspace; and $N_{\sigma,\alpha}$ the regularized formal monodromy of the $\lambda = -\sigma$ summand of $\tilde{\nabla}_{\partial_Q}$, restricted correspondingly.

(i) If $\alpha \neq 1$, $M_{\sigma,\alpha}$ and $N_{\sigma,\alpha}$ are conjugate.

(ii) There is a bijective correspondence between Jordan blocks of $M_{\sigma,1}$ and of $N_{\sigma,1}$, under which sizes change by at most $1$. Here, we think of having an infinite reservoir of size $0$ Jordan blocks on each side, so that size $1$ Jordan blocks can appear and disappear under the correspondence.
\end{corollary}

\subsection{The Gauss-Manin system\label{subsec:classical-gauss-manin}}

\subsubsection{Definition\label{subsubsec:define-gauss-manin}}
Let $X$ be a smooth complex algebraic variety, and $W: X \rightarrow S = \mathbb{A}^1$ a proper nonconstant morphism. The Gauss-Manin system is defined \cite[Ch.~2, Section 15]{pham79} as the derived pushforward of the $D$-module sheaf $\scrO_X$ under $W$. To compute it, one factors $W$ as the composition of the embedding $(\mathit{id},W): X \rightarrow X \times S$ and the projection $X \times S \rightarrow S$. Concretely, the outcome is as follows \cite[p.~159]{pham79}. Take the complex of sheaves
\begin{equation} \label{eq:e-complex}
\begin{aligned}
& \scrE_q = \Omega^*_X[q], \\
& d_{\scrE_q}\theta = d\theta - q\, dW \wedge \theta,
\end{aligned}
\end{equation}
with the operation
\begin{equation} \label{eq:wtheta}
\nabla_{\partial_q}(\theta q^k) = k \theta\, q^{k-1} - W\theta\, q^k.
\end{equation}
The hypercohomology $E_q$ of $\scrE_q$ becomes a $W_q$-module in each degree. The Gauss-Manin system in degree $k$ is defined as the Fourier-Laplace transform of $E_q^{k+1}$. From the general theory (see e.g.\ \cite[Theorem 3.2.3]{hotta}), one gets:

\begin{lemma} \label{th:classical-holonomic}
In each degree, $E_q$ is a holonomic $W_q$-module.
\end{lemma}

%

The picture simplifies away from the singular fibers. Namely, let $S^{\mathrm{reg}} \subset S$ be the set of regular values, and $X^{\mathrm{reg}} = W^{-1}(S^{\mathrm{reg}})$ its preimage. The restriction of the Gauss-Manin system to $S^{\mathrm{reg}}$ can be computed directly using the pushforward by the submersion $W|X^{\mathrm{reg}}$, which yields the classical definition of the Gauss-Manin connection on the hypercohomology of the relative de Rham complex $(\Omega^*_{X^{\mathrm{reg}}/S^{\mathrm{reg}}}, d)$. 

\subsubsection{Application}
In the algebro-geometric context, the counterpart of Conjecture \ref{th:main-conjecture} is:

\begin{proposition} \label{th:classical}
(i) In each degree, restricting $E_q$ to $\bC^* = \{q \neq 0\}$ yields a connection which has a regular singularity at $q = 0$, and a singularity of unramified exponential type at $q = \infty$.

(ii) The regularized formal monodromies at $q = \infty$ are quasi-unipotent. 

(iii) For each regularized formal monodromy, the spectrum on $E_q$ (combining degrees) is invariant under $\lambda \leftrightarrow 1/\lambda$.
\end{proposition}

\begin{proof}
By the (Griffiths-Landman-Grothendieck) monodromy theorem, the Gauss-Manin connection for $W: X^{\mathrm{reg}} \rightarrow S^{\mathrm{reg}}$ has regular singularities, and quasi-unipotent monodromy around each of those singularities. Moreover, each monodromy endomorphism is compatible with the Poincar{\'e} duality pairing on the cohomology of the fibers, and therefore has the same eigenvalues as its inverse. This, together with Lemma \ref{th:classical-holonomic}, allows us to apply Proposition \ref{th:fourier}, which gives the desired result.
\end{proof}

Part of enumerative mirror symmetry, as formulated e.g.\ in \cite{givental95}, is that the Gauss-Manin system of the mirror superpotential should give the quantum connection. In situations where this is proved, one can use Proposition \ref{th:classical} to derive the corresponding case of Conjecture \ref{th:main-conjecture}. Strictly speaking, to use Proposition \ref{th:classical} as stated, one has to have a proper mirror superpotential, which can be achieved if the mirror is constructed relative to a smooth anticanonical divisor; however, the algebro-geometric considerations can be generalized beyond the proper case, under suitable assumptions on $W$. We had already mentioned one of the results that have been obtained in this way \cite{sanda-shamoto19}.

\begin{example}
The mirror of $\bC P^1$ is the superpotential $X = \bC^*$, $W(z) = z+z^{-1}$. 
%
The complex \eqref{eq:e-complex} reduces to
\begin{equation}
\xymatrix{
\bC[z^{\pm 1}]\mathit{dz} && q\bC[z^{\pm 1}]\mathit{dz} && q^2\bC[z^{\pm 1}]\mathit{dz} && \cdots
\\
&&
\ar@/^1pc/[ull]^-{d}
\ar[u]^-{q(z^{-2}-1)\mathit{dz}}
\bC[z^{\pm 1}] && q\bC[z^{\pm 1}] \ar@/^1pc/[ull]^-{d} \ar[u]^-{q(z^{-2}-1)\mathit{dz}} && \cdots \ar@/^1pc/[ull]
}
\end{equation}
We have drawn that so as to exhibit an increasing filtration. Using that filtration, one sees that the only nontrivial cohomology group $E_q^1 = H^1(\scrE_q)$, as a $W_q$-module, has generators and relations
%
\begin{equation} \label{eq:a0a1-generators}
\begin{aligned}
& a_0 = [z^{-1} dz], \;\; a_1 = [dz], \\
& q \partial_q a_0 = -2 qa_1, \\
& \partial_q q a_1 = -2qa_0.
\end{aligned}
\end{equation}
Our module contains $q$-torsion elements: $q[(1-z^{-2})\mathit{dz}] = q(\partial_q a_0 + 2a_1) = 0$. If we tensor with $\bC[q^{\pm 1}]$ and replace the generators with $(a_0,-qa_1)$, the relations \eqref{eq:a0a1-generators} yield the quantum connection \eqref{eq:p1-connection-0}. Take the Fourier-Laplace transform, and write the relations \eqref{eq:a0a1-generators} as
\begin{equation}
\begin{aligned}
& a_0 = \half (t^2-4) \partial_t a_1, \\
& 2\partial_t a_0 = -t \partial_t a_1.
\end{aligned}
\end{equation}
After tensoring with $\bC[t,(t^2-4)^{-1}]$, this becomes the connection with
\begin{equation}
A_t = \begin{pmatrix} -\frac{t}{t^2-4} & \frac{2}{t^2-4} \\ 0 & 0 \end{pmatrix}.
\end{equation}
As one would expect from its geometric origin, the monodromy around $t = \pm 2$ has eigenvalues $\{-1,1\}$ (it swaps the two sheets of $W$). Via Proposition \ref{th:fourier}(ii), that explains the occurrence of the eigenvalues $-1$ in the regularized formal monodromy computation from Example \ref{th:p1}. 
%
\end{example}

\begin{example} \label{th:cubic-surface-2}
The mirror to the cubic surface (relative to a smooth anticanonical divisor) is obtained from the extremal rational elliptic surface $X_{431}$, in the notation of \cite[Theorem 4.1]{miranda-persson86}, by removing the $I_3$ fiber. More specifically, the base should be parametrized so that the resulting $W: X \rightarrow \bC$ is a partial compactification of the superpotential given e.g.\ in \cite[Example 10]{coates-kasprzyk-prince19} with a constant $-6$ added (which brings the critical values into their expected position $\{-6,21\}$; for the origin of that constant, see \cite[Section 10]{givental95} and \cite[Appendix B]{sheridan16}). This $W$ has one nondegenerate singular point, and a more complicated singular fiber which consists of an $\tilde{E}_6$ configuration of rational spheres. The monodromy around the last-mentioned fiber has the eigenvalues $\exp(\pm 2\pi i/3)$ seen in Example \ref{th:cubic-surface}.
\end{example}

\subsection{The noncommutative theory\label{subsec:u-theory}}


\subsubsection{The $(u,q)$-Weyl algebra\label{subsubsec:uq-weyl}}
(Heisenberg-)Weyl algebras depending on an additional formal variable are of course what occurs in the original quantum mechanics context. In noncommutative geometry the setup is a little different, since the additional formal variable $u$ has degree $2$. As we will not actually be solving any differential equations in this part of our argument, we can work over an arbitrary field $\bK$.

Let $W_{q,u}$ be the graded $\bK[u]$-algebra generated by $q$ (of degree $2$) and $u\partial_q$ (of degree $0$), with
\begin{equation} \label{eq:q-relation}
[u\partial_q,q] = u.
\end{equation}
(In spite of the notation, there is no element $\partial_q$ in this algebra.) Let's look at graded (left) modules over $W_{q,u}$, understood to be graded $\bK[u]$-modules with  $u$-linear actions of $q$ and $\partial_q$, which satisfy \eqref{eq:q-relation}.
\begin{itemize} \itemsep.5em
\item
We say that a module is $u$-torsionfree if multiplication by $u$ is injective. 
\item
The $u$-adic completion of a module $A_{q,u}$ (in the graded sense, meaning that we complete in each degree separately) is
\begin{equation} \label{eq:completion}
\hat{A}_{q,u} \stackrel{\mathrm{def}}{=} \underleftarrow{\lim}_m \, A_{q,u}/u^mA_{q,u} =
\underleftarrow{\lim}_m \, \big((\bK[u]/u^m) \otimes_{\bK[u]} A_{q,u}\big).
\end{equation}
We call $A_{q,u}$ complete if the canonical map $A_{q,u} \rightarrow \hat{A}_{q,u}$ is an isomorphism. (This implies that no nonzero element of $A_{q,u}$ can be divisible by arbitrarily high powers of $u$.) Completions are always complete \cite[Lemma 00MC]{stacks}. 

\item
Setting $u = 0$ in $A_{q,u}$ yields a module over a graded two-variable polynomial ring, since the actions of $q$ and $u\partial_q$ then commute; we denote that by
\begin{equation} \label{eq:set-u-to-zero}
A_q \stackrel{\mathrm{def}}{=} A_{q,u}/uA_{q,u} = (\bK[u]/u) \otimes_{\bK[u]} A_{q,u}.
\end{equation}

\end{itemize}
Just like the notions above, the next Lemmas really concern only the $u$-module structure:

\begin{lemma} \label{th:completion}
Suppose that $A_{q,u}$ is $u$-torsionfree. Then so is its completion $\hat{A}_{q,u}$. Moreover, the $u = 0$ reduction of the completion agrees with that of the original module.
\end{lemma}

\begin{proof}
We have short exact sequences
\begin{equation}
0 \rightarrow A_{q,u}/u^m A_{q,u} \stackrel{u}{\longrightarrow} A_{q,u}/u^{m+1} A_{q,u}
\longrightarrow A_q \rightarrow 0.
\end{equation}
Passing to the limit (which is exact because the maps that decrease $m$ are surjective) yields
\begin{equation}
0 \rightarrow \hat{A}_{q,u} \stackrel{u}{\longrightarrow} \hat{A}_{q,u} \longrightarrow A_q \rightarrow 0.
\end{equation}
\end{proof}

\begin{lemma} \label{th:completion-2}
Let $A_{q,u}$ and $B_{q,u}$ be complete $u$-torsionfree modules, and $f_{q,u}: A_{q,u} \rightarrow B_{q,u}$ a $W_{q,u}$-linear map whose $u = 0$ reduction $f_q: A_q \rightarrow B_q$ is injective. Then $f_{q,u}$ is itself injective, and $C_{q,u} = \mathrm{coker}(f_{q,u})$ is a complete $u$-torsionfree module. Moreover, the $u= 0$ reduction $C_q$ is the cokernel of $f_q$.
\end{lemma}

\begin{proof}
Injectivity of $f_{q,u}$ is elementary: suppose that $f_{q,u}(a) = 0$ for some nonzero $a$. We can write $a = u^m a'$ with a maximal $m$, and then $f_{q,u}(a') = 0$, which after reduction to $u = 0$ shows that $a'$ must again be divisible by $u$, a contradiction. The fact that $C_{q,u}$ is $u$-torsionfree is also elementary: if $uc = 0$ in the cokernel, then a lift of $c$ to $b \in B_{q,u}$ would satisfy $ub = f_{q,u}(a)$ for some $a$. This means that the $u=0$ reduction of $a$ lies in the kernel of $f_q$, hence must be zero, and we can write $a = ua'$. But then $u(b-f_{q,u}(a')) = 0$, hence $b = f_{q,u}(a')$ and $c = 0$. At this point, we can tensor with $\bK[u]/u^m$ to get short exact sequences
\begin{equation}
0 \rightarrow A_{q,u}/u^m A_{q,u} \stackrel{f_{q,u}}{\longrightarrow} B_{q,u}/u^m B_{q,u} \longrightarrow C_{q,u}/u^m C_{q,u} \rightarrow 0,
\end{equation}
which for $m = 1$ shows the desired fact about $C_q$. Passing to the limit $m \rightarrow \infty$ yields
\begin{equation}
0 \rightarrow A_{q,u} \stackrel{f_{q,u}}{\longrightarrow} B_{q,u} \longrightarrow \hat{C}_{q,u} \rightarrow 0,
\end{equation}
which proves that $C_{q,u} = \hat{C}_{q,u}$. 
\end{proof}

At this point, we add the action of $q$ to the discussion.
\begin{itemize} \itemsep.5em
\item Given a complete module $A_{q,u}$, one can invert $q$ and then take the $u$-completion of that. The outcome of this process will be denoted (slightly clumsily, since the tensor product uses $q$ and the completion uses $u$) by
\begin{equation} \label{eq:invert-q}
A_{q^{\pm 1},u} \stackrel{\mathrm{def}}{=} \bK[q^{\pm 1}] \hat\otimes_{\bK[q]} A_{q,u}.
\end{equation}
Here, just like in the framework of classical $D$-modules, the action of $u\partial_q$ on the $q$-inverted module is given by $u\partial_q(q^k \otimes a) = uk q^{k-1} \otimes a + q^k \otimes u\partial_q a$. 

\item
We write (see Section \ref{subsec:conventions}(b) for the notation; completion is in the same sense as before)
\begin{equation} \label{eq:quotient-q}
q^{-1} A_{q^{-1},u} \stackrel{\mathrm{def}}{=} 
q^{-1}\bK[q^{-1}] \hat\otimes_{\bK[q]} A_{q,u}.
\end{equation}
\end{itemize}

\begin{lemma} \label{th:invert-q}
Suppose that $A_{q,u}$ is complete and $u$-torsionfree. Then \eqref{eq:invert-q} is $u$-torsionfree, and its $q=0$ reduction is related to that $A_{q,u}$ in the obvious way:
\begin{equation}
A_{q^{\pm 1}} \stackrel{\mathrm{def}}{=} A_{q^{\pm 1},u}/uA_{q^{\pm 1},u} \iso \bK[q^{\pm 1}] \otimes_{\bK[q]} A_q.
\end{equation}
\end{lemma}

\begin{proof}
Tensoring with $\bK[q^{\pm 1}]$ yields a short exact sequence
\begin{equation}
0 \rightarrow \bK[q^{\pm 1}] \otimes_{\bK[q]} A_{q,u} \stackrel{u}{\longrightarrow}
\bK[q^{\pm 1}] \otimes_{\bK[q]} A_{q,u} \longrightarrow
\bK[q^{\pm 1}] \otimes_{\bK[q]} A_q \rightarrow 0.
\end{equation}
In words, $\bK[q^{\pm 1}] \otimes_{\bK[q]} A_{q,u}$ is $u$-torsionfree, and its $u = 0$ reduction is $\bK[q^{\pm 1}] \otimes_{\bK[q]} A_q$. Lemma \ref{th:completion} does the rest.
\end{proof}

\begin{lemma} \label{th:invert-q-2}
Suppose that $A_{q,u}$ is complete and $u$-torsionfree; and that its $u=0$ reduction $A_q$ is $q$-torsionfree ($q$ acts injectively on it). Then \eqref{eq:quotient-q} is $u$-torsionfree; its $u = 0$ reduction is related to that of $A_{q,u}$ in the obvious way,
\begin{equation}
q^{-1}A_{q^{-1}} \stackrel{\mathrm{def}}{=} q^{-1}A_{q^{-1},u}/uq^{-1}A_{q^{-1},u}
\iso q^{-1}\bK[q^{-1}] \otimes_{\bK[q]} A_q;
\end{equation}
and it fits into a short exact sequence
\begin{equation} \label{eq:ses-q}
0 \rightarrow A_{q,u} \longrightarrow A_{q^{\pm 1},u} \longrightarrow q^{-1}A_{q^{-1},u} \rightarrow 0.
\end{equation}
\end{lemma}

\begin{proof}
Because $A_q$ is $q$-torsionfree, so is $A_{q,u}$ (if $x \in A_{q,u}$ satisfies $qx = 0$, then it must be a multiple of $u$; that argument can be iterated to prove that $x$ is arbitrarily often $u$-divisible, hence zero by completeness).
Therefore, the following diagram has exact columns:
\begin{equation} \label{eq:3x3}
\xymatrix{
& 0 \ar[d] & 0 \ar[d] & \ar[d] 0 & \\
0 \ar[r] & A_{q,u} \ar[d] \ar[r]^-{u} & \ar[d] A_{q,u} \ar[r] & \ar[d] \ar[r] A_q & 0 \\
0 \ar[r] & \bK[q^{\pm 1}] \otimes_{\bK[q]} A_{q,u} \ar[d] \ar[r]^-{u} & \bK[q^{\pm 1}] \otimes_{\bK[q]} A_{q,u} \ar[d] \ar[r] & \bK[q^{\pm 1}] \otimes_{\bK[q]} A_q \ar[d] \ar[r] & 0 \\
0 \ar[r] & q^{-1}\bK[q^{-1}] \otimes_{\bK[q]} A_{q,u} \ar[d] \ar[r]^-{u} & q^{-1}\bK[q^{-1}] \otimes_{\bK[q]} A_{q,u} \ar[r] \ar[d] & q^{-1}\bK[q^{-1}] \otimes_{\bK[q]} A_q \ar[d] \ar[r] & 0 \\
& 0 & 0 & 0 &
}
\end{equation}
The top two rows are exact, hence so is the bottom one. In words, $q^{-1}\bK[q^{-1}] \otimes_{\bK[q]} A_{q,u}$ is $u$-torsionfree, and its $u = 0$ reduction is $q^{-1}\bK[q^{-1}] \otimes_{\bK[q]} A_q$. We can then apply Lemma \ref{th:completion} to carry over those results to $q^{-1}A_{q^{-1}}$. Define \eqref{eq:ses-q} to be the $u$-completion of the left or middle column in \eqref{eq:3x3}. Lemma \ref{th:completion-2} tells us that the first map in \eqref{eq:ses-q} is injective. Moreover, concerning the map from its cokernel to $q^{-1}A_{q^{-1},u}$, we then know that its $u = 0$ reduction is an isomorphism, which implies that the map itself must be an isomorphism.
\end{proof}

\begin{example} \label{th:example-modules}
Let $V$ be a graded $\bK$-vector space. Consider the graded $\bK[q,u]$-module $A_{q,u} = V[[q,u]]$. Then
\begin{equation} \label{eq:v-q-u-0}
\bK[q^{\pm 1}] \otimes_{\bK[q]}\, A_{q,u} = \underrightarrow{\lim}_k\, q^{-k}A_{q,u} =
V[[u]]((q))
\end{equation}
is the space of Laurent series in $q$ with coefficients in $V[[u]]$ (each such Laurent series has a lower bound on the powers of $q$ that can appear). The completion is
\begin{equation} \label{eq:v-q-u}
A_{q^{\pm 1},u} = V((q))[[u]] = V((q))[[u/q]]
\end{equation}
which means power series in $u$, or equivalently $u/q$, whose coefficients are Laurent series in $q$. Concretely, an element in \eqref{eq:v-q-u} of degree $d$ is a series 
\begin{equation} \label{eq:qu-series}
\sum_{i \geq 0} \sum_{j \geq m_i} (u/q)^i q^j v_{ij}
= \sum_{i \geq 0} \sum_{k \geq m_i-i} u^i q^k v_{i,i+k}
\end{equation}
for some $m_i \in \bZ$, $v_{ij} \in V^{d-2j}$. In the special case where $V$ is bounded, it follows that $v_{ij} = 0$ once $|j|$ exceeds some $d$-dependent bound, and therefore: 
\begin{equation}
A_{q^{\pm 1},u} = (V[[u/q]])[q^{\pm 1}] \quad \text{if $V$ is bounded.}
\end{equation}

Similarly, we have
\begin{equation}
q^{-1}\bK[q^{-1}] \otimes_{\bK[q]} A_{q,u} = \underrightarrow{\lim}_k \,(q^{-k}A_{q,u}/A_{q,u}) = q^{-1} V[[u]][q^{-1}],
\end{equation}
which is the space of polynomials in $q^{-1}$ with zero constant term and coefficients in $V[[u]]$. Completion yields
\begin{equation}
q^{-1}A_{q^{-1},u} = (q^{-1}V[q^{-1}])[[u]],
\end{equation}
the space of power series in $u$ with coefficients in $q^{-1}V[q^{-1}]$. One can think of this as in \eqref{eq:qu-series} but where the entries are restricted to $j<i$ (respectively $k<0$).
\end{example}

Take the graded $\bK[u]$-algebra $W_{t,u}$, with generators $t$ of degree $0$ and $u\partial_t$ of degree $2$, such that
\begin{equation}
[u\partial_t,t] = u.
\end{equation}
There is an isomorphism $W_{t,u} \iso W_{q,u}$, 
\begin{equation} \label{eq:Fourier-Laplacenc}
t = u\partial_q, \;\; u\partial_t = -q.
\end{equation}
That gives rise to the notion of Fourier-Laplace transform appropriate to our context (completeness and $u$-torsionfreeness are independent of whether one thinks of a module as lying over $W_{t,u}$ or $W_{q,u}$). There is also a localisation process with respect to $t$, which is more flexible because that variable has degree zero. Namely, take a nonzero $p(t) \in \bK[t]$ and a complete $W_{t,u}$-module $A_{t,u}$, and form the $u$-adically completed tensor product
\begin{equation} \label{eq:quotient-t}
A_{t,1/p,u} = \bK[t,1/p] \hat\otimes_{\bK[t]} A_{t,u}.
\end{equation}
As before, this becomes a $W_{t,u}$-module by $u\partial_t( p(t)^k \otimes a) = u kp(t)^{k-1} p'(t) \otimes a + p(t)^k \otimes u\partial_t a$. In parallel with Lemma \ref{th:invert-q}, we have:

\begin{lemma} \label{th:invert-t}
Suppose that $A_{t,u}$ is complete and $u$-torsionfree. Then \eqref{eq:quotient-t} is $u$-torsionfree, and its $u = 0$ reduction is 
\begin{equation}
A_{t,1/p} \stackrel{\mathrm{def}} = A_{t,1/p,u}/uA_{t,1/p,u}
\iso \bK[t,1/p] \otimes_{\bK[t]} A_t.
\end{equation}
\end{lemma}

\subsubsection{The derived category\label{subsubsec:derived}}
At this point, we consider differential graded modules over $W_{q,u}$, which means $A_{q,u}$ that additionally come with a $W_{q,u}$-linear differential $d_{A_{q,u}}$.

\begin{definition}
Take the category whose objects are $u$-torsionfree and complete dg modules over $W_{q,u}$, and whose morphisms are chain maps. By passing to chain homotopy classes, we obtain the homotopy category $K(W_{q,u})$. Call a morphism $A_{q,u} \rightarrow B_{q,u}$ a filtered quasi-isomorphism if it induces a quasi-isomorphism $A_q \rightarrow B_q$. The category obtained from $K(W_{q,u})$ by inverting such quasi-isomorphisms is called the derived category $D(W_{q,u})$. 
\end{definition}

Both $K(W_{q,u})$ and $D(W_{q,u})$ are triangulated categories. This uses nothing more than the standard mapping cone construction. 

\begin{lemma} \label{th:filtered-les}
Take a sequence of two chain maps which compose to zero,
\begin{equation} \label{eq:abc}
A_{q,u} \longrightarrow B_{q,u} \longrightarrow C_{q,u}.
\end{equation}
Suppose that after setting $u = 0$, this becomes a short exact sequence. Then, in $D(W_{q,u})$ there is a canonical morphism that completes it to an exact triangle.
\end{lemma}

\begin{proof}
By assumption, the map $\mathit{Cone}(A_{q,u} \rightarrow B_{q,u}) \rightarrow C_{q,u}$ is a filtered quasi-isomorphism. One defines the desired morphism by combining the inverse of that map with the projection from the cone to $A_{q,u}[1]$.
\end{proof}

The localisation process \eqref{eq:invert-q} also applies to dg modules. Because of Lemma \ref{th:invert-q}, it preserves filtered quasi-isomorphisms, hence gives rise to an exact endofunctor of $D(W_{q,u})$. Under the extra assumption that $A_q$ is $q$-torsionfree, Lemmas \ref{th:invert-q-2} and \ref{th:filtered-les} say that we have an exact triangle
\begin{equation} \label{eq:aq-triangle}
\xymatrix{A_{q,u} \ar[r] & A_{q^{\pm 1},u} \ar[r] & q^{-1} A_{q^{-1},u}
\ar@/^1pc/[ll]^-{[1]}
}
\end{equation}
One can of course also think of $D(W_{q,u}) = D(W_{t,u})$ as a derived category of dg modules over $W_{t,u}$. Localisation in the sense of \eqref{eq:quotient-t} preserves filtered quasi-isomorphisms, hence gives rise to an exact endofunctor of the derived category.

In applications, geometrically defined chain maps are often strictly $u$-linear and $q$-linear, but commute with differentiation only up to homotopy. That can be remedied in the derived category:

\begin{lemma} \label{th:homotopy-dq}
Take two $u$-torsionfree complete dg modules $A_{q,u}$ and $B_{q,u}$. Suppose that we have $(q,u)$-linear maps
\begin{equation} \label{eq:fh}
\begin{aligned}
& f_{q,u}: A_{q,u} \longrightarrow B_{q,u} \text{ of degree $0$}, && d_{B_{q,u}} f_{q,u}(x) = f_{q,u}(d_{A_{q,u}} x), \\
& h_{q,u}: A_{q,u} \longrightarrow B_{q,u} \text{ of degree $-1$}, &&
d_{B_{q,u}} h_{q,u}(x) + h_{q,u}(d_{A_{q,u}} x) = u\partial_q f_{q,u}(x) - f_{q,u}(u\partial_q x).
\end{aligned}
\end{equation}
This gives rise to a canonical morphism $A_{q,u} \rightarrow B_{q,u}$ in $D(W_{q,u})$. Moreover, if $f_{q,u}$ is a filtered quasi-isomorphism, then the associated morphism is an isomorphism.
\end{lemma}

\begin{proof}
Equip the mapping cone $C_{q,u} = A_{q,u}[1] \oplus B_{q,u}$ of $f_{q,u}$ with the standard differential and $(q,u)$-action, and with the differentiation operation
\begin{equation}
u\partial_q(a,b) = (u\partial_q a, u\partial_q b + h_{q,u}(a)).
\end{equation}
This is an object of our category, and Lemma \ref{th:filtered-les} says that the inclusion and projection maps are part of a canonical exact triangle
\begin{equation} \label{eq:abc-2}
\xymatrix{
B_{q,u} \ar[r] & C_{q,u} \ar[r] & A_{q,u}[1] \ar@/^1pc/[ll]
}
\end{equation}
The boundary homomorphism of that triangle, meaning the left-pointing arrow in \eqref{eq:abc-2}, is the morphism we wanted to define. One can make this construction entirely explicit: namely, take the mapping cone of the shifted map $C_{q,u}[-1] \rightarrow A_{q,u}$ from \eqref{eq:abc-2}; this should more appropriately be called the mapping cylinder of $(f_{q,u},h_{q,u})$, and we denote it by $Z_{q,u}$. It comes with natural maps
\begin{equation} \label{eq:leftrightarrow}
A_{q,u} \longrightarrow Z_{q,u} \longleftarrow B_{q,u},
\end{equation}
of which the $\leftarrow$ is a filtered quasi-isomorphism. Inverting that gives rise to the desired morphism in the derived category. Finally, if $f_{q,u}$ is a filtered quasi-isomorphism, then so is the $\rightarrow$ in \eqref{eq:leftrightarrow}.
\end{proof}

Visibly, Lemma \ref{th:homotopy-dq} is asymmetric with respect to $(t,q)$. There is an analogue with the two variables switched, proved in the same way. (One might hope for more general and symmetric statements, possibly involving some $A_\infty$-version of $W_{q,u}$-module homomorphisms, but what we have will be sufficient for our purpose.)

\section{Noncommutative geometry\label{sec:algebra}}

The Getzler-Gauss-Manin connection \cite{getzler95} on periodic cyclic homology, and the theorem of Petrov-Vaintrob-Vologodsky  \cite{petrov-vaintrob-vologodsky18} concerning its behaviour for smooth and proper families, play a key role in our argument. The relevance of these results for the quantum connection depends on another piece of noncommutative geometry, which appears to be new; namely, a Fourier-Laplace duality for Getzler-Gauss-Manin connections (Theorem \ref{th:noncommutative-fourier-transform}), which resembles the construction of Gauss-Manin systems (see Section \ref{subsec:classical-gauss-manin}). In Section \ref{subsec:dga}, we explain that duality and its consequences, for differential graded algebras deformed by a superpotential (a central cocycle). After that, the purely expository Section \ref{subsec:ainfty} sets up the corresponding more general context for curved deformations of $A_\infty$-algebras. Section \ref{subsec:fiber} contains a technical argument used to reduce the $A_\infty$-situation to that of dga's. The outcome of these purely algebraic considerations is summarized in Corollary \ref{th:end-of-algebra}. The final Section \ref{subsec:linfty} has a separate purpose: it recalls some definitions from the world of $L_\infty$-algebras, which will be useful when discussing symplectic cohomology and its deformations.

\subsection{Differential graded algebras\label{subsec:dga}}

\subsubsection{The setup\label{subsubsec:dga-setup}}
In this section, $\scrA$ is a (nonzero) differential graded algebra over a field $\bK$. Denote the unit by $e_{\scrA} \in \scrA^0$, and write $\bar\scrA = \scrA/\bK e_{\scrA}$. Fix a central element
\begin{equation} \label{eq:w}
W \in \scrA^0,\;\; d_{\scrA}W = 0, \;\; Wa = aW \text{ for all $a \in \scrA$.} 
\end{equation}
There are two ways in which one can consider $W$ as part of the structure of $\scrA$.

\begin{itemize} \itemsep1em
\item Multiplication by $W$ makes $\scrA$ into a dga over a one-variable polynomial ring. It is not necessarily free as a module over that ring, but we can replace it by a better-behaved model, the $\bK[t]$-linear dga
\begin{equation} \label{eq:cofibrant-t}
\begin{aligned}
& \scrA_t = \scrA[t,\epsilon], \\
& d_{\scrA_t} = d_{\scrA} + (t-W)\partial_\epsilon.
\end{aligned}
\end{equation}
Here $(t,\epsilon)$ are formal variables of degree $0$ and $-1$, respectively; see Section \ref{subsec:conventions}(a). The inclusion $\scrA \hookrightarrow \scrA_t$ is a quasi-isomorphism, and the induced map on cohomology takes $[W]$ to $t[e_\scrA]$.

\item Let $q$ be a formal variable of degree $2$. We can regard $\scrA[[q]]$ as a differential graded algebra over $\bK[[q]]$ with a curvature term, namely $qW$ (this is a special case of the notion of curved $A_\infty$-deformation). Let's denote that curved dga by $\scrA_q$. All constructions involving $\scrA_q$ need to be carried out in $q$-adically completed versions.
\end{itemize}

\subsubsection{Bar resolutions\label{subsubsec:bar}}
Because we are dealing with differential graded algebras, all modules and bimodules are understood in the dg sense. A bimodule over $\scrA$ is the same as a module over $\scrA \otimes \scrA^{\operatorname{opp}}$. Recall the (normalized) bar resolution of the diagonal bimodule, 
\begin{equation} \label{eq:bar-diagonal}
\begin{aligned}
\bar{B}\scrA & = \scrA \otimes T(\bar\scrA[1]) \otimes \scrA \\
& = (\scrA \otimes \scrA) \oplus (\scrA \otimes \bar\scrA[1] \otimes \scrA) \oplus (\scrA \otimes \bar\scrA[1] \otimes \bar\scrA[1] \otimes \scrA) \oplus \cdots,
\end{aligned}
\end{equation}
with the obvious left and right $\scrA$-module structure, and differential
\begin{equation} \label{eq:bar-differential}
\begin{aligned}
& d_{\bar{B}\scrA}(a_0 (a_1|\dots|a_l) a_{l+1}) = 
\\[.5em] & \quad 
d_{\scrA}a_0 (a_1|\dots|a_l) a_{l+1}
- \sum_j (-1)^{|a_0|+\|a_1\|+\cdots+\|a_j\|} a_0 (a_l|\dots|d_{\scrA}a_{j+1}|\dots|a_l) a_{l+1} 
\\ & \quad
+ (-1)^{|a_0|+\|a_1\|+\cdots+\|a_l\|} a_0 (a_1|\dots|a_l) d_{\scrA} a_{l+1}
+ (-1)^{|a_0|} a_0a_1 (a_2|\dots|a_l) a_{l+1} \\[.5em] & \quad +
\sum_j (-1)^{|a_0|+\|a_1\|\cdots+ \|a_{j+1}\|} a_0(a_1|\dots|a_{j+1}a_{j+2}|\dots|a_l) a_{l+1} 
\\ & \quad - (-1)^{|a_0|+\|a_1\|+\cdots+\|a_{l-1}\|} a_0(a_1|\dots|a_{l-1}) a_la_{l+1}.
\end{aligned}
\end{equation}
The quasi-isomorphism $\bar{B}\scrA \rightarrow \scrA$ is given by $a_0()a_1 \mapsto a_0a_1$. (Our reason for working with normalized complexes will become clear later, see Example \ref{th:k}.) We will need two variants:
\begin{itemize}
\itemsep1em 
\item When talking about $\scrA_t$-bimodules, those are always assumed to be $t$-linear, which means that they are $\bK[t]$-linear modules over $\scrA_t \otimes_{\bK[t]} \scrA_t^{\operatorname{opp}}$. An example of this is the bar resolution $\bar{B}\scrA_t$, defined as before but with all tensor products taken over $\bK[t]$.

\item The bar resolution of $\scrA_q$ is similarly defined by working over $\bK[[q]]$, but with $q$-completion built in, and including an additional term in the differential which uses the curvature $qW$. Explicitly,
\begin{equation} \label{eq:daq-resolution}
\begin{aligned}
& \bar{B}\scrA_q = \bar{B}\scrA [[q]], \\
& d_{\bar{B}\scrA_q}(a_0(a_1|\dots|a_l)a_{l+1}) = \; d_{\bar{B}\scrA}(a_0(a_1|\dots|a_l)a_{l+1}) \\ &
 \qquad \qquad - q\sum_j (-1)^{|a_0|+\|a_1\|+\cdots+\|a_j\|} a_0(a_1|\dots|a_j|W|a_{j+1}|\dots|a_l)a_{l+1}.
\end{aligned}
\end{equation}
\end{itemize}

Pushing forward the bar resolution of $\scrA$ via $\scrA \hookrightarrow \scrA_t$ yields an $\scrA_t$-bimodule
\begin{equation} \label{eq:q-complex-0}
\begin{aligned}
& (\scrA_t \otimes_{\bK[t]} \scrA_t^{\operatorname{opp}}) \otimes_{(\scrA \otimes \scrA^{\operatorname{opp}})} \bar{B}\scrA 
=
\scrA_t \otimes_{\bK[t]} T_{\bK[t]}(\bar{\scrA}[1] \otimes \bK[t]) \otimes_{\bK[t]} \scrA_t \\ & \qquad
= (\scrA[\epsilon] \otimes T(\bar\scrA[1]) \otimes \scrA[\epsilon])[t].
\end{aligned}
\end{equation}
The differential is, by definition, derived from that on $B\scrA$ and $\scrA_t$. More precisely, given 
$a_0 \in \scrA[\epsilon]$, $a_1,\dots,a_l \in \bar\scrA$, and $a_{l+1} \in \scrA[\epsilon]$, one defines 
$d_{(\scrA_t \otimes_{\bK[t]} \scrA_t^{\operatorname{opp}}) \otimes_{(\scrA \otimes \scrA^{\operatorname{opp}})} \bar{B}\scrA}(a_0(a_1|\dots|a_l)a_{l+1})$ as in \eqref{eq:bar-differential}, but replacing $d_{\scrA}a_0$, $d_{\scrA}a_{l+1}$ by their $\scrA_t$-counterparts; this is then extended $t$-linearly. We next define an $\scrA_t$-bimodule with an additional action of $q$, which means a module over $(\scrA_t \otimes_{\bK[t]} \scrA_t^{\operatorname{opp}})[[q]]$, by combining \eqref{eq:q-complex-0} with a term resembling that from \eqref{eq:daq-resolution}:
\begin{equation} \label{eq:q-complex}
\begin{aligned}
& \bar{Q}\scrA_t = (\scrA[\epsilon] \otimes T(\bar\scrA[1]) \otimes \scrA[\epsilon])[t,q^{-1}], \\[.5em]
& d_{\bar{Q}\scrA_t}(a_0(a_1|\dots|a_l)a_{l+1} q^{-k}) = 
d_{(\scrA_t \otimes_{\bK[t]} \scrA_t^{\operatorname{opp}}) \otimes_{(\scrA \otimes \scrA^{\operatorname{opp}})} \bar{B}\scrA}(a_0(a_1|\dots|a_l)a_{l+1}) q^{-k}
\\ & \quad
+ \Big( -(-1)^{|a_0|} a_0\epsilon(a_1|\dots|a_l)a_{l+1} 
\\
& \quad - \sum_j (-1)^{|a_0|+\|a_1\|+\cdots+\|a_j\|} a_0(a_1|\dots|a_j|W|a_{j+1}|\dots|a_l)a_{l+1}
\\[-.5em]
& \quad
 + (-1)^{|a_0|+\|a_1\|+\cdots+\|a_l\|} a_0(a_1|\dots|a_l) \epsilon a_{l+1} \Big) q^{-k+1}
\\ &
\qquad \qquad \qquad \qquad \qquad \text{for } a_0 \in \scrA[\epsilon], \, a_1,\dots,a_l \in \scrA, \, a_{l+1} \in \scrA[\epsilon].
\end{aligned}
\end{equation}
For $k = 0$, the $q^{-k+1}$ term becomes zero; also, $\epsilon a = (-1)^{|a|} a\epsilon \in \scrA_t$; both are parts of our general conventions, see Section \ref{subsec:conventions}(a), (b). The shuffle map is the following map of $\scrA_t$-bimodules:
\begin{equation} \label{eq:shuffle}
\begin{aligned}
& \mathit{sh}: \bar{Q}\scrA_t \longrightarrow \bar{B}\scrA_t, \\
& \mathit{sh}(a_0(a_1|\cdots|a_l)a_{l+1}\, q^{-k}) = (-1)^k \!\!\! \sum_{0 \leq i_1 \leq \cdots \leq i_k \leq l}  \!\!\!
a_0(a_1|\dots|a_{i_1}|e_{\scrA}\epsilon|\dots|a_{i_2}|e_{\scrA}\epsilon|\dots|a_l)a_{l+1} 
\\
& \qquad \qquad \qquad \qquad \qquad \text{for } a_0 \in \scrA[\epsilon], \, a_1,\dots,a_l \in \scrA,\, a_{l+1} \in \scrA[\epsilon], \text{ and } k \geq 0.
\end{aligned}
\end{equation}

\begin{example} \label{th:k}
Write $\scrK = \bK$ for the coefficient field thought of as a dga, with $W = 0$ as the central element, so that $(\scrK_t,d_{\scrK_t}) = (\bK[t,\epsilon], d\epsilon = t)$. Then \eqref{eq:q-complex} simplifies to
\begin{equation}
\begin{aligned}
& \bar{Q}\scrK_t = (\bK[\epsilon] \otimes \bK[\epsilon])[t,q^{-1}], \\
& d_{\bar{Q}\scrK_t}(a_0()a_1 \, q^{-k}) = \partial_\epsilon a_0()a_1 \, q^{-k}t + 
(-1)^{|a_0|} a_0()\partial_\epsilon a_1 \, q^{-k} t \\
& \qquad \qquad
-(-1)^{|a_0|} a_0\epsilon()a_1 q^{-k+1} + (-1)^{|a_0|}a_0()\epsilon a_1 q^{-k+1}
\end{aligned}
\end{equation}
for $a_0,a_1 \in \bK[\epsilon]$; on the other hand,
\begin{equation}
\begin{aligned}
& \bar{B}\scrK_t = (\bK[\epsilon] \otimes T(\bK \epsilon[1]) \otimes \bK[\epsilon])[t], \\
& d_{\bar{B}\scrK_t}( a_0 (\underbrace{\epsilon|\dots|\epsilon}_k) a_1) = 
\partial_\epsilon a_0 (\underbrace{\epsilon|\dots|\epsilon}_k) a_1\, t 
+ (-1)^{|a_0|} a_0(\underbrace{\epsilon|\dots|\epsilon}_k) \partial_\epsilon a_1\, t
\\ & \qquad \qquad +
(-1)^{|a_0|} a_0\epsilon  (\underbrace{\epsilon|\dots|\epsilon}_{k-1}) a_1
- (-1)^{|a_0|} a_0  (\underbrace{\epsilon|\dots|\epsilon}_{k-1}) \epsilon a_1.
\end{aligned}
\end{equation}
The map \eqref{eq:shuffle} takes $a_0()a_1 q^{-k}$ to $(-1)^k a_0 (\underbrace{\epsilon|\cdots|\epsilon}_k) a_1$, and is an isomorphism of complexes.
\end{example}

\begin{example} \label{th:w-is-zero}
Take an arbitrary $\scrA$, but still assuming the central element to be $W = 0$. Then, $\scrA_t$ is the tensor product (over $\bK$) of $\scrA$ and the previously considered $\scrK_t$. Along similar lines,
\begin{equation}
\bar{Q}\scrA_t \iso \bar{B}\scrA \otimes \bar{Q}\scrK_t.
\end{equation}
If we then use the identification $\bar{Q}\scrK_t \iso \bar{B}\scrK_t$ from the previous example, the map \eqref{eq:shuffle} turns into a form of the classical shuffle product (see e.g.\ \cite[Section 9.4]{weibel}), and fits into a commutative diagram
\begin{equation}
\xymatrix{
\bar{B}\scrA \otimes \bar{B}\scrK_t \ar[rr]^-{\mathit{sh}} \ar[dr] && \bar{B}(\scrA \otimes \scrK_t) \ar[dl]
\\ &
\scrA \otimes \scrK_t.
}
\end{equation}
Here, the diagonal maps express the fact that both $\bar{B}\scrA \otimes \bar{B}\scrK_t$ and $\bar{B}(\scrA \otimes \scrK_t)$ are resolutions of $\scrA \otimes \scrK_t$. Since those maps are quasi-isomorphisms, so is the shuffle map (a well-known fact, of course).
\end{example}

\begin{prop} \label{th:q-resolution}
For any $(\scrA,W)$, the map \eqref{eq:shuffle} is a quasi-isomorphism.
\end{prop}

\begin{proof}
Let's say that the formal variables $\epsilon$, $t$, $q^{-1}$ all have weight $1$. Consider the increasing filtration of $\bar{Q}\scrA_t$ obtained by putting an upper bound of the weights. This filtration is bounded below and exhaustive, and on the associated graded space, the differential is precisely what one would obtain if $W = 0$. One can use the same weights for $t$ and $\epsilon$ to obtain a filtration on $\bar{B}\scrA_t$, and again, passing to the associated quotient has the same effect as setting $W = 0$. The shuffle map is homogeneous with respect to weights. Hence, the induced map on graded spaces is exactly what we looked at in Example \ref{th:w-is-zero}. That being a quasi-isomorphism, an obvious spectral sequence argument yields the desired result.
\end{proof}

The bar resolution $\bar{B}\scrA$ is homotopically flat ($K$-flat in the terminology of \cite{spaltenstein88, bernstein-lunts94}), meaning that if $P$ is any acyclic $\scrA$-bimodule, then $P \otimes_{\scrA \otimes \scrA^{\operatorname{opp}}} \bar{B}\scrA \iso P \otimes_{\bK} T(\bar{\scrA}[1])$ is an acyclic chain complex. The pushforward \eqref{eq:q-complex-0} inherits the corresponding property as an $\scrA_t$-bimodule, because by definition
\begin{equation} \label{eq:push-tensor}
\begin{aligned}
& P_t \otimes_{(\scrA_t \otimes_{\bK[t]} \scrA_t^{\operatorname{opp}})} 
(\scrA_t \otimes_{\bK[t]} \scrA_t^{\operatorname{opp}}) \otimes_{(\scrA \otimes \scrA^{\operatorname{opp}})} \bar{B}\scrA 
\iso P_t \otimes_{(\scrA \otimes \scrA^{\operatorname{opp}})} \bar{B}\scrA
\\ & \qquad\qquad\qquad
\text{for any $(\scrA_t \otimes_{\bK[t]} \scrA_t^{\operatorname{opp}})$-module $P_t$.}
\end{aligned}
\end{equation}
From that and a $q$-filtration argument, it follows that $\bar{Q}\scrA_t$ is homotopically flat. So is the bar resolution $\bar{B}\scrA_t$ (for the same reason as $\bar{B}\scrA$).

\begin{corollary} \label{th:k-flat}
Let $P_t$ be an $\scrA_t$-bimodule. Then, the map \eqref{eq:shuffle} induces a quasi-isomorphism 
\begin{equation} \label{eq:tensor-with-sh}
P_t \otimes_{(\scrA_t \otimes_{\bK[t]} \scrA_t^{\operatorname{opp}})} \bar{Q}\scrA_t \stackrel{\htp}{\longrightarrow} P_t \otimes_{(\scrA_t \otimes_{\bK[t]} \scrA_t^{\operatorname{opp}})} \bar{B}\scrA_t.
\end{equation}
\end{corollary}

\begin{proof}
By the general results of \cite[14.8]{drinfeld04} or \cite[\S 3.1.]{keller94}, we may choose a homotopically flat resolution 
\begin{equation} \label{eq:flat-resolution}
\tilde{P}_t \longrightarrow P_t. 
\end{equation}
 If we replace $P_t$ by $\tilde{P}_t$, then the map in \eqref{eq:tensor-with-sh} is a quasi-isomorphism, just because the shuffle map is a quasi-isomorphism and $\tilde{P}_t$ is homotopically flat. On the other hand, tensoring the map \eqref{eq:flat-resolution} with $\bar{Q}\scrA_t$ or with $\bar{B}\scrA_t$ yields a quasi-isomorphism, because $\bar{Q}\scrA_t$ and $\bar{B}\scrA_t$ are homotopically flat. The combination of those facts yields the desired result.
\end{proof}

Similar observations work for $\mathit{hom}$ instead of the tensor product. $\bar{B}\scrA$ is homotopically projective ($K$-projective), meaning that if $P$ is acyclic, then so is $\mathit{hom}_{\scrA \otimes \scrA^{\operatorname{opp}}}(\bar{B}\scrA, P) \iso \mathit{hom}_{\bK}(T(\bar{\scrA}[1]),P)$. One carries over this property to \eqref{eq:q-complex-0} by the adjunction
\begin{equation} \label{eq:hom-adjunction}
\mathit{hom}_{\scrA_t \otimes_{\bK[t]} \scrA_t^{\operatorname{opp}}}\big(
(\scrA_t \otimes_{\bK[t]} \scrA_t^{\operatorname{opp}}) \otimes_{(\scrA \otimes \scrA^{\operatorname{opp}})} \bar{B}\scrA, P_t\big) \iso
\mathit{hom}_{\scrA \otimes \scrA^{\operatorname{opp}}}(\bar{B}\scrA, P_t).
\end{equation}
As before, a further filtration argument then shows that $\bar{Q}\scrA_t$ is homotopically projective; so is $\bar{B}\scrA_t$, leading to the following analogue of Corollary \ref{th:k-flat}:

\begin{corollary} \label{th:k-projective}
Let $P_t$ be an $\scrA_t$-bimodule. Then, the map \eqref{eq:shuffle} induces a quasi-isomorphism 
\begin{equation} \label{eq:tensor-with-sh-2}
\mathit{hom}_{\scrA_t \otimes_{\bK[t]} \scrA_t^{\operatorname{opp}}}(\bar{B}\scrA_t,P_t)
\stackrel{\htp}{\longrightarrow}
\mathit{hom}_{\scrA_t \otimes_{\bK[t]} \scrA_t^{\operatorname{opp}}}(\bar{Q}\scrA_t,P_t).
\end{equation}
\end{corollary}

There is some duplication in the discussion above, because homotopically projective implies flat \cite[Corollary 10.12.4.4]{bernstein-lunts94}.

\subsubsection{Hochschild (co)homology}
The sources for the following exposition, as well as its generalization in Section \ref{subsubsec:hh-again} later on, are \cite{shklyarov13, getzler95, sheridan20}. Take the (normalized) standard chain complex underlying the Hochschild homology $\mathit{HH}_*(\scrA)$, namely
\begin{equation}
\bar{C}_*(\scrA) = \scrA \otimes_{(\scrA \otimes \scrA^{\operatorname{opp}})} \bar{B}\scrA.
\end{equation}
It is worth while spelling this out:
\begin{equation} \label{eq:hochschild-homology-complex}
\begin{aligned}
& \bar{C}_*(\scrA) = \scrA \otimes T(\bar{\scrA}[1]), \\
& d_{\bar{C}_*(\scrA)}(a_0(a_1|\dots|a_l)) \\
& \quad = d_{\scrA}a_0 (a_1|\dots|a_l) + \sum_j 
(-1)^{\|a_0\|+\|a_1\|+\cdots+\|a_j\|}
a_0(a_1|\dots|d_{\scrA} a_{j+1}|\dots|a_l) \\[-.5em] & \quad + (-1)^{|a_0|} a_0a_1 (a_2|\dots|a_l) 
 + \sum_j (-1)^{|a_0|+\|a_1\|+\cdots+ \|a_{j+1}\|} a_0(a_1|\dots|a_{j+1}a_{j+2}|\dots|a_l) \\ & \quad -
(-1)^{\|a_l\|(|a_0|+\|a_1\|+\cdots+\|a_{l-1}\|)} a_la_0 (a_1|\dots|a_{l-1}).
\end{aligned}
\end{equation}
Hochschild cohomology $\mathit{HH}^*(\scrA)$ is similarly computed by 
\begin{equation} \label{eq:c-upper}
\begin{aligned}
& 
\bar{C}^*(\scrA) = \mathit{hom}_{\scrA \otimes \scrA^{\operatorname{opp}}}(\bar{B}\scrA,\scrA)
= \mathit{hom}_{\bK}(T(\bar\scrA[1]),\scrA) = \prod_{l \geq 0} \mathit{hom}_{\bK}((\bar{\scrA}[1])^{\otimes l},\scrA), \\[-.5em]
&
(d_{\bar{C}^*(\scrA)}\phi)^l(a_1,\dots,a_l) \\ 
& \qquad = d_{\scrA} \phi^l(a_1,\dots,a_l) 
 + \sum_j (-1)^{\|a_1\|+\cdots+\|a_j\|+|\phi|} \phi^l(a_1,\dots,d_{\scrA}a_{j+1},\dots,a_l) 
\\[-.5em] & \qquad
- (-1)^{|\phi|\,\|a_1\|} a_1\phi^{l-1}(a_2,\dots,a_l) \\ & \qquad  - \sum_j 
(-1)^{\|a_1\|+\cdots+\|a_{j+1}\|+|\phi|}
 \phi^{l-1}(a_1,\dots,a_{j+1}a_{j+2},\dots,a_l)
\\ & \qquad
+ (-1)^{\|a_1\|+\cdots+\|a_{l-1}\| + |\phi|}\phi^{l-1}(a_1,\dots,a_{l-1})a_l.
\end{aligned}
\end{equation}
It is well-known that $\mathit{HH}^*(\scrA)$ is a graded commutative algebra. The product is induced by 
\begin{equation} \label{eq:smile}
(\phi \smile \psi)^l(a_1,\dots,a_l) = \sum_j 
(-1)^{|\psi|(\|a_1\|+\cdots+\|a_j\|)}
\phi^j(a_1,\dots,a_j)\psi^{l-j}(a_{j+1},\dots,a_l).
\end{equation}
Moreover, $\mathit{HH}_*(\scrA)$ is a module over $\mathit{HH}^*(\scrA)$, with the underlying chain level structure being
\begin{equation} \label{eq:iota-action}
\begin{aligned}
& \iota_\phi: \bar{C}_*(\scrA) \longrightarrow \bar{C}_{*+|\phi|}(\scrA), \\
& \iota_\phi( a_0(a_1|\dots|a_l)) \\ & \quad = \sum_j (-1)^{|\phi|+(|a_0|+\cdots+\|a_j\|)(\|a_{j+1}\|+\cdots+\|a_l\|)} 
\phi^{l-j}(a_{j+1},\dots,a_l)
a_0 (a_1|\dots|a_j).
\end{aligned}
\end{equation}
$\mathit{HH}^*(\scrA)$ also carries a Lie bracket of degree $-1$, induced by
\begin{equation} \label{eq:bracket}
\begin{aligned}
& [\phi,\psi]^l(a_1,\dots,a_l)  = \sum_{jk} 
(-1)^{\|\psi\|(\|a_1\|+\cdots+\|a_j\|)}
\phi^{l-k+1}(a_1,\dots,\psi^k(a_{j+1},\dots,a_{j+k}),\dots,a_l) \\[-.5em] & \qquad -
(-1)^{\|\phi\|(\|a_1\|+\cdots+\|a_j\|+\|\psi\|)}
\psi^{l-k+1}(a_1,\dots,\phi^{k}(a_{j+1},\dots,a_{j+k}),\dots,a_l).
\end{aligned}
\end{equation}
Finally, $\mathit{HH}_*(\scrA)$ is a Lie module with respect to that bracket, by
\begin{equation} \label{eq:lie-module}
\begin{aligned}
& L_\phi: \bar{C}_*(\scrA) \longrightarrow \bar{C}_{*+|\phi|-1}(\scrA), \\
& L_\phi( a_0(a_1|\dots|a_l)) 
\\ & \quad
= \sum_{jk} (-1)^{(\|a_0\|+\|a_1\|+\cdots+\|a_k\|)(\|a_{k+1}\|+\cdots+\|a_l\|)} \phi^{l-k+j+1}(a_{k+1},\dots,a_0,\dots,a_j)
\\[-1em] & \qquad \qquad \qquad \qquad \qquad \qquad \qquad \qquad \qquad \qquad \qquad \qquad
 (a_{j+1}|\dots|a_k) \\
& \quad + \sum_{jk} (-1)^{\|\phi\|(\|a_0\|+\|a_1\|+\cdots+\|a_j\|)}a_0 (a_1|\dots|\phi^k(a_{j+1},\dots,a_{j+k})|\dots|a_l).
\end{aligned}
\end{equation}
In our context, the central element $W$ is a Hochschild cocycle. The formulae above simplify to
\begin{equation} 
\begin{aligned}
& (W \smile \psi)^l(a_1,\dots,a_l) = W \psi^l(a_1,\dots,a_l), \\
& \iota_W( a_0(a_1|\dots|a_l)) = W a_0 (a_1|\dots|a_l), \\
& [W,\psi]^l(a_1,\dots,a_l) = -(-1)^{\|\psi\|(\|a_1\|+\cdots+\|a_j\|+1)}
\psi^{l+1}(a_1,\dots,a_j,W,\dots,a_l), \\
& L_W (a_0(a_1|\dots|a_l)) = \sum_j (-1)^{\|a_0\|+\cdots+\|a_j\|} a_0(a_1|\cdots|a_j|W|a_{j+1}|\dots|a_l).
\end{aligned}
\end{equation}

The same constructions apply to $\scrA_q$ (working over $\bK[[q]]$, and $q$-adically completing) and $\scrA_t$ (over $\bK[t]$). As a special case of Corollary \ref{th:k-flat}, tensoring the shuffle map \eqref{eq:shuffle} with the diagonal bimodule $\scrA_t$ yields a quasi-isomorphism, for which we use the same notation,
\begin{equation} \label{eq:shuffle-2}
\mathit{sh}: \overline{CQ}_*(\scrA_t) \stackrel{\mathrm{def}}{=} \scrA_t \otimes_{(\scrA_t \otimes_{\bK[t]} \scrA_t^{\operatorname{opp}})} \bar{Q}\scrA_t \longrightarrow \bar{C}_*(\scrA_t).
\end{equation}
Explicitly, the domain of \eqref{eq:shuffle-2} is the $(t,q)$-linear complex
\begin{equation} \label{eq:p-hochschild}
\begin{aligned}
& 
\overline{CQ}_*(\scrA_t) = (\scrA[\epsilon] \otimes T(\bar\scrA[1]))[t,q^{-1}], \\
& d_{\overline{CQ}_*(\scrA_t)} ( a_0 (a_1|\dots|a_l) q^{-k}) = d_{\bar{C}_*(\scrA)}(a_0(a_1|\dots|a_l)) q^{-k}
\\ & \qquad
- q^{-k+1} \sum_j (-1)^{|a_0| + \|a_1\| + \cdots + \|a_j\|} a_0(a_1|\dots|a_j|W|a_{j+1}|\dots|a_l)
\\[-.5em] & \qquad + (t e_\scrA-W) \partial_\epsilon a_0  (a_1|\cdots|a_l) q^{-k}
 \qquad \qquad \qquad
\text{for } a_0 \in \scrA[\epsilon], \, a_1,\dots,a_l \in \scrA.
\end{aligned}
\end{equation}
The formula for \eqref{eq:shuffle-2}, derived directly from \eqref{eq:shuffle}, is
\begin{equation} \label{eq:shuffle-2b}
\begin{aligned}
&
\mathit{sh}(a_0 (a_1|\dots|a_l)q^{-k}) = (-1)^k \sum_{0 \leq i_1 \leq \cdots \leq i_k \leq l} a_0(a_1|\dots|a_{i_1}|e_{\scrA}\epsilon|\dots|a_{i_k}|e_{\scrA}\epsilon|\dots|a_l)
\\[-.5em] & 
\qquad \qquad \qquad \qquad \qquad \qquad \qquad \qquad \qquad \qquad
\text{for } a_0 \in \scrA[\epsilon], \, a_1,\dots,a_l \in \scrA.
\end{aligned}
\end{equation}
By definition, $\partial_\epsilon: \scrA_t \rightarrow \scrA_t$ satisfies
\begin{equation}
\begin{aligned}
& \partial_\epsilon d_{\scrA_t} + d_{\scrA_t} \partial_\epsilon = 0, \\
& \partial_\epsilon (a_1a_2) = (\partial_\epsilon a_1)a_2 + (-1)^{|a_1|} a_1 (\partial_\epsilon a_2),
\end{aligned}
\end{equation} 
hence is a cocycle of degree $2$ in $\bar{C}^*(\scrA_t)$, with
\begin{equation} \label{eq:action-of-depsilon}
\begin{aligned}
& \iota_{\partial_\epsilon}(a_0(a_1|\cdots|a_l)) = (-1)^{(|a_0|+\|a_1\|+\cdots+\|a_{l-1}\|)\|a_l\|}(\partial_\epsilon a_l) a_0 (a_1|\dots|a_j), \\
& L_{\partial_\epsilon}(a_0(a_1|\cdots|a_l)) = (\partial_\epsilon a_0)(a_1|\dots|a_l) + 
\sum_j (-1)^{\|a_0\|+\cdots+\|a_j\|} a_0(a_1|\dots|\partial_\epsilon a_j|\dots|a_l).
\end{aligned}
\end{equation}
From the definitions, one immediately sees that \eqref{eq:shuffle-2} fits into a commutative diagram
\begin{equation} \label{eq:hochschild-mult-q}
\xymatrix{
\overline{CQ}_*(\scrA_t)
\ar[d]_-{q} \ar[r]^-{\mathit{sh}} & \bar{C}_*(\scrA_t) \ar[d]^-{-\iota_{\partial_\epsilon}} \\
\overline{CQ}_{*+2}(\scrA_t)
\ar[r]^-{\mathit{sh}} & \bar{C}_{*+2}(\scrA_t).
}
\end{equation}

The first two terms in the formula for the differential \eqref{eq:p-hochschild} are as in the Hochschild complex for $\scrA_q$. One can therefore separate out the parts with and without $\epsilon$, and write
\begin{equation} \label{eq:cq-cone}
\overline{CQ}_*(\scrA_t) \iso
\mathit{Cone}\Big( (\bK[q^{-1}] \otimes_{\bK[q]} \bar{C}_*(\scrA_q))[t] \xrightarrow{t-\iota_W} (\bK[q^{-1}] \otimes_{\bK[q]} \bar{C}_*(\scrA_q))[t] \Big).
\end{equation}
The map $t-\iota_W$  is injective, and its image is a subcomplex which is complementary to the $(\epsilon,t)$-constant subcomplex
\begin{equation} \label{eq:hook-0}
\bK[q^{-1}] \otimes_{\bK[q]} \bar{C}_*(\scrA_q)  \subset \overline{CQ}_*(\scrA_t).
\end{equation}
Hence, the inclusion of that subcomplex is a quasi-isomorphism. Moreover, because of the interpretation as mapping cone, it follows that this diagram is homotopy commutative:
\begin{equation}
\xymatrix{
\bK[q^{-1}] \otimes_{\bK[q]} \bar{C}_*(\scrA_q) \; \ar@{^{(}->}[r]
\ar[d]_{\iota_W} 
&
\overline{CQ}_*(\scrA_t) \ar[d]^-{t}
\\
\bK[q^{-1}] \otimes_{\bK[q]} \bar{C}_*(\scrA_q) \; \ar@{^{(}->}[r] &
\overline{CQ}_*(\scrA_t).
}
\end{equation}
We summarize the conclusions of our discussion:

\begin{prop} \label{th:hochschild-homology}
The map \eqref{eq:shuffle-2b}, restricted to \eqref{eq:hook-0}, defines a quasi-isomorphism
\begin{equation} \label{eq:shuffle-2c}
\mathit{sh}: \bK[q^{-1}] \otimes_{\bK[q]} \bar{C}_*(\scrA_q) \longrightarrow \bar{C}_*(\scrA_t).
\end{equation}
This quasi-isomorphism fits into a strictly commutative diagram
\begin{equation}
\xymatrix{
\bK[q^{-1}] \otimes_{\bK[q]} \bar{C}_*(\scrA_q) 
\ar[d]_-{q} \ar[r]^-{\mathit{sh}} & \bar{C}_*(\scrA_t) \ar[d]^-{-\iota_{\partial_\epsilon}} \\
\bK[q^{-1}] \otimes_{\bK[q]} \bar{C}_{*+2}(\scrA_q) 
\ar[r]^-{\mathit{sh}} & \bar{C}_{*+2}(\scrA_t),
}
\end{equation}
as well as into a homotopy commutative diagram
\begin{equation} \label{eq:minus-diagram}
\xymatrix{
\bK[q^{-1}] \otimes_{\bK[q]} \bar{C}_*(\scrA_q) 
\ar[d]_-{\iota_W} \ar[r]^-{\mathit{sh}} & \bar{C}_*(\scrA_t) \ar[d]^-{t} \\
\bK[q^{-1}] \otimes_{\bK[q]} \bar{C}_*(\scrA_q) 
\ar[r]^-{\mathit{sh}} & \bar{C}_*(\scrA_t).
}
\end{equation}
\end{prop}

\begin{remark}
One can specialize $\scrA_t$ to a single value $t = t_0 \in \bK$, which means considering the dga $\scrA_{t_0} = \bK[t]/(t-t_0) \otimes_{\bK[t]} \scrA_t$ over $\bK$. The corresponding specialization of \eqref{eq:shuffle-2}, \eqref{eq:cq-cone} then computes the Hochschild homology of $\scrA_{t_0}$:
\begin{equation} \label{eq:iota-cone}
\mathit{Cone}\Big(\bK[q^{-1}] \otimes_{\bK[q]} \bar{C}_*(\scrA_q)
\xrightarrow{t_0 - \iota_W} \bK[q^{-1}] \otimes_{\bK[q]} \bar{C}_*(\scrA_q) \Big)
\htp \bar{C}_*(\scrA_{t_0}).
\end{equation}
This imitates a familiar expression in algebraic geometry (compare Section \ref{subsec:classical-gauss-manin}). Namely, let $W: X \rightarrow \bC$ be a function on a smooth algebraic variety, and $t_0$ a regular value. Then, the sheaves $\Omega^*_{X_{t_0}}$ of differential forms on that smooth fiber have a resolution on $X$,
\begin{equation}
\mathit{Cone}\Big(   (\Omega^*_X[q^{-1}], -q\, dW \wedge)
\xrightarrow{t_0 - W} (\Omega^*_X[q^{-1}], - q\, dW \wedge) \Big) \htp \Omega^*_{X_{t_0}}.
\end{equation}
\end{remark}

There is a parallel story for Hochschild cohomology. From \eqref{eq:shuffle} and Corollary \ref{th:k-projective}, we get a quasi-isomorphism
\begin{equation} \label{eq:shuffle-3}
\mathit{sh}: \bar{C}^*(\scrA_t) \longrightarrow \overline{CQ}^*(\scrA_t) \stackrel{\mathrm{def}}{=} \mathit{hom}_{\scrA_t \otimes_{\bK[t]} \scrA_t^{\operatorname{opp}}}\big(
\bar{Q}\scrA_t, \scrA_t\big).
\end{equation}
Explicitly,
\begin{equation} \label{eq:cq-star}
\begin{aligned}
&
\overline{CQ}^*(\scrA_t) = \mathit{hom}_{\bK}(T(\bar{\scrA}[1]), \scrA[\epsilon])[t][[q]], 
\\
&
(d_{\overline{CQ}^*(\scrA_t)}\phi)^l(a_1,\dots,a_l) = (d_{\bar{C}^*(\scrA)}\phi)^l(a_1,\dots,a_l) \\
& \qquad
+ q\sum_j (-1)^{\|a_1\|+\cdots+\|a_j\|+|\phi|}\phi^{l+1}(a_1,\dots,a_j,W,a_{j+1},\dots,a_l) \\[-.5em]
& \qquad + (t-W) \partial_\epsilon \phi^l(a_1,\dots,a_l), \\
&
\mathit{sh}(\phi)^l(a_1,\dots,a_l) = \sum_{k, \, i_1 \leq \cdots \leq i_k} (-1)^k \phi^{l+k}(a_1,\dots,
a_{i_1},e_{\scrA}\epsilon,\dots,a_{i_k},e_{\scrA}\epsilon,\dots,a_l) q^k. 
\end{aligned}
\end{equation}
By definition,
\begin{equation}
(\partial_\epsilon \smile \phi)^l(a_1,\dots,a_l) = (-1)^{|\phi|\, \|a_1\|} (\partial_\epsilon a_1) \phi^{l-1}(a_2,\dots,a_l).
\end{equation}
The analogue of \eqref{eq:hochschild-mult-q} is the commutative diagram
\begin{equation} \label{eq:hochschild-mult-q-2}
\xymatrix{
\ar[d]_-{\partial_\epsilon \smile}
\bar{C}^*(\scrA_t) \ar[r]^-{\mathit{sh}} &
\overline{CQ}^*(\scrA) \ar[d]^-{-q}
\\
\bar{C}^{*+2}(\scrA_t) \ar[r]^-{\mathit{sh}} &
\overline{CQ}^{*+2}(\scrA).
}
\end{equation}
A look at the differential shows that we can write \eqref{eq:cq-star} as
\begin{equation} \label{eq:cq-upper}
\overline{CQ}^*(\scrA) = \mathit{Cone}\Big( (\bar{C}^*(\scrA_q)[t])^\wedge \xrightarrow{t - W \smile} (\bar{C}^*(\scrA_q)[t])^\wedge \Big),
\end{equation}
where $(\cdots)^\wedge$ is $q$-adic completion. The map that appears here is injective, and its image is complementary to the $t$-constant subcomplex (this remains true after $q$-adic completion; one can check it separately for each power of $q$, and then take the product of all of them). As a consequence, 
\begin{equation} \label{eq:hook}
\bar{C}^*(\scrA_q) \hookrightarrow \overline{CQ}^*(\scrA_t)
\end{equation}
is a quasi-isomorphism (in fact, it is a filtered quasi-isomorphism with respect to the $q$-filtration, and therefore a homotopy equivalence). We now state the counterpart of Proposition \ref{th:hochschild-homology}, which follows from this discussion. 

\begin{prop} \label{th:hochschild-cohomology}
The quasi-isomorphisms 
\begin{equation}
\bar{C}^*(\scrA_t) \stackrel{\mathit{sh}}{\longrightarrow} \overline{CQ}^*(\scrA_q)
\hookleftarrow \bar{C}^*(\scrA_q)
\end{equation}
fit into a commutative diagram
\begin{equation}
\xymatrix{
\bar{C}^*(\scrA_t) \ar[r]^-{\mathit{sh}} \ar[d]_-{\partial_\epsilon \smile}
& \overline{CQ}^*(\scrA_t) & \ar@{_{(}->}[l] \, \bar{C}^*(\scrA_q) \ar[d]^-{-q}
\\
\bar{C}^{*+2}(\scrA_t) \ar[r]^-{\mathit{sh}}
& \overline{CQ}^{*+2}(\scrA_t) & \ar@{_{(}->}[l] \, \bar{C}^{*+2}(\scrA_q),
}
\end{equation}
as well as into a homotopy commutative diagram
\begin{equation}
\xymatrix{
\bar{C}^*(\scrA_t) \ar[r]^-{\mathit{sh}} \ar[d]_-{t}
& \overline{CQ}^*(\scrA_t) & \ar@{_{(}->}[l] \, \bar{C}^*(\scrA_q) \ar[d]^-{W \smile}
\\
\bar{C}^*(\scrA_t) \ar[r]^-{\mathit{sh}}
& \overline{CQ}^*(\scrA_t) & \ar@{_{(}->}[l] \, \bar{C}^*(\scrA_q).
}
\end{equation}
\end{prop}

From the previous discussion, it follows that endomorphisms of $\bar{Q}\scrA_t$ compute the Hochschild cohomology of $\scrA_q$:
\begin{equation} \label{eq:endo-q-2}
H^*(\mathit{hom}_{(\scrA_t \otimes_{\bK[t]} \scrA_t^{\operatorname{opp}})}(\bar{Q}\scrA_t,\bar{Q}\scrA_t)) \iso \mathit{HH}^*(\scrA_q).
\end{equation}
Explicitly, this isomorphism is induced by a chain of quasi-isomorphisms
\begin{equation} \label{eq:endo-q}
\begin{aligned} &
\mathit{hom}_{(\scrA_t \otimes_{\bK[t]} \scrA_t^{\operatorname{opp}})}(\bar{Q}\scrA_t,\bar{Q}\scrA_t) \stackrel{\htp}{\longrightarrow}
\mathit{hom}_{(\scrA_t \otimes_{\bK[t]} \scrA_t^{\operatorname{opp}})}(\bar{Q}\scrA_t,\bar{B}\scrA_t)
\\ & \qquad
\stackrel{\htp}{\longrightarrow}
\mathit{hom}_{(\scrA_t \otimes_{\bK[t]} \scrA_t^{\operatorname{opp}})}(\bar{Q}\scrA_t,\scrA_t) =
\overline{CQ}^*(\scrA_t)
\stackrel{\htp}{\longleftarrow}
\bar{C}^*(\scrA_q);
\end{aligned}
\end{equation}
the first one comes from \eqref{eq:shuffle}, the second from the standard map $\bar{B}\scrA_t \rightarrow \scrA_t$, and the third one is \eqref{eq:hook}.

\begin{corollary} \label{th:endo-rings}
The isomorphism \eqref{eq:endo-q-2}sends $q^k$ times the identity (as an endomorphism of $\bar{Q}\scrA_t)$ to the element of the same name in $\mathit{HH}^*(\scrA_q)$. Moreover, it is an isomorphism of $\bK[t]$-modules, where $t$ acts on $\mathit{HH}^*(\scrA_q)$ by $[W] \smile$.
\end{corollary}

\begin{proof}
Under the first map in \eqref{eq:endo-q}, $q^k$ times the identity endomorphism of $\bar{Q}\scrA_t$ is mapped to $\mathit{sh}(q^k \cdot)$, seen as an element of $\mathit{hom}(\bar{Q}\scrA_t,\bar{B}\scrA_t)$. From there it is mapped to $q^k e_{\scrA} \in \overline{CQ}^*(\scrA_t)$, which is of course the image of $q^k e_{\scrA} \in \bar{C}^*(\scrA_q)$ under \eqref{eq:hook}. The first two maps in \eqref{eq:endo-q} are obviously $t$-linear, and under the last one, $t$ corresponds to $(W \smile)$ up to chain homotopy, for the same reason as in Proposition \ref{th:hochschild-cohomology}.
\end{proof}


\subsubsection{Smoothness\label{subsubsec:smoothness}}
Let $D(\scrA \otimes \scrA^{\operatorname{opp}})$ be the derived category of $\scrA$-bimodules. Recall that a bimodule $P$ is perfect if and only if it is a compact object of the derived category, meaning that $\mathit{Hom}_{D(\scrA \otimes \scrA^{\operatorname{opp}})}(P,\cdot)$ commutes with colimits. The dga $\scrA$ is called homologically smooth if the diagonal bimodule is perfect.
Similar concepts apply to $\scrA_t$, using the $\bK[t]$-linear category $D(\scrA_t \otimes_{\bK[t]} \scrA_t^{\operatorname{opp}})$. We will use our previous results to relate the smoothness of $\scrA$ and of $\scrA_t$ (this is inspired by arguments in \cite{preygel11}).

\begin{lemma} \label{th:0-perfect}
Suppose that $\scrA$ is smooth over $\bK$. Then \eqref{eq:q-complex-0} is a perfect $\scrA_t$-bimodule.
\end{lemma}

\begin{proof}
Both $\bar{B}\scrA$ and \eqref{eq:q-complex-0} are homotopically projective (see the discussion preceding Corollary \ref{th:k-projective}). Hence, the adjunction \eqref{eq:hom-adjunction} descends to the derived category,
\begin{equation}
\begin{aligned}
& \mathit{Hom}_{D(\scrA_t \otimes_{\bK[t]} \scrA_t^{\operatorname{opp}})}( (\scrA_t \otimes_{\bK[t]} \scrA_t^{\operatorname{opp}}) \otimes_{(\scrA \otimes \scrA^{\operatorname{opp}})} \bar{B}\scrA,P_t) \iso \mathit{Hom}_{D(\scrA \otimes \scrA^{\operatorname{opp}})}(\bar{B}\scrA,P_t)
\\ & \qquad \qquad \qquad \qquad \iso \mathit{Hom}_{D(\scrA \otimes \scrA^{\operatorname{opp}})}(\scrA,P_t).
\end{aligned}
\end{equation}
By assumption, the $\mathit{Hom}$ space on the right hand side commutes with colimits, hence the same is true on the left hand side.
\end{proof}

\begin{proposition} \label{th:smooth-1}
Suppose that: 
\begin{itemize} 
\item[(i)] $\scrA$ is smooth (over $\bK$);
\item[(ii)] $q^r [e_\scrA] \in \mathit{HH}^{2r}(\scrA_q)$ vanishes for some $r>0$.
\end{itemize}
Then $\scrA_t$ is smooth over $\bK[t]$.
\end{proposition}

\begin{proof}
Consider the increasing filtration $F_r \bar{Q}\scrA_t \subset \bar{Q}\scrA_t$, $r \geq 0$, given by restricting the powers of $q$ to be $\geq -r$. Assumption (i) and Lemma \ref{th:0-perfect} say that $F_0\bar{Q}\scrA_t$ is perfect. By definition, there are short exact sequences
\begin{equation}
0 \rightarrow F_0 \bar{Q}\scrA_t \longrightarrow F_r \bar{Q}\scrA_t \stackrel{q}{\longrightarrow} F_{r-1}\bar{Q}\scrA_t[2] \rightarrow 0.
\end{equation}
From the resulting exact triangles in the derived category, it follows (inductively) that each $F_r \bar{Q}\scrA_t$ is perfect. Along the same lines, we have short exact sequences
\begin{equation} \label{eq:qm}
0 \rightarrow F_r\bar{Q}\scrA_t \longrightarrow \bar{Q}\scrA_t \stackrel{q^r}{\longrightarrow} \bar{Q}\scrA_t[2r] \rightarrow 0.
\end{equation}
Assumption (ii), together with Corollary \ref{th:endo-rings}, tells us that there is some $r$ such that the $q^r$ map in \eqref{eq:qm} is nullhomotopic. In the derived category, this means that $\bar{Q}\scrA_t$ is isomorphic to a retract (direct summand) of $F_r\bar{Q}\scrA_t$. Since we already know that $F_r\bar{Q}\scrA_t$ is perfect, the result follows.
\end{proof}

The assumption (ii) in Proposition \ref{th:smooth-1} is rarely satisfied. What we actually need is a generalization, where one removes finitely many values of $t$. This amounts to taking a nonzero polynomial $p(t)$, and looking at the dga over $\bK[t,1/p]$ obtained by extending constants,
\begin{equation} \label{eq:remove-fibers}
\scrA_{t,1/p} = \bK[t,1/p] \otimes_{\bK[t]}\scrA_t.
\end{equation}

\begin{proposition} \label{th:smooth-2}
Suppose that we have $(\scrA,W)$ and a nonzero polynomial $p(t)$, such that:
\begin{itemize} 
\item[(i)] $\scrA$ is smooth (over $\bK$);
\item[(ii)] $q^r p([W]) \in \mathit{HH}^{2r}(\scrA_q)$ vanishes for some $r>0$. 
\end{itemize}
Then $\scrA_{t,1/p}$ is smooth over $\bK[t,1/p]$.
\end{proposition}

\begin{proof}
The argument from Lemma \ref{th:0-perfect} implies that $\bK[t,1/p] \otimes_{\bK[t]} F_0 \bar{Q}\scrA_t$ is perfect as a $\scrA_{t,1/p}$- bimodule. The proof of Proposition \ref{th:smooth-1} then carries over to show that $\bK[t,1/p] \otimes_{\bK[t]} F_r\bar{Q}\scrA_t$ is also perfect. Consider the sequence of bimodules over $\scrA_{t,1/p}$ obtained by tensoring \eqref{eq:qm} with $\bK[t,1/p]$, and then multiplying the second map with the invertible element $p(t) \in \bK[t,1/p]$:
\begin{equation} \label{eq:qm-2}
\begin{aligned} &
0 \rightarrow \bK[t,1/p] \otimes_{\bK[t]} F_r\bar{Q}\scrA_t
\longrightarrow \bK[t,1/p] \otimes_{\bK[t]} \bar{Q}\scrA_t \xrightarrow{q^r p(t)}\\
& \qquad \qquad \qquad \qquad \qquad \qquad \qquad \rightarrow \bK[t,1/p] \otimes_{\bK[t]} \bar{Q}\scrA_t[2r] \rightarrow 0.
\end{aligned}
\end{equation}
Assumption (ii), together with Corollary \ref{th:endo-rings}, implies that $q^r p(t)$ is nullhomotopic (in fact, it was already nullhomotopic before passing to $\bK[t,1/p]$-coefficients). The last step is as before.
\end{proof}

\begin{remark} It is a general fact that if $P_t$ is a perfect bimodule over $\scrA_t$, then $\bK[t,1/p] \otimes_{\bK[t]} P_t$ is a perfect bimodule over $\scrA_{t,1/p}$. We have given a direct argument for the specific case we need in Proposition \ref{th:smooth-2}.  \end{remark}

\begin{corollary} \label{th:degree-bound}
In the situation of Proposition \ref{th:smooth-2}, suppose additionally that $\mathit{HH}_*(\scrA)$ is concentrated in degrees $\leq d$, and that $\bK(t) \otimes_{\bK[t]} H^*(\scrA)$ is finite-dimensional over $\bK(t)$. Then $\mathit{HH}_*(\bK(t) \otimes_{\bK[t]} \scrA_t)$ is concentrated in degrees $[-d,d]$.
\end{corollary}

\begin{proof}
Take \eqref{eq:shuffle-2}, \eqref{eq:cq-cone} and tensor all groups involved with $\bK(t)$. The outcome is a quasi-isomorphism
\begin{equation} \label{eq:1-over-p-quasi}
\mathit{Cone}\Big( 
\bK[q^{-1}](t) \otimes_{\bK[q]} \bar{C}_*(\scrA_q) \xrightarrow{t-\iota_W} 
\bK[q^{-1}](t) \otimes_{\bK[q]} \bar{C}_*(\scrA_q) \Big) 
\htp \bar{C}_*(\bK(t) \otimes_{\bK[t]} \scrA_t).
\end{equation}
Using the (bounded above exhausting) filtration by powers of $q$, and the associated spectral sequence, one sees that under our assumption, the cohomology of $\bK[q^{-1}] \otimes_{\bK[q]} \bar{C}_*(\scrA_q)$ is concentrated in degrees $\leq d$. The same therefore holds for the left hand side of \eqref{eq:1-over-p-quasi}. This shows the upper bound in our statement; since $\bK(t) \otimes_{\bK[t]} \scrA_t$ is proper and smooth over $\bK(t)$, the lower bound follows from the existence of the nondegenerate Shklyarov pairing \cite{shklyarov07}.
\end{proof}

\subsubsection{Cyclic homology\label{subsubsec:cyclic-homology}}
By cyclic homology we mean what's usually referred to as negative cyclic homology. Let $u$ be a formal variable of degree $2$; see Section \ref{subsec:conventions}(e) for sign conventions. The cyclic complex is obtained by adding an extra $u$ term (the Connes operator) to the Hochschild differential:
\begin{equation} \label{eq:connes}
\begin{aligned}
& \overline{CC}_*(\scrA) = \bar{C}_*(\scrA)[[u]], \\
& d_{\overline{CC}_*(\scrA)}(a_0(a_1|\dots|a_l)) = d_{\bar{C}_*(\scrA)}(a_0(a_1|\dots|a_l)) \\
& \qquad
- u  \sum_j (-1)^{(\|a_0\|+\cdots+\|a_{j}\|)(\|a_{j+1}\|+\cdots+\|a_l\|)} e_{\scrA}(a_{j+1}|\dots|a_l|a_0|\dots|a_j).
\end{aligned}
\end{equation}
As before, Hochschild cochains $\phi$ act on $\overline{CC}_*(\scrA)$, $u$-linearly, in two different ways. The Lie action is given by the same formula \eqref{eq:lie-module}, and we continue to write it as $L_\phi$. The module action acquires an extra term, and we correspondingly change notation:
\begin{equation} \label{eq:big-iota-action}
\begin{aligned} &
I_\phi( a_0(a_1|\dots|a_l) ) = \iota_\phi(a_0(a_1|\dots|a_l)) 
\\ & \quad
-u
\sum_{ijm} (-1)^{(\|a_0\|+\cdots+\|a_{j}\|)(\|a_{j+1}\|+\cdots+\|a_l\|) + \|\phi\|( \|a_{j+1}\|+\cdots+\|a_i\|)} \\[-1em] & \qquad \qquad \qquad \qquad
e_{\scrA} (a_{j+1}|\dots|\phi^m(a_{i+1},\dots,a_{i+m})|\dots|a_0|\dots|a_j).
\end{aligned}
\end{equation}
(In the new term, $a_0$ must lie to the right of $\phi$.) These operations satisfy the Cartan formula
\begin{equation} \label{eq:cartan}
d_{\overline{CC}_*(\scrA)} I_\phi - (-1)^{|\phi|} I_\phi d_{\overline{CC}_*(\scrA)} - I_{d_{\bar{C}^*(\scrA)}\phi} = uL_\phi.
\end{equation}
This formula underlies the general definition of the Getzler-Gauss-Manin connection. For now, we only need two rather special cases:
\begin{itemize} \itemsep1em
\item 
Take $\overline{CC}_*(\scrA_t)$, defined as before but working over $\bK[t]$. The naive operation of $t$-differentiation, applied to cyclic cochains, satisfies
\begin{equation} \label{eq:dt-d}
\partial_t d_{\overline{CC}_*(\scrA_t)} - d_{\overline{CC}_*(\scrA_t)} \partial_t = L_{\partial_\epsilon}.
\end{equation}
The expression $\partial_\epsilon$ appears here because it is the $t$-derivative of $d_{\scrA_t}$, the rest of the dga structure being $t$-independent. The $t$-connection is the $u$-linear map
\begin{equation}
\begin{aligned}
& \nabla_{u\partial_t}: \overline{CC}_*(\scrA_t) \longrightarrow \overline{CC}_{*+2}(\scrA_t), \\
& \nabla_{u\partial_t} = u\partial_t + I_{\partial_\epsilon}.
\end{aligned}
\end{equation}
This is a chain map, because of \eqref{eq:dt-d} and  \eqref{eq:cartan}. Spelling out the definition, see \eqref{eq:action-of-depsilon}, yields
\begin{equation}
\begin{aligned}
& \nabla_{u\partial_t}( a_0 (a_1|\dots|a_l)t^k) \\ & \quad = uk \, a_0(a_1|\dots|a_l) t^{k-1}
 + (-1)^{(|a_0|+\|a_1\|+\cdots+\|a_{l-1}\|)\|a_l\|} (\partial_\epsilon a_l) a_0 (a_1|\dots|a_{l-1}) t^k
\\ & \quad
 - u\sum_{ij}
(-1)^{(\|a_0\|+\|a_1\|+\cdots+\|a_j\|)(\|a_{j+1}\|+\cdots+\|a_l\|) + (\|a_{j+1}\| + \cdots + \|a_i\|)}
\\[-1em] & \qquad \qquad \quad \qquad
 e_{\scrA} (a_{j+1}|\dots|\partial_\epsilon a_{i+1}| \dots|a_0|\dots|a_j) t^k  
\qquad \qquad \qquad \text{for } a_0, \dots, a_l \in \scrA[\epsilon].
\end{aligned}
\end{equation}

\item
On the other hand, consider the cyclic complex of the curved dga $\scrA_q$. This is understood to be complete with respect to both $q$ and $u$, hence is $\overline{CC}_*(\scrA_q) = \bar{C}_*(\scrA)[[q,u]]$. The differential includes $W$ in the same way as in \eqref{eq:daq-resolution}, and has the same Connes operator term as in \eqref{eq:connes}. The analogue of \eqref{eq:dt-d} is
\begin{equation} \label{eq:dq-d}
\partial_q d_{\overline{CC}_*(\scrA_q)} - d_{\overline{CC}_*(\scrA_q)} \partial_q = L_W,
\end{equation}
and the associated $q$-connection is correspondingly
\begin{equation} \label{eq:q-ggm}
\begin{aligned}
& \nabla_{u\partial_q}: \overline{CC}_*(\scrA_q) \longrightarrow \overline{CC}_*(\scrA_q), \\
& \nabla_{u\partial_q} = u\partial_q + I_W;
\end{aligned}
\end{equation}
or explicitly
\begin{equation}
\begin{aligned}
& \nabla_{u\partial_q}( a_0(a_1|\dots|a_l) q^k) = uk \,a_0(a_1|\dots|a_l)q^{k-1}  + 
Wa_0(a_1|\dots|a_l)q^k \\ & \quad - u\sum_{ij} 
(-1)^{(\|a_0\|+\cdots+\|a_j\|)(\|a_{j+1}\|+\cdots+\|a_l\|) + (\|a_{j+1}\| + \cdots + \|a_i\|)}
\\[-1em] & \qquad \qquad \quad \qquad
e_{\scrA}(a_{j+1}|\dots|a_i|W|\dots|a_0|\dots|a_j)q^k
 \qquad \qquad \qquad
\text{for }a_0,\dots,a_l \in \scrA.
\end{aligned}
\end{equation}
There is also a version of cyclic homology with negative powers of $q$, more precisely $q^{-1}\bK[q^{-1}] \hat\otimes_{\bK[q]} \overline{CC}_*(\scrA_q)$ (see Section \ref{subsec:conventions}(b) for notation). Because $\partial_q$ acts on $q^{-1}\bK[q^{-1}]$, this version inherits a connection which we also denote by $\nabla_{u\partial_q}$.
\end{itemize}

Take the map \eqref{eq:shuffle-2c}, shift the powers of $q$ involved by $1$, and extend it $u$-linearly. We denote the outcome, which is easily seen to be a chain map, by
\begin{equation} \label{eq:cyclic-shuffle}
\begin{aligned}
& \mathit{SH}: q^{-1}\bK[q^{-1}] \hat\otimes_{\bK[q]} \overline{CC}_*(\scrA_q)  \longrightarrow \overline{CC}_{*+2}(\scrA_t), \\
& \mathit{SH}(a_0(a_1|\dots|a_l) q^{-k-1}) = \mathit{sh}(a_0(a_1|\dots|a_l)q^{-k}).
\end{aligned}
\end{equation}

\begin{theorem} \label{th:noncommutative-fourier-transform}
The map \eqref{eq:cyclic-shuffle} is a filtered quasi-isomorphism (with respect to the $u$-filtration). It fits into a strictly commutative diagram
\begin{equation} \label{eq:gm-compatible-1}
\xymatrix{
 q^{-1}\bK[q^{-1}] \hat\otimes_{\bK[q]}\overline{CC}_*(\scrA_q)  \ar[r]^-{\mathit{SH}} \ar[d]_-{q} &
\overline{CC}_{*+2}(\scrA_t) \ar[d]^-{-\nabla_{u\partial_t}}
\\
 q^{-1}\bK[q^{-1}] \hat\otimes_{\bK[q]}\overline{CC}_{*+2}(\scrA_q)  \ar[r]^-{\mathit{SH}} &
\overline{CC}_{*+4}(\scrA_t), 
}
\end{equation}
as well as into a homotopy commutative diagram 
\begin{equation} \label{eq:gm-compatible-2}
\xymatrix{
 q^{-1}\bK[q^{-1}] \hat\otimes_{\bK[q]}\overline{CC}_*(\scrA_q)  \ar[r]^-{\mathit{SH}} \ar[d]_-{\nabla_{u\partial_q}} &
\overline{CC}_{*+2}(\scrA_t) \ar[d]^-{t}
\\
 q^{-1}\bK[q^{-1}] \hat\otimes_{\bK[q]} \overline{CC}_*(\scrA_q) \ar[r]^-{\mathit{SH}} &
\overline{CC}_{*+2}(\scrA_t). 
}
\end{equation}
\end{theorem}

\begin{proof}
The quasi-isomorphism statement follows from that in Proposition \ref{th:hochschild-homology}. The commutativity of \eqref{eq:gm-compatible-1} is elementary, especially so because elements in the image of \eqref{eq:cyclic-shuffle} are constant in $t$. As for \eqref{eq:gm-compatible-2}, we have
\begin{equation} 
\begin{aligned}
& \mathit{SH} \,\partial_q = -L_{\epsilon e_\scrA} \mathit{SH}, \\
& \mathit{SH} I_W = I_W \mathit{SH},
\end{aligned}
\end{equation}
and therefore, using \eqref{eq:cartan},
\begin{equation}
\begin{aligned}
& \mathit{SH}\, (\nabla_{u\partial_q}) = (-uL_{\epsilon e_{\scrA}} + I_W) \mathit{SH} \\
& \quad = (-d_{\overline{CC}_*(\scrA_t)} I_{\epsilon e_{\scrA}} - I_{\epsilon e_{\scrA}} d_{\overline{CC}_*(\scrA_t)} + I_{d_{\scrA_t}(\epsilon e_{\scrA})} + I_W) \mathit{SH} \\
& \quad = (-d_{\overline{CC}_*(\scrA_t)} I_{\epsilon e_{\scrA}} - I_{\epsilon e_{\scrA}} d_{\overline{CC}_*(\scrA_t)} + t) \mathit{SH}.
\end{aligned}
\end{equation}
In other words, if we replaced $t$ by the homotopic endomorphism in brackets in the last line, then the diagram \eqref{eq:gm-compatible-2} would commute strictly.
\end{proof}

\subsubsection{The monodromy theorem\label{subsubsec:monodromy}}
In the terminology of Section \ref{subsec:u-theory}, 
\begin{equation} \label{eq:final-cyclic}
A_{q,u} = \overline{CC}_*(\scrA_q),
\end{equation}
with its $q$-connection, is a complete $u$-torsionfree dg module over $W_{q,u} \iso W_{t,u}$; and its $u = 0$ reduction is $q$-torsionfree,
\begin{equation}
A_q = \bar{C}_*(\scrA_q).
\end{equation}
What appears in Theorem \ref{th:noncommutative-fourier-transform} is the modified version $q^{-1}A_{q^{-1},u}$ from \eqref{eq:quotient-q}. This is again complete (by definition) and $u$-torsionfree (by Lemma \ref{th:invert-q-2}; one can also easily check it directly). On that version, we can carry out the completed localisation process with respect to some nonzero polynomial $p(t)$, as in \eqref{eq:quotient-t}. Denote the outcome by
\begin{equation} \label{eq:cc-localise-t}
q^{-1}A_{q^{-1},1/p,u} \stackrel{\mathrm{def}}{=} \bK[t,1/p] \hat\otimes_{\bK[t]} q^{-1}A_{q^{-1},u};
\end{equation}
it is again complete (by definition) and $u$-torsionfree (by Lemma \ref{th:invert-t}; this time, direct verification is not straightforward, since $t = \nabla_{u\partial_q}$ by definition of the Fourier-Laplace transform). Finally, we can invert $u$ in a purely algebraic sense, again using shorthand notation:
\begin{equation} \label{eq:cc-periodic-final}
q^{-1} A_{q^{-1},1/p,u^{\pm 1}} \stackrel{\mathrm{def}}{=} \bK[u^{\pm 1}] \otimes_{\bK[u]} q^{-1}A_{q^{-1},1/p,u}.
\end{equation}
This is now $2$-periodically graded; and the action of $\nabla_{\partial_t} = -q/u$ and $t = \nabla_{u\partial_q}$ make it into a chain complex of modules over the classical Weyl algebra.

\begin{corollary} \label{th:monodromy}
Take $\bK = \bC$. Suppose that we have $(\scrA,W)$ and a nonzero polynomial $p(t)$, such that:
\begin{itemize} \itemsep.5em
\item[(i)] $\scrA$ is smooth;
\item[(ii)] $q^k p([W]) \in \mathit{HH}^{2k}(\scrA_q)$ vanishes for some $k>0$. 
\item[(iii)] $\bK[t,1/p] \otimes_{\bK[t]} H^*(\scrA)$, where $t$ acts by $[W]$, is a finitely generated $\bK[t,1/p]$-module.
\end{itemize}
Then, in each degree, the cohomology of $q^{-1} A_{q^{-1},1/p,u^{\pm 1}}$ is finitely generated over $\bC[t,1/p]$, where $t = \nabla_{u\partial_q}$. Moreover, the action of $\nabla_{\partial_t} = -q/u$ on that cohomology is a connection with regular singularities, and quasi-unipotent monodromy around each singularity (for both statements, this includes $t = \infty$; quasi-unipotency means that the eigenvalues are roots of unity).
\end{corollary}

\begin{proof}
First consider $q^{-1}A_{q^{-1},u}$. Theorem \ref{th:noncommutative-fourier-transform}, together with Lemma \ref{th:homotopy-dq}, says that in $D(W_{q,u}) \iso D(W_{t,u})$, this is isomorphic to $\overline{CC}_*(\scrA_t)$, with its $t$-action and Getzler-Gauss-Manin connection. As a consequence, we get an isomorphism in that category,
\begin{equation} \label{eq:cc-is-cc}
q^{-1} A_{q^{-1},1/p,u} \iso \big(\bK[t,1/p] \hat\otimes_{\bK[t]} \overline{CC}_*(\scrA_t)\big).
\end{equation}
The rightmost expression is the complex underlying the negative cyclic homology, $\mathit{HC}_*(\scrA_{t,1/p})$ taken over $\bK[t,1/p]$; for this to be true it is essential to $u$-adically complete the tensor product in \eqref{eq:cc-is-cc} as one does in the definition of cyclic homology \eqref{eq:connes}. We then invert $u$ to obtain periodic cyclic homology (again relative to $\bK[t,1/p]$) 
\begin{equation}
\mathit{HP}_*(\scrA_{t,1/p}) = \mathit{HC}_*(\scrA_{t,1/p}) \otimes_{\bK[u]} \bK[u^{\pm 1}]. 
\end{equation}
  On that cohomology, $\nabla_{\partial_t}$ acts as the Getzler-Gauss-Manin connection. Now, $\scrA_{t,1/p}$ is smooth over $\bK[t,1/p]$ by Proposition \ref{th:smooth-2}, and proper by (iii). Both regularity and quasi-unipotency hold by \cite[Theorem 3]{petrov-vaintrob-vologodsky18}.
\end{proof}

\begin{corollary} \label{th:unipotence-bound}
In the situation of Corollary \ref{th:monodromy}, suppose additionally that $\mathit{HH}_*(\scrA)$ is concentrated in degrees $\leq d$. Then, the monodromy of $\nabla_{\partial_t}$ around each singular point (including $\infty$) has Jordan blocks of size at most $d+1$.
\end{corollary}

\begin{proof}
Assumption (iii) from Corollary \ref{th:monodromy} implies that $\bK(t) \otimes_{\bK(t)} H^*(\scrA)$ is finite-dimensional over $\bK(t)$, and Proposition \ref{th:smooth-2} guarantees the smoothness of $\bK(t) \otimes_{\bK[t]} \scrA_t$ over $\bK(t)$. By Corollary \ref{th:degree-bound}, the Hochschild homology of $\bK(t) \otimes_{\bK[t]} \scrA_t$ is concentrated in degrees $[-d,\dots,d]$. One then applies \cite[Theorem 8]{petrov-vaintrob-vologodsky18} (the ``exponent'' in \cite{petrov-vaintrob-vologodsky18} is the maximal Jordan block size; see \cite[(0.2.2)]{katz}).
\end{proof}

\subsection{$A_\infty$-algebras\label{subsec:ainfty}}

\subsubsection{Basic notions}
We again work over a field $\bK$. An $A_\infty$-algebra consists of a graded vector space $\scrA$, together with operations 
\begin{equation} \label{eq:mu}
\mu^d_{\scrA}: \scrA^{\otimes d} \longrightarrow \scrA, \;\; d \geq 1,
\end{equation}
of degree $2-d$, satisfying the $A_\infty$-associativity equations \eqref{eq:associativity}. For now, we assume that our $A_\infty$-algebras are strictly unital: there is an $e_{\scrA} \in \scrA^0$ such that \eqref{eq:strict-unit} holds. Similarly, for an $A_\infty$-homomorphism $\Phi: \scrA \rightarrow \scrB$, with components
\begin{equation}
\Phi^d: \scrA^{\otimes d} \longrightarrow \scrB
\end{equation}
of degree $1-d$, the strict unitality condition is that
\begin{equation} \label{eq:strict-unit-2}
\Phi^1(e_{\scrA}) = e_{\scrB}, \;\; \Phi^d(\dots,e_{\scrA},\dots) = 0 \text{ for all $d \neq 1$.}
\end{equation}

A curved $A_\infty$-algebra is an $\scrA$ as before, with operations 
\begin{equation} \label{eq:curved-mu}
\mu^d_{\scrA_q}: \scrA^{\otimes d} \longrightarrow \scrA[[q]], \;\; d \geq 0,
\end{equation}
again of degree $2-d$, and where $q$ is a formal variable of degree $2$. The curvature term $\mu_{\scrA_q}^0$ must have zero $q$-constant part, and the entire sequence of operations (extended $q$-linearly to multilinear maps on $\scrA_q = \scrA[[q]]$) must satisfy the extended $A_\infty$-associativity equations. We also require the analogue of \eqref{eq:strict-unit}, again with $e_{\scrA} \in \scrA$ (not $\scrA_q$). Of course, setting $q = 0$ then gives us an ordinary $A_\infty$-algebra $\scrA$; we will also refer to $\scrA_q$ as a curved deformation of $\scrA$. Expanding in orders of $q$, we write
\begin{equation} \label{eq:q-expansion}
\mu_{\scrA_q}^d = \mu_{\scrA}^d + q\mu_{\scrA_q}^{d,(1)} + O(q^2).
\end{equation}
The $q$-derivative of the deformation is denoted by
\begin{equation} \label{eq:kappa}
\kappa_{\scrA_q}^d = \partial_q \mu^d_{\scrA_q}.
\end{equation}
There is a curved version $\Phi_q: \scrA_q \rightarrow \scrB_q$ of $A_\infty$-homomorphisms, where one makes the operations $q$-dependent, includes a $\Phi^0_q \in \scrB_q^1$ with zero $q$-constant part, and still maintains the condition \eqref{eq:strict-unit-2}. We say that $\Phi_q$ is a filtered quasi-isomorphism if its $q = 0$ reduction $\Phi: \scrA \rightarrow \scrB$ is a quasi-isomorphism.

\subsubsection{Hochschild and cyclic homology\label{subsubsec:hh-again}}
The constructions of Hochschild and cyclic homology generalize to the $A_\infty$-world (see e.g.\ \cite{sheridan20}). For Hochschild homology, the differential becomes
\begin{equation} \label{eq:unnormalized-c}
\begin{aligned}
& d_{\bar{C}_*(\scrA)}(a_0(a_1|\dots|a_l)) 
= \sum_{ij} (-1)^{\|a_0\|+\cdots+\|a_j\|} a_0 (a_1|\dots|a_j|\mu_{\scrA}^i(a_{j+1},\dots,a_{j+i})|\dots)
\\ &
 +\sum_{ij} (-1)^{(\|a_{j+1}\|+\cdots+\|a_l\|)(\|a_0\|+\cdots+\|a_j\|)}
\mu_{\scrA}^{l-j+i+1}(a_{j+1}, \dots,a_l,a_0,\dots,a_i) (a_{i+1}|\dots|a_j). 
\end{aligned}
\end{equation}
For cyclic homology, one adds the same Connes operator term \eqref{eq:connes} as before. For Hochschild cohomology, one similarly has 
\begin{equation} \label{eq:diff-c-a}
\begin{aligned} &
(d_{\bar{C}^*(\scrA)}\phi)^l(a_1,\dots,a_l) 
\\ & \quad
= \sum_{ij} (-1)^{\|\phi\|(\|a_1\|+\cdots+\|a_j\|)}\mu_{\scrA}^{l-j+1}(a_1,\dots,a_j,\phi^i(a_{j+1},\dots,a_{j+i}),\dots,a_l)
\\[-1em] & \qquad \qquad \quad
+ (-1)^{\|a_1\|+\cdots+\|a_j\|+|\phi|} \phi^{l-i+1}(a_1,\dots,a_j,\mu_{\scrA}^i(a_{j+1},\dots,a_{j+i}),\dots,a_l).
\end{aligned}
\end{equation}
For instance, the first order term $\mu_{\scrA_q}^{(1)}$ of a curved deformation is a cocycle in $\bar{C}^0(\scrA)$. As before, Hochschild cohomology is a graded commutative algebra. 
The analogue of \eqref{eq:iota-action} for the action of $\phi \in \bar{C}^*(\scrA)$ on the Hochschild complex is
\begin{equation} \label{eq:generalized-iota-action}
\begin{aligned}
&
\iota_\phi(a_0(a_1|\dots|a_l)) =
 \sum_{ijmn} (-1)^{(\|a_0\|+\cdots+\|a_j\|)(\|a_{j+1}\|+\cdots+\|a_l\|)
+ \|\phi\|(\|a_{j+1}\| + \cdots + \|a_i\|)}
\\[-.5em] & \qquad \mu_{\scrA}^{n+l+2-j}(a_{j+1},\dots,\phi_{\scrA}^m(a_{i+1},\dots,a_{i+m}),\dots,a_0,\dots,a_n) (a_{n+1}|\dots|a_j).
\end{aligned}
\end{equation}
For the action $I_\phi$ on cyclic homology, one adds the same $u$ term as in \eqref{eq:big-iota-action}. 

The same constructions apply to curved $A_\infty$-algebras $\scrA_q$, obviously including the $\mu^0_{\scrA_q}$-term and making sure that everything is $q$-linear and complete. For instance, the $q$-derivative \eqref{eq:kappa} is a cocycle in $\bar{C}^0(\scrA_q)$. One defines the $q$-connection on $\overline{CC}_*(\scrA_q)$ by
\begin{equation} \label{eq:kappa-connection}
\nabla_{u\partial_q} = u\partial_q + I_{\kappa_{\scrA_q}},
\end{equation}
or explicitly:
\begin{equation} \label{eq:explicit-gauss-manin}
\begin{aligned} &
\nabla_{u\partial_q}(a_0(a_1|\dots|a_l)q^k) = uk\, a_0(a_1|\dots|a_l)q^{k-1} 
\\ & 
\quad  + \sum_{ijmn} (-1)^{(\|a_0\|+\cdots+\|a_j\|)(\|a_{j+1}\|+\cdots+\|a_l\|)
+ (\|a_{j+1}\| + \cdots + \|a_i\|)}
\\[-1em] & \qquad \qquad \quad \mu_{\scrA_q}^{n+l+2-j}(a_{j+1},\dots,\partial_q \mu_{\scrA_q}^m(a_{i+1},\dots,a_{i+m}),\dots,a_0,\dots,a_n) (a_{n+1}|\dots|a_j)
\\[.5em] &
\quad - 
u
\sum_{ijm} (-1)^{(\|a_0\|+\cdots+\|a_{j}\|)(\|a_{j+1}\|+\cdots+\|a_l\|) + ( \|a_{j+1}\|+\cdots+\|a_i\|)} \\[-1em] & \qquad \qquad \qquad \qquad
e_{\scrA} (a_{j+1}|\dots|\partial_q \mu_{\scrA_q}^m(a_{i+1},\dots,a_{i+m})|\dots|a_0|\dots|a_j).
\end{aligned}
\end{equation}
The compatibility of this connection with the (covariant) functoriality of cyclic homology was addressed in \cite[Theorem 3.32 and Appendix B]{sheridan20} (for $A_\infty$-algebras without curvature term, but the inclusion of that term is straightforward). The statement is:

\begin{lemma} \label{th:induced-map}
Let $\Phi_q: \scrA_q \rightarrow \scrB_q$ be a curved $A_\infty$-homomorphism. Then, there is a canonical $(q,u)$-linear induced map $\overline{CC}_*(\scrA_q) \longrightarrow \overline{CC}_*(\scrB_q)$, which fits into a homotopy commutative diagram
\begin{equation}
\xymatrix{
\ar[d]_-{\nabla_{u\partial_q}} \overline{CC}_*(\scrA_q) \ar[rr] && \overline{CC}_*(\scrB_q) \ar[d]^-{\nabla_{u\partial_q}}
\\
\overline{CC}_*(\scrA_q) \ar[rr] && \overline{CC}_*(\scrB_q).
}
\end{equation}
Finally, if $\Phi_q$ is a filtered quasi-isomorphism, then so is the induced map on cyclic chains.
\end{lemma}

\subsubsection{Deformation theory\label{subsubsec:deformation-theory}}
Classically, Hochschild cohomology appears in the $A_\infty$-context as the obstruction theory governing curved $A_\infty$-deformations. We will need a slight variation of the theory, which concerns curved $A_\infty$-homomorphisms. Suppose that we have an $A_\infty$-homomorphism $\Phi: \scrA \rightarrow \scrB$. Associated to that is a Hochschild cohomology theory $\mathit{HH}^*(\scrA,\scrB)$, with underlying complex
\begin{equation} \label{eq:ab-complex}
\begin{aligned} &
\bar{C}^*(\scrA,\scrB) = \mathit{hom}_{\bK}(T(\bar\scrA[1]),\scrB), \\[.5em]
& (d_{\bar{C}^*(\scrA,\scrB)}\phi)^l(a_1,\dots,a_l) 
\\ & \quad
= \sum_{\substack{r, j \\ i_1+\cdots+i_r = l}}
(-1)^{\|\phi\|(\|a_1\|+\cdots+\|a_{i_1+\cdots+i_j}\|)}
\mu_{\scrB}^r(\Phi^{i_1}(a_1,\dots,a_{i_1}),\dots,\Phi^{i_j}(\dots,a_{i_1+\cdots+i_j}),
\\[-1.5em] & \qquad \qquad \qquad \qquad \qquad
\phi^{i_{j+1}}(a_{i_1+\cdots+i_j+1},\dots,a_{i_1+\cdots+i_{j+1}}),\Phi^{i_{j+2}}(\dots),\dots)
\\[.5em] & \qquad \quad
+ \sum_{ij} (-1)^{\|a_1\|+\cdots+\|a_i\|+|\phi|} \phi^{l-j+1}(a_1,\dots,a_i,\mu_{\scrA}^j(a_{i+1},\dots,a_{i+j}),\dots,a_l).
\end{aligned}
\end{equation}
This comes with maps (compare e.g. \cite[Section 4.2]{sheridan20}, which discusses the more general situation of Hochschild cohomology with bimodule coefficients)
\begin{equation} \label{eq:hochschild-mixed-functoriality}
\begin{aligned}
& \bar{C}^*(\scrA) \stackrel{\Phi_*}{\longrightarrow} \bar{C}^*(\scrA,\scrB) \stackrel{\Phi^*}{\longleftarrow} \bar{C}^*(\scrB), \\[.5em]
& (\Phi_*\phi)^l(a_1,\dots,a_l) = \sum_{ij} (-1)^{\|\phi\|(\|a_1\|+\cdots+\|a_j\|)} \Phi^{l-i+1}(a_1,\dots,a_j,
\\[-1em] & \qquad \qquad \qquad \qquad \qquad \qquad \qquad
\phi^i(a_{j+1},\dots,a_{j+i}),\dots), \\[.5em]
& (\Phi^*\phi)^l(a_1,\dots,a_l) = \sum_{\substack{r \\ i_1+\cdots+i_r = l}}
\phi^r(\Phi^{i_1}(a_1,\dots,a_{i_1}),\Phi^{i_2}(a_{i_1+1},\dots,a_{i_1+i_2}),\dots).
\end{aligned}
\end{equation}
Moreover, $\mathit{HH}^*(\scrA,\scrB)$ is an algebra (see e.g.\ \cite[p.~11, Equation (1.9)]{seidel04} for the product), and the maps induced by $\Phi_*$, $\Phi^*$ are maps of algebras. Finally, if $\Phi$ is a quasi-isomorphism, then so are the maps $\Phi_*$, $\Phi^*$.

\begin{remark} \label{th:functorial}
One can get a higher-level picture by considering $A_\infty$-algebras as categories with a single object. $\bar{C}^*(\scrA)$ is the chain complex of natural transformations from the identity functor to itself; correspondingly, $\bar{C}^*(\scrA,\scrB)$ are the natural transformations from $\Phi$ to itself, in the category of functors $\scrA \rightarrow \scrB$; and the maps \eqref{eq:hochschild-mixed-functoriality} are left and right composition with $\Phi$ (compare e.g.\ \cite[Section 1e]{seidel04}). This makes it clear why those are maps of algebras.
\end{remark}

\begin{lemma} \label{th:def-theory}
Suppose that we have an $A_\infty$-homomorphism $\Phi: \scrA \rightarrow \scrB$, and 
curved deformations $\scrA_q$, $\scrB_q$, which are a priori independent of each other. If
\begin{equation} \label{eq:matching-first-order-classes}
[\Phi_*(\mu_{\scrA_q}^{(1)})] =[\Phi^*(\mu_{\scrB_q}^{(1)})] \in \mathit{HH}^0(\scrA,\scrB)
\end{equation}
and
\begin{equation} \label{eq:no-negative-hh}
\mathit{HH}^{2k}(\scrA,\scrB) = 0 \;\; \text{for all $k<0$},
\end{equation}
then $\Phi$ can be extended to a curved $A_\infty$-homomorphism $\Phi_q: \scrA_q \rightarrow \scrB_q$.
\end{lemma}

\begin{proof}[Sketch of proof]
This is a straightforward obstruction theory argument, order by order in $q$. Spelling out the equation for $\Phi_q$ at first order in $q$ shows that we are looking for $\phi \in \bar{C}^{-1}(\scrA,\scrB)$ which satisfies
\begin{equation}
\Phi_*(\mu_{\scrA_q}^{(1)}) - \Phi^*(\mu_{\scrB_q}^{(1)}) = d_{\bar{C}^*(\scrA,\scrB)}\phi \in \bar{C}^0(\scrA,\scrB),
\end{equation}
and that of course can be done iff \eqref{eq:matching-first-order-classes} holds. The next order equation will take place in $\bar{C}^{-2}(\scrA,\scrB)$, and so on, with the vanishing of the obstruction groups ensured by \eqref{eq:no-negative-hh}.
\end{proof}

We will also need a result concerning homomorphisms of curved $A_\infty$-algebras, whose proof uses the same techniques.

\begin{lemma} \label{th:phi-q}
A filtered quasi-isomorphism $\Phi_q: \scrA_q \rightarrow \scrB_q$ induces an isomorphism of Hochschild cohomologies, $\mathit{HH}^*(\scrA_q) \iso \mathit{HH}^*(\scrB_q)$. This is an isomorphism of algebras over $\bK[q]$. Moreover, it sends $[\kappa_{\scrA_q}]$ to $[\kappa_{\scrB_q}]$.
\end{lemma}

\begin{proof}[Sketch of proof]
One introduces a mixed group $\mathit{HH}^*(\scrA_q,\scrB_q)$, which comes with maps $\Phi_{q,*}$ and $\Phi_q^*$ as in \eqref{eq:hochschild-mixed-functoriality}. This has an interpretation in terms of categories of curved $A_\infty$-functors and their natural transformations, as in Remark \ref{th:functorial}, and from that, one sees that the maps are compatible with the algebra structures. Given that $\Phi_q$ is a filtered quasi-isomorphism, the maps $\Phi_{q,*}$ and $\Phi_q^*$ are quasi-isomorphisms; this argument goes by $q$-filtration, which reduces it to the uncurved case. Differentiating $\Phi_q$ itself yields a cochain $\lambda_{\Phi_q} \in \bar{C}^{-1}(\scrA_q,\scrB_q)$, which satisfies
\begin{equation}
d_{C^*(\scrA_q,\scrB_q)}(\lambda_{\Phi_q}) = \Phi_{q,*}\kappa_{\scrA_q} - \Phi_q^*\kappa_{\scrB_q}.
\end{equation}
\end{proof}

\subsubsection{Cohomological unitality\label{subsubsec:c-unital}}
We will now drop the condition of strict unitality, and only require that $H^*(\scrA)$ be a unital algebra. For curved $A_\infty$-algebras $\scrA_q$, we require that the $q = 0$ reduction should be cohomologically unital. Similar adaptations can be made to the notion of $A_\infty$-homomorphism. 

Hochschild homology and cohomology can be defined in the cohomologically unital context simply by dropping the normalization condition, which means replacing $\bar{\scrA}$ by $\scrA$ in the definition of the relevant complexes.  We denote the outcome by $C_*(\scrA)$ and $C^*(\scrA)$. For Hochschild homology, there is also another approach, which then extends to cyclic homology as well. Namely, take
\begin{equation} \label{eq:adjoin}
\scrA^+ = \scrA \oplus \bK e_{\scrA}^+, 
\end{equation}
where the $A_\infty$-operations are extended by making $e_{\scrA}^+$ a strict unit. The normalized Hochschild complex of $\scrA^+$ can be written (as a graded vector space) as
\begin{equation} 
\bar{C}_*(\scrA^+) = C_*(\scrA) \oplus (e_{\scrA}^+ \otimes T(\scrA[1])).
\end{equation}
Write 
\begin{equation} \label{eq:bar-complex}
D_*(\scrA) \stackrel{\mathrm{def}}{=} T^{>0}(\scrA[1]) \subset T(\scrA[1])
\end{equation}
for the space of tensor expressions of positive length. This comes with a standard bar differential
\begin{equation} \label{eq:dbar-differential}
d_{D_*(\scrA)}(a_1|\cdots|a_l) = \sum_{ij} (-1)^{\|a_1\|+\cdots+\|a_j\|} (a_1|\dots|a_j|\mu_{\scrA}^i(a_{j+1},\dots,a_{j+i})|\dots|a_l).
\end{equation}
The non-unital version of the Hochschild complex 
\begin{equation} \label{eq:plus-complex}
C_*^+(\scrA) \stackrel{\mathrm{def}}{=} C_*(\scrA) \oplus (e_{\scrA}^+ \otimes D_*(\scrA)) \subset \bar{C}_*(\scrA^+)
\end{equation}
Explicitly, the differential is
\begin{equation} \label{eq:hochschild-homology-complex-2}
\begin{aligned}
& d_{C_*^+(\scrA)}(a_0(a_1|\dots|a_l))
\text{ as in \eqref{eq:unnormalized-c},} \\
& d_{C_*^+(\scrA)}(e^+_{\scrA}(a_1|\dots|a_l)) = 
 - e^+_{\scrA}\, d_{D_*(\scrA)}(a_1|\dots|a_l)
\\ & \qquad
+ a_1(a_2|\dots|a_l) - (-1)^{\|a_l\|(\|a_1\|+\cdots+\|a_{l-1}\|)} a_l(a_1|\dots|a_{l-1}).
\end{aligned}
\end{equation}
For cyclic homology, one correspondingly uses the subcomplex
\begin{equation} \label{eq:plus-complex-cyclic}
\mathit{CC}_*^+(\scrA) \stackrel{\mathrm{def}}{=} C_*^+(\scrA)[[u]] 
\stackrel{\mathrm{def}}{=} C_*(\scrA)[[u]] \oplus (e_{\scrA}^+ \otimes D_*(\scrA)[[u]])
\subset \overline{CC}_*(\scrA^+),
\end{equation}
and we again spell out the differential:
\begin{equation} \label{eq:c-unital-cyclic}
\begin{aligned}
& d_{\mathit{CC}_*^+(\scrA)}(a_0(a_1|\dots|a_l)) = d_{C_*^+(\scrA)}(a_0(a_1|\dots|a_l)) 
\\ & \qquad
- u \sum_j (-1)^{(\|a_0\|+\cdots+\|a_j\|)(\|a_{j+1}\|+\cdots+\|a_l|)} e_{\scrA}^+(a_{j+1}|\dots|a_l|a_0|\dots|a_j), \\
& d_{\mathit{CC}_*^+(\scrA)}(e_{\scrA}^+ (a_1|\dots|a_l)) = 
d_{C_*^+(\scrA)}(e_{\scrA}^+(a_1|\dots|a_l)).
\end{aligned}
\end{equation}

All this also works in the curved ($q$-linear) setup. The connection on $\overline{CC}_*(\scrA_q^+)$ defined in \eqref{eq:kappa-connection} preserves the subspace $\mathit{CC}_*^+(\scrA_q)$, and we will use that as the definition of Getzler-Gauss-Manin connection in the cohomologically unital case. The explicit formula is
\begin{equation} \label{eq:recap-gauss-manin}
\begin{aligned} 
&
\nabla_{u\partial_q}(a_0(a_1|\dots|a_l)) 
= u(\partial_q a_0)(a_1|\dots|a_l) + u\sum_i a_0(a_1|\dots|\partial_qa_i|\dots|a_l)
\\ &
\qquad + \sum_{ijrs} 
(-1)^{(\|a_0\|+\cdots+\|a_j\|)(\|a_{j+1}\|+\cdots+\|a_l\|)+ (\|a_{j+1}\| + \cdots + \|a_i\|)}
\\[-.5em] 
& \qquad \qquad \qquad 
\mu_{\scrA_q}^{s+l+2-j}(a_{j+1},\dots,\partial_q \mu_{\scrA_q}^r(a_{i+1},\dots,a_{i+r}),\dots,a_0,\dots,a_s) (a_{s+1}|\dots|a_j)
\\[.5em] &
\qquad - u\sum_{ijk}
(-1)^{(\|a_0\|+\cdots+\|a_{j}\|)(\|a_{j+1}\|+\cdots+\|a_l\|) + ( \|a_{j+1}\|+\cdots+\|a_i\|)} \\[-1em] & \qquad \qquad \qquad 
e^+_{\scrA} (a_{i+1}|\dots|\partial_q \mu_{\scrA_q}^r(a_{j+1},\dots,a_{j+k})|\dots|a_0|\dots|a_i), 
\\[.5em]
& \nabla_{u\partial_q}(e_{\scrA}^+(a_1|\cdots|a_l)) = u \sum_i e_{\scrA}^+(a_1|\dots|\partial_q a_i|\dots|a_l)
\\ & \qquad + 
\sum_{j \geq 0} (-1)^{(\|a_1\|+\cdots+\|a_{j}\|)(\|a_{j+1}\|+\cdots+\|a_l\|)} 
\partial_q\mu_{\scrA_q}^j(a_{l-j+1},\dots,a_l) (a_1|\dots|a_{l-j}).
\end{aligned}
\end{equation}

We now digress to discuss a source of equivalent formulae for the Getzler-Gauss-Manin connection (this could also be applied to strictly unital $A_\infty$-algebras, but the present framework is where we will need it for geometric applications). Recall from Sections \ref{subsubsec:cyclic-homology} and \ref{subsubsec:hh-again} that the Getzler-Gauss-Manin connection can be understood in terms of the operation $\iota_{\kappa_{\scrA_q}}$, where $\kappa_{\scrA_q} = \partial_q \mu_{\scrA_q}$, and its cyclic extension $I_{\kappa_{\scrA_q}}$. Here, we want to factor $\iota_{\kappa_{\scrA_q}}$ into two steps. There is a map from the (unreduced) chain complex underlying Hochschild cohomology to a suitable morphism space, in the category of strictly unital $A_\infty$-bimodules over $\scrA_q^+$:
\begin{equation} \label{eq:hochschild-to-diagonal}
\begin{aligned}
& \Delta: C^*(\scrA_q) \longrightarrow \overline{\mathit{hom}}_{(\scrA_q^+, \scrA_q^+)}(\scrA_q^+,\scrA_q), \\
& (\Delta\phi)^{j,1,k}(a_1,\dots,a_j,\underline{a_{j+1}},a_{j+2},\dots,a_{j+k+1}) = \sum_{i l} (-1)^{\|\phi\|(\|a_i\|+\cdots+\|a_i\|)} 
\\[-.5em] & \qquad \qquad \qquad \mu_{\scrA_q}^{j+k-l+2}(a_1,\dots,a_i,\phi^l(a_{i+1},\dots,a_{i+l}),\dots,\underline{a_{j+1}},\dots,a_{j+k+1}), \\
& (\Delta\phi)^{j,1,0}(a_1,\dots,a_j,\underline{e_{\scrA}^+}) = (-1)^{|\phi|+\|a_1\|+\cdots+\|a_j\|} \phi^j(a_1,\dots,a_j), \\
& (\Delta\phi)^{j,1,k}(\dots,\underline{e_{\scrA}^+},\dots) = 0 \;\; \text{for $k>0$.}
\end{aligned}
\end{equation}
Here, the $\overline{\hom}$ reminds us that we are talking about strictly unital morphisms over $\scrA_q^+$; and we have underlined those entries which are the distinguished (central) ones of the $A_\infty$-bimodule map. Note that in the second line of \eqref{eq:hochschild-to-diagonal}, the distinguished entry must lie to the right of $\phi(\dots)$ (the opposite convention would lead to a different but chain homotopic map). Next, endomorphisms of the diagonal act on the Hochschild complex (and more generally, Hochschild homology is a functor on bimodules). Adapted to our situation, this yields the chain map
\begin{equation} \label{eq:endomorphisms-act-on-hochschild}
\begin{aligned}
& \Gamma: \overline{\mathit{hom}}_{(\scrA_q^+,\scrA_q^+)}(\scrA_q^+,\scrA_q) \longrightarrow \mathit{hom}_{\bK[[q]]}(C_*^+(\scrA_q),C_*(\scrA_q)), 
\\
& (\Gamma \xi)(a_0(a_1|\cdots|a_l)) = \sum_{jk} (-1)^{(\|a_0\|+\cdots+\|a_{j}\|)(\|a_{j+1}\|+\cdots+\|a_l\|)} 
\\[-.75em] & \qquad \qquad \qquad \qquad \qquad \qquad
\xi^{l-j,1,k}(a_{j+1},\dots,a_l,\underline{a_0},a_1,\dots,a_k) (a_{k+1}|\dots|a_{j}), \\
& (\Gamma \xi)(e_{\scrA}^+(a_1|\dots|a_l)) = \sum_{jk} 
(-1)^{(\|a_1\|+\cdots+\|a_j\|+1)(\|a_{j+1}\|+\cdots+\|a_l\|)}
\\[-.75em] & \qquad \qquad \qquad \qquad \qquad \qquad
\xi^{l-j,1,k}(a_{j+1},\dots,a_l,\underline{e_{\scrA}^+},a_1,\dots,a_k)(a_{k+1}|\dots|a_j).
\end{aligned}
\end{equation}
The composition of the two is $\iota_{\phi} = \Gamma \Delta \phi$. Now, take an arbitrary cochain 
\begin{equation}
\xi \in \hom_{(\scrA_q,\scrA_q)}^{-1}(\scrA_q,\scrA_q),
\end{equation}
and its coboundary $d_{(\scrA_q,\scrA_q)}\xi$. Let's extend $\xi$ trivially to an element of $\overline{\hom}_{(\scrA_q,\scrA_q)}(\scrA_q^+,\scrA_q)$, which means $\xi^{j,1,k}(\dots,\underline{e_{\scrA}^+},\dots) = 0$. Then define a modified Getzler-Gauss-Manin connection by
\begin{equation} \label{eq:modified-gauss-manin}
\tilde{\nabla}_{u\partial_q} \stackrel{\mathrm{def}}{=} \nabla_{u\partial_q} + d_{\mathit{CC}_*^+(\scrA_q)}\Gamma(\xi) + \Gamma(\xi)d_{\mathit{CC}_*^+(\scrA_q)}.
\end{equation}
By construction, this induces the same connection on cohomology as $\nabla_{u\partial_q}$. Explicitly, the added term is
\begin{align}
\label{eq:modified-gauss-manin-2}
&
\begin{aligned}
&
\big(d_{\mathit{CC}_*^+(\scrA_q)}\Gamma(\xi) + \Gamma(\xi)d_{\mathit{CC}_*^+(\scrA_q)}\big)
\big( a_0(a_1|\dots|a_l) \big) \\
& \qquad = 
\sum_{jk} (-1)^{(\|a_{l-j+1}\|+\cdots+\|a_l\|)(\|a_0\|+\cdots+\|a_{l-j}\|)} 
\\[-1em] & \qquad \qquad \qquad \quad
(d_{(\scrA_q,\scrA_q)}\xi)^{j,1,k}(a_{l-j+1},\dots,a_l,\underline{a_0},a_1,\dots,a_k) (a_{k+1}|\dots|a_{l-j})
\\ & \qquad -
u\sum_{ijk} (-1)^{(\|a_{i+1}\|+\cdots+\|a_{l-j}\|) +
(\|a_{i+1}\|+\cdots+\|a_l\|)(\|a_0\|+\cdots+\|a_i\|)}
\\[-1em] & \qquad \qquad \quad \quad
e^+_{\scrA} \otimes (a_{i+1}|\dots|\xi^{j,1,k}(a_{l-j+1},\dots,a_l,\underline{a_0},a_1,\dots,a_k)|a_{k+1}|\dots|a_i),
\end{aligned}
\\ 
\label{eq:modified-gauss-manin-3}
& 
\begin{aligned} &
\big(d_{\mathit{CC}_*^+(\scrA_q)}\Gamma(\xi) + \Gamma(\xi)d_{\mathit{CC}_*^+(\scrA_q)}\big)
\big( e_{\scrA^+} (a_1|\cdots|a_l) \big) 
\\ & \qquad =
\sum_{jk} (-1)^{(\|a_1\|+\cdots+\|a_{l-j}\|)(\|a_{l-j+1}\|+\cdots+\|a_l\|)} 
\\[-1em] & \qquad \qquad \qquad \quad
\xi^{j,1,k-1}(a_{l-j+1},\dots,a_l,\underline{a_1},a_2,\dots,a_k)(a_{k+1}|\dots|a_{l-j}) 
\\ & \qquad - \sum_{jk} (-1)^{(\|a_1\|+\cdots+\|a_{l-j}\|)(\|a_{l-j+1}\|+\cdots+\|a_l\|)} 
\\[-1em] & \qquad \qquad \qquad \quad
\xi^{j-1,1,k}(a_{l-j+1},\dots,\underline{a_l},a_1,\dots,a_k)(a_{k+1}|\dots|a_{l-j}).
\end{aligned}
\end{align} 
Note that \eqref{eq:modified-gauss-manin-3} is zero if $l = 1$. 

\subsubsection{Carrying results over to the cohomologically unital case}
Once the definitions have been set up, all the results we have obtained generalize to the cohomologically unital case.  For instance, Lemmas \ref{th:def-theory} and \ref{th:phi-q} apply to cohomologically unital $A_\infty$-homomorphisms, with the same proof. The same holds for Lemma \ref{th:induced-map} (indeed, it's the cohomologically unital situation which is primarily considered in \cite{sheridan20}).
In cases where the results require more effort, one can avoid re-doing the work by using a few tricks that reduce things to the strictly unital case. 

\begin{lemma} \label{th:make-strictly-unital}
Let $\scrA$ be a cohomologically unital $A_\infty$-algebra. Then, there is a strictly unital $\scrB$ and a quasi-isomorphism $\scrA \rightarrow \scrB$. Similarly, if $\scrA_q$ has the property that its $q = 0$ reduction is cohomologically unital, there is a strictly unital $\scrB_q$ and a filtered quasi-isomorphism $\scrA_q \rightarrow \scrB_q$.
\end{lemma}

\begin{proof}[Sketch of proof]
A short way to prove this is via the Yoneda embedding \cite[Corollary 9.4]{fukaya02}, which for a given $\scrA$, produces a quasi-isomorphism to a (strictly unital) differential graded algebra $\scrB$. This also extends to the curved case, where it produces a curved dga $\scrB_q$. Alternatively, one can prove Lemma \ref{th:make-strictly-unital} using deformation theory (based on the fact that the classification of $A_\infty$-structures is governed by the Hochschild complex, and that of strictly unital ones by the quasi-isomorphic normalized complex).
\end{proof}

Even with Lemma \ref{th:make-strictly-unital} at hand, there's a small gap to bridge. Namely, for strictly unital $\scrA$, one has to show that constructions involving the enlarged ``$+$-version'' of the Hochschild complex agree with their classical counterparts. To do that, one considers the collapse map
\begin{equation} \label{eq:collapse}
\begin{aligned}
& \mathit{C}_*^+(\scrA) \longrightarrow \overline{C}_*(\scrA), \\
& a_0(a_1|\dots|a_l) \longmapsto a_0(a_1|\dots|a_l), \\
& e^+_{\scrA}(a_1|\dots|a_l) \longmapsto e_{\scrA}(a_1|\dots|a_l).
\end{aligned}
\end{equation}
Think of the decomposition $\scrA^+ = \scrA \oplus \bK(e^+_{\scrA} - e_{\scrA})$ (which is a direct sum of $A_\infty$-algebras). Then, \eqref{eq:collapse} is induced by projection to the first summand.

\begin{lemma} \label{th:collapse}
For any strictly unital $A_\infty$-algebra $\scrA$, the map \eqref{eq:collapse} is a quasi-isomorphism.
\end{lemma}

\begin{proof}[Sketch of proof]
By \eqref{eq:plus-complex}, the unreduced complex $C_*(\scrA)$ is contained in $C_*^+(\scrA)$ as a subcomplex. The quotient $C_*^+(\scrA)/C_*(\scrA)$, which is the second summand in \eqref{eq:plus-complex} with the induced differential, is acyclic. Therefore, the inclusion $C_*(\scrA) \subset C_*^+(\scrA)$ is a quasi-isomorphism. If we compose that inclusion with \eqref{eq:collapse}, we get the standard projection $C_*(\scrA) \rightarrow \bar{C}_*(\scrA)$, which is known to be a quasi-isomorphism.
\end{proof}

The same formula \eqref{eq:collapse} applies to cyclic homology. In the curved case, the resulting map $\mathit{CC}_*^+(\scrA_q) \rightarrow \overline{CC}_*(\scrA_q)$ is compatible with the connections on those groups. It is also, as a consequence of Lemma \ref{th:collapse}, a filtered quasi-isomorphism.

\subsubsection{$A_\infty$-categorical terminology\label{subsubsec:categories}}
We record here the straightforward extension to the case of $A_\infty$-categories, assumed to be small. (In all our applications the categories have finitely many objects, so one could equivalently think of them as algebras over a semisimple ring $\bK \oplus \cdots \oplus \bK$). The standard complexes underlying Hochschild homology and cohomology are (see Section \ref{subsec:conventions}(c) for notation)
\begin{equation} \label{eq:categorical-hochschild}
\begin{aligned}
& C_*(\scrA) = \bigoplus_{\substack{l \geq 0 \\ X_0,\dots,X_l}} \scrA(X_l,X_0,X_1,\dots,X_l)[l], \\
& C^*(\scrA) = \prod_{\substack{l \geq 0 \\ X_0,\dots,X_l}}
\mathit{Hom}(\scrA(X_0,\dots,X_l), \scrA(X_0,X_l))[-l],
\end{aligned}
\end{equation}
where the $l = 0$ term is $\bigoplus_{X_0} \scrA(X_0,X_0)$ respectively $\prod_{X_0} \scrA(X_0,X_0)$. To construct the non-unital version, one uses an enlarged category $\scrA^+$ which has an added unit $e_X^+$ for each object $X$. The outcome, following \eqref{eq:plus-complex}, is
\begin{equation} 
\label{eq:categorical-nonunital-hochschild}
C_*^+(\scrA) = C_*(\scrA) \oplus 
\bigoplus_{\substack{l > 0 \\ X_0,\dots,X_{l-1}}}
e_{X_0}^+ \otimes \scrA(X_0,X_1,\dots,X_{l-1},X_0)[l];
\end{equation}
the cyclic complex is $\mathit{CC}_*^+(\scrA) = \mathit{C}_*^+(\scrA)[[u]]$. In practice, the $e^+$ in \eqref{eq:categorical-nonunital-hochschild} merely serves as as a reminder to distinguish the two summands in $C^*_+(\scrA)$; we will therefore write it as $e_{X_0}^+ = e_{\scrA}^+$, freeing us from always having to keep track of the objects involved.  

Finally, given a curved $A_\infty$-deformation $\scrA_q$, we define $C_*(\scrA_q) = C_*(\scrA)[[q]]$, and similarly for the other homological invariants, with correspondingly deformed differentials.

\subsection{Fiber categories for curved deformations\label{subsec:fiber}}

\subsubsection{The fiber at zero\label{subsubsec:zero-fiber}}
Throughout this section, $\scrA$ is a strictly unital $A_\infty$-algebra, and $\scrA_q$ a strictly unital curved $A_\infty$-deformation. To make the notation a little lighter, we will often omit the subscript $\scrA$, writing $e$ for the unit, $\mu$ for the $A_\infty$-structure on $\scrA$, and $\mu_q$ for its $q$-deformation.

Let $\scrE$ be the graded algebra given by the endomorphisms of the graded vector space $\bK \oplus \bK[-1]$. The tensor product 
\begin{equation} \label{eq:f-fiber}
\scrB_0 = \scrA_q \otimes \scrE
\end{equation}
inherits the structure of a strictly unital curved $A_\infty$-algebra. We find it convenient to write elements of this tensor product as matrices
\begin{equation}
b = \begin{pmatrix} b^{11} & b^{12} \\ b^{21} & b^{22} \end{pmatrix}, \;\;
\text{ where $b^{11}, b^{22} \in \scrA_q^{|b|}$, $b^{21} \in \scrA_q^{|b|-1}$, $b^{12} \in \scrA_q^{|b|+1}$,}
\end{equation}
since that makes the formula for the $A_\infty$-operations intuitive: they combine those of $\scrA_q$ and matrix multiplication, with suitable Koszul signs (involving the reduced degree in $\scrA_q$, and the actual degree in $\scrE$). For instance,
\begin{equation}
\begin{aligned}
& \mu^2_{\scrA_q \otimes \scrE}
\Big( \begin{pmatrix} b_1^{11} & b_1^{12} \\ b_1^{21} & b_1^{22} \end{pmatrix},
\begin{pmatrix} b_2^{11} & b_2^{12} \\ b_2^{21} & b_2^{22} \end{pmatrix} \Big)
\\ & = \begin{pmatrix} 
\mu^2_{q}(b^{11}_1,b^{11}_2) + (-1)^{\|b^{21}_2\|}\mu^2_{q}(b^{12}_1,b^{21}_2) & 
(-1)^{\|b^{22}_2\|}\mu^2_{q}(b^{12}_1,b^{22}_2) + \mu^2_{q}(b^{11}_1,b^{12}_2) \\
(-1)^{\|b^{11}_2\|} \mu^2_{q}(b^{21}_1,b^{11}_2) + \mu^2_{q}(b^{22}_1,b^{21}_2) & 
\mu^2_{q}(b^{22}_1,b^{22}_2) + (-1)^{\|b^{12}_2\|} \mu^2_{q}(b^{21}_1,b^{12}_2) \\
\end{pmatrix}.
\end{aligned}
\end{equation}
Rather than using the given $A_\infty$-structure on the tensor product, we will deform it based on the Maurer-Cartan element
\begin{equation} \label{eq:delta-element}
\delta = \begin{pmatrix} 0 & qe \\ q^{-1}\mu^0_q & 0 \end{pmatrix}.
\end{equation}
Some explanation is necessary. By saying that $\delta$ is a Maurer-Cartan element, we mean that 
\begin{equation} \label{eq:delta-is-mc}
\mu^0_{\scrA_q \otimes \scrE} + \mu^1_{\scrA_q \otimes \scrE}(\delta) + \mu^2_{\scrA_q \otimes \scrE}(\delta,\delta) + \cdots = 0.
\end{equation}
Note that even though we have written the equation as an infinite sum, all terms other than $i = 0,2$ vanish in the case of \eqref{eq:delta-element}. By saying the $A_\infty$-structure is deformed using $\delta$, we mean that we consider
\begin{equation} \label{eq:mc-deformed-ainfty}
\mu^d_{\scrB_0}(b_1,\dots,b_d) = \sum_{i_0,\dots,i_d \geq 0} \mu_{\scrA_q \otimes \scrE}^{d+i_0+\cdots + i_d}\big(\delta^{\otimes i_0},b_1,\delta^{\otimes i_1},
b_2,\dots,b_d, \delta^{\otimes i_d}\big).
\end{equation}
Here, the notation is that $\delta^{\otimes i}$ means $i$ subsequent entries of $\delta$; and the sum is again finite. In those terms, the Maurer-Cartan equation \eqref{eq:delta-is-mc} says that $\mu^0_{\scrB_0} = 0$. The $d = 1$ case of \eqref{eq:mc-deformed-ainfty} becomes
\begin{equation} \label{eq:f-differential-0}
\mu^1_{\scrB_0}(b) = \mu^1_{\scrA_q \otimes \scrE}(b) + \mu^2_{\scrA_q \otimes \scrE}(\delta,b) + \mu^2_{\scrA_q \otimes \scrE}(b,\delta) + \mu^3_{\scrA_q \otimes \scrE}(\delta,b,\delta),
\end{equation}
since all terms with two adjacent $\delta$ vanish. It is useful to spell this out:
\begin{equation} \label{eq:f-differential}
\begin{aligned}
& \mu^1_{\scrB_0}\begin{pmatrix} b^{11} & b^{12} \\ b^{21} & b^{22} \end{pmatrix} 
= 
\\ & 
\left(\begin{matrix}
\mu^1_{q}(b^{11}) + (-1)^{\|b^{21}\|}q b^{21} 
- q^{-1}\mu^2_{q}(b^{12},\mu^0_q) 
\\
\mu_{q}^1(b^{21}) + (-1)^{\|b^{11}\|} q^{-1}\mu^2_{q}(\mu^0_q,b^{11}) + q^{-1}\mu_q^2(b^{22},\mu_q^0) 
+ (-1)^{\|b^{12}\|} q^{-2}\mu^3_q(\mu_q^0,b^{12},\mu_q^0)
\end{matrix} \right.
\\ & \qquad \qquad \qquad \qquad \qquad \qquad \qquad
\left.\begin{matrix}
\mu^1_q(b^{12}) + (-1)^{\|b^{22}\|}qb^{22} - (-1)^{\|b^{11}\|} qb^{11}
\\
\mu^1_q(b^{22}) + (-1)^{\|b^{12}\|} q^{-1}\mu^2_q(\mu^0_q,b^{12}) + (-1)^{\|b^{21}\|} qb^{21}
\end{matrix}\right). 
\end{aligned}
\end{equation}
The formulae for $\mu^d_{\scrB_0}$, $d \geq 2$, are actually a bit simpler, because the identity term in \eqref{eq:delta-element} no longer contributes at all, which means that that one can replace $\delta$ by 
\begin{equation}
\tilde{\delta} = \begin{pmatrix} 0 & 0 \\ q^{-1}\mu^0_q & 0 \end{pmatrix}.
\end{equation}
The first of these formulae is
\begin{equation} \label{eq:f-composition-2}
\begin{aligned} &
\mu^2_{\scrB_0}(b_1,b_2) =
\mu^2_{\scrA_q \otimes \scrE}(b_1,b_2) \\ & \quad + 
\mu^3_{\scrA_q \otimes \scrE}(\tilde{\delta},b_1,b_2) + \mu^3_{\scrA_q \otimes \scrE}(b_1,\tilde{\delta},b_2) +
\mu^3_{\scrA_q \otimes \scrE}(b_1,b_2,\tilde{\delta}) + 
\\ & \quad \mu^4_{\scrA_q \otimes \scrE}(\tilde{\delta},b_1,\tilde{\delta},b_2) +
\mu^4_{\scrA_q \otimes \scrE}(\tilde{\delta},b_1,b_2,\tilde{\delta}) + \mu^4_{\scrA_q \otimes \scrE}(b_1,\tilde{\delta},b_2,\tilde{\delta}) + \mu^5_{\scrA_q \otimes \scrE}(\tilde{\delta},b_1,\tilde{\delta},b_2,\tilde{\delta}).
\end{aligned}
\end{equation}

The structure of $\scrB_0$ may look mysterious, but is actually related to the previously considered \eqref{eq:cofibrant-t}, or more precisely to its $t = 0$ specialization. To see that, take $\epsilon$ to be a formal variable of degree $-1$, and consider the chain complex 
\begin{equation} \label{eq:tilde-f}
\begin{aligned}
& \scrA_0 = \scrA[\epsilon], \\
& \mu^1_{\scrA_0}(a) = \mu^1(a), \\
& \mu^1_{\scrA_0}(a\epsilon) = \mu^1(a)\epsilon - \mu^2(a,\mu_{q}^{0,(1)}).
\end{aligned}
\end{equation}
Here, $\mu_{q}^{0,(1)}$ is the $q$-linear term of $\mu_{q}^0$, which lies in $\scrA^0$. That term is a $\mu^1$-cocycle, and therefore the differential \eqref{eq:tilde-f} squares to zero. We introduce maps
\begin{align} \label{eq:transfer-start}
&
\begin{aligned}
& i: \scrA_0 \longrightarrow \scrB_0, \quad 
i(a) = \begin{pmatrix} 
a & 0 \\ 
(-1)^{\|a\|} q^{-1}(\mu^1_q a - \mu^1 a) & a 
\end{pmatrix}, \\
& \quad i(a \epsilon) = \begin{pmatrix} 0 & a \\ 
(-1)^{\|a\|} q^{-1}\big(\mu^2_q(a,q^{-1}\mu^0_q) - \mu^2(a,\mu_{q}^{0,(1)}) \big)
& (-1)^{\|a\|} q^{-1}(\mu^1_q a - \mu^1 a)\end{pmatrix}
\end{aligned}
\\ &
p: \scrB_0 \longrightarrow \scrA_0, 
\quad
p\begin{pmatrix} b^{11} & b^{12} \\ b^{21} & b^{22} \end{pmatrix} = b^{11,(0)} + b^{12,(0)} \epsilon,
\\ & \label{eq:transfer-end}
\begin{aligned} &
h: \scrB_0 \longrightarrow \scrB_0[-1], \\ & \quad
h\begin{pmatrix} b^{11} & b^{12} \\ b^{21} & b^{22} \end{pmatrix} = 
\begin{pmatrix} 
0 & 
0 \\
-(-1)^{\|b^{11}\|} q^{-1}(b^{11}-b^{11,(0)}) &
-(-1)^{\|b^{12}\|} q^{-1}(b^{12}-b^{12,(0)})
\end{pmatrix} 
\end{aligned}
\end{align}
where $b^{ij,(0)}$ stands for the $q$-constant term of $b^{ij}$. These maps satisfy
\begin{align}
& \mu^1_{\scrB_0} i = i \mu^1_{\scrA_0}, \\
& \mu^1_{\scrA_0} p = p \mu^1_{\scrB_0}, \\
& pi = \mathit{id}_{\scrA_0}, \\
& ip = \mathit{id}_{\scrB_0} + \mu^1_{\scrB_0} h + h \mu^1_{\scrB_0}. 
\end{align}
Starting with that, one can apply the Homological Perturbation Lemma to equip $\scrA_0$ with higher $A_\infty$-operations which make it quasi-isomorphic to $\scrB_0$:
\begin{equation} \label{eq:tilde-f-structure}
\mu^d_{\scrA_0} = p \big(\mu^d_{\scrB_0} + \text{terms involving $h$ as well as $\mu_{\scrB_0}$}\big)i^{\otimes d}. 
\end{equation}
This can be conveniently expressed as a sum over planar trees (\cite[Section 6.4]{kontsevich-soibelman00}; for the signs, or rather absence thereof, see e.g.\ \cite[Remark 3.1]{seidel03b}), where the first summand in \eqref{eq:tilde-f-structure} comes from the tree with a single vertex. The following discussion requires the reader to have that formulation in mind.

\begin{lemma} \label{th:structure-0}
For $a_1,\dots,a_d \in \scrA \subset \scrA_0$, we have
\begin{equation} \label{eq:simple-mu}
\mu_{\scrA_0}^d(a_1,\dots,a_d) = \mu^d_{\scrA}(a_1,\dots,a_d).
\end{equation}
\end{lemma}

\begin{proof}
All steps in the computation involve morphisms in $\scrB_0$ given by lower triangular matrices. The ingredients are (for the sake of brevity, we have replaced terms that are irrelevant for our argument by ?):
\begin{align}
\label{eq:all-triangle}
& \mu^d_{\scrB_0}\big( \begin{pmatrix} b_1^{11} & 0 \\ ? & ? \end{pmatrix},
\dots,\begin{pmatrix} b_d^{11} & 0 \\ ? & ? \end{pmatrix}\big) =
\begin{pmatrix} \mu^d_q(b_1^{11},\dots,b_d^{11}) & 0 \\ ? & ? \end{pmatrix},\;\; d \geq 2, \\
& i(a) = \begin{pmatrix} a & 0 \\ ? & ? \end{pmatrix}, \;\;
h\begin{pmatrix} ? & 0 \\ ? & ? \end{pmatrix}
= \begin{pmatrix} 0 & 0 \\ ? & 0 \end{pmatrix}, \;\;
p\begin{pmatrix} b^{11} & 0 \\ ? & ? \end{pmatrix} = b^{11,(0)}.
\end{align}
From that, one sees that the single-vertex tree contributes the expression on the right-hand side of \eqref{eq:simple-mu}. For any other tree, there is a finite edge which corresponds to an occurrence of $h$, whose output has zero upper left entry. That leads to an overall output with the same property at the root of the tree, which is then mapped to zero by $p$.
\end{proof}

\begin{lemma} \label{th:structure-1}
For $a_1,\dots,a_d \in \scrA \subset \scrA_0$ and any $k$, we have
\begin{equation} \label{eq:epsilon-term}
\begin{aligned}& 
\mu_{\scrA_0}^d(a_1,\dots,a_{k-1},a_k \epsilon,a_{k+1},\dots,a_d) =
(-1)^{\|a_{k+1}\|+\cdots+\|a_d\|} \mu^d(a_1,\dots,a_d) \epsilon
\\ & 
- \sum_{j \geq k} (-1)^{\|a_{k+1}\|+\cdots+\|a_j\|} \mu^{d-l+1}\big(a_1,\dots,a_k,\dots,a_j,
\\[-1em] & \qquad \qquad \qquad \qquad \qquad \qquad
\mu^{l,(1)}_q(a_{j+1},\dots,a_{j+l}),a_{j+l+1},\dots,a_d\big);
\end{aligned}
\end{equation} 
note that in the second expression, $a_k$ must lie to the left of $\mu^{l,(1)}_q$.
\end{lemma}

\begin{proof}
We begin by computing two of the matrix entries of $\mu^d_{\scrB_0}(i a_1,\dots,i(a_k \epsilon),\dots, ia_d)$, $d \geq 2$:
\begin{equation}
\mu^d_{\scrB_0}(i a_1,\dots,i(a_k \epsilon),\dots,i a_d)^{12} =
(-1)^{\|a_{k+1}\|+\cdots+\|a_d\|}\mu^d_q(a_1,\dots,a_d),
\end{equation}
and
\begin{equation}
\begin{aligned}
& \mu^d_{\scrB_0}(i a_1,\dots, i(a_k \epsilon),\dots,i a_d)^{11}
= \mu^d_{\scrA_q \otimes \scrE}(i a_1,\dots,i(a_k \epsilon),\dots, i a_d)^{11}
\\ & \quad
+ \sum_{j \geq k} \mu^{d+1}_{\scrA_q \otimes \scrE}(i a_1,\dots,i(a_k\epsilon),\dots,a_j,\tilde{\delta},a_{j+1},\dots,a_d)^{11}
\\ & 
= -\sum_{j \geq k} (-1)^{\|a_{k+1}\|+\cdots+\|a_j\|} \big(\mu_q^d(a_1,\dots, a_k,\dots,q^{-1}(\mu_q^1 a_{j+1} - \mu^1 a_{j+1}),\dots, a_d) 
\\[-.5em] & \quad \qquad \qquad
- \mu_q^{d+1}(a_1\dots,a_k,\dots,a_j,q^{-1}\mu_q^0,a_{j+1},\dots,a_d) \big).
\end{aligned}
\end{equation}
Hence, the contribution from the one-vertex tree is
\begin{equation} \label{eq:1-vertex}
\begin{aligned}
& p \mu^d_{\scrB_0}(i a_1,\dots,i(a_k \epsilon),\dots,i a_d) =
(-1)^{\|a_{k+1}\|+\cdots+\|a_d\|} \epsilon\mu^d(a_1,\dots,a_d) \\
& \quad
- \sum_{j \geq k} (-1)^{\|a_{k+1}\|+\cdots+\|a_j\|} \big(\mu^d(a_1,\dots,a_k,\dots,a_j,\mu^{1,(1)}_q(a_{j+1}),\dots,a_d)
\\[-1em] & \qquad \qquad \qquad \qquad
+ \mu^{d+1}(a_1,\dots,a_k,\dots,a_j,\mu^{0,(1)}_q,\dots,a_d) \big).
\end{aligned}
\end{equation}

Let's look at the contributions from other trees. As before, $h\mu^d_{\scrB_0}(i a_1,\dots,i(a_k\epsilon),\dots,i a_d)$ is lower-triangular and has vanishing upper left entry, and obviously the same is true for any sub-expression $h\mu^l_{\scrB_0}(i a_{j+1},\dots,i(a_k \epsilon),\dots, ia_{j+l})$ ($l \geq 2$). For the same reason as in Lemma \ref{th:structure-0}, this means that any term arising from the Perturbation Lemma which contains such an expression contributes zero to our computation. Similarly, any tree with more than two vertices cannot contribute. The remaining terms are (again with $l \geq 2$):
\begin{align}
\label{eq:term-1}
& p \mu^{d-l+1}_{\scrB_0}(ia_1,\dots,h\mu_{\scrB_0}^l(ia_{j+1},\dots,ia_{j+l}), \dots, i(a_k \epsilon),\dots, ia_d) = 0; \\
\label{eq:term-2}
& p \mu^{d-l+1}_{\scrB_0}(ia_1,\dots,i(a_k \epsilon), \dots, h\mu_{\scrB_0}^l(ia_{j+1},\dots,ia_{j+l}),\dots,ia_d) \\
\notag & \qquad = -(-1)^{\|a_{k+1}\| + \cdots + \|a_j\|}
\mu^{d-l+1}(a_1,\dots, a_k,\dots,\mu^{l,(1)}_q(a_{j+1},\dots,a_{j+l}),\dots,a_d).
\end{align}
The sum of \eqref{eq:1-vertex} and \eqref{eq:term-2} precisely gives the desired \eqref{eq:epsilon-term}.
\end{proof}

\begin{example} \label{th:aw-fiber}
Suppose that we have $(\scrA,W)$ as in Section \ref{subsubsec:dga-setup}, and the associated $\scrA_q$. In that case, \eqref{eq:tilde-f} says that
\begin{equation}
\mu^1_{\scrA_0}(a\epsilon) = (d_{\scrA}a)\epsilon - W\partial_\epsilon(a\epsilon),
\end{equation}
and we also have
\begin{equation}
\mu^2_{\scrB_0} = \mu^2_{\scrA \otimes \scrE}, \;\; \mu^d_{\scrB_0} = 0 \text{ for $d>2$.}
\end{equation}
Moreover, \eqref{eq:transfer-start} simplifies to
\begin{equation}
i(a) = \begin{pmatrix} a & 0 \\ 0 & a \end{pmatrix}, \;\; i(a\epsilon) = \begin{pmatrix} 0 & a \\ 0 & 0 \end{pmatrix}.
\end{equation}
As a consequence, $h \circ \mu^2_{\scrB_0} \circ (i \otimes i) = 0$. This implies that $\mu^2_{\scrA_0} = \mu^2_{\scrA[\epsilon]}$, and $\mu^d_{\scrA_0} = 0$ for $d>2$. In other words, we get precisely the dga from \eqref{eq:cofibrant-t} specialized to $t = 0$, which explains our notation.
\end{example}

\subsubsection{The general fiber\label{subsubsec:general-fiber}}
We define a curved $A_\infty$-algebra $\scrC_{t,q} = \scrA[t][[q]]$ over $\bK[t]$, by setting
\begin{equation}
\mu^0_{\scrC_{t,q}} = \mu^0_{\scrA_q} - qt \, e_\scrA,
\end{equation}
and with all other operations extended $t$-linearly from $\scrA_q$. Generalizing \eqref{eq:f-fiber}, we consider
\begin{equation}
\scrB_t = \scrC_{t,q} \otimes \scrE,
\end{equation}
with the $A_\infty$-structure deformed by the Maurer-Cartan element
\begin{equation}
\delta_t = 
\begin{pmatrix} 0 & qe_{\scrA} \\ q^{-1}\mu^0_q - t  e_{\scrA} & 0 \end{pmatrix}. 
\end{equation}
Concretely, the differential here has an additional $t$-dependent term
\begin{equation} \label{eq:f-differential-2}
(\mu^1_{\scrB_t} - \mu^1_{\scrB_0})\begin{pmatrix} b^{11} & b^{12} \\ b^{21} & b^{22} \end{pmatrix} 
= t\begin{pmatrix}
-(-1)^{\|b^{21}\|} b^{21} & 0 
\\ (-1)^{\|b^{22}\|} b^{22} - (-1)^{\|b^{11}\|} b^{11} & -(-1)^{\|b^{12}\|} b^{12}
\end{pmatrix};
\end{equation}
while $\mu^d_{\scrB_t}$, $d \geq 2$, are the $t$-linear extensions of the corresponding operations in $\scrB_0$. Following \eqref{eq:tilde-f}, we define
\begin{equation} \label{eq:tilde-mu1-s}
\begin{aligned}
& \scrA_t = \scrA[t,\epsilon], \\
& \mu^1_{\scrA_t}(a t^k) = \mu^1_{\scrA}(a) t^k, \\
& \mu^1_{\scrA_t}(a t^k \epsilon) = \mu^1_{\scrA}(a) t^k \epsilon
- \mu^2(a,\mu_q^{0,(1)})t^k + a t^{k+1}.
\end{aligned}
\end{equation}
Consider the increasing filtration of $\scrA_t$ given by those $a_1(t) + a_2(t) \epsilon$ where $a_1(t)$ is a polynomial in $t$ of degree $\leq k$, and $a_2(t)$ of degree $<k$. The initial term of that filtration is $\scrA$; and it follows from \eqref{eq:tilde-mu1-s} that all subsequent quotients of the filtration are acyclic. From that, we see:

\begin{lemma} \label{th:easy-acyclic}
The inclusion $\scrA \subset \scrA_t$ is a quasi-isomorphism.
\end{lemma}

One can use the same formulae as in \eqref{eq:transfer-start}--\eqref{eq:transfer-end} to transfer the $A_\infty$-structure from $\scrB_t$ to $\scrA_t$. Because the operations $\mu_{\scrB_t}^d$, $d \geq 2$, are $t$-independent, the same will then hold for the corresponding operations in $\scrA_t$. In other words, the only part of the $A_\infty$-structure of $\scrA_t$ which is $t$-dependent is the differential \eqref{eq:tilde-mu1-s}. In particular, Lemmas \ref{th:structure-0} and \ref{th:structure-1}, for $d \geq 2$, carry over. Combining the analogue of Lemma \ref{th:structure-0} with Lemma \ref{th:easy-acyclic}, we see that $\scrA \subset \scrA_t$ is a quasi-isomorphic $A_\infty$-subalgebra.

\begin{example}
Continuing the discussion from Example \ref{th:aw-fiber}, in that case $\scrA_t$ is exactly \eqref{eq:cofibrant-t}.
\end{example}

\subsubsection{Reduction to the dga case}
Let's return briefly to the $A_\infty$-algebra $\scrA_0$ from Section \ref{subsubsec:zero-fiber}. Applying the definition from Section \ref{subsubsec:deformation-theory} to the particularly simple case of the inclusion $\scrA \subset \scrA_0$, we have a Hochschild complex $\bar{C}^*(\scrA,\scrA_0)$, with maps \eqref{eq:hochschild-mixed-functoriality}.

\begin{lemma}
The primary deformation class $[\mu_{\scrA_q}^{(1)}] \in \mathit{HH}^0(\scrA)$ maps to zero in $\mathit{HH}^0(\scrA,\scrA_0)$.
\end{lemma}

\begin{proof}
Let $\eta \in \bar{C}^{-1}(\scrA,\scrA_0)$ be the Hochschild cochain whose only nonzero term is the constant $\eta^0 = e_{\scrA} \epsilon$. By \eqref{eq:diff-c-a} and (a particularly simple special case of) Lemma \ref{th:structure-1},
\begin{equation}
(d_{\bar{C}^*(\scrA,\scrA_0)}\eta)^l(a_1,\dots,a_l) = \sum_i \mu_{\scrA_0}^{l+1}(a_1,\dots,a_i,e_{\scrA} \epsilon,a_{i+1},\dots,a_l) = -\mu_{\scrA_q}^{l,(1)}(a_1,\dots,a_l).
\end{equation}
\end{proof}

In the corresponding situation from Section \ref{subsubsec:general-fiber}, the analogue for the inclusion $\scrA \subset \scrA_t$ is:

\begin{lemma} \label{th:primary-class-s}
The primary deformation class $[\mu_{\scrA_q}^{(1)}]$ maps to $[t\,e_\scrA] \in \mathit{HH}^0(\scrA,\scrA_t)$.
\end{lemma}

\begin{proof}
We use the same $\eta$, but pick up an additional term from $\mu^1_{\scrA_t}(e_{\scrA}\epsilon)$:
\begin{equation}
(d_{\bar{C}^*(\scrA,\scrA_t)}\eta)^l(a_1,\dots,a_l) = \begin{cases} 
t\,e_{\scrA} -\mu_{\scrA_q}^{0,(1)} & l = 0, \\
-\mu_{\scrA_q}^{l,(1)}(a_1,\dots,a_l) & l>0.
\end{cases}
\end{equation}
\end{proof}

\begin{proposition} \label{th:first-order}
Take a strictly unital $\scrA_q$, such that $\mathit{HH}^{2k}(\scrA) = 0$ for all $k < 0$. Then this is filtered quasi-isomorphic to the curved $A_\infty$-algebra obtained from some dga and a central element, as in Section \ref{subsubsec:dga-setup}.
\end{proposition}

\begin{proof}
Consider $\scrA_t$ as an $A_\infty$-algebra over $\bK[t]$. One can apply the Yoneda embedding over that ring, to obtain a differential graded algebra $\scrA'_t$ and a $\bK[t]$-linear $A_\infty$-quasi-isomorphism $\scrA_t \rightarrow \scrA'_t$. For this to work, it is important that $\scrA_t$ is a free graded $\bK[t]$-module, which is true by definition; the resulting $\scrA'_t$ may not have that property, since it's defined as an infinite product, but that is irrelevant for our purposes. Indeed, we will now forget the $\bK[t]$-linear structure, and just consider $\scrA_t'$ as a dga over $\bK$. Combining this with Lemma \ref{th:easy-acyclic}, we have quasi-isomorphisms
\begin{equation} \label{eq:combo-quasi-iso}
\scrA \stackrel{\htp}{\longrightarrow} \scrA_t \stackrel{\htp}{\longrightarrow} \scrA'_t.
\end{equation}
Under the induced isomorphism (of Hochschild cohomologies formed over $\bK$)
\begin{equation}
\mathit{HH}^*(\scrA) \iso \mathit{HH}^*(\scrA_t) \iso \mathit{HH}^*(\scrA'_t), 
\end{equation}
the element $[\mu_{\scrA_q}^{(1)}]$ goes to $[t\,e_{\scrA_t}]$ and then to $[t\, e_{\scrA'_t}]$; the first part is by Lemma \ref{th:primary-class-s}, and the second part follows from the fact that the underlying quasi-isomorphism was $\bK[t]$-linear. At this point, we can apply Lemma \ref{th:def-theory}, which says that the quasi-isomorphism $\scrA \rightarrow \scrA'_t$ can be extended to a curved $A_\infty$-homomorphism from $\scrA_q$ to the deformation of $\scrA'_t$ obtained by turning on the curvature term $qW'$, where $W' = t \, e_{\scrA'_t}$.
\end{proof}

\subsubsection{Conclusion\label{subsubsec:manipulate-a}}
With these techniques at hand, we can now extend the results of Section \ref{subsubsec:monodromy} to $A_\infty$-algebras. With a view to applications in Floer theory, we will formulate the outcome in the cohomologically unital context. Namely, given a cohomologically unital $\scrA$ and a curved deformation $\scrA_q$, consider
\begin{equation}
A_{q,u} = \mathit{CC}^+_*(\scrA_q),
\end{equation}
with its Getzler-Gauss-Manin connection. This is a complete $u$-torsionfree dg module over $W_{q,u} \iso W_{t,u}$, and we can apply the same manipulation to it as in \eqref{eq:cc-periodic-final}, leading to an appropriate version of $q^{-1} A_{q^{-1},1/p,u^{\pm 1}}$. 

\begin{corollary} \label{th:end-of-algebra}
Take $\bK = \bC$. Let $p(t)$ be a nonzero polynomial. Suppose that:
\begin{itemize} \itemsep.5em
\item[(i)] $\scrA$ is smooth.
\item[(ii)] $q^r p([\kappa_{\scrA_q}]) \in \mathit{HH}^{2r}(\scrA_q)$ vanishes for some $r>0$. 
\item[(iii)] $\bK[t,1/p] \otimes_{\bK[t]} H^*(\scrA)$, where $t$ acts by multiplication with $[\mu_{\scrA_q}^{0,(1)}]$, is a finitely generated $\bK[t,1/p]$-module.
\item[(iv)] $\mathit{HH}^{2k}(\scrA) = 0$ for $k<0$.
\end{itemize}
Then, in each degree, the cohomology of $q^{-1} A_{q^{-1},1/p,u^{\pm 1}}$ is finitely generated over $\bC[t,1/p]$, where $t = \nabla_{u\partial_q}$. On that cohomology, $\nabla_{\partial_t} = -q/u$ is a connection with regular singularities and quasi-unipotent monodromy around each singular point (including $\infty$). If in addition,
\begin{itemize}
\item[(v)] $\mathit{HH}_*(\scrA)$ is concentrated in degrees $\leq d$,
\end{itemize}
the monodromy of around each singularity (again including $\infty$) has Jordan blocks of size $\leq d+1$.
\end{corollary}

\begin{proof}
In view of (iv), Lemma \ref{th:make-strictly-unital}, and Proposition \ref{th:first-order}, the given $\scrA_q$ is filtered quasi-isomorphic to the deformation $\scrA_q'$ obtained from a dga $\scrA'$ (our previous $\scrA'_t$, where we forget the $\bK[t]$-linear structure and accordingly adjust the notation) and central element $W'$. Moreover, (i)-(iii) carry over to that dga; the only nontrivial issue is (ii), but that is taken care of by (the cohomologically unital version of) Lemma \ref{th:phi-q}. Corollary \ref{th:monodromy} gives results parallel to the ones stated here, but for the cohomology of an appropriately manipulated version of $\overline{CC}_*(\scrA_q')$. As in Lemma \ref{th:collapse} and the discussion following it, one can use $\mathit{CC}_*^+(\scrA_q')$ instead. We now need to bring the results back to $\scrA_q$. Lemma \ref{th:homotopy-dq} and \ref{th:induced-map} show that $\mathit{CC}_*^+(\scrA_q)$ and $\mathit{CC}_*^+(\scrA_q')$ are isomorphic objects of $D(W_{q,u})$. Passing to negative powers of $q$, and inverting $p(t)$, are both operations within that category. The final step leading to $q^{-1} A_{q^{-1},1/p,u^{\pm 1}}$ is the purely algebraic inversion of $u$, and that is of course compatible with quasi-isomorphisms. The proof of the last part, involving (v), uses the same argument and Corollary \ref{th:unipotence-bound}.
\end{proof}

As usual, we have formulated the discussion for $A_\infty$-algebras, but the translation to $A_\infty$-categories (see Section \ref{subsubsec:categories}) is straightforward.

\begin{remark}
It is unclear whether assumption (iv) is more than a technical one. We have used it in the proof of Proposition \ref{th:first-order}, to determine the deformation of $\scrA'_t$ which corresponds to $\scrA_q$, but it is possible that the Proposition holds more generally. Alternatively (requiring changes on a larger scale), one could try to carry out something like the Fourier-Laplace transform argument from Section \ref{subsec:dga} directly for curved $A_\infty$-deformations.
\end{remark}

\subsection{$L_\infty$-formalism\label{subsec:linfty}}

\subsubsection{$L_\infty$-algebras} 
We begin by recalling the basic definitions (see e.g.\ \cite[Section 4]{getzler95}). Fix a ground field $\bigK$ of characteristic zero. Let $\mathit{Sh}(k, m) \subset \mathit{Sym}(m)$ denote the set of $k$-shuffles, meaning permutations $\sigma$ of $\{1,\dots,m\}$ which satisfy 
\begin{equation} \label{eq:shuffle2}
\sigma(1)< \cdots < \sigma(k) \text{ and } \sigma(k+1) < \cdots < \sigma(m).
\end{equation}

An $L_\infty$-structure on a graded $\bK$-vector space $\scrG$ is a sequence of linear maps $\ell^m: \scrG^{\otimes m} \rightarrow \scrG$, $m \geq 1$, of degree $2-m$. They must be graded symmetric when viewed as defined on the shifted space $\scrG[-1]$, which means that
\begin{equation} \label{eq:gradedsymmetry} 
\ell^m(x_1,\dots,x_i,x_{i+1},\dots,x_m) = (-1)^{\|x_i\|\,\|x_{i+1}\|}
\ell^d(x_1,\dots,x_{i+1},x_{i},\dots, x_m);
\end{equation} 
and they should satisfy the $L_\infty$-relations,
\begin{equation} \label{eq:Linfty} 
\sum_{\sigma \in \mathit{Sh}(k, m)} (-1)^{\dagger}\ell^{m-k+1}(\ell^{k}(x_{\sigma(1)}, \dots, x_{\sigma(k)}), x_{\sigma(k+1)},\dots, x_{\sigma(m)}) = 0,
\end{equation}
where $(-1)^{\dagger}$ is the Koszul sign associated to $\sigma$ acting on $(\scrG[1])^{\otimes m}$. Note that because of the previous condition, one can equivalently write this as
\begin{equation}
\sum_{\sigma \in \mathit{Sym}(m)} (-1)^{\dagger} \textstyle \frac{1}{k!(m-k)!} \ell^{m-k+1}(\ell^{k}(x_{\sigma(1)}, \dots, x_{\sigma(k)}), x_{\sigma(k+1)},\dots, x_{\sigma(m)}) = 0. 
\end{equation}

\begin{remark}
The sign conventions here are convenient for Floer theory, but less so when it comes to comparing the situation with classical Lie theory. One could instead use
\begin{equation}
\tilde{\ell}^m(x_1,x_2, \dots, x_m) \stackrel{\mathrm{def}}{=} (-1)^{\sum_i (m-i)|x_i|} \ell^m(x_1,x_2,\dots,x_m),
\end{equation}
which satisfies the symmetry condition
\begin{equation}  
\tilde{\ell}^m(x_1,\dots,x_i,x_{i+1},\dots, x_m) = -(-1)^{|x_i|\,|x_{i+1}|} \tilde{\ell}^m(x_1,\dots,x_{i+1},x_{i},\dots,x_m).
\end{equation}
Writing $dx = \tilde{\ell}^1(x) = \ell^1(x)$ and $[x_1,x_2] = \tilde{\ell}^2(x_1,x_2) = (-1)^{|x_1|} \ell^2(x_1,x_2)$, the $m=2,3$ equations in \eqref{eq:Linfty} turn into 
\begin{align}
& d[x_1,x_2] = [dx_1,x_2]+(-1)^{|x_1|} [x_1,dx_2], \\
\label{eq:jacobi}
& d\tilde{\ell}^3(x_1,x_2,x_3) + \tilde{\ell}^3(dx_1,x_2,x_3]) + \tilde{\ell}^3(x_1,dx_2,x_3) + \tilde{\ell}^3(x_1,x_2,dx_3) \\
\notag & \qquad +
[[x_1,x_2],x_3] + (-1)^{|x_1|(|x_2|+|x_3|)} [[x_2,x_3],x_1] + (-1)^{|x_3|(|x_1|+|x_2|)} [[x_3,x_1],x_2] = 0,
\end{align}
where the latter is a homotopical version of the Jacobi identity. Hence, an $L_\infty$-algebra with $\ell^m = 0$ for $m \geq 3$ is the same as a dg Lie algebra; and for a general $L_\infty$-algebra, $H^*(\scrG)$ inherits a graded Lie algebra structure, induced by $[x_1,x_2]$.
\end{remark}

An $L_\infty$-module over $\scrG$ is a graded vector space $\scrM$ with operations $\ell^{m,1}: \scrG^{\otimes m} \otimes \scrM \rightarrow \scrM[1-m]$, $m \geq 0$. They must be symmetric in the first $m$ entries, in the same sense as in \eqref{eq:gradedsymmetry}, and satisfy 
\begin{equation} \label{eq:LinftyM} 
\begin{aligned}
&
0 = \sum_{\sigma \in \mathit{Sh}(k, m)} (-1)^{\dagger}\, \ell^{m-k+1,1}(\ell^{k}(x_{\sigma(1)}, \dots, x_{\sigma(k)}), x_{\sigma(k+1)},\cdots, x_{\sigma(m)},y)\\ 
& 
+ \sum_{\sigma \in \mathit{Sh}(k, m)} (-1)^{\dagger + \|x_{\sigma(1)}\|+\cdots+\|x_{\sigma(k)}\|}\, \ell^{k,1}(x_{\sigma(1)},\dots,x_{\sigma(k)}, \ell^{m-k,1}(x_{\sigma(k+1)},\dots,x_{\sigma(m)},y)).
\end{aligned}
\end{equation}
Note that in the second sum, $k = 0$ and $k = m$ are both allowed.
%

\begin{example} \label{ex:diagonal} (The diagonal module)
Take $\scrM = \scrG$ with $\ell^{m,1} = \ell^{m+1}$. Any $\sigma \in \mathit{Sh}(k,m+1)$ satisfies $\sigma(m+1) = m+1$ or $\sigma(k) = m+1$, and this yields an identification $\mathit{Sh}(k,m+1) \iso \mathit{Sh}(k,m) \sqcup \mathit{Sh}(k-1,m)$. Using that, \eqref{eq:Linfty} turns into \eqref{eq:LinftyM}.
\end{example} 

\subsubsection{Maurer-Cartan theory\label{sec:mcelements}}
We have encountered Maurer-Cartan elements for $A_\infty$-algebras, as a technical ingredient in Section \ref{subsubsec:general-fiber}. A version of that notion in the  $L_\infty$ framework plays a much more fundamental role in our geometric constructions later on. Let $\scrG$ be an $L_\infty$-algebra, and $q$ a formal variable of degree $2$. A Maurer-Cartan element \cite[Definition 4.3]{getzler95} is
\begin{equation} \label{eq:lie-mc}
\alpha \in (q\scrG[[q]])^1, \quad
\sum_{m \geq 1} {\textstyle \frac{1}{m!}} \ell^{m}(\alpha^{\otimes m}) = 0 \in (q\scrG[[q]])^2.
\end{equation}
Any such $\alpha$ gives rise to a formal deformation of the $L_\infty$-structure on $\scrG[[q]]$, namely \cite[Proposition 4.4]{getzler95}
\begin{equation} \label{eq:deform-the-linfty-algebra}
\ell^m_{\alpha}(x_1,\dots,x_m) \stackrel{\mathrm{def}}{=} \sum_{j \geq 0} {\textstyle \frac{1}{j!}} \ell^{m+j}(\alpha^{\otimes j},x_1,\dots,x_m)
\end{equation}
(because $\alpha$ is of odd degree, it doesn't really matter where we insert $\alpha^{\otimes j}$ in the entries here). By differentiating \eqref{eq:lie-mc} with respect to $q$, one obtains: 
\begin{equation} \label{eq:ell1alpha}
\ell^1_\alpha(\partial_q\alpha) = \sum_{m \geq 1} {\textstyle\frac{1}{(m-1)!}} 
\ell^m(\alpha^{\otimes m-1}, \partial_q\alpha)
 = \sum_{m \geq 1} \textstyle{\frac{1}{m!}} \partial_q \ell^m(\alpha^{\otimes m}) = 0.
\end{equation}  
More generally, a Maurer-Cartan element deforms the structure of an arbitrary $L_\infty$-module, in the sense that the operations on $\scrM[[q]]$ defined by
\begin{equation} \label{eq:deform-the-linfty-module}
\ell^{m,1}_\alpha(x_1,\cdots,x_m,y) \stackrel{\mathrm{def}}{=} \sum_{j\geq 0} \textstyle{\frac{1}{j!}} \ell^{m+j,1}(\alpha^{\otimes j},x_1,\cdots,x_m,y)
\end{equation}
satisfy the $L_\infty$-module relations. It is again useful to look at what one gets by differentiating this formula. In the simplest instance, 
\begin{equation} \label{eq:commutationalphaeq}
\partial_q \ell^{0,1}_{\alpha}(y) - 
\ell^{0,1}_{\alpha}(\partial_q y) =
\sum_{m \geq 1} \textstyle{\frac{1}{(m-1)!}} \ell^{m,1}(\alpha^{\otimes m-1},\partial_q\alpha,y) = \ell^{1,1}_{\alpha}(\partial_q\alpha,y).
\end{equation}

\section{Floer theory preliminaries\label{sec:floer}}
This section sets up a version of Floer cohomology on Liouville manifolds, using Hamiltonians with quadratic growth. This allows us to work with a single Hamiltonian (as opposed to a sequence of Hamiltonians with linear growth and increasing slopes). For Floer-theoretic operations with more than one input, the Hamiltonian at the output end cannot be equal to that at the inputs, but one can arrange that the Floer complexes are still the same, by applying a rescaling trick to the Liouville manifold as in \cite{abouzaid10}. On a more technical level, we follow \cite{ganatra13} but modify the construction by adding ``unperturbed shells'', in order to reduce analytic aspects to their most elementary part; the downside being that the choices of Hamiltonians, and of almost complex structures, must be monitored carefully (see in particular Section \ref{subsubsec:operations}). Because of those changes, it has seemed appropriate to give a reasonably self-contained presentation. We will be considering only Riemann surfaces of genus zero, with one negative end; this restriction is not strictly necessary for developing the theory, but it does cover all cases relevant to us.

\subsection{The Hamiltonian theory\label{sec:hamiltonian}}

\subsubsection{Basic notation}
We begin by recalling some elementary notions underlying the construction of Floer cohomology on an exact symplectic manifold $(\hat{N}, \omega_{\hat{N}} = d\theta_{\hat{N}})$ (the notation $\hat{N}$ is chosen in view of later applications; at the moment, any exact symplectic manifold will do). Write $\scrH(\hat{N}) = \smooth(\hat{N},\bR)$, and $\scrJ(\hat{N})$ for the space of compatible almost complex structures.

Let $S$ be an oriented surface. Suppose that we are given a one-form on $S$ with values in $\scrH(\hat{N})$, or equivalently, a one-form on $S \times \hat{N}$ which vanishes in $T\hat{N}$-direction; we call this a ``Hamiltonian term'' and denote it by $K_S$. We remind the reader of our sign conventions for Hamiltonian vector fields from Section \ref{subsec:conventions}(f).  By passing from functions to Hamiltonian vector fields, one associates to $K_S$ a one-form $Y_S$ with values in $\smooth(T\hat{N})$, or equivalently a map $TS \rightarrow T\hat{N}$ of bundles pulled back to $S \times \hat{N}$. To summarize, we have
\begin{equation} \label{eq:k-y}
K_S \in \Omega^1(S,\scrH(\hat{N})) \subset \Omega^1(S \times \hat{N}) \; \Rightarrow \; Y_S \in \Omega^1(S,\smooth(T\hat{N})) \iso \mathit{Hom}_{S \times \hat{N}}(TS,T\hat{N}).
\end{equation}
This describes a Hamiltonian connection on the trivial bundle $S \times \hat{N} \rightarrow S$. More precisely, the connection is $d-K_S$ (where the relevant Lie algebra is $\scrH$, with the Poisson bracket) or $d-Y_S$ (where the relevant Lie algebra is that of Hamiltonian vector fields). One can think of a connection in terms of parallel transport: given a path $c(t) \in S$, the associated parallel transport is the family $(\phi_t)$ of symplectomorphisms satisfying
\begin{equation} \label{eq:phi-parallel-transport}
d\phi_t/dt = \phi_t^*(Y_S(c'(t))).
\end{equation}
The curvature of $d-K_S$ (again see Section \ref{subsec:conventions}(f) for sign conventions) is
\begin{equation} \label{eq:curvature}
F_S =  -dK_S + \half \{K_S,K_S\} \in \Omega^2(S,\scrH(\hat{N})).
\end{equation}
In local coordinates $(s,t)$ on $S$ (and omitting the $\mathit{ds} \wedge \mathit{dt}$), this means that
\begin{equation} \label{eq:local-curvature}
F_S = -\partial_s K_S(\partial_t) + \partial_t K_S(\partial_s) + \{K_S(\partial_s), K_S(\partial_t)\}.
\end{equation}
Since $S$ is oriented, it makes sense to say that the curvature is nonnegative (meaning, $F_S$ evaluates nonnegatively on any oriented basis of tangent vectors, or the function \eqref{eq:local-curvature} is nonnegative in oriented local coordinates).

Take a map $u: S \rightarrow \hat{N}$. Given a loop $c: S^1 \rightarrow S$, the action of $u$ along $c$ is
\begin{equation} \label{eq:action}
A(u|c) = \int_c c^*(-u^*\theta_{\hat{N}} + u^*K_S),
\end{equation}
where $u^*K_S$ is really the pullback by the graph $(z,u(z)): S \rightarrow S \times \hat{N}$. One extends this additively to finite collections of loops ($1$-cycles). If $C \subset S$ is a compact subdomain (always assumed to have smooth boundary), the topological energy of $u$ on $C$ is
\begin{equation} \label{eq:action-energy}
E^{\mathrm{top}}(u|C) = \int_C u^*\omega_{\hat{N}} - d(u^*K_S) = -A(u|\partial C).
\end{equation}
Suppose that $S$ carries the structure $j_S$ of a Riemann surface, and the target space a family $J_S = (J_{S,z})_{z \in S}$, $J_{S,z} \in \scrJ(\hat{N})$. The associated Cauchy-Riemann equation is
\begin{equation} \label{eq:cauchy-riemann}
(du-Y_S)^{0,1} = \half \big ( (du - Y_S) + J_{S,z} \circ (du - Y_S) \circ j_S \big) = 0.
\end{equation}
The geometric energy of a solution, restricted to a compact subdomain, is
\begin{equation} \label{eq:two-energies}
E^{\mathrm{geom}}(u|C) = \int_C \half \| du - Y_S \|^2 = E^{\mathrm{top}}(u|C) - \int_C u^*F_S.
\end{equation}
In a local complex coordinate $z = s+it$ on $S$, the integrand for the geometric energy (again omitting the $\mathit{ds} \wedge \mathit{dt}$) is
\begin{equation} 
\half\|du - Y_S\|^2 = \|\partial_s u - Y_S(\partial_s)\|^2 = \|\partial_t u - Y_S(\partial_t)\|^2.
\end{equation}

\begin{example}
On $S = \bR \times S^1$, take $K_S = H_{s,t} \mathit{dt}$, $J_S = J_{s,t}$. The Cauchy-Riemann equation is the familiar continuation map equation $\partial_s u + J_{s,t}(\partial_t u - X_{s,t}) = 0$, where $X$ is the Hamiltonian vector field of $H$; and having nonnegative curvature reduces to the monotonicity condition $\partial_s H_{s,t} \leq 0$. Floer's equation is the translation-invariant special case, $H_{s,t} = H_t$ and $J_{s,t} = J_t$, for which the curvature is zero.
\end{example} 

\subsubsection{Liouville manifolds\label{subsubsec:liouville}}
We will now be more specific about the class of symplectic manifolds under consideration. A Liouville domain is a compact exact symplectic manifold with boundary $N$, such that the associated Liouville vector field $Z_N$, defined by $\omega_N(Z_N,\cdot) = \theta_N$, satisfies
\begin{equation} \label{eq:liouville-condition}
Z_N|\partial N \text{ points strictly outwards.}
\end{equation}
Then, $\alpha_{\partial N} = \theta_N|\partial N$ is a contact one-form on $\partial N$, whose Reeb field we denote by $R_{\partial N}$. The completion of $N$ is the Liouville manifold $(\hat{N}, \omega_{\hat{N}} = d\theta_{\hat{N}})$ obtained by attaching a cone to $\partial N$:
\begin{equation} \label{eq:liouville-completion}
\begin{aligned}
& \hat{N} = N\cup_{\partial N} ([1,\infty) \times \partial N), \\
& \theta_{\hat{N}}\, |\, ([1,\infty) \times \partial N) = \rho\, \alpha_{\partial N},
\end{aligned}
\end{equation}
where $\rho$ is the $[1,\infty)$ coordinate. 

\begin{itemize}
\itemsep.5em
\item
For any $r \geq 1$, we write $N_r = \hat{N} \setminus \{\rho > r\} = N \cup_{\partial N} ([1,r] \times \partial N)$ for the compact piece of $\hat{N}$ bounded by $\{r\} \times \partial N$ ($N_1 = N$ is the original Liouville domain).

\item
By an $r$-shell, for $r>1$, we mean a subset $[r-\epsilon,r+\epsilon] \times \partial N \subset \hat{N}$, for some $\epsilon \in (0,r-1]$.

\item 
We write $Z_{\hat{N}}$ for the Liouville vector field of $\theta_{\hat{N}}$, and $\lambda_{\hat{N}}$ for its flow, which is defined for all times. These satisfy
\begin{equation} \label{eq:liouville-flow}
\begin{aligned}
& Z_{\hat{N}}|([1,\infty) \times \partial N) = \rho\partial_\rho, \\
& \lambda_{\hat{N},t}(\rho,x) = (e^t\rho,x) \text{ for } (\rho,x) \in [1,\infty) \times \partial N \text{ with } e^t\rho \geq 1.
\end{aligned}
\end{equation}

\item
On the cone, the Hamiltonian vector field of the function $\rho: [1,\infty) \times \partial N \rightarrow \bR$ is the Reeb field extended by $0$ in $\rho$-direction, which we simply write as
\begin{equation}
X_\rho = R_{\partial N}.
\end{equation}
The standard quadratic Hamiltonian is the function $\half\rho^2$, which therefore satisfies
\begin{equation}
X_{\frac12\rho^2} = \rho R_{\partial N}.
\end{equation}

\item A compatible almost complex structure $J$ is called of contact type if, on the cone,
\begin{equation} \label{eq:contact-type}
\alpha_{\partial N} \circ J = d\log(\rho) \;\; \Leftrightarrow \;\;
\theta_{\hat{N}} \circ J = d\rho \;\; \Leftrightarrow \;\; J(Z_{\hat{N}}) = R_{\partial N}.
\end{equation}
For the metric associated to any such $J$, we have the equalities (again on the cone)
\begin{equation} \label{eq:unit-norms}
\|R_{\partial N}\| = \|Z_{\hat{N}}\| = \|d\rho\| = \|\theta_{\hat{N}}\| = \rho^{1/2}.
\end{equation}
\end{itemize}

\begin{example}
Let $N$ be a disc of radius 2 in $\bR^2$ with $\theta_N = \half (x \mathit{dy} - y \mathit{dx})$. The Reeb field on the boundary is given by $R_{\partial N} = \half(x\partial_y - y\partial_x)$. The completion $\hat{N}$ can be identified with all of $\bR^2$ with its standard symplectic structure. The Liouville vector field is given by $Z_{\hat{N}} = \half(x \partial_x + y \partial_y)$ and the Liouville coordinate is given by $\rho = \quarter(x^2+y^2)$, from which it follows that the standard complex structure is of contact type.
\end{example}

Our Riemann surfaces $S$ will come with a one-form
\begin{equation} \label{eq:d-beta}
\beta_S \in \Omega^1(S), \;\; d\beta_S \leq 0.
\end{equation}
This singles out a particularly simple class of Hamiltonian terms, namely those which on the cone satisfy $K_S = \half \rho^2 \beta_S$; the condition \eqref{eq:d-beta} is then equivalent to nonnegativity of the curvature on the cone. The actual construction of Hamiltonian Floer cohomology will be a bit more complicated, as the Hamiltonian term will have to be perturbed at least on part of the cone, so as to make it non-autonomous.

\subsubsection{A priori bounds}
The following is a simple special case of arguments from \cite{floer-hofer94, cieliebak94,oancea08,ganatra13} (in the last-mentioned reference, it corresponds to Lemma A.4). 

\begin{lemma} \label{th:action-bounds}
Take a (holomorphically embedded) cylinder $C = [0,1] \times S^1 \subset S$, over which:
\begin{itemize} 
\item the almost complex structures $J_S$ are of contact type on the cone;
\item the Hamiltonian term is $K_S = (\sigma/2) \rho^2 \mathit{dt}$ on the cone, for some constant $\sigma>0$. 
\end{itemize}
Choose constants $e>0$ and $a \in \bR$. Then there is a constant $r \geq 1$ (depending on $\sigma$ and $K_S$, as well as $e$ and $a$) such that the following holds. Given any solution $u$ of \eqref{eq:cauchy-riemann} which satisfies
\begin{align}
\label{eq:e-bound} & E^{\mathrm{geom}}(u|C) \leq e, \\
\label{eq:a-bound} & A(u|\{1\} \times S^1) \geq a,
\end{align}
there is an $s \in [0,1]$ for which the restriction $u(s,\cdot): S^1 \rightarrow \hat{N}$ takes values in $N_r$.
\end{lemma}

\begin{proof}
From \eqref{eq:unit-norms}, $\|d(\rho^{1/2})\| = \half$. Hence, if $u(s,t)$ lies in the cone, then at that point,
\begin{equation}
\begin{aligned} \label{eq:estimatetderivativerootrho}
& |\partial_t(\rho^{1/2}(u))| = |d(\rho^{1/2})(\partial_tu)|  = |d(\rho^{1/2})(\partial_tu - \sigma \rho R_{\partial N})|
\\ & \qquad \leq \half \|\partial_t u - \sigma \rho R_{\partial N}\| = \half \|\partial_t u - Y_S(\partial_t)\|.
\end{aligned}
\end{equation}
For the second equality in \eqref{eq:estimatetderivativerootrho}, we have used the fact that $d(\rho^{1/2})(R_{\partial N})=0$. It follows that, if $\{s\} \times [t_0,t_1]$ is an interval whose entire image under $u$ lies in the cone, then 
\begin{equation} \label{eq:change-of-rho}
\begin{aligned}
& |\rho^{1/2}(u(s,t_1))-\rho^{1/2}(u(s,t_0))| \leq \half \int_{\{s\} \times [t_0,t_1]} \|\partial_t u-Y_S(\partial_t)\|
\\ & \; \leq \half (t_1-t_0)^{1/2} \Big( \int_{\{s\} \times [t_0,t_1]} \|\partial_t u - Y_S(\partial_t)\|^2 \Big)^{1/2} \leq
\half \Big( \int_{\{s\} \times S^1} \|\partial_t u - Y_S(\partial_t)\|^2 \Big)^{1/2}.
\end{aligned}
\end{equation}
Similarly, using $\|\theta_{\hat{N}}\| = \rho^{1/2}$, we see that for any $s$ such that $u(\{s\} \times S^1)$ lies inside the cone,
\begin{equation} \label{eq:lower-action}
\begin{aligned}
&
A(u|\{s\} \times S^1) = \int_{\{s\} \times S^1} -\theta_{\hat{N}}(\partial_t u) + \smallfrac{\sigma}{2}\rho^2(u)
\\ & \qquad
= \int_{\{s\} \times S^1} -\theta_{\hat{N}}(\partial_t u - \sigma \rho R_{\partial N}) - \smallfrac{\sigma}{2}\rho^2(u) 
\\ & \qquad
\leq \int_{\{s\} \times S^1} \rho^{1/2}(u) \|\partial_t u-Y_S(\partial_t)\| - \smallfrac{\sigma}{2} \rho^2(u)
\\ & \qquad
\leq \Big(\int_{\{s\} \times S^1} \rho(u)\Big)^{\half} \Big(
 \int_{\{s\} \times S^1} \|\partial_t u-Y_S(\partial_t)\|^2\Big)^{\half} - \smallfrac{\sigma}{2} \int_{\{s\} \times S^1} \rho^2(u) 
\\ & \qquad
\leq \Big(\int_{\{s\} \times S^1} \rho^2(u)\Big)^{\half} \Big(
 \int_{\{s\} \times S^1} \|\partial_t u-Y_S(\partial_t)\|^2\Big)^{\half} - \smallfrac{\sigma}{2} \int_{\{s\} \times S^1} \rho^2(u) 
\\ & \qquad \qquad \qquad + 
 \left( \frac{1}{\sigma^{1/2}} \Big(\int_{\{s\} \times S^1} \|\partial_t u - Y_S(\partial_t)\|^2 \Big)^{\half} -
\frac{\sigma^{1/2}}{2} \Big(\int_{\{s\} \times S^1} \rho^2(u)\Big)^{\half}\right)^2
\\ & \qquad
= \frac{1}{\sigma} \int_{\{s\} \times S^1} \|\partial_t u-Y_S(\partial_t)\|^2 - \frac{\sigma}{4} \int_{\{s\} \times S^1} \rho^2(u).
\end{aligned}
\end{equation}
The second equality used that $\theta_{\hat{N}}(R_{\partial N})= \rho.$ The last step involved the silly inequality $\rho \leq \rho^2$ on the cone, and an equally elementary Peter-Paul trick. Next, we use the available action and energy bounds. The curvature \eqref{eq:curvature} is bounded on our finite cylinder, because it vanishes when restricted to the cone. Fix a bound,
\begin{equation} \label{eq:curvature-bound}
|F_S| \leq b.
\end{equation}
Then for any $s \in [0,1]$, we get from \eqref{eq:a-bound} that
\begin{equation} \label{eq:a-bound-2}
\begin{aligned}
& A(u|\{s\} \times S^1) = A(u|\{1\} \times S^1) + E^{\mathrm{top}}(u|[s,1] \times S^1)
\\ & \qquad = A(u|\{1\} \times S^1) + E^{\mathrm{geom}}(u|[s,1] \times S^1) + \int_{[s,1] \times S^1} u^*F_S
\geq a - b.
\end{aligned}
\end{equation}
From \eqref{eq:e-bound}, we see that there must be some $s \in [0,1]$ such that 
\begin{equation} \label{eq:low-diameter} 
\int_{\{s\} \times S^1} \|\partial_t u - Y_S(\partial_t)\|^2 = \int_{\{s\} \times S^1} \half \|du - Y_S\|^2 
\leq e.
\end{equation}
For that value of $s$, the inequalities \eqref{eq:change-of-rho} and \eqref{eq:lower-action}, \eqref{eq:a-bound-2} imply
\begin{align} \label{eq:diameter-bound}
& |\rho^{1/2}(u(s,t_1))-\rho^{1/2}(u(s,t_0))| \leq \half e^{1/2} && \text{if $u(\{s\} \times [t_0,t_1])$ lies in the cone,}
\\
\label{eq:lambda-bound}
& \int_{\{s\} \times S^1} \rho^2(u) 
\leq \smallfrac{4}{\sigma}(b-a) + \smallfrac{4}{\sigma^2}e && \text{if $u(\{s\} \times S^1)$ lies in the cone.}
\end{align}

The argument now concludes as follows. Take a value of $s$ for which \eqref{eq:low-diameter} holds. Suppose that there is some $t \in S^1$ such that $u(s,t)$ does not lie in the cone. Then, by applying \eqref{eq:diameter-bound} to maximal intervals $\{s\} \times [t_0,t_1]$ which are mapped to the cone, one sees that
\begin{equation} \label{eq:12-one}
\rho^{1/2}(u(s,t)) \leq 1 + \half e^{1/2} \;\; \text{ for all $t$ such that $u(s,t)$ is in the cone.}
\end{equation}
On the other hand, suppose that $u(\{s\} \times S^1)$ is entirely contained in the cone. Then \eqref{eq:lambda-bound} applies, which shows that that there must be some $t$ such that $\rho^2(u(s,t)) \leq \smallfrac{4}{\sigma}(b-a) + \smallfrac{4}{\sigma^2} e$. In conjunction with \eqref{eq:diameter-bound}, one gets
\begin{equation} \label{eq:12-two}
\rho^{1/2}(u(s,t)) \leq \big( \smallfrac{4}{\sigma}(b-a) + \smallfrac{4}{\sigma^2} e\big)^{1/4} + \half e^{1/2}\;\; \text{ for all $t$.}
\end{equation}
These inequalities provide the required $r$.
\end{proof}

The following is the integrated maximum principle \cite{abouzaid-seidel07, abouzaid10, ganatra13} (in the last reference, Section A.2). 

\begin{lemma} \label{th:integrated-maximum-principle}
Let $C \subset S$ be a connected compact subdomain with nonempty boundary. For some $r$, suppose that:
\begin{itemize} 
\item there is an $r$-shell on which: $K_S = \half \rho^2 \beta_S$, and $J_S$ is of contact type, over $C$.
\item The curvature $F_S$ is nonnegative on $C \times ([r,\infty) \times \partial N)$. 
\end{itemize}
Let $u: S \rightarrow \hat{N}$ be a solution of \eqref{eq:cauchy-riemann} such that $u(\partial C) \subset N_r$. Then $u(C) \subset N_r$. 
\end{lemma}

\begin{proof} 
Suppose that the contrary is true. Then there is an $\tilde{r}$ which is a little larger than $r$, but still within our shell, such that $u|C$ intersects $\{\tilde{r}\} \times \partial N$ transversely in a nonempty subset, which must be disjoint from $\partial C$. Take $D$ to be one of the connected components of $u^{-1}([\tilde{r},\infty) \times \partial N) \cap C$. Because of the curvature assumption, we have
\begin{equation} \label{eq:e-top-geq0}
E^{\mathrm{top}}(u|D) \geq E^{\mathrm{geom}}(u|D) \geq 0.
\end{equation}
Let $\xi$ be a tangent vector at $z \in \partial D$ pointing in positive direction along the boundary. Then $j_S\xi$ points inwards, hence $du(j_S\xi)$ points in positive $\rho$-direction, by the transverse intersection assumption. Hence,
\begin{equation}
\begin{aligned}
& 0 < d\rho(du(j_S\xi)) = d\rho( du - \rho R_{\partial N} \beta_S)(j_S\xi) = (d\rho \circ J_z) (du - \rho R_{\partial N} \beta_S)(\xi) 
\\ & \qquad = \theta_{\hat{N}}(-du + \rho R_{\partial N} \beta_S)(\xi) = -(u^*\theta_{\hat{N}})(\xi) + \tilde{r}^2 \beta_S (\xi).
\end{aligned}
\end{equation} 
Hence,
\begin{equation} \label{eq:action-is-geq0}
A(u|\partial D) = \int_{\partial D} -u^*\theta_{\hat{N}} + \half \rho^2(u) \beta_S \geq -\half \tilde{r}^2 \int_{\partial D} \beta_S = -\half \tilde{r}^2 \int_D d\beta_S \geq 0.
\end{equation}
Comparing \eqref{eq:e-top-geq0} and \eqref{eq:action-is-geq0}, one sees that $E^{\mathrm{geom}}(u|D) = 0$. But that means $du = \rho R_{\partial N} \beta_S$ on $D$, hence $\rho(u)|D = \tilde{r}$ is constant, which is a contradiction to the transverse intersection assumption.
\end{proof}

\subsubsection{Floer data\label{subsubsec:floer}}
To simplify the bookkeeping involved, we fix once and for all a number $P > 1$, such that:

\begin{assumption} \label{th:no-p-orbits}
The Reeb flow on $\partial N$ has no periodic orbits whose period is a positive integer multiple of $P$.
\end{assumption}

A Floer datum is given by a time-dependent Hamiltonian $H = (H_t)_{t \in S^1}$, $H_t \in \scrH(\hat{N})$, with vector field $X = (X_t)$, and $J = (J_t)_{t \in S^1}$, $J_t \in \scrJ(\hat{N})$. We require the following (along the lines of \cite[Section 5]{abouzaid10}).
\begin{itemize} \itemsep.5em
\item
On the cone, 
\begin{equation} \label{eq:perturbed-hamiltonian}
H_t = \half\rho^2 + \tilde{H}_t, 
\end{equation}
where the perturbation $\tilde{H}$ is bounded and has bounded derivative $\partial_\rho \tilde{H}$. 

\item For each integer $i \geq 1$, there is an $iP$-shell in which $\tilde{H} = 0$ and $J$ is of contact type. The one-periodic orbits of $H$ must be disjoint from the shells (this is made possible by the choice of $P$), and nondegenerate.
\end{itemize}
Fix such a Floer datum $(H,J)$. Write $A(x)$ for the action of a one-periodic orbit, defined as in \eqref{eq:action} to be $A(x) = \int_{S^1} -x^*\theta_{\hat{N}} + H_t(x(t)) \mathit{dt}$. Consider Floer solutions asymptotic to one-periodic orbits,
\begin{equation} \textstyle
x_0(t) = \lim_{s \rightarrow -\infty} u(s,t), \;\; x_1(t) = \lim_{s \rightarrow +\infty} u(s,t). 
\end{equation}
For such solutions, the version of \eqref{eq:action-energy} that applies to the whole cylinder is
\begin{equation}
E^{\mathrm{geom}}(u) = \int_{\bR \times S^1} \|\partial_su\|^2 = E^{\mathrm{top}}(u) = 
\int_{\bR \times S^1} u^*\omega_{\hat{N}} - d(u^*H_t\,\mathit{dt}) = 
A(x_0) - A(x_1).
\end{equation}
Following \cite{abouzaid10, ganatra13} (Lemma A.1 of the latter), one has:

\begin{lemma} \label{th:action-growth}
For any $a \in \bR$, there are only finitely many one-periodic orbits with $A(x) \geq a$.
\end{lemma}

\begin{proof}
The argument is similar to \eqref{eq:lower-action}. Fix constants $b$, $c$ such that
\begin{equation}
|\tilde{H}| \leq b, \quad
|\partial_\rho \tilde{H}| \leq c.
\end{equation}
Then,
\begin{equation} \label{eq:lower-action-2}
\begin{aligned}
& A(x) = \int_{S^1} -\theta_{\hat{N}}(X_t)_{x(t)} + \half \rho(x(t))^2 + \tilde{H}_t(x(t)) \\ & \quad = \int_{S^1} (-\rho \partial_\rho H_t)(x(t)) + \half \rho(x(t))^2 + \tilde{H}_t(x(t))
\\ & \quad = \int_{S^1} (-\rho \partial_\rho\tilde{H}_t)(x(t)) - \half \rho(x(t))^2 + \tilde{H}_t(x(t))
\\ & \quad \leq \int_{S^1} \rho(x(t)) c - \half \rho(x(t))^2 + b  \leq c^2+b - \int_{S^1} \quarter \rho(x(t))^2.
\end{aligned}
\end{equation}
The last inequality is an integrated Peter-Paul inequality. Each orbit is either contained in $N_P \setminus \partial N_P$, or in one of the subsets $(Pi,P(i+1)) \times \partial N$, $i \in \bN$. Because of nondegeneracy, there are only finitely many orbits in each of those classes. For those orbits lying in $(Pi,P(i+1)) \times \partial N$, \eqref{eq:lower-action-2} provides a bound $A(x) \leq c^2 + b - \quarter P^2i^2$.
\end{proof}

The next statement is the main a priori ($C^0$) bound on solutions.

\begin{proposition} \label{th:c0-bound}
Fix a one-periodic orbit $x_1$. There is an $r$ such that any Floer solution with positive limit $x_1$, and arbitrary negative limit $x_0$, is contained in $N_r$.
\end{proposition}

\begin{proof}
If a solution $u$ with asymptotics $(x_0,x_1)$ exists, we must have $A(x_0) \geq A(x_1)$, which by Lemma \ref{th:action-growth} leaves only finitely possibilities for $x_0$. Hence, we may as well assume that $x_0$ is fixed as well. Choose some $i$ so that $(x_0,x_1)$ are contained in $N_{iP}$. We can then apply Lemma \ref{th:integrated-maximum-principle} to show that $u$ is contained in the same subset.
\end{proof}

\begin{proposition} \label{th:floer-transversality}
For fixed $H$, a generic choice of $J$ makes all Floer trajectories regular.
\end{proposition}

\begin{proof}[Sketch of proof]
The limits lie outside the $iP$-shells, by assumption on $H$. Therefore, any trajectory $u$ has the property that for $|s| \gg0$, $u(s,t)$ lies outside the $iP$-shells. For those $(s,t)$, we can vary $J_t$ at $u(s,t)$ freely. Given that, the argument is as in \cite[Theorem 5.1]{floer-hofer-salamon94}.
\end{proof}

Following a strategy from \cite{abouzaid10}, we will also use rescaled version of Floer data. Take $(H,J)$ as before, and some $\sigma \geq 1$. Then, the rescaled datum is obtained by pullback via the Liouville flow for time $\log(\sigma) \geq 0$:
\begin{equation} \label{eq:pullback-floer}
\begin{aligned}
& H^{(\sigma)}_t = \lambda_{\hat{N},\log(\sigma)}^*H_t/\sigma \;\;\Rightarrow\;\; X^{(\sigma)}_t = \lambda_{\hat{N},\log(\sigma)}^*X_t, \\
& J^{(\sigma)}_t = \lambda_{\hat{N},\log(\sigma)}^*J_t.
\end{aligned}
\end{equation}
This inherits the following properties:
\begin{itemize}
\item on the cone, $H^{(\sigma)}_t = \smallfrac{\sigma}{2} \rho^2 + (\lambda_{\hat{N},\log(\sigma)}^*\tilde{H}_t)/\sigma$;
\item
on $(iP/\sigma)$-shells, we have $H^{(\sigma)}_t = \smallfrac{\sigma}{2}\rho^2$ and $J_t^{(\sigma)}$ of contact type.
\end{itemize}
There is an obvious correspondence between Floer trajectories for $(H,J)$ and $(H^{(\sigma)}, J^{(\sigma)})$. It is also worth recording the effect of this rescaling on the action functional. If $x$ is a one-periodic orbit of $H$, and $x^{(\sigma)} = \lambda_{\hat{N},-\log(\sigma)}(x)$ is the corresponding one-periodic orbit of $H^{(\sigma)}$, then \begin{align} \label{eq:rescaledaction} A_{H^{(\sigma)}}(x^{(\sigma)})= A_H(x)/\sigma. \end{align}

\subsubsection{Gradings and orientations\label{subsubsec:gradings-and-orientations}}
At this point, let's impose a Calabi-Yau type assumption: 
\begin{equation} \label{eq:calabi-yau}
c_1(N) = 0, \text{ and we fix a trivialization of the underlying complex line bundle.}
\end{equation}
This leads to a Floer complex defined as a $\bZ$-graded $\bK$-vector space, for an arbitrary coefficient field $\bK$. Even though that is entirely a standard construction, we give a short account here, because the notation and conventions will be used later on. For a more detailed discussion we recommend \cite[\S 1.4]{abouzaid14}, though the constructions go back as far as \cite[Section 2e]{floer89}.

Choose a Floer datum satisfying the transversality property from Proposition \ref{th:floer-transversality}. For each pair of orbits $\bfx = (x_0,x_1)$, let $\mathring{\frakC}(\bfx)$ be the manifold of trajectories asymptotic to those orbits, up to $\bR$-translation; and $\frakC(\bfx)$ its Gromov compactification (which is indeed compact, thanks to Proposition \ref{th:c0-bound}).

For every one-periodic orbit $x$, the trivialization  \eqref{eq:calabi-yau} gives rise to a homotopy class of trivializations of the pull-back $x^*T\hat{N}$, viewed as a unitary vector bundle. Choosing a trivialization within the prescribed homotopy class, the differential of the Hamiltonian flow defines a path of symplectic matrices. One can associate to this path an orientation operator $D_x$, which is a Cauchy-Riemann operator on the thimble \cite[\S 1.4.1]{abouzaid14}. In an appropriate sense, the space of such operators is simply-connected. One uses it to define the degree (Conley-Zehnder index) and orientation line (determinant line)
\begin{align}
& \mathrm{deg}(x) = \mathrm{index}(D_x), \\
& \frako_x = \mathit{det}(D_x) = \lambda^{\mathrm{top}}(\mathit{coker}(D)^\vee) \otimes \lambda^{\mathrm{top}}(\mathit{ker}(D)). \label{eq:orientation-line}
\end{align}
Here, $\frako_x$ is well defined as a line bundle over a simply-connected space. Hence, one has a well-defined $\bK$-normalized orientation line
\begin{equation}
|\frako_x|_{\bK} = \left\{\text{\parbox{25em}{$\bK$-vector space generated by orientations of $\frako_x$, with the sum of opposite orientations set to zero}}\right\}.
\end{equation}
The Floer complex is defined as
\begin{equation} \label{eq:hamiltonian-floer}
\mathit{CF}^k(H) = \bigoplus_{\mathrm{deg}(x)=k} |\frako_x|_{\bK},
\end{equation}
Any Floer trajectory, with linearized operator $D_u$ and asymptotics $(x_0,x_1)$, satisfies
\begin{equation} \label{eq:index-formula}
\mathrm{index}(D_u) = i(x_0) - i(x_1),
\end{equation}
and gives rise (via gluing of determinant lines) to an isomorphism
\begin{equation} \label{eq:ou}
\mathit{det}(D_u) \otimes \frako_{x_1} \iso \frako_{x_0},
\end{equation}
again canonical up to multiplication with a positive number. One has a short exact sequence
\begin{equation} \label{eq:tangent-to-moduli-1}
0 \rightarrow \bR \longrightarrow \mathit{ker}(D_u) \rightarrow T_u\mathring{\frakC}(\bfx) \rightarrow 0,
\end{equation}
where the first map takes $1$ to $-\partial_su$. From that, one obtains an isomorphism 
\begin{equation} \label{eq:tangent-to-moduli-2}
\bR \otimes \lambda^{\mathrm{top}}(T_u\mathring{\frakC}(\bfx)) \iso
\mathit{det}(D_u),
\end{equation}
where we carry around the copy of $\bR$ only because it formally has dimension $1$, hence may influence Koszul signs. Combining that with \eqref{eq:ou} yields an isomorphism
\begin{equation} \label{eq:tangent-ou}
o(u): \bR \otimes \lambda^{\mathrm{top}}(T_u\mathring{\frakC}(\bfx)) \otimes \frako_{x_1} \iso \frako_{x_0}.
\end{equation}
In the special case of isolated trajectories, the $T\mathring{\frakC}$ factor in \eqref{eq:tangent-ou} is trivial, and \eqref{eq:tangent-ou} therefore reduces to an isomorphism $\frako_{x_1} \iso \frako_{x_0}$. The differential $\delta = \delta_{H,J}$ on the Floer complex is the sum of the $\bK$-normalized versions of those isomorphisms; finiteness of that sum is again ensured by Proposition \ref{th:c0-bound}. 

When considering the rescaled Floer datum $(H^{(\sigma)}, J^{(\sigma)})$ from \eqref{eq:pullback-floer}, the correspondence between Floer trajectories gives rise to a canonical isomorphism of complexes, $(CF^*(H),\delta_{H,J}) \cong (CF^*(H^{(\sigma)}),\delta_{H^{(\sigma)},J^{(\sigma)}})$.

\subsubsection{Thick-thin decompositions\label{subsubsec:operations}}
Take a punctured sphere
\begin{equation} \label{eq:punctured-sphere}
S = \bC P^1 \setminus \{z_0,\dots,z_m\}.
\end{equation}
This surface will carry a one-form \eqref{eq:d-beta}. We want to assume that it comes with a decomposition into thick and thin pieces. The thin pieces are semi-infinite (one surrounding each puncture) cylinders, or finite cylinders, and are themselves divided into regions that will be treated technically differently (called ``Floer'', ``transitional'', and ``standard'' regions). Concretely, the (pairwise disjoint) thin pieces are of the following kind:
\begin{itemize} \itemsep1em
\item A single negative semi-infinite cylindrical end 
\begin{equation} \label{eq:minus-end}
(-\infty,0] \times S^1 \hookrightarrow S, \text{ surrounding $z_0$.}
\end{equation}
This comes with a constant $\sigma_0 \geq 1$ and function $\psi_0: (-\infty,0] \rightarrow \bR$,
\begin{equation} \label{eq:psi-1} 
\begin{cases}
\psi_0(s) = \sigma_0 & s \leq -2, \\
\psi_0'(s) < 0 & s \in (-2,-1), \\
\psi_0'(s) = 0 & s \geq -1.
\end{cases} 
\end{equation}
The one-form is $\beta_S = \psi_0(s) \mathit{dt}$. The terminology is that $(-\infty,-2] \times S^1$ is the Floer region, $(-2,-1) \times S^1$ the transitional one, and $[-1,0] \times S^1$ the standard region.

\item $m$ positive semi-infinite cylindrical ends 
\begin{equation} \label{eq:plus-end}
[0,\infty) \times S^1 \hookrightarrow S, \text{ surrounding $z_1,\dots,z_m$.}
\end{equation}
On the $e$-th positive cylinder we have a constant $\sigma_e \geq 1$ and function $\psi_e: [0,\infty) \rightarrow [0,1]$,
\begin{equation} \label{eq:psi-2}
\begin{cases}
\psi_e'(s) = 0 & s \leq 1, \\
\psi_e'(s) < 0 & s \in (1,2), \\
\psi_e(s) = \sigma_e & s \geq 2.
\end{cases} 
\end{equation}
As before, we then set $\beta_S = \psi_e(s) \mathit{dt}$. Symmetrically to the previous case, the Floer regions are $[2,\infty) \times S^1$, the transitional ones are $(1,2) \times S^1$, and the standard regions are $[0,1] \times S^1$. 

\item Some number (which can be zero) of finite cylinders. For simplicity of notation, let's consider a single such cylinder, which is of the form $[-l,l] \times S^1 \hookrightarrow S$ with $l > 2$.
The parametrization is always such that $\{-l\} \times S^1$ borders the component of $S \setminus ((-l,l) \times S^1)$ which contains the negative semi-infinite cylinder. The finite cylinder comes with a constant $\sigma \geq 1$ and function $\psi: [-l,l] \rightarrow [0,1]$,
\begin{equation} \label{eq:psi-3}
\begin{cases}
\psi'(s) = 0 & s \leq -l+1 \text{ or } s \geq l-1, \\
\psi'(s) < 0 & s \in (-l+1,-l+2) \text{ or } s \in (l-2,l-1), \\
\psi(s) = \sigma & -l+2 \leq s \leq l-2.
\end{cases}
\end{equation}
We again set $\beta_S = \psi(s) \mathit{dt}$. The Floer region is $[-l+2,l-2] \times S^1$; the transitional ones are $(-l+1,-l+2) \times S^1$ and $(l-2,l-1) \times S^1$; and the standard ones, $[-l,-l+1] \times S^1$ and $[l-1,l] \times S^1$. (The notation is intentionally shorthand: each of the finite cylinders comes with its own $\sigma$, $\psi$ and $l$.)
\end{itemize}
\begin{figure}
\begin{centering}
\includegraphics{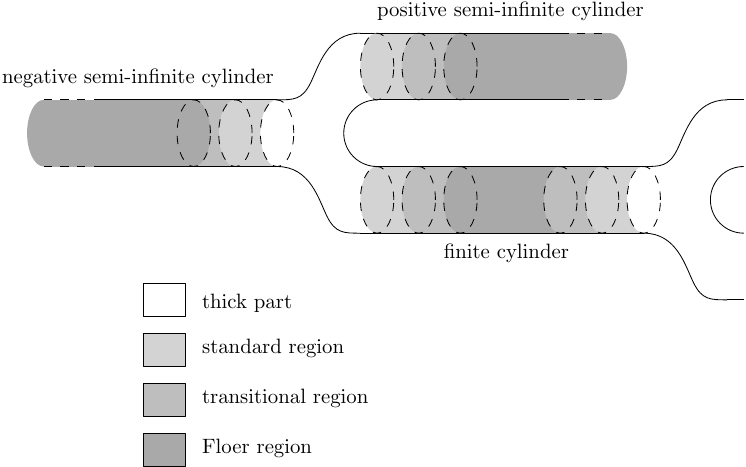}
\caption{\label{fig:surface}The decomposition of our surfaces. For simplicity, we have shown only part of the surface, omitting some of the positive semi-infinite cylinders.}
\end{centering}
\end{figure}
The thick part is the rest of the surface, and can carry any one-form \eqref{eq:d-beta} (see Figure \ref{fig:surface} for a summary). 

\begin{remark}
The conditions on the $\psi_e$ and on $\beta_S$ imply, via Stokes, that $\sigma_0 > \sigma_1 +\cdots + \sigma_m$.
\end{remark}

We fix a Floer datum as in Section \ref{subsubsec:floer}, denoting it by $(H^{\mathrm{Floer}},J^{\mathrm{Floer}})$ for clarity. Concerning $(K_S,J_S)$ we make the following assumptions:
\begin{itemize} \itemsep.5em
\item over the thick part of the surface, as well the standard regions of the thin parts, we have $K_S = \half\rho^2 \beta_S$ on the cone of $\hat{N}$, and each $J_{S,z}$ is of contact type. 

\item Over the Floer region of the semi-infinite cylinders, we use the Floer Hamiltonian rescaled as in \eqref{eq:pullback-floer}, $K_S = (H^{\mathrm{Floer}}_t)^{(\sigma_e)} \mathit{dt}$. For the almost complex structures, we allow any $J_{S,s,t}$ which are of contact type on $(iP/\sigma_e)$-shells (for all $i \in \bN$ such that $iP/\sigma_e > 1$), and which as $s \rightarrow \pm\infty$ converge exponentially fast to $(J^{\mathrm{Floer}}_t)^{(\sigma_e)}$. This condition ensures good asymptotic behaviour of finite energy solutions of our Cauchy-Riemann equation, and is at the same time flexible enough for transversality purposes.

\item There is more freedom on the Floer regions of the finite cylinders. We allow Hamiltonian terms of the form $K_S = H^{\mathrm{finite}}_t \mathit{dt}$, where on the cone $H^{\mathrm{finite}}_t = \frac{\sigma}{2} \rho^2 + \tilde{H}^{\mathrm{finite}}_{t}$, for a perturbation $\tilde{H}^{\mathrm{finite}}$ which vanishes on $(iP/\sigma)$-shells. Moreover, on those shells, the almost complex structures should be of contact type. Of course, the $H^{\mathrm{finite}}$ do not have to be the same for two different finite cylinders.

\item On the transition regions the Hamiltonian term should, on the cone of $\hat{N}$, satisfy
\begin{equation} \label{eq:ds-bound}
K_S = \half \rho^2 \psi(s) \mathit{dt} + \tilde{H}_{s,t} \mathit{dt}, \; \text{ with } \;
|\partial_s \tilde{H}_{s,t}| \leq -c\,\psi'(s)
\end{equation}
for some positive constant $c$, and where $\psi$ is the appropriate one of the functions \eqref{eq:psi-1}, \eqref{eq:psi-2}, \eqref{eq:psi-3}. Moreover, $\tilde{H}_{s,t}$ should vanish on a sequence of shells, which is the same as for the (unique) adjacent Floer region. On those shells, we also require that the almost complex structures $J_{s,t}$ should be of contact type. To avoid confusion, let's clarify that each transition region is treated separately here.
\end{itemize}

\begin{lemma} \label{th:curvature-piece-by-piece}
Over the Floer regions, the curvature $F_S$ is zero. Over the standard regions, the curvature is zero on the cone of $\hat{N}$. Over the transitional regions, the curvature is nonnegative (in fact positive) outside a compact subset of $\hat{N}$. Finally, over the thick part, the curvature is nonnegative on the cone.
\end{lemma}

\begin{proof}
Most of this is straightforward. Over the standard regions, we have $K_s = \half \rho^2 \beta_S$ on the cone, and $\beta_S = \psi(s)\mathit{dt}$ with $\psi'(s) = 0$, so that $d\beta_S = 0$. Less obviously, over a transitional region, \eqref{eq:ds-bound} implies that on the cone of $\hat{N}$,
\begin{equation} \label{eq:transitional-curvature}
F_S(\partial_s,\partial_t) = -\psi'(s) \half\rho^2 - \partial_s \tilde{H}_{s,t} \geq -\psi'(s) (\half \rho^2 - c).
\end{equation}
That makes the curvature positive where $\rho$ is large.
\end{proof}

Take one-periodic orbits $(x_0,\dots,x_m)$ of $(H_t^{\mathrm{Floer}})$. Consider solutions of the Cauchy-Riemann equation \eqref{eq:cauchy-riemann} whose asymptotics are appropriately rescaled version of these orbits:
\begin{equation} \label{eq:rescaled-asymptotics-operations}
\textstyle\lim_{s \rightarrow \pm\infty} \lambda_{\hat{N},\log(\sigma_e)} u(s,\cdot) = x_e \;\; \text{on the $e$-th semi-infinite cylinder.}
\end{equation}
For such a solution, we have (bearing in mind \eqref{eq:rescaledaction})
\begin{equation} \label{eq:curvature-on-s}
\frac{A(x_0)}{\sigma_0} - \sum_{e=1}^m \frac{A(x_e)}{\sigma_e} = E^{\mathrm{top}}(u) = E^{\mathrm{geom}}(u) + \int_S u^*F_S.
\end{equation}
Lemma \ref{th:curvature-piece-by-piece} shows that there is a uniform lower bound on the curvature integral, which holds for any $u$. Namely, the curvature $F_S$ is supported over a compact surface $S'$ which is the closure of the complement of the Floer regions. Over $S'$, the curvature is non-negative outside a compact set of $\hat{N}$. Combining these two observations gives the required lower bound. Hence, given $(x_0,\dots,x_m)$, there is an upper bound on the geometric energy, which is of course fundamental for compactness. The appropriate version of Proposition \ref{th:c0-bound} is:

\begin{proposition} \label{th:c0-bound-2}
Fix one-periodic orbits $(x_1,\dots,x_m)$. Then, there is an $r$ such that every solution $u$ of the Cauchy-Riemann equation on $S$, with those limits over the positive semi-infinite cylinders, and arbitrary limit $x_0$ on the negative semi-infinite cylinder, is contained in $N_r$.
\end{proposition}
\begin{figure}
\begin{centering}
\includegraphics{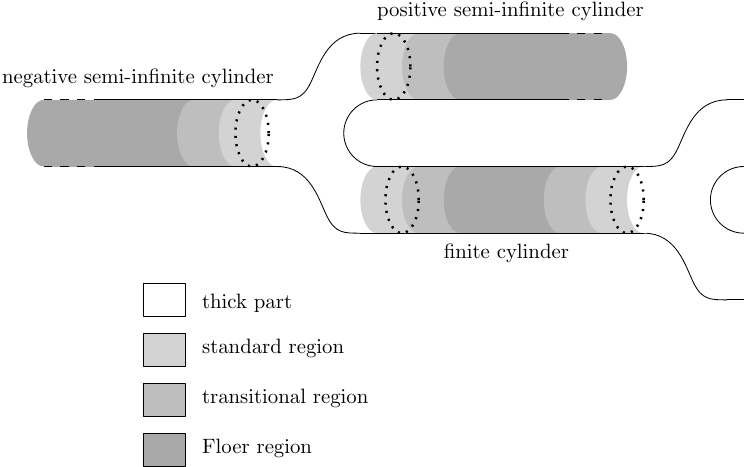}
\caption{\label{fig:surface3}The decomposition from Figure \ref{fig:surface}, with extra circles in each standard region, as they appear in the proof of Proposition \ref{th:c0-bound-2}. In that argument, the integrated maximum principle is applied separately to each part of the surface bounded by those circles.}
\end{centering}
\end{figure}

\begin{proof}
From \eqref{eq:curvature-on-s} we get a lower bound on $A(x_0)$, given $(x_1,\dots,x_m)$. Hence, we may just as well work with a fixed $x_0$. 

The first step is to bound the behaviour of $u$ on certain circles, one for each standard region. Take $C^{\mathrm{standard}}$ to be the closure of a standard region, identifying it with $[0,1] \times S^1$ by translation in $s$-direction. From \eqref{eq:curvature-on-s} we get an a priori bound on $E^{\mathrm{geom}}(u|C^{\mathrm{standard}}) \leq E^{\mathrm{geom}}(u)$. Let $T \subset S$ be the closure of the connected component of $S \setminus C^{\mathrm{standard}}$ which bounds $\{1\} \times S^1$; equivalently, this is the component not containing the negative semi-infinite cylinder. Let's say that $T$ contains the $e$-th positive semi-infinite cylinders for $e \in E \subset \{1,\dots,m\}$. Then, 
\begin{equation}
A(u|\{1\} \times S^1) = \sum_{e \in E} \frac{A(x_e)}{\sigma_e} + E^{\mathrm{top}}(u|T) \geq \sum_{e \in E} \frac{A(x_e)}{\sigma_e} + \int_T u^*F_S.
\end{equation}
As before, there is an a priori lower bound for the curvature integral, and therefore for $A(u|\{1\} \times S^1)$. Together, these bounds allow us to apply Lemma \ref{th:action-bounds}, and to obtain an $r \geq 1$ such that any solution $u$ must satisfy $u(s \times \{S^1\}) \subset N_r$ for some $s \in [0,1]$. Since there are finitely many standard regions, we may choose $r$ so that the bound applies to all of them.

Let's look at a connected component of the thick part. This is surrounded by standard regions, and in the closure of each standard region there is a circle that gets mapped to $N_r$. We can therefore enlarge our component by adding cylindrical pieces of those standard regions, and obtain a compact connected $C^{\mathrm{thick}} \subset S$ with nonempty boundary, such that $u(\partial C^{\mathrm{thick}}) \subset N_r$. Notice that over $C^{\mathrm{thick}}$, we have $K_S = \half\rho^2\beta_S$ on the cone, and all the almost complex structures are of contact type. We can therefore apply Lemma \ref{th:integrated-maximum-principle} to conclude that $u(C^{\mathrm{thick}}) \subset N_r$.

Next, consider the negative semi-infinite cylinder. Since its piece $[-1,0] \times S^1$ is a standard region, we know that there is an $s \in [-1,0]$ such that $u(\{s\} \times S^1) \subset N_r$. Moreover, $[s,0] \times S^1$ belonged to one of the previously considered regions $C^{\mathrm{thick}}$, so it remains to bound $u$ on $(-\infty,s] \times S^1$. From Lemma \ref{th:curvature-piece-by-piece} we know that over this cylinder, $F_S$ is nonnegative outside $N_{\tilde{r}}$ for some $\tilde{r}$, which we can assume to be $\geq r$. We also know the limit $x_0$, hence may enlarge $\tilde{r}$ so that $x_0$ is contained in the interior of $N_{\tilde{r}}$. Finally, again making it larger, we can assume that $\tilde{r}$ is $P/\sigma_0$ times an integer. At this point, for any sufficiently negative $\tilde{s}$, we know that $u| [\tilde{s},s] \times S^1$ has the property that the boundary maps to $N_{\tilde{r}}$. One can apply Lemma \ref{th:integrated-maximum-principle} to conclude that the whole of $[\tilde{s},s] \times S^1$ gets mapped to $N_{\tilde{r}}$. For this, it is crucial that on the $\tilde{r}$-shell, the Hamiltonian is of standard form $\half \rho^2 \beta$, and the almost complex structure of contact type; that is ensured by the choices we have made, including those for the transitional region $(-2,-1) \times S^1$. 

The argument for the positive semi-infinite cylinders is parallel. That for the finite cylinders is also similar, but still deserves a short discussion. Such a cylinder $[-l,l] \times S^1$ has standard regions at its ends, so there are $s_- \in [-l,-l+1]$ and $s_+ \in [l-1,l]$ such that $u(\{s_{\pm}\} \times S^1) \subset N_r$. Moreover, as part of our discussion of the thick parts, we have already obtained a priori bounds on $u$ on $[-l,s_-] \times S^1$ and $[s_+,l] \times S^1$. On the remaining part $[s_-,s_+] \times S^1$, we will again apply the integrated maximum principle, bearing in mind that the $(iP/\sigma)$-shells can be used over that entire part.
\end{proof}

\begin{proposition} \label{th:s-transversality}
For a fixed choice of Floer datum $(H^{\mathit{Floer}},J^{\mathit{Floer}})$ and of one-form $\beta_S$, a generic choice of $(K_S,J_S)$ within the class above ensures that every solution to the Cauchy-Riemann equation is regular.
\end{proposition}

\begin{proof}[Sketch of proof]
Our main tool to achieve transversality is a generic choice of almost complex structures on the Floer regions of the semi-infinite cylinders. Of course, this freedom holds only away from the relevant shells, but because the limits are disjoint from those shells, so is $u(s,t)$ provided that $|s| \gg 0$. 

One small wrinkle: there is a very restricted class of solutions whose transversality can't be established in that way, namely ones which satisfy $\partial_s u = 0$ for $|s| \gg 0$. Let's look at such a $u$ on the negative semi-infinite cylinder (any other semi-infinite cylinder would also do). By unique continuation, we must have
\begin{equation} \label{eq:exceptional-solution}
u(s,t) = \lambda_{\hat{N},-\log(\sigma_0)}(x_0(t)) \;\; \text{for all $s \leq -2$.}
\end{equation}
In particular, on the Floer region, $u(s,t)$ is disjoint from the $(iP/\sigma_0)$-shells. Hence, the same is true for $u(s,t)$ when $s$ is slightly larger than $-2$. But by looking at such $s$, we have reached into the transitional region, and the Hamiltonian term can then be chosen freely near $u(s,t)$ (subject to the bound \eqref{eq:ds-bound}, which is irrelevant since transversality is about infinitesimal considerations). Hence, one can use the freedom to change that Hamiltonian term to ensure that this special class of solutions is also regular.
\end{proof}

Finally, note that there is a rescaling process for data $(\beta_S, J_S,K_S)$ as in \eqref{eq:pullback-floer}: for $\tau \geq 1$, take
\begin{equation} \label{eq:pullback-data}
\begin{aligned}
& \beta_S^{(\tau)} = \tau \beta_S, \\
& K_S^{(\tau)} = \lambda_{\hat{N},\log(\tau)}^* K_S/\tau, \\
& J_S^{(\tau)} = \lambda_{\hat{N},\log(\tau)}^* J_S.
\end{aligned}
\end{equation}
This still satisfies all the conditions imposed above, and the associated moduli spaces of maps can be identified with each other by taking $u^{(\tau)} = \lambda_{\hat{N},-\log(\tau)} u$.

\subsubsection{Parametrized moduli spaces\label{subsubsection:Floerfamilies}}
For many applications, one has to go beyond a single Riemann surface, and instead consider a manifold $\mathring{\frakP}$ which parametrizes a family of punctured spheres, usually denoted by
\begin{equation} \label{eq:universal-family}
U_{\mathring{\frakP}} \longrightarrow \mathring{\frakP}.
\end{equation}
This should carry a fiberwise one-form 
\begin{equation} \label{eq:family-beta}
\beta_{\mathring{\frakP}} \in \Omega^1(U_{\mathring{\frakP}}/\mathring{\frakP}),
\end{equation}
satisfying the conditions from Section \ref{subsubsec:operations} on each fiber. There are a few notable points:
\begin{itemize} \itemsep.5em
\item the numbers $\sigma_0,\dots,\sigma_m$ describing the behaviour of the one-form at infinity can vary throughout the family, which means that they are functions on $\mathring{\frakP}$.

\item The topological structure of the thick-thin decomposition is not necessarily constant throughout the family. Instead, the following happens: the semi-infinite cylinders are defined over the entire parameter space. Additionally, we have a finite partially ordered set $A$ and, for each $a \in A$, an open subset $U_a \subset \mathring{\frakP}$, such that
\begin{equation} \label{eq:a1a2}
a_1 \leq a_2 \;\; \Longrightarrow \;\; \bar{U}_{\fraka_1} \subset \bar{U}_{\fraka_2} 
\end{equation}
(in applications, this structure will come from collar neighbourhoods of the boundary strata of a compactification of $\mathring{\frakP}$ to a manifold with corners). Over each $\bar{U}_{\fraka}$ we have a smoothly varying choice of finite cylinders; and passing to a larger subset as in \eqref{eq:a1a2} amounts to forgetting some of the finite cylinders (the ``disappearance'' of finite cylinders is more intuitive in terms of neck-stretching, which we will discuss later; see e.g.\ \cite[Remark 9.1]{seidel04} for a previous occurrence of such artificially constructed thick-thin decompositions). 
\end{itemize}
Similarly, we want to have a family version of the perturbation data from Section \ref{subsubsec:operations},
\begin{equation} \label{eq:family-data}
K_{\mathring{\frakP}} \in \Omega^1(U_{\mathring{\frakP}}/\mathring{\frakP}, \mathcal{H}(\hat{N})) \text{ and } J_{\mathring{\frakP}} \in C^{\infty}(\mathcal{U}_{\mathring{\frakP}}, \scrJ(\hat{N})),
\end{equation}
Again, the devil lies in the details:
\begin{itemize}
\itemsep.5em
\item The Floer datum involved is always the same (constant over $\mathring{\frakP}$).

\item As before, it is sufficient to have thick-thin decompositions locally on $\mathring{\frakP}$. Again locally on $\mathring{\frakP}$, one should be able to choose the constant $c$ which appears in \eqref{eq:ds-bound} independently of the parameters.

\item For the behaviour of the almost complex structures on the Floer regions of the semi-infinite cylinders, the $C^\infty$ condition as formulated in \eqref{eq:family-data} is not enough. Instead, in a local trivialization of our family (around some point of $\mathring{\frakP}$) where $r$ are the parameters, we must have that $J_{\mathring{\frakP},r,s,t} \rightarrow (J_t^{\mathrm{Floer}})^{(\sigma_e)}$ exponentially as $s \rightarrow \pm\infty$, and uniformly in $r$. This means that in any $C^k$ norm (derivatives with respect to $r,s,t$), the difference between the two is bounded by $Ae^{\mp Bs}$, where $A,B>0$ are independent of $r$. (In addition, one of course has restrictions on the behaviour on shells.)
\end{itemize}

Given a collection of orbits $\bfx = (x_0,x_1,\cdots,x_m)$, let $\mathring{\frakP}(\mathbf{x})$ denote the moduli space of pairs 
\begin{equation} 
p \in \mathring{\frakP},\; u: S = U_p \longrightarrow \hat{N},
\end{equation}
where $U_p$ is the fiber of \eqref{eq:universal-family}, and $u$ is a solution of \eqref{eq:cauchy-riemann} with asymptotics \eqref{eq:rescaled-asymptotics-operations}. There is an straightforward parametrized version of Proposition \ref{th:s-transversality}, saying that this space is regular for generic choice of \eqref{eq:family-data}. 

The analogue of \eqref{eq:tangent-to-moduli-2} is a canonical isomorphism
\begin{equation} \label{eq:turn-off}
\lambda^{\mathrm{top}}T_{(p,u)}\mathring{\frakP}(\bfx) \iso 
T_p\mathring{\frakP} \otimes \mathit{det}(D_u).
\end{equation}
Here, $D_u$ is the ordinary linearized operation, which comes with versions of \eqref{eq:index-formula} and \eqref{eq:ou}:
\begin{align}
& \mathrm{index}(D_u) = \mathrm{deg}(x_0) - \mathrm{deg}(x_1) - \cdots - \mathrm{deg}(x_m), \\
\label{eq:det-gluing}
& \mathit{det}(D_u) \otimes \frako_{x_1} \otimes \cdots \otimes \frako_{x_m} \iso \frako_{x_0}.
\end{align}
The combination of this and \eqref{eq:turn-off} yields
\begin{align} 
& \mathrm{dim}\,\mathring{\frakP}(\bfx) = \mathrm{dim}\,\mathring{\frakP} + \mathrm{deg}(x_0) - \mathrm{deg}(x_1) - \cdots - \mathrm{deg}(x_m),
\\ \label{eq:oru} &
o(p,u): \lambda^{\mathrm{top}}(T_{(p,u)}\mathring{\frakP}(\bfx))
\otimes \frako_{x_1} \otimes \cdots \otimes \frako_{x_m} \iso 
\lambda^{\mathrm{top}}(T_p\mathring{\frakP}) \otimes \frako_{x_0}.
\end{align}
Let's assume that an orientation of $\mathring{\frakP}$ has been chosen, and that $(p,u)$ is an isolated point in the parametrized moduli space. In that case, \eqref{eq:oru} simplifies to an isomorphism $\frako_{x_1} \otimes \cdots \otimes \frako_{x_m} \iso \frako_{x_0}$. In principle, the $\bK$-normalization of this isomorphism describes the contribution of $(p,u)$ to a map
\begin{equation} \label{eq:pop}
\mathit{CF}^*(H^{\mathrm{Floer}})^{\otimes m} \longrightarrow \mathit{CF}^*(H^{\mathrm{Floer}})[-\mathrm{dim}(\mathring{\frakP})].
\end{equation}
As just described, the construction would only work for compact $\mathring{\frakP}$, which rarely applies. Instead, one usually encounters compactifications of the parameter space obtained by allowing the surfaces to split into pieces via neck-stretching, and that will be the next topic in our exposition.

\subsubsection{Neck-stretching\label{subsubsec:neck}}
Let's consider a simple situation, where we have surfaces $S_-$ and $S_+$, each of them coming with all the structure introduced above: one-forms $\beta_{S_\pm}$, and numbers $(\sigma_{\pm,e})$ which describe those one-forms on the Floer regions of the semi-infinite cylinders. Let $\varepsilon_-: (-\infty,0] \times S^1 \rightarrow S_-$ be the unique negative end, and $\varepsilon_+: [0,\infty) \times S^1 \rightarrow S_+$ one of the positive ends. The glued surface, for some length $l>2$, is 
\begin{equation} \label{eq:glued-surface}
S_l \stackrel{\mathrm{def}}{=} \frac{(S_+ \setminus \varepsilon_+((l,\infty) \times S^1)) \sqcup (S_- \setminus \varepsilon_-((-\infty,-l) \times S^1))}{\varepsilon_+(s,t) \sim \varepsilon_-(s-l,t)}.
\end{equation}
To make the one-forms compatible, we multiply them by some constants 
\begin{equation} \label{eq:taupm}
\tau_{\pm} \geq 1, \;\;\text{ such that } \tau_- \sigma_{-,e} = \tau_+ \sigma_{+,0}.
\end{equation}
Then, $S_l$ will come with one-forms $\beta_{S_l}$. Gluing creates a finite cylinder $[0,l] \times S^1 \subset S_l$. For large values $l \gg 0$, we consider $S_l$ as having the thick-thin decomposition inherited from $S_{\pm}$ including that cylinder. However, as $l$ becomes smaller, we will forget that cylinder, which means that it will belong to the thick part of $S_l$ (this is the previously mentioned phenomenon that the thick-thin decomposition can be parameter-dependent).

Instead of the gluing length, we also often use a gluing parameter $\gamma \in \mathring{\frakP} = (0,e^{-2})$, with the relation between the two given by
\begin{equation} \label{eq:log-gamma}
l = -\log(\gamma).
\end{equation}
One can think of the glued surfaces as fibers of a family of Riemann surfaces parametrized by $\gamma$. That family can then be extended over $\gamma = 0$ by setting the fiber there to be $S_- \sqcup S_+$. The outcome is a three-manifold $U_\frakP$ with boundary, with a local submersion
\begin{equation} \label{eq:extended-u}
U_\frakP \longrightarrow \frakP, 
\end{equation}
where $\frakP = [0,e^{-2})$. The one-forms we have constructed, and the original $\beta_{S_\pm}$, then yield a smooth fiberwise one-form
\begin{equation}
\beta_{\frakP} \in \Omega^1(U_{\frakP}/\frakP).
\end{equation}

Suppose that $S_{\pm}$ come with perturbation data. Concerning the data on the glued family, these should smoothly extend to $\gamma = 0$, which means we have
\begin{equation} \label{eq:extended-kj}
K_{\frakP} \in \Omega^1(U_{\frakP}/\frakP, \mathcal{H}(\hat{N})) \text{ and } J_{\frakP} \in C^{\infty}(U_{\frakP}, \scrJ(\hat{N})),
\end{equation}
which on the fiber $S_+ \sqcup S_-$ reduce to versions of $(K_{S_{\pm}},J_{S_{\pm}})$ rescaled by $\tau_{\pm}$ in the sense of \eqref{eq:pullback-data}. Fiberwise, this should satisfy our usual conditions, but some additional restrictions have to be imposed as well.
\begin{itemize}
\itemsep.5em
\item The constant $c$ appearing in \eqref{eq:ds-bound} can be chosen the same for all $S_l$, $l \gg 0$, as well as for $S_{\pm}$. 
\item On the semi-infinite ends of the glued surfaces, one needs to have exponential convergence to $(J_t^{\mathit{Floer}})^{(\sigma_e)}$ uniformly in the gluing parameter $\gamma$.
\item On the finite cylinder $[0,l] \times S^1$ obtained from gluing, the difference between the almost complex structure on the glued surface and the appropriate $(J_t^{\mathit{Floer}})^{(\sigma)}$ must be bounded by $A(e^{-Bs} + e^{B(s-l)})$, where $A,B>0$ are constants independent of the gluing length.
\end{itemize}
The aim of these requirements is to ensure suitable compactness properties in the parametrized moduli space. These are based on the following:

\begin{proposition}
Energy bounds on solutions $u: S_l \rightarrow \hat{N}$ of the Cauchy-Riemann equation, as well as the key a priori bound (Proposition \ref{th:c0-bound-2}), apply uniformly as $l \rightarrow \infty$.
\end{proposition}

\begin{proof}[Sketch of proof]
What's important here is that the curvature integral appearing in \eqref{eq:curvature-on-s} will have an a priori lower bound that holds for all $l$, because the variable-length part $[-l+2,l-2] \times S^1$ contributes zero. For the constant appearing in \eqref{eq:ds-bound}, we have explicitly asked that an $l$-independent bound should hold. The outcome is that the bound on the geometric energy, and Proposition \ref{th:c0-bound-2}, apply uniformly over all $S_l$.
\end{proof}

The discussion above is a model for what happens near codimension-one faces of general parameter spaces $\frakP$. Our spaces $\frakP$ will be manifolds with corners, where the strata of codimension $k$ correspond to splittings of the surfaces into $k$ pieces. They carry extended universal families which generalize that in \eqref{eq:extended-u}, and which come with thick-thin decompositions (and in particular ends) compatible with gluing. We then choose perturbation data over the compactified parameter space generalizing \eqref{eq:extended-kj} which near each codimension $k$ face satisfies the above conditions with respect to the gluing parameters $\gamma_1,\cdots,\gamma_k$ near that stratum. We summarize this by saying that $(K_{\frakP},J_{\frakP})$ is {\em conformally consistent}.

\subsection{The Lagrangian (and mixed) theory}

\subsubsection{Geometric setup}
We again begin with a review of generalities. Suppose $(N,\theta_N)$ is a Liouville domain and $(\hat{N},\theta_{\hat{N}})$ is its symplectic completion (see Section \ref{subsubsec:liouville}). Let $S$ be an oriented surface with boundary, which comes with a Hamiltonian term \eqref{eq:k-y}. In addition, we want our surface to be equipped with an exact Lagrangian boundary condition, meaning a submanifold and function
\begin{equation} \label{eq:lagrangian}
\begin{aligned}
& L_{\partial S} \subset \partial S \times \hat{N}, \\
& G_{\partial_S}: L_{\partial S} \longrightarrow \bR,
\end{aligned}
\end{equation}
such that $L_{\partial S} \rightarrow \partial S$ is a fiber bundle, and $dG_{\partial S}|L_{\partial S,z} = \theta_{\hat{N}}|L_{\partial S,z} \in \Omega^1(L_{\partial S,z})$ for each $z \in \partial S$. The boundary curvature $F_{\partial S}$ is a one-form on $L_{\partial S}$ which vanishes when restricted to a fiber, given by
\begin{equation} \label{eq:boundary-curvature}
F_{\partial S} = dG_{\partial S} - \theta_{\hat{N}}|L_{\partial S} + K_S|L_{\partial S}.
\end{equation}
Thinking of the connection $d-K_S$, one can explain the geometric meaning of \eqref{eq:boundary-curvature} as follows: $L_{\partial S}$ is compatible with parallel transport \eqref{eq:phi-parallel-transport} along $\partial S$ if and only if $(\omega_{\hat{N}} - dK_S)|L_{\partial S} = 0$, which is the case iff $F_{\partial S}$ is locally the pullback of a one-form on $\partial S$. (Actual vanishing of $F_{\partial S}$ lifts that compatibility statement from Hamiltonian vector fields to the level of functions.)

Take a map $u: S \rightarrow \hat{N}$, satisfying
\begin{equation} \label{eq:boundary-condition}
u(z) \in L_{\partial S,z} \;\; \text{ for all $z \in \partial S$.}
\end{equation}
Given a path $c: [0,1] \rightarrow S$ with $c(0), c(1) \in \partial S$, the action of $u$ along $c$ is defined as
\begin{equation} \label{eq:action-2}
A(u|c) = \int_{[0,1]} c^*(-u^*\theta_{\hat{N}} + u^*K_S) + G_{\partial S}(c(1),u(c(1))) - G_{\partial S}(c(0),u(c(0))).
\end{equation}
One can add up that term for several paths, and also allow loops as in \eqref{eq:action}, to get a definition of action for relative one-cycles in $(S,\partial S)$. Let $C \subset S$ be a compact subdomain, whose boundary consists of $\partial^{\mathrm{para}}C = C \cap \partial S$ and additional intervals and circles $\partial^{\mathrm{trans}}C$; the two parts are required to meet transversally, so $C$ is a surface with corners. The topological energy of $u$ on $C$ is
\begin{equation}
E^{\mathrm{top}}(u|C) = \int_C u^*\omega_{\hat{N}} - d(u^*K_S) + \int_{\partial^{\mathrm{para}}C} u^*F_{\partial S} = 
-A(u|\partial^{\mathrm{trans}} C).
\end{equation}
For solutions of the Cauchy-Riemann equation \eqref{eq:cauchy-riemann} satisfying \eqref{eq:boundary-condition}, we now have
\begin{equation}
E^{\mathrm{geom}}(u|C) = E^{\mathrm{top}}(u|C) - \int_C u^*F_S - \int_{\partial^{\mathrm{para}}C} u^*F_{\partial S}.
\end{equation}

In our application, we will only use exact Lagrangian submanifolds $L$ which are of Legendrian type on the cone, which means
\begin{equation} \label{eq:l-cone}
L \cap ([1,\infty) \times \partial N) = [1,\infty) \times \Lambda
\end{equation}
for some Legendrian $\Lambda \subset \partial N$. Each $L$ should come with a primitive $\theta_{\hat{N}}|L = dG_L$ which vanishes on \eqref{eq:l-cone}. When dealing with surfaces $S$, we want the one-forms \eqref{eq:d-beta} to also satisfy
\begin{equation} \label{eq:d-beta-2}
\beta_S|\partial S = 0 \in \Omega^1(\partial S).
\end{equation}
For the associated Lagrangian boundary conditions, we require that the part of $L_{\partial S}$ lying in the cone of $\hat{N}$ is locally constant in $z$, and that $G_{\partial S}$ vanishes on the cone. If we then suppose that $K_S = \half \rho^2 \beta_S$ on the cone, it follows that the boundary curvature \eqref{eq:boundary-curvature} is zero there.

\subsubsection{A priori bounds}
We consider solutions of \eqref{eq:cauchy-riemann} with boundary conditions \eqref{eq:boundary-condition}. The analogue of Lemma \ref{th:action-bounds} is:

\begin{lemma} \label{th:action-bounds-2}
Take a rectangular region $C = [0,1]^2 \subset S$ (with $\partial^{\mathrm{para}}C = [0,1] \times \{0,1\}$ and $\partial^{\mathrm{trans}}C = \{0,1\} \times [0,1]$). Suppose that on this region, the same conditions on $(J_S,K_S)$ as in Lemma \ref{th:action-bounds} apply, and so does \eqref{eq:e-bound}. Then, with \eqref{eq:a-bound} replaced by its obvious analogue
\begin{align}
& A(u|\{1\} \times [0,1]) \geq a,
\end{align}
the same conclusion holds, where the constant $r$ additionally depends on $G_{\partial S}$.
\end{lemma}

\begin{proof}
One has \eqref{eq:change-of-rho} as before. Using the fact that $G_{\partial S}=0$ on the cone, the analogue of \eqref{eq:lower-action} says that if $u(\{s\} \times [0,1])$ lies in the cone,
\begin{equation}
A(u|\{s\} \times [0,1]) \leq \frac{1}{\sigma} \int_{\{s\} \times [0,1]} \|\partial_t u - Y_S(\partial_t)\|^2 - \frac{\sigma}{4} \int_{s \times [0,1]} \rho^2(u).
\end{equation}
In addition to the bound \eqref{eq:curvature-bound} on the curvature, we also have 
\begin{equation}
|F_{\partial S}| \leq c.
\end{equation}
The analogue of \eqref{eq:a-bound-2} is
\begin{equation}
A(u|\{s\} \times [0,1]) \geq a-b-2c.
\end{equation}
From the energy bound, there must be some $s \in [0,1]$ such that
\begin{equation}
\int_{\{s\} \times [0,1]} \|\partial_t u - Y_S(\partial_t)\|^2 \leq e.
\end{equation}
For that value of $s$, we have \eqref{eq:diameter-bound} as well as
\begin{equation}
\int_{\{s\} \times [0,1]} \rho^2(u) \leq \frac{4}{\sigma} (b+2c-a) + \frac{4}{\sigma^2}e
\quad \text{ if $u(\{s\} \times [0,1])$ lies in the cone.}
\end{equation}
One concludes the argument as before.
\end{proof}

The original version of the integrated maximum principle \cite[Lemma 7.2]{abouzaid-seidel07} already allows for Lagrangian boundary conditions. We reproduce it here for ready citeability. The proof is the same as in Lemma \ref{th:integrated-maximum-principle} (with $\partial D$ replaced by $\partial^{\mathrm{trans}}D$).

\begin{lemma} \label{th:integrated-maximum-principle-2}
Let $C \subset S$ be a connected compact subdomain, such that $\partial^{\mathit{trans}}C \neq \emptyset$. For some $r \geq 1$, suppose that the assumptions from Lemma \ref{th:integrated-maximum-principle} on $(K_S,J_S,F_S)$ hold. Let $u: S \rightarrow \hat{N}$ be a solution of \eqref{eq:cauchy-riemann}, \eqref{eq:boundary-condition}, such that $u(\partial^{\mathrm{trans}}C) \subset N_r$. Then $u(C) \subset N_r$.
\end{lemma}

\subsubsection{Floer cohomology} \label{section:LagrangianFloercomplex}
Fix a pair of Lagrangian submanifolds $L_k$ ($k = 0,1$), as well as some $P>1$. The following assumption will allow us to use Hamiltonians which are (unperturbed) quadratic on the cone.

\begin{assumption} \label{th:no-p-chords}
For the associated Legendrian submanifolds $(\Lambda_0,\Lambda_1)$, all Reeb chords of length $\geq 1$ are nondegenerate. Moreover, there are no Reeb chords whose length is a positive integer multiple of $P$.
\end{assumption}

A Floer datum is given by $H = (H_t)_{t \in [0,1]}$ and $J = (J_t)_{t \in [0,1]}$, such that:
\begin{itemize}
\itemsep.5em
\item
on the cone, $H_t = \half \rho^2$. As a consequence, Floer chords which intersect the cone must lie entirely inside it, and can be identified with Reeb chords of length $\rho \geq 1$; these are nondegenerate by assumption. We then also require all other chords to be nondegenerate.
\item
For each $i \in \bN$, there is an $iP$-shell on which $J$ is contact type. These shells must be disjoint from the Floer chords.
\end{itemize}
The counterpart of Lemma \ref{th:action-growth} is trivially satisfied, since the Floer chords that correspond to length $\rho$ Reeb chords have action $-\half\rho^2$. For solutions $u: \bR \times [0,1] \rightarrow \hat{N}$ of Floer's equation with boundary conditions $(L_0,L_1)$, we have the analogue of Lemma \ref{th:c0-bound}, using Lemma \ref{th:integrated-maximum-principle-2} in the same way as before: 

\begin{lemma} \label{th:co-bound-3}
Fix a Floer chord $y_1$. There is an $r$ such that any Floer solution with positive limit $y_1$, and arbitrary negative limit $y_0$, is contained in $N_r$.
\end{lemma}

The rescaling trick \eqref{eq:pullback-floer} now additionally involves moving the Lagrangian submanifolds (and the corresponding primitives)
\begin{equation} \label{eq:rescale-l-g}
\begin{aligned}
& L^{(\sigma)} = \lambda_{\hat{N},-\log(\sigma)}(L), \\
& G^{(\sigma)} = \lambda_{\hat{N},\log(\sigma)}^*G_L/\sigma.
\end{aligned}
\end{equation}
This pushes $L$ inwards, so we still have $L^{(\sigma)} = [1,\infty) \times \Lambda$ on the cone. One gets a canonical isomorphism between the Floer complex for $(L_0,L_1,H,J)$ and $(L_0^{(\sigma)},L_1^{(\sigma)},H^{(\sigma)},J^{(\sigma)})$.

Again, because the almost complex structure is unconstrained near the Floer chords, a generic choice of that structure achieves regularity of the moduli spaces (c.f. Proposition \ref{th:s-transversality}). To define the Floer complex with a $\bZ$-grading and $\bK$-coefficients, we assume the following standard ``brane conditions'':
\begin{equation} \label{eq:brane}
\parbox{35em}{our Lagrangian submanifolds carry gradings with respect to the trivialization of the line bundle from \eqref{eq:calabi-yau} (and hence come with orientations). Additionally they are equipped with {\em Spin} structures.}
\end{equation}
Then, as in the closed-string case, each chord $y$ has an associated degree $\mathrm{deg}(y)$ and determinant line $\mathfrak{o}_y = \operatorname{det}(D_y)$. In parallel with \eqref{eq:hamiltonian-floer} one sets
\begin{equation} 
\mathit{CF}^k(L_0,L_1,H) = \bigoplus_{\mathrm{deg}(y)=k} |\mathfrak{o}_y|_\bigK.  
\end{equation} 
Given two chords, $\mathring{\frakR}(y_0,y_1)$ denotes the moduli space of Floer strips up to $\bR$-translation. Assuming regularity, any $u$ in that space determines an isomorphism $o(u)$ as in \eqref{eq:tangent-ou}, which for isolated strips reduces to an isomorphism $\frako_{y_1} \iso \frako_{y_0}$. One again defines the Floer differential by adding up the $\bK$-normalizations of those isomorphisms.

\subsubsection{Open string operations\label{subsec:open-string-operations}}
Let $(S,j_S)$ be a boundary-punctured disc, obtained by removing $n+1 > 0$ boundary points from a closed disc; we label those missing boundary points by $\{0,\dots,n\}$, going counterclockwise around the boundary. We correspondingly label the connected components of $\partial S$, so that the $0$-th boundary component joins the $0$-th and $1$-st boundary puncture. Our surface should come with a one-form satisfying \eqref{eq:d-beta} and \eqref{eq:d-beta-2}. As before, we want a decomposition of $S$ into thick and thin pieces, and of the latter pieces into three types of regions (see Figure \ref{fig:surface-2}; note though that the ``transitional'' regions play a much smaller role than before, and exist only for a very  technical transversality reason). The thin pieces are:
\begin{figure}
\begin{centering}
\includegraphics{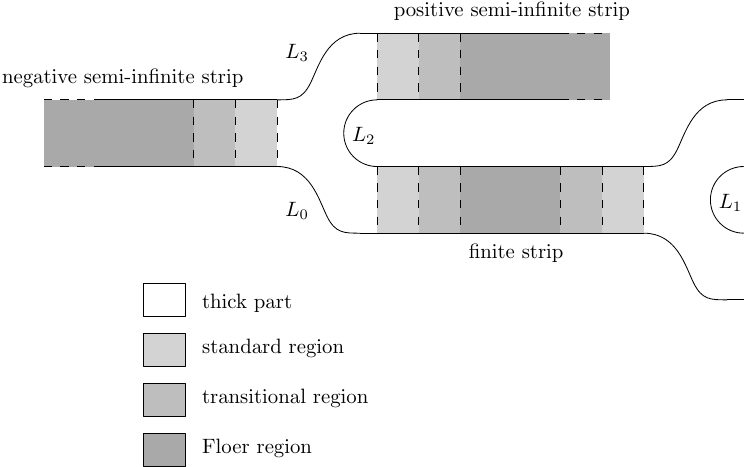}
\caption{\label{fig:surface-2}The decomposition of surfaces in the open string case.}
\end{centering}
\end{figure}
\begin{itemize} \itemsep.5em
\item A single negative semi-infinite strip $(-\infty,0] \times [0,1] \hookrightarrow S$ asymptotic to the $0$-th boundary puncture (here and in similar situations below, it is understood that $(-\infty,0] \times \{0,1\}$ is mapped to $\partial S$). This comes with a constant $\sigma_0$ and function $\psi_0$ as in \eqref{eq:psi-1}, with $\beta_S = \psi_0(s) \mathit{dt}$. The Floer region is $(-\infty,-2] \times [0,1]$; the transitional one, $(-2,-1) \times [0,1]$; and the standard region, $[-1,0] \times [0,1]$.

\item Positive semi-infinite strips $[0,\infty) \times [0,1] \rightarrow S$ for the boundary punctures labeled $e \in \{1,\dots,n\}$, with constants $\sigma_e \geq 1$, functions $\psi_e$ as in \eqref{eq:psi-2}, and $\beta_S = \psi_e(s) \mathit{dt}$. The Floer regions are $[2,\infty) \times [0,1]$; the transitional ones, $(1,2) \times [0,1]$; and the standard regions, $[0,1] \times [0,1]$.

\item Some number (which can be zero) of finite strips, $[-l,l] \times [0,1] \hookrightarrow 0$ for $l > 2$. 
The parametrization is always such that $\{-l\} \times [0,1]$ borders the component of $S \setminus ((-l,l) \times [0,1])$ which contains the negative semi-infinite strip. Each finite strip comes with its own constant $\sigma$ and function $\psi$ as in \eqref{eq:psi-3}, with $\beta_S = \psi(s) \mathit{dt}$. The Floer region of such a strip is $[-l+2,l-2] \times [0,1]$; the transitional regions are $(-l+1,-l+2) \times [0,1]$ and $(l-2,l-1) \times [0,1]$; and the standard regions, $[-l,-l+1] \times [0,1]$ and $[l-1,1] \times [0,1]$.
\end{itemize}
Fix Lagrangian submanifolds $(L_0,\dots,L_n)$. Choose Floer data $(H^{\mathrm{Floer}}_{e}, J^{\mathrm{Floer}}_e)$, $e \in \{0,\dots,n\}$, associated to the pairs $(L_0,L_n)$ (for $e = 0$) and $(L_{e-1},L_e)$ (for $e = 1,\dots,n$). Each such pair is required to satisfy Assumption \ref{th:no-p-chords} (with the same constant $P$). Our surface will come with a Lagrangian boundary condition $(L_{\partial S}, G_{\partial S})$, as well as the usual $(K_S,J_S)$. Concerning those, we make the following assumptions:
\begin{itemize} \itemsep.5em
\item Everywhere on the surface, we have that $K_S = \half \rho^2 \beta_S$ on the cone of $\hat{N}$. Similarly, everywhere on the $i$-th boundary component of $S$, we have $L_{\partial S} = L_i$ on the cone.

\item Over the thick part of $S$, as well as the standard regions, $J_S$ is of contact type.

\item Take the Floer region of the $e$-th semi-infinite strip (for any $e$). On that region, the Lagrangian boundary condition is
\begin{equation}
\left.
\begin{aligned}
& L_{\partial S,0,t} = L_0^{(\sigma_0)} \\
& L_{\partial S,1,t} = L_n^{(\sigma_0)}
\end{aligned}
\right\}
\text{ if $e = 0$, or } \quad
\left.
\begin{aligned}
& L_{\partial S,0,t} = L_{e-1}^{(\sigma_0)} \\
& L_{\partial S,1,t} = L_e^{(\sigma_0)}
\end{aligned}
\right\}
\text{ if $e>0$,}
\end{equation}
with corresponding functions as in \eqref{eq:rescale-l-g}. The Hamiltonian term is $K_S = (H^{\mathrm{Floer}}_{e,t})^{(\sigma_e)} \mathit{dt}$. The almost complex structure $J_{S,s,t}$ is of contact type on $(iP/\sigma_e)$-shells and as $s \rightarrow \pm\infty$, converges exponentially to $(J^{\mathrm{Floer}}_{e,t})^{(\sigma_e)}$.

\item On the Floer region of any finite strip, we want $K_S = H^{\mathrm{finite}}_t \mathit{dt}$, where $H^{\mathrm{finite}}_t =\frac{\sigma}{2} \rho^2$ on the cone. Similarly, the Lagrangian boundary condition and function $G_{\partial S}$ should be independent of $s$. Finally, the almost complex structures should be of contact type on $(iP/\sigma)$-shells.

\item On any transitional region, we want the almost complex structures to be of contact type for the same shells as the (unique) adjacent Floer region.
\end{itemize}
We consider solutions of \eqref{eq:cauchy-riemann}, \eqref{eq:boundary-condition} with asymptotics which are rescaled Lagrangian chords,
\begin{equation} \label{eq:asymptoticslagrangian}
\textstyle \lim_{s \rightarrow \pm \infty} \lambda_{\hat{N},\log(\sigma_e)} u(s,\cdot) = x_e.
\end{equation}
The analogue of \eqref{eq:curvature-on-s} is
\begin{equation} \label{eq:curvature-on-s-2}
\frac{A(x_0)}{\sigma_0} - \sum_{e=1}^n \frac{A(x_e)}{\sigma_e} = E^{\mathrm{top}}(u) = E^{\mathrm{geom}}(u) + \int_S u^*F_S + \int_{\partial S} u^*F_{\partial S}.
\end{equation}
As before, we have a uniform lower bound on the two curvature terms (in fact, the $\partial S$ term is bounded both below and above). The following is the counterpart of Proposition \ref{th:c0-bound-2}; the proof remains the same, with the basic analytic ingredients are replaced by their open string counterparts (Lemma \ref{th:action-bounds-2} and \ref{th:integrated-maximum-principle-2}).

\begin{proposition} \label{th:c0-bound-4}
Fix Floer chords $(y_1,\dots,y_m)$. Then, there is an $r$ such that every solution $u$ of the Cauchy-Riemann equation on $S$, with those limits over the positive semi-infinite strips, and arbitrary limit $y_0$ on the negative semi-infinite strip, is contained in $N_r$.
\end{proposition}

We also have a transversality result, analogous to Proposition \ref{th:s-transversality}.
As in that situation, there is an unlikely exceptional case which has to be considered separately. Namely, suppose we have a solution which, on the negative semi-infinite strip, is of the form \eqref{eq:exceptional-solution}. There are two sub-cases to consider:
\begin{itemize}
\itemsep.5em
\item {\em The rescaled orbit $\lambda_{\hat{N},-\log(\sigma_0)}(y_0(t))$ is disjoint from the cone of $\hat{N}$.} In that case, at a point $u(s,t)$ with $s$ slightly larger than $-2$, one can choose $K_S$ freely, and transversality can be established by varying it there, just as in the closed string case.

\item {\em The rescaled orbit $\lambda_{\hat{N},-\log(\sigma_0)}(y_0(t))$ lies in the cone of $\hat{N}$.} Suppose temporarily that there is some $s_* \in (-2,-1)$ such that on the region $[-2,s_*] \times [0,1]$, we have $\partial_s u = 0$. This also means that $\partial_t u = \psi(s) \rho R$, where $\rho \geq 1$ is constant; but those two conditions contradict each other, because $\psi'(s) < 0$. As a consequence, we see that there points $(s,t) \in (-2,-1) \times [0,1]$ with $s$ arbitrarily close to $-2$, where $\partial_s u$ is nonzero. At such a point, $u(s,t)$ will be disjoint from $(iP/\sigma)$-shells, hence the almost complex structure can be varied freely, which ensures transversality.
\end{itemize}

The discussion of parametrized moduli spaces is parallel to that in Section \ref{subsubsec:gradings-and-orientations}, and there is a notion of conformally consistent perturbations parallel to that from Section \ref{subsubsec:neck}. We will not go through the details.

\subsubsection{Closed-open and open-closed situations} \label{subsubsec:closedopen}
Finally, we will encounter setups that mix the open and closed string theory, of two different kinds. In the first situation (closed-open) we consider Riemann surfaces obtained from the closed disc by removing $m$ interior points, labeled as $\{1,\dots,m\}$, and $(d+1)$ boundary points, labeled as $\{0,\dots,d\}$. The $0$-th boundary point corresponds to a negative semi-infinite strip, and the others have positive (cylindrical or strip) ends. In the second situation (open-closed), we have $(m+1)$ interior points, one of which is negative, and $d$ positive boundary punctures. 

In this context, the thick pieces can be more general than before, since they can both include parts of $\partial S$ and have boundary circles. However, that does not matter: we continue to impose the same conditions on $(K_S,J_S)$ on such pieces, which makes sense since our previous discussions of closed and open string operations were in agreement for thick pieces. All other pieces (finite and semi-infinite cylinders, finite and semi-infinite strips) belong to the one of the two setups encountered before, and we follow our previous treatment.

\section{Deformed symplectic cohomology\label{sec:closed-string}}

This section sets up certain operations in Hamiltonian Floer cohomology. For the analysis, we use the framework from Section \ref{sec:hamiltonian} without further elaboration. We apply that framework to a number of specific families of punctured Riemann surfaces, and the bulk of the discussion will be taken up with defining those. Here is how the various parts fit together.
\begin{itemize} \itemsep.5em
\item At the start, we review Deligne-Mumford spaces $\frakM_m$ (Section \ref{subsec:deligne-mumford}), as well as a version $\frakL_{m,r}$ of the lollipop spaces from \cite{abouzaid-seidel07} (Section \ref{sec:lollipop}). These won't be used as such, but they are convenient for introducing terminology (and also as a technical tool: $\frakM_m$ and $\frakL_{m,r}$ are smooth complex varieties, and this can be used to understand the manifold structures on related spaces).

\item Fulton-MacPherson spaces $\frakF_m$ (Section \ref{section:FMd}) are a form of the topological $L_\infty$-operad. Correspondingly, their application is to construct the $L_\infty$-structure on symplectic cohomology (Section \ref{section:Linf}; see \cite{abouzaid-groman-varolgunes,borman-el-alami-sheridan24} for other approaches). We use the $L_\infty$-structure only for one purpose, namely to define the deformed differential associated to a Maurer-Cartan element (part (i) of Definition \ref{th:deformed-structures}), with a formal parameter $q$.

\item One can rewrite Fulton-MacPherson spaces in terms of configurations on the cylinder rather than the plane (Section \ref{section:FM-cylinder}, see in particular Remark \ref{th:compare-c-f}). This viewpoint, which we emphasize by changing notation to $\frakC_m$, is natural for constructing the $L_\infty$-module on symplectic cohomology. Note that this is not the diagonal module of the $L_\infty$-algebra, because of the choice of asymptotic markers (Figure \ref{fig:markers-cylinder}). As before, a Maurer-Cartan elements leads to a $q$-deformation of the differential (part (ii) of Definition \ref{th:deformed-structures}).

\item Section \ref{section:angles} introduces parameter space $\frakA_r$ which then, in Section \ref{subsubsec:s1-equivariant}, will be used to define $S^1$-equivariant symplectic cohomology (compare \cite{ganatra19}). As in the non-equivariant situation, the underlying complex carries an $L_\infty$-module structure, hence admits $q$-deformations associated to Maurer-Cartan elements. The underlying parameter spaces are the $\AC_{m,r}$ from Section \ref{section:anglesFM}; they are used to set up the $L_\infty$-module structure in Section \ref{sec:s1module}; and the $q$-deformed differential is in Definition \ref{th:deformed-structures}(iii).

\item The $q$-deformed equivariant symplectic cohomology carries a connection $\nabla_{u\partial_q}$. This is the most complicated construction in this section: it is carried out in Section \ref{section:connection}, and combines parameter spaces $\AC^{(A)}_{m,r,w}$ and $\AC^{(B)}_{m,r,w}$ from Sections \ref{section:cartan-A} and \ref{section:cartan-B}.
\end{itemize}

\subsection{Deligne-Mumford spaces and their relatives\label{subsec:51}}

\subsubsection{Deligne-Mumford spaces\label{subsec:deligne-mumford}}
Consider spheres \eqref{eq:punctured-sphere} with $m+1 \geq 3$ punctures. We assume $z_0 = \infty$ and write
\begin{equation} \label{eq:punctured-sphere-2}
S = \bC \setminus \{z_1,\dots,z_m\}.
\end{equation}
The parameter space is
\begin{equation}
\mathring{\frakM}_m = \mathit{Conf}_m(\bC)/(\bC \rtimes \bC^*),
\end{equation}
where $\mathit{Conf}$ is ordered configuration space. Strata of the Deligne-Mumford compactification $\frakM_m$ are labeled by trees, for which we use the following terminology. 
\begin{itemize} \itemsep.5em
\item A tree $T$ is a contractible graph with directed edges, such that every vertex has exactly one outgoing edge. We write $|v|_{\operatorname{in}} = |v|-1$ for the number of incoming edges.

\item Our trees have $(m+1)$ semi-infinite edges, of which $1$ is directed towards infinity (outgoing), and $m$ are directed away from infinity (incoming). There is a unique vertex (the root vertex $v_{\operatorname{root}}$) adjacent to the outgoing edge. 

\item As part of the data of a tree, we fix a labeling of the incoming semi-infinite edges by $\{1,\dots,m\}$ (if two trees are isomorphic, the isomorphism between them is unique, since it has to preserve the labels). The outgoing edge can then be labeled as $0$.
\end{itemize}
Given that,
\begin{equation} \label{eq:dm} 
\frakM_m = \bigsqcup_T \mathring{\frakM}_T, \quad \mathring{\frakM}_T = \prod_v \mathring{\frakM}_{|v|_{\operatorname{in}}}, 
\end{equation}
where the union is over all stable trees $T$, meaning trees such that each vertex has $\geq 2$ incoming edges. (More precisely, we pick a representative in each isomorphism class of trees; the resulting $\mathring{\frakM}_T$ is independent of that choice up to canonical isomorphism.) The tree with only one vertex corresponds to the inclusion $\mathring{\frakM}_m \subset \frakM_m$. The space $\frakM_m$ has a natural structure of a compact complex manifold, and the stratification in \eqref{eq:dm} is given by a normal crossing divisor inside that manifold. Each point of $\frakM_m$ corresponds to a possibly disconnected Riemann surface $S = \bigsqcup_v S_v$, with components as in \eqref{eq:punctured-sphere-2}. 

\subsubsection{The framed version\label{sec:framings}}
A framing at $z_k \in \bar{S} = \bC P^1$ is a distinguished tangent direction
\begin{equation}
\tau_{z_k} \in (T_z\bar{S} \setminus \{0\})/\bR^{>0} \iso S^1.
\end{equation}
Since we work primarily with the punctured surface $S$ obtained by removing the $z_k$, the framings may be more appropriately called by their Floer-theoretic name, asymptotic markers. Punctured spheres with asymptotic markers (at all $z_k$, including $z_0 = \infty$) are parametrized by a space $\mathring{\frakM}_m^{\operatorname{fr}}$, which is an $(S^1)^{m+1}$-bundle over $\mathring{\frakM}_m$. It has a compactification to a smooth manifold with corners $\frakM_m^{\operatorname{fr}}$, whose strata are analogous to \eqref{eq:dm}:
\begin{equation} \label{eq:ksv-strata}
\frakM_m^{\operatorname{fr}} = \bigsqcup_T \mathring{\frakM}_T^{\operatorname{fr}}, \quad \mathring{\frakM}_T^{\operatorname{fr}} = \big(\prod_v \mathring{\frakM}^{\operatorname{fr}}_{|v|_{\operatorname{in}}}\big)/(S^1)^{E_{\mathit{fin}}(T)}.
\end{equation}
Here, $E_{\mathit{fin}}(T)$ is the set of finite edges of $T$. Each such edge singles out two punctures, lying on different $S_v$, and the $S^1$-action rotates their markers in opposite directions. More geometrically, if we consider the compact connected nodal surface $\bar{S} = \bigcup_v \bar{S}_v$ corresponding to a point of $\frakM_m^{\operatorname{fr}}$, this comes equipped with framings at the $z_k$, as well as relative framings at all nodes $z = \bar{S}_{v_1} \cap \bar{S}_{v_2}$, meaning distinguished elements
\begin{equation} \label{eq:relative-framing}
\tau_z^{\mathit{rel}} \in \big((T_z\bar{S}_{v_1} \otimes T_z\bar{S}_{v_2}) \setminus \{0\}\big)/\bR^{>0} \iso S^1.
\end{equation}

\begin{remark}
The space $\frakM_m^{\operatorname{fr}}$ was constructed in \cite{KSV}. As an intermediate object, they first introduced another space $\frakM_m^{\mathrm{bl}}$, obtained by taking a real oriented blowup of all the strata in Deligne-Mumford space. Since the stratification has normal crossings, blowing up yields a smooth (in fact real-analytic) manifold with corners. Geometrically, this corresponds to adding relative framings \eqref{eq:relative-framing} only. One then has a principal bundle
\begin{equation}
(S^1)^{m+1} \longrightarrow \frakM_m^{\operatorname{fr}} \longrightarrow \frakM_m^{\mathrm{bl}}.
\end{equation}
\end{remark}

\subsubsection{The symmetric group}
The action of $\mathit{Sym}(m)$ on $\mathring{\frakM}_m^{\operatorname{fr}}$, by permuting $(z_1,\dots,z_m)$, is free. This is easy to see: a group element which preserves a point of $\mathring{\frakM}_m^{\operatorname{fr}}$ must be a finite order automorphism of $\bC$ which fixes the asymptotic marker at $z_0 = \infty$; that leaves a remaining automorphism group $\bC \rtimes \bR^{>0}$, which has no nontrivial finite order elements. As a consequence of the definition of differentiable structure, the $\mathit{Sym}(m)$-action extends smoothly to $\frakM_m^{\operatorname{fr}}$, and that extension remains free. One could prove the last-mentioned claim by an explicit combinatorial argument, but we prefer to apply the following general idea to the inclusion $\mathring{\frakM}_m^{\operatorname{fr}} \hookrightarrow \frakM_m^{\operatorname{fr}}$ (which is a homotopy equivalence, because it is the inclusion of the interior of a manifold with corners into the whole manifold).

\begin{lemma} \label{lem:freeabs} 
Let $X, Y$ be topological manifolds with boundary, equipped with continuous actions of a finite group (or more generally compact Lie group) $G$. Suppose  $f: X \rightarrow Y$ is a $G$-equivariant map which (non-equivariantly) is a homotopy equivalence. If the $G$-action on $X$ is free, the same is true for $Y$.
\end{lemma}

\begin{proof} 
Suppose on the contrary that $y \in Y$ has nontrivial stabilizer $G_y$ (in the case of a Lie group, $G_y \subset G$ is closed and hence again a Lie group). Take a cyclic subgroup $\bZ/p \subset G_y$, for $p$ prime. Inclusion of $y$ and the constant map induce homomorphisms
\begin{equation}
\xymatrix{
\ar@/_1pc/[rr]_-{\mathit{id}}
H_*^{\bZ/p}(\mathit{point};\bF_p) \ar[r] &
H_*^{\bZ/p}(Y;\bF_p) \ar[r] &
H_*^{\bZ/p}(\mathit{point};\bF_p)
}
\end{equation}
Since $H_*^{\bZ/p}(\mathit{point};\bF_p) = \bF_p$ for all $\ast \geq 0$, it follows that $H_*^{\bZ/p}(Y;\bF_p) \neq 0$ for all $\ast \geq 0$. On the other hand, $H_*^{\bZ/p}(X;\bF_p) \iso H_*(X/(\bZ/p);\bF_p)$, and $X/(\bZ/p)$ is again a manifold with boundary, so that homology group is zero in high degrees. That is a contradiction, since $f$ induces an isomorphism on equivariant homology.
\end{proof} 

\subsubsection{Cylindrical ends\label{sec:cylindrical}}
Take a punctured plane \eqref{eq:punctured-sphere-2}. Cylindrical ends are as in \eqref{eq:minus-end}, \eqref{eq:plus-end}. We adopt the standard gluing process from \eqref{eq:glued-surface}. If $S$ comes with asymptotic markers, one can ask for ends compatible with those markers. By this, we mean that the ends $\varepsilon_0,\dots,\varepsilon_m$ satisfy
\begin{equation}
\begin{aligned}
& \textstyle
\tau_{z_0} = \lim_{s\rightarrow -\infty} (\bR^{>0}\partial_s \varepsilon_0)_{(s,0)},
\\ & \textstyle
\tau_{z_k} = \lim_{s \rightarrow \infty} (-\bR^{>0}\partial_s \varepsilon_k)_{(s,0)}
\;\; \text{ for $k>0$.}
\end{aligned}
\end{equation}
Compatible ends form a contractible space. In particular, one can find a choice of such ends, for the universal family over $\mathring{\scrM}_m^{\operatorname{fr}}$, with $\smooth$ dependence on the modular parameters. One can simplify some technical aspects by restricting to the smaller class of rational ends, which form a finite-dimensional space of choices; here, rationality means that each end extends to a biholomorphic map $\{\pm\infty\} \cup (\bR \times S^1) \rightarrow \bC P^1$.

Suppose that a compatible choice of ends for the universal families has been made.
Given a surface $S = \bigsqcup_v S_v$ corresponding to a point in $\mathring{\frakM}_T^{\operatorname{fr}}$, and gluing parameters $\gamma = (\gamma_e)$, $\gamma_e \in [0,1)$ indexed by the set of finite edges, one can glue along the edges with $\gamma_e>0$ to obtain another surface, which belongs to the stratum of $\frakM^{\operatorname{fr}}_m$ labeled by the tree with those edges collapsed. This construction yields a map
\begin{equation}\label{eq:gluingKSV}
[0,1)^{\mathit{E}_{\operatorname{fin}}(T)} \times \mathring{\frakM}_T^{\operatorname{fr}} \longrightarrow 
\frakM_m^{\operatorname{fr}},
\end{equation}
which is $\smooth$ for the given differentiable structure. When restricted to a neighbourhood of $\{0\}^{\mathit{E}_{\operatorname{fin}}(T)} \times \mathring{\frakM}_T^{\operatorname{fr}}$,
it provides a collar neighbourhood of the $T$-stratum. We say that the choice of ends is consistent if the ends not used up in the gluing process, for small values of the gluing parameters, agree with those that are part of the universal choice for the glued surface. A consistent choice can be used to define a thick-thin decomposition on the universal family of surfaces, which is what was needed for the technical Floer-theoretic setup in Section \ref{subsubsec:operations}.

\subsubsection{Lollipops\label{sec:lollipop}}
Take an ordered configuration $(z_1,\dots,z_m)$ on the cylinder $\bR \times S^1$, and the associated surface
\begin{equation} \label{eq:cylinder-configuration}
S = (\bR \times S^1) \setminus \{z_1,\dots,z_m\} = 
(\overline{\bR \times S^1}) \setminus \{z_0 = -\infty,z_1,\dots,z_m,z_{m+1} = +\infty\}.
\end{equation}
In addition, we suppose that this comes with a choice of holomorphic isomorphisms 
\begin{equation} \label{eq:lollipop-maps}
\begin{aligned}
& \phi_1,\dots,\phi_r: \bar{S} \longrightarrow \bC P^1, \\ 
& \phi_i(-\infty) = \infty, \; \phi_i(\infty) = 0.
\end{aligned}
\end{equation}
One can think of each $\phi_i$ as determined by the point $\phi_i^{-1}(1) = (\sigma_i,\theta_i) \in \bR \times S^1$. Pairs \eqref{eq:cylinder-configuration}, \eqref{eq:lollipop-maps} are correspondingly parametrized by
\begin{equation}
\mathring{\frakL}_{m,r} \stackrel{\mathrm{def}}{=} (\mathit{Conf}_m(\bR \times S^1) \times (\bR \times S^1)^r)/(\bR \times S^1).
\end{equation}
A natural compactification $\frakL_{m,r}$ is the space of stable maps from genus $0$ curves with $(m+2)$ marked points to $(\bC P^1)^r$, of degree $(1,\dots,1)$, and with incidence conditions reflecting those above: the marked point $z_0$ goes to $(\infty,\dots,\infty)$, and $z_{m+1}$ to $(0,\dots,0)$. Examination of the first order deformation theory shows that this space is regular, hence a compact complex manifold, with a normal crossing stratification given by the number of nodes of the domain. Adopting the terminology from \cite[Section 6a]{abouzaid-seidel07} where a closely related concept was considered, we call this the space of lollipops.
\begin{figure}
\begin{centering}
\includegraphics{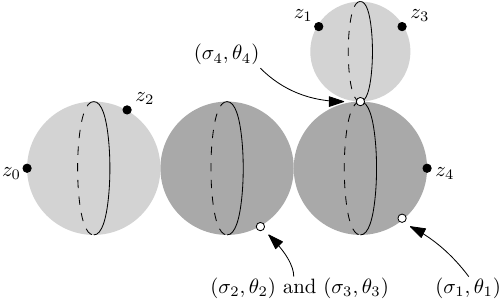}
\caption{\label{fig:lollipop}An example of the marked points from \eqref{eq:sigma-tau-point} (for $m = 3$, $r = 3$, $l = 4$). The lighter shaded components are where the map $\phi$ is constant.}
\end{centering}
\end{figure}

A point of $\frakL_{m,r}$ represents a connected compact nodal surface $\bar{S}$ and map $\phi = (\phi_1,\dots,\phi_r): \bar{S} \rightarrow \bC P^1$, but it is worth describing that information a little more explicitly. Let's write $\bar{S}_1,\dots,\bar{S}_l$ for those irreducible components of $\bar{S}$ which form a chain connecting the marked point $z_0 \in \bar{S}_1$ to $z_{m+1} \in \bar{S}_l$. For concreteness, let's choose identifications $\bar{S}_i = \overline{\bR \times S^1}$, so that $-\infty$ is either $z_0$ or the node closest to it on $\bar{S}_i$, and $+\infty$ is either $z_{m+1}$ or the node closest to it on $\bar{S}_i$. Let $p: \{1,\dots,r\} \rightarrow \{1,\dots,l\}$ be the map such that $\phi_j$ is non-constant precisely on $\bar{S}_{p(j)}$. We then mark the point
\begin{equation} \label{eq:sigma-tau-point}
(\sigma_j,\theta_j) = \phi_j^{-1}(1) \in \bR \times S^1 \subset \bar{S}_{p(j)}. 
\end{equation}
Note that these points can be nodes which connect $\bar{S}_i$ to one of the irreducible components of $\bar{S}$ which do not lie on our chain; or they can agree with one of the $(z_1,\dots,z_m)$; and moreover, several of those points can be the same (see Figure \ref{fig:lollipop}). The stability condition says that if some irreducible component does not contain any of the points \eqref{eq:sigma-tau-point}, then it must have at least three other special points. The data of the stable map is completely encoded in the surface $\bar{S}$ and additional points \eqref{eq:sigma-tau-point}.

As in the case of Deligne-Mumford space, one can modify the space by oriented real blowups, which geometrically adds relative framings at the nodes. This results in a compact manifold with corners $\frakL_{m,r}^{\mathrm{bl}}$. One can then further add framings at the marked points, which yields a principal bundle
\begin{equation} \label{eq:framed-lollipops}
(S^1)^{m+2} \longrightarrow \frakL_{m,r}^{\operatorname{fr}} \longrightarrow
\frakL_{m,r}^{\mathrm{bl}}.
\end{equation}

\subsection{Parameter spaces}

\subsubsection{Fulton-MacPherson spaces\label{section:FMd}}
We again consider configurations $(z_1,\dots,z_m)$ in $\bC$, for $m \geq 2$, but where now two configurations are isomorphic if they are related by a transformation $z \mapsto az+b$, with $a \in \bR^{>0}$ and $b \in \bC$. The interior of Fulton-MacPherson space is the parameter space
\begin{equation}
\mathring{\frakF}_m \stackrel{\mathrm{def}}{=} \mathit{Conf}_m(\bC)/(\bC \rtimes \bR^{>0}).
\end{equation}
The compactification $\frakF_m$ is a manifold with corners, stratified as in \eqref{eq:dm}:
\begin{equation} \label{eq:fm}
\frakF_m = \bigsqcup_T \mathring{\frakF}_T, \quad \mathring{\frakF}_T = \prod_v \mathring{\frakF}_{|v|_{\operatorname{in}}}.
\end{equation}
\begin{figure}
\begin{centering}
\includegraphics{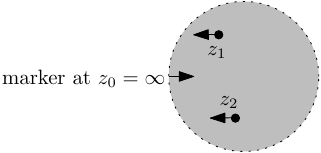}
\caption{\label{fig:punctured-plane}A punctured plane, with aligned asymptotic markers.}
\end{centering}
\end{figure}%
An application of Lemma \ref{lem:freeabs} shows that the $\mathit{Sym}(m)$-action on $\frakF_m$ is free. 

\begin{remark}
The name ``Fulton-MacPherson space'' as used here comes from the operad literature, e.g. \cite{salvatore21}. Briefly, the original paper of Fulton-MacPherson \cite{fulton-macpherson94} introduced a compactification of configuration space in $\bC^d$, for any $d$, which is a smooth algebraic variety. The strata in this compactification involve configurations divided by complex translation and rescaling. The spaces \eqref{eq:fm} are the version of that construction over $\bR$, for $d = 2$ (one can think of this as part of the real locus of the corresponding space over $\bC$, and use that to define the differentiable structure). The earliest occurrence of this version seems to be \cite[Section 3.2]{getzler-jones94} (see e.g.\ \cite[Section II.4.3]{markl-shnider-stasheff} for an exposition).
\end{remark}

Let's choose asymptotic markers which are {\em aligned}, in following sense. At the finite $z_i$, markers go in negative real direction; and at $z_0 = \infty$, the marker points in direction of the path $[0,1) \rightarrow \bC P^1$, $s \mapsto -1/s$ (Figure \ref{fig:punctured-plane}). Correspondingly, the convenient choice of ends is 
\begin{equation} \label{eq:punctured-plane-ends}
\begin{aligned}
& [0,\infty) \times S^1 \longrightarrow S, \; (s,t) \longmapsto z_j - \rho_j\exp(-2\pi(s+it)) \;\; \text{ $\rho_j>0$, near $z_j$ ($j = 1,\dots,m$);} \\
& (-\infty,0] \times S^1 \longrightarrow S, \; (s,t) \longmapsto \chi - \rho_0\exp(-2\pi(s+it)) \;\; \text{ $\rho_0>0$, $\chi \in \bC$, near $z_0 = \infty$.}
\end{aligned}
\end{equation}
These are a sub-class of rational ends, following the terminology of Section \ref{sec:framings}, and compatible with our choice of asymptotic markers. This particular class of ends is preserved under gluing, hence a consistent choice can be made within it.

\begin{remark}
A point in Fulton-MacPherson space describes a collection of configurations in $\bC$. One can compactify each component to $\bC P^1$, and glue those together to get a nodal surface, with $(m+1)$ smooth marked points. This surface inherits a marker at the output marked point $z_0$, and relative markers at the nodes. The aligned condition corresponds to specific choices of markers at $(z_1,\dots,z_m)$, but we can also rotate these by an arbitrary amount. The outcome of that yields a diffeomorphism
\begin{equation} \label{eq:m-rotate}
(S^1)^m \times \frakF_m \stackrel{\iso}{\longrightarrow} \frakM_m^{\operatorname{fr}},
\end{equation}
(this is well-known, compare e.g.\ the end of the proof of \cite[Proposition 2.1]{giansiracusa-salvatore}). The embedding $\frakF_m \hookrightarrow \frakM_m^{\operatorname{fr}}$ obtained by taking trivial rotations is particularly important, since it is compatible with the operad structure, or more concretely with gluing surfaces together.
\end{remark}

\subsubsection{Configurations on the cylinder\label{section:FM-cylinder}}
Consider configurations on the cylinder \eqref{eq:cylinder-configuration}, for some $m>0$, up to translation in $\bR$-direction. The parameter space is
\begin{equation} \label{eq:frakc}
\mathring{\frakC}_m = \mathit{Conf}_m(\bR \times S^1)/\bR.
\end{equation}
The construction of the compactification $\frakC_m$ can be divided into two steps. 
\begin{itemize} \itemsep.5em
\item
One first introduces a partial compactification $\frakC_m^\heart$ where, as points on the cylinder collide, they bubble into a configuration on $\bC$, determined up to $\bC \rtimes \bR^{>0}$. The stratification is correspondingly
\begin{equation} \label{eq:heart0}
\frakC_m^\heart = \bigsqcup_T \frakC_T^\heart, \quad
\frakC_T^\heart = 
\mathring{\frakC}_{|v_{\operatorname{root}}|_{\operatorname{in}}} \times
\prod_{v \neq v_{\operatorname{root}}} \mathring{\frakF}_{|v|_{\operatorname{in}}}.
\end{equation}
\item The full compactification builds in breaking of the cylinder, as points go to $\pm\infty$:
\begin{equation} \label{eq:cm-space0}
\frakC_m = \bigsqcup_{\substack{l \geq 1 \\ m_1+\cdots+m_l = m}}
\frakC_{m_1}^\heart \times \cdots \times \frakC_{m_l}^\heart.
\end{equation}
\end{itemize}
\begin{figure}
\begin{centering}
\includegraphics{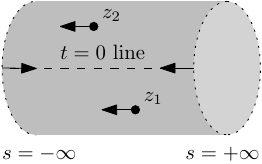}
\caption{\label{fig:markers-cylinder}Asymptotic markers on a cylinder.}
\end{centering}
\end{figure}%

We choose {\em $S^1$-invariant} asymptotic markers for the punctured cylinder (see Figure \ref{fig:markers-cylinder}). At $\pm\infty$, these point in the direction of the paths $s \mapsto (\pm 1/s,0)$, and at $z_1,\dots,z_m$ they go in negative $s$-direction. Correspondingly, one picks ends
\begin{align} \label{eq:endoncylinder1}
&
(-\infty,0] \times S^1 \longrightarrow S, \;\; (s,t) \longmapsto (s+\sigma_0,t)
\;\; \text{ $\sigma_0 \in \bR$, near $z_0 = -\infty$;}
\\ &
\label{eq:log-end}
[0,\infty) \times S^1 \longrightarrow S, \;\; (s,t) \longmapsto  z_j -\rho_j \exp(-2\pi(s+it)) \;\; \text{ $\rho_j>0$, near $z_j$, $j = 1,\dots,m$;}
\\ &
\label{eq:endoncylinder2} 
[0,\infty) \times S^1 \longrightarrow S, \;\; (s,t) \longmapsto (s+\sigma_{m+1},t)
\;\; \text{ $\sigma_{m+1} \in \bR$, near $z_{m+1} = +\infty$.}
\end{align}
These are not rational ends in the previously defined sense, but they play a similar role of reducing the amount of choice involved. If one uses these and the previously introduced ends on Fulton-MacPherson spaces, then the whole can be made consistent with respect to both gluing processes that occur here (merging two cylinders along $\pm\infty$, as well as inserting a punctured plane at some point in a cylinder).
\begin{figure}
\begin{centering}
\includegraphics{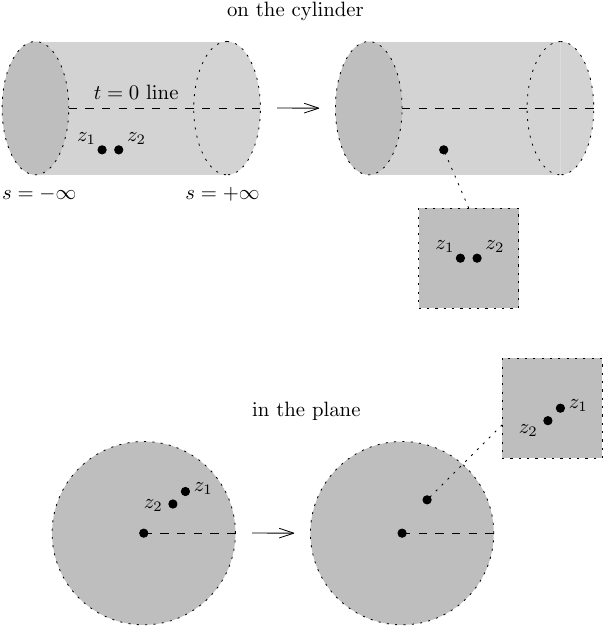}
\caption{\label{fig:compare-c-f}The degeneration from Remark \ref{th:compare-c-f}.}
\end{centering}
\end{figure}%

\begin{remark} \label{th:compare-c-f}
It turns out that 
\begin{equation} \label{eq:compare-c-f}
\frakC_m \iso \frakF_{m+1}. 
\end{equation}
On the interior, take
\begin{equation}
\bR \times S^1 \stackrel{\iso}{\longrightarrow} \bC^*, \quad s+it \longmapsto \exp(-2\pi (s+it))
\end{equation}
This induces an isomorphism $\mathring{\frakC}_m \iso \mathring{\frakF}_{m+1}$. In terms of combinatorics, a stratum of $\frakC_m$ is associated to a collection of trees $T_1,\dots,T_m$. For each $i>0$, let's turn the semi-infinite outgoing edge of each $T_{i+1}$ into a finite edge going towards the root of $T_i$; and add a semi-infinite edge going towards of the root of $T_m$. The outcome is a single tree $T$ with $(m+1)$ semi-infinite ends, which indexes the corresponding stratum of $\frakF_{m+1}$. The problem with \eqref{eq:compare-c-f} is that it is not compatible, at least not in the most straightforward way, with writing the stratifications as products in \eqref{eq:fm} and \eqref{eq:heart0}. To see that, take Figure \ref{fig:compare-c-f}, which shows the same degeneration (two points colliding) in $\mathring\frakC_2$ and $\mathring\frakF_3$. This yields limits in the strata $\mathring{\frakC}_1 \times \mathring{\frakF}_2 \subset \frakC_2$ respectively $\mathring{\frakF}_2 \times \mathring{\frakF}_2 \subset \frakF_3$, but those limits are not obtained from each other by just applying \eqref{eq:compare-c-f} in each factor: instead, the second factor should be rotated in a way which depends on the position of the point in the first factor. For that reason, we prefer to treat $\frakC_m$ and $\frakF_{m+1}$ separately. In accordance with that distinction, our choices of framings yield embeddings $\frakF_{m+1} \rightarrow \frakM_{m+1}^{\operatorname{fr}}$, $\frakC_m \rightarrow \frakM_{m+1}^{\operatorname{fr}}$ with different images.
\end{remark}

\subsubsection{Angle-decorated cylinders\label{section:angles}}
Fix some $r > 0$, and consider an $r$-tuple 
\begin{equation} \label{eq:angle-decorated}
(\sigma_1,\theta_1), \dots, (\sigma_r, \theta_r) \in \bR \times S^1,
\quad \sigma_1 \leq \cdots \leq \sigma_r.
\end{equation}
Geometrically, one thinks of marking the cylinder $S = \bR \times S^1$ with the circles $\{s = \sigma_i\}$, and where each circle comes equipped with an angle $\theta_i$. We call this an {\em angle-decorated cylinder}. Instead of the separate angles, it can often be convenient to work with sums 
\begin{equation} \label{eq:tildecoordinates} 
\theta_{\leq i} = \theta_1 + \cdots + \theta_i, \;\;
\theta_{\geq i} = \theta_i + \theta_{i+1} + \cdots + \theta_r, \;\;
\theta_{\mathrm{tot}} = \theta_{\leq r} = \theta_{\geq 1}.
\end{equation}
Suppose we have two angle-decorated cylinders $S_{\pm}$. When gluing the $+\infty$ end of $S_+$ to the $-\infty$ end of $S_-$ to form a new cylinder, our convention is to always use {\em angle-twisted gluing}, which means that $(s,t) \in S_+$ gets identified with $(s-l,t+\theta_{\mathrm{tot},-}) \in S_-$ ($l$ is the gluing length; what's important is the appearance of the total angle of $S_-$).

Write $\Sigma_r \subset \bR^r$ for the space of $(\sigma_1,\dots,\sigma_r)$ satisfying the condition from \eqref{eq:angle-decorated}. The parameter space of angle-decorated cylinders is
\begin{equation}
\mathring{\frakA}_r = (\Sigma_r \times (S^1)^r)/\bR.
\end{equation}
Unlike the previous situations, this already has codimension $1$ boundary faces before compactification, which occur when $\sigma_i = \sigma_{i+1}$ for some $i$. Those faces comes with natural maps
\begin{equation} \label{eq:forget-circle}
\partial_{\sigma_{i+1}=\sigma_i} \mathring{\frakA}_r \longrightarrow \mathring{\frakA}_{r-1},
\end{equation}
where we replace the adjacent pair $(s_i,\theta_i)$, $(s_{i+1} = s_i, \theta_{i+1})$ with $(s_i,\theta_i+\theta_{i+1})$, preserving $\theta_{\mathrm{tot}}$.
\begin{figure}
\begin{centering}
\includegraphics{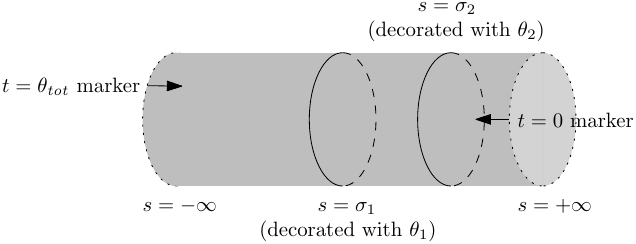}
\caption{\label{fig:angle-decorated}An angle-decorated cylinder.}
\end{centering}
\end{figure}%

The compactification $\frakA_r$ is defined by allowing some of the $\sigma_i$ to go to $\pm\infty$, leading to a finite collection of cylinders of the same kind, which means that
\begin{equation} \label{eq:a-space-compactified}
\frakA_r = \bigsqcup_{\substack{l \geq 1 \\ r_1+\cdots+r_l = r}} \mathring{\frakA}_{r_1} \times \cdots \times \mathring{\frakA}_{r_l}.
\end{equation}
The map $(\theta_1,\dots,\theta_r)$ extends smoothly to the compactification, and on each stratum \eqref{eq:a-space-compactified}, that extension is the product of the corresponding maps on the factors. The forgetful map \eqref{eq:forget-circle} also extends smoothly to compactifications.

One always chooses the asymptotic marker at $+\infty$ to point in direction of $s \mapsto (1/s,0)$, but that at $-\infty$ to point in the direction of $s \mapsto (-1/s,\theta_{\mathrm{tot}})$ (Figure \ref{fig:angle-decorated}). The ends are correspondingly taken to be
\begin{equation} \label{eq:rotated-end}
\begin{aligned}
& (-\infty,0] \times S^1 \longrightarrow S, \quad (s,t) \longmapsto (s+\sigma_0,\, t+\theta_{\mathrm{tot}}), \\
& [0,\infty) \times S^1 \longrightarrow S, \quad (s,t) \longmapsto (s+\sigma_1,t)
\end{aligned}
\end{equation}
As usual, one wants to make a universal choice of ends over $\mathring{\frakA}_r$, consistent with angle-twisted gluing. One additionally requires that the choice of ends should be compatible with \eqref{eq:forget-circle}. 

\subsubsection{Angle-decorated configurations\label{section:anglesFM}}
We next define parameter spaces $\AC_{m,r}$, for $m + r >0$, which are a mashup of those from Section \ref{section:FM-cylinder} and \ref{section:angles}. Consider configurations of $m$ marked points on the cylinder, and also data \eqref{eq:angle-decorated}, both up to common translation in $\bR$-direction:
\begin{equation}
\mathring{\AC}_{m,r} = \big(\mathit{Conf}_m(\bR \times S^1) \times \Sigma^r \times (S^1)^r\big)/\bR.
\end{equation}
As in \eqref{eq:forget-circle}, this space has boundary faces where $\sigma_i = \sigma_{i+1}$, and forgetful maps
\begin{equation} \label{eq:ac-forget}
\partial_{\sigma_{i+1}=\sigma_i} \mathring{\AC}_{m,r} \longrightarrow \mathring{\AC}_{m,r-1}.
\end{equation}
The compactification is constructed as in Section \ref{section:FM-cylinder}, by first taking
\begin{equation} \label{eq:ac-heart-space}
\AC_{m,r}^\heart = \bigsqcup_T \AC_T^\heart, \quad
\AC_T^\heart = 
\mathring{\AC}_{|v_{\operatorname{root}}|_{\operatorname{in}},r} \times
\prod_{v \neq v_{\operatorname{root}}} \mathring{\frakF}_{|v|_{\operatorname{in}}},
\end{equation}
amd then
\begin{equation} \label{eq:ac-space}
\AC_{m,r} = \bigsqcup_{\substack{l \geq 1 \\ m_1+\cdots+m_l = m \\ r_1+\cdots + r_l = r}}
\AC_{m_1,r_1}^\heart \times \cdots \times \AC_{m_l,r_l}^\heart.
\end{equation}
These identifications are chosen so that angle-twisted gluing is continuous (Figure \ref{fig:angle-twisted-stretching}). The outcome is a manifold with corners. One chooses asymptotic markers and ends at $\pm\infty$ as in \eqref{eq:rotated-end}; and around $(z_1,\dots,z_m)$ as in \eqref{eq:log-end}, consistently with gluing and compatibly with \eqref{eq:ac-forget}.
\begin{figure}
\begin{centering}
\includegraphics{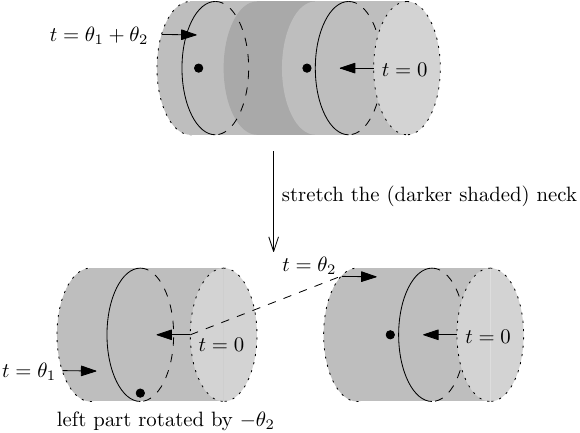}
\caption{\label{fig:angle-twisted-stretching}A degeneration in $\AC_{2,2}$, with limit in $\mathring{\AC}_{1,1} \times \mathring{\AC}_{1,1}$.
}
\end{centering}
\end{figure}

%

\begin{remark}
Starting with our choice of asymptotic markers, and then allowing those at $(z_0,z_1,\dots,z_m)$ to rotate, yields an analogue of \eqref{eq:m-rotate}, namely a map into the space of framed lollipops from \eqref{eq:framed-lollipops},
\begin{equation}
(S^1)^{m+1} \times \AC_{m,r} \longrightarrow \frakL_{m,r}^{\operatorname{fr}}.
\end{equation}
These spaces in fact only differ by the condition $\sigma_i \leq \sigma_{i+1}$. One can show that the equality $\sigma_i = \sigma_{i+1}$ is transverse to all the boundary strata of $\frakL_{m,r}^{\operatorname{fr}}$. This is one way to prove that $\AC_{m,r}$ is a manifold with corners, since that reduces the question to the more familiar study of the stable map spaces that underlie the notion of lollipop.
\end{remark}

\subsubsection{Cartan homotopy moduli spaces (A)\label{section:cartan-A}}
We now turn to the two spaces used to define the connection on equivariant Floer cohomology. Both are constructed inside $\AC_{m,r}$, and use the same kinds of asymptotic markers and ends.

Take a configuration \eqref{eq:cylinder-configuration}, $m>0$, with angle decorations \eqref{eq:angle-decorated}. Fix some $w \in \{0,\dots,r\}$ and impose the following constraint on the first point (Figure \ref{fig:cartan-A}):
\begin{equation} \label{eq:s-moves}
z_1 = (s_1,t_1 = \theta_{\geq w+1}), \quad
s_1 \in \begin{cases}
\bR & r = 0, \\
(-\infty,\sigma_1] & w = 0, \, r>0 \\
[\sigma_w,\sigma_{w+1}] & w = 1,\dots,r-1, \\
[\sigma_r,\infty) & w = r, \, r>0
\end{cases}
\end{equation}
The space of those, up to $\bR$-translation, is written as $\mathring{\AC}_{m,r,w}^{(A)}$. This has different kinds of codimension $1$ boundary faces:
\begin{itemize} \itemsep.5em
\item
We can have $\sigma_i = \sigma_{i+1}$ for some $i \neq w$ (the $i = w$ case is of codimension $2$, hence not listed here). Those faces comes with the usual map (note that the forgetting process is compatible with the condition that $t_1 = \theta_{\geq w+1}$)
\begin{equation}
\partial_{\sigma_i = \sigma_{i+1}} \mathring{\AC}_{m,r,w}^{(A)} \longrightarrow
\begin{cases}
\mathring{\AC}_{m,r,w-1}^{(A)} & i<w, \\
\mathring{\AC}_{m,r,w}^{(A)} & i>w.
\end{cases}
\end{equation}

\item
We can have $s_1 = \sigma_w$ (when $w \neq 0$) or $s_1 = \sigma_{w+1}$ (when $w \neq r$). Those boundary faces are written as $\partial_{s_1 = \sigma_w} \mathring{\AC}_{m,r,w}^{(A)}$, $\partial_{s_1 = \sigma_{w+1}} \mathring{\AC}_{m,r,w}^{(A)}$.
\end{itemize}
\begin{figure}
\begin{centering}
\includegraphics{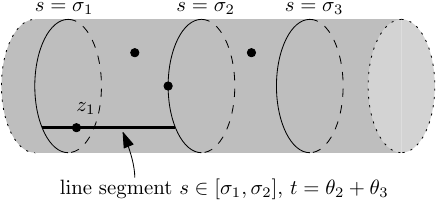}
\caption{\label{fig:cartan-A}The position constraint from \eqref{eq:s-moves}.}
\end{centering}
\end{figure}

Clearly, $\mathring{\AC}_{m,r,w}^{(A)}$ is a codimension $1$ submanifold with boundary in $\mathring{\AC}_{m,r}$. We define $\AC_{m,r,i}^{(A)}$ to be its closure in $\AC_{m,r}$. Concretely, this means that one has a partial compactification $\AC_{m,r,w}^{(A),\heart}$ as in \eqref{eq:ac-heart-space}, and then the analogue of \eqref{eq:ac-space} is
\begin{equation} \label{eq:FMdkb3} 
\AC_{m,r,w}^{(A)} = \bigsqcup\, \AC_{m_1,r_1}^{\heart} \times \cdots \times \AC_{m_k,r_k,w-r_1-\cdots-r_{k-1}}^{(A),\heart} \times \cdots \times \AC_{m_l,r_l}^{\heart}.
\end{equation}
The union is over all $l \geq 1$, $m_1 + \cdots + m_l = m$, $r_1 + \cdots + r_l = r$, and $1 \leq k \leq l$ such that the spaces in the formula make sense. This closure is a submanifold with corners inside $\AC_{m,r}$. The main point here is that $\theta_{\geq w+1}$ extends to a function on $\AC_{m,r}$ without any critical points, and the same for its restriction to any closed stratum; hence setting $t_1 = \theta_{\geq w+1}$ is automatically a transverse constraint, even though $t_1$ does have critical points. The same applies to the other part of \eqref{eq:s-moves} in a slightly more complicated way (one can use translation-invariance to fix $s_1 = 0$, and then the constraints become $\sigma_w \leq 0$, $\sigma_{w+1} \geq 0$).
The compactification has additional codimension $1$ boundary faces:
\begin{itemize} \itemsep.5em
\item The cylinder can split into two, with the part carrying $z_1$ lying either on the left or right. This is the case $l = 2$ of \eqref{eq:FMdkb3}, which means that the open stratum is
\begin{equation}
\mathring{\AC}_{m_1,r_1} \times \mathring{\AC}_{m_2,r_2,w-r_1}^{(A)} \;\; \text{ or } \;\;
\mathring{\AC}_{m_1,r_1,w}^{(A)} \times \mathring{\AC}_{m_2,r_2}, \;\;
m_1+m_2 = m, \; r_1+r_2=r.
\end{equation}
\item We can have Fulton-MacPherson bubbling, where a group of punctures collide. Strictly speaking, there are two sub-cases here, depending on whether $z_1$ ends up on the bubble of not. 
\end{itemize}
\begin{figure}
\begin{centering}
\includegraphics{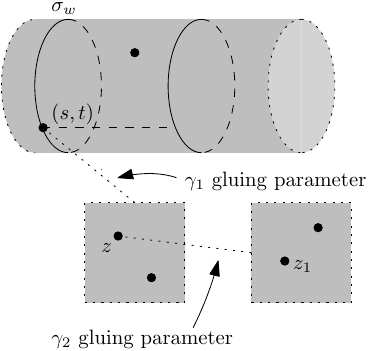}
\caption{\label{fig:fm-bubble}The gluing process from Remark \ref{th:glue-constraint}.}
\end{centering}
\end{figure}%

\begin{remark} \label{th:glue-constraint}
To further illustrate the submanifold-with-corners structure, it is useful to look at the coordinates given by the gluing process. Take a stratum in \eqref{eq:ac-heart-space} where two levels of Fulton-MacPherson bubbling have occurred at a point on the $\sigma_w$ circle, and where the point $z_1 \in \bC$ lies on one of the resulting bubbles, as in Figure \ref{fig:fm-bubble}. Write $z \in \bC$ and $(s,t) \in \bR \times S^1$ for the relevant attaching points (which can move as part of the parameter space, as can $z_1$). After gluing the surfaces together using parameters $\gamma_1, \gamma_2 > 0$  (suppressing some irrelevant constants), the point $z_1$ would end up at $(s,t) + \gamma_1(z + \gamma_2 z_1) \in \bR \times S^1$. Then, \eqref{eq:s-moves} is given by the following conditions on $(\gamma_1,\gamma_2,\sigma_w,s,t,z,z_1)$:
\begin{equation}
\begin{aligned}
& s + \gamma_1 \mathrm{re}(z) + \gamma_1 \gamma_2 \mathrm{re}(z_1) \geq \sigma_w, \\
& t + \gamma_1 \mathrm{im}(z) + \gamma_1 \gamma_2 \mathrm{im}(z_1) = \theta_{\geq w+1}.
\end{aligned}
\end{equation}
As one can see, if one turns the first into an equality, it is always transversally cut out at $\gamma_1 = \gamma_2 = 0$; and so is the second one.
\end{remark}
\begin{figure}
\begin{centering}
\includegraphics{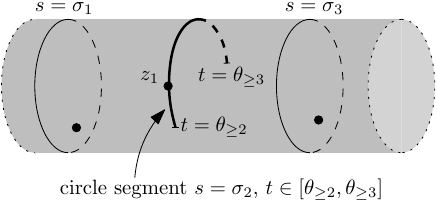}
\caption{\label{fig:cartan-B}The position constraint from \eqref{eq:t-moves}.}
\end{centering}
\end{figure}

\subsubsection{Cartan homotopy moduli spaces (B)\label{section:cartan-B}}
We start with a cylinder configuration as before, some $r \geq 1$ and $1 \leq w \leq r$. As part of the data, we include a lift of $\theta_w$ to $[0,1]$, and an extra variable $\xi$, which enter into the position constraint for $z_1$ (Figure \ref{fig:cartan-B} shows a slightly simplified picture of the outcome):
\begin{equation} \label{eq:t-moves} 
\theta_w^{\operatorname{lift}} \in [0,1],\;\; \xi \in [\theta_w^{\operatorname{lift}},1], \quad
z_1 = (s_1 = \sigma_w, \, t_1= \theta_{\geq w+1} + \xi).
\end{equation}
The resulting parameter space $\mathring{\AC}_{m,r,w}^{(B)}$ has the following codimension $1$ boundary faces:
\begin{itemize}
\itemsep.5em
\item
The by now standard ones where $\sigma_i = \sigma_{i+1}$ for some $i$. In the cases where $i = w-1$ and $i = w$, the forgetful maps should be thought of as ignoring the condition on $t_1$, which means they have the form
\begin{equation}
\begin{aligned}
&
\partial_{\sigma_{w-1} = \sigma_w} \mathring{\AC}_{m,r,w}^{(B)} \longrightarrow 
\{
\text{subset of $\mathring{\AC}_{m,r-1}$ where $z_1 = (s_1,t_1)$, $s_1 = \sigma_{w-1}$}
\}, \\
&
\partial_{\sigma_w = \sigma_{w+1}} \mathring{\AC}_{m,r,w}^{(B)} \longrightarrow
\{\text{subset of $\mathring{\AC}_{m,r-1}$ where $z_1 = (s_1,t_1)$, $s_1 = \sigma_w$}\}.
\end{aligned}
\end{equation}
It is still true that they decrease dimensions, which is what's important for us.

\item
We can have $t_1 = \theta_{\geq w}$. This comes with a map to one of the boundary faces of the (A) type spaces,
\begin{equation} \label{eq:cylinder-left-right}
\partial_{t_1 = \theta_{\geq w}} \mathring{\AC}_{m,r,w}^{(B)} \longrightarrow
\partial_{s_1 = \sigma_w} \mathring{\AC}_{m,r,w-1}^{(A)},
\end{equation}
which is an isomorphism away from a subset of positive codimension (the issue being the lifts to $[0,1]$ in \eqref{eq:t-moves}, which carry nontrivial information only at the endpoints $\{0,1\}$). At the other extreme, we can have $t_i = \theta_{\geq w+1}$, and a parallel map
\begin{equation} \label{eq:ABisomorphism2}
\partial_{t_1 = \theta_{\geq w+1}} \mathring{\AC}_{m,r,w}^{(B)} \longrightarrow \partial_{s_1 = \sigma_w} \mathring{\AC}_{m,r,w}^{(A)}.
\end{equation}

\item
Finally, if $\theta_w^{\operatorname{lift}} = 0$, then $t_1$ can lie anywhere on the circle $\{s = \sigma_w\}$. We can forget $\sigma_w$, and then the condition is that $z_1$ can lie anywhere in the annulus $[\sigma_{w-1},\sigma_{w+1}] \times S^1$. Renumbering of the angle decorations therefore yields a map
\begin{equation} \label{eq:0-forget}
\partial_{\theta_i^{\operatorname{lift}} = 0} \mathring{\AC}_{m,r,w}^{(B)} \longrightarrow \mathring{\AC}_{m,r-1}.
\end{equation}
The image of this is the subset of configurations decorated with $(r-1)$ angles, such that $z_1$ lies in the annulus $[\sigma_{w-1},\sigma_w] \times S^1$.
\end{itemize}
As before we have a compactification $\AC_{m,r,w}^{(B)}$, which contributes the same additional codimension $1$ boundary strata (cylinder-breaking, bubbling off of Fulton-MacPherson configuration) as in the (A) type situation.

\subsubsection{Orientation conventions\label{subsec:or-con}}
We end by listing the choices of orientations for the main parameter spaces. The general convention for orientations on manifolds with boundary $N$ is this: if $(\tau_1,\tau_2,\dots)$ is a positively oriented basis for $T(\partial N)$, and $\nu$ is an outwards pointing normal vector, then $(\nu,\tau_1,\tau_2,\dots)$ is a positively oriented basis for $TN$. 
\begin{itemize}
\itemsep.5em
\item \parindent0em \parskip1em
For every configuration $(z_1,\dots,z_m)$ there is a short exact sequence
\begin{equation} \label{eq:orient} 
0 \rightarrow \bC \oplus \bR \longrightarrow \bC^m \longrightarrow T_{[z_1,\dots,z_m]} \mathring{\frakF}_m \rightarrow 0,
\end{equation} 
where $1 \in \bC$ is mapped to $(1,\dots,1)$, and $1 \in \bR$ to $(-z_1,\dots,-z_m)$ (infinitesimal shrinking of a configuration). We use the orientation of $\frakF_m$ compatible with that sequence. Let's illustrate the example $m = 2$: there, the left hand side of \eqref{eq:orient} gives three linearly independent elements in $\bC^2$, which are naturally completed to an oriented basis
\begin{equation} \label{eq:2-basis}
(1,1), \, (i,i), \, (-z_1,-z_2), \, (-iz_1,-iz_2) \in \bC^2.
\end{equation}
The relative angle $\alpha(z_1,z_2) = \arg(z_2-z_1)$ satisfies $\alpha'_{(z_1,z_2)}(-iz_1,-iz_2) = -1$, hence induces an {\em orientation-reversing} diffeomorphism $\frakF_2 \iso S^1$. 

Note that by definition, $\mathit{Sym}_m$ acts orientation-preservingly. One can also check that the inclusions of boundary faces 
\begin{equation} \label{eq:b-to-r}
\mathring{\frakF}_{m_2} \times \mathring{\frakF}_{m_1} \hookrightarrow \partial\frakF_m,\;\; m = m_1+m_2-1
\end{equation}
are compatible with orientations. Here, in terms of \eqref{eq:fm}, $\frakF_{m_1}$ corresponds to the root vertex of the tree, so we are ordering the factors in \eqref{eq:b-to-r} branch-to-root (the use of this ordering, specifically for orientation purposes, is motivated by the bottom left entry in Figure \ref{fig:orientation-spaces}, where it appears naturally). 

\item
The orientation of $\frakF_{m+1}$ carries over to $\frakC_m$ via \eqref{eq:compare-c-f}. More directly, this orientation is given by the short exact sequence
\begin{equation} \label{eq:right-moving}
0 \rightarrow \bR \longrightarrow \bC^m \longrightarrow T_{[z_1,\dots,z_m]} \mathring{\frakC}_m \rightarrow 0,
\end{equation}
where $1 \in \bR$ is mapped to the infinitesimal translation $(1,\dots,1)$ (because the marked point $z_0$ is now at $-\infty$, and our previous sign convention was to move away from $z_0$). Again, it is instructive to look at the simplest case $m = 1$: there, 
$\mathrm{im}(z_1)$ gives an {\em orientation-preserving} diffeomorphism $\frakC_1 \iso S^1$.

As before, this will cause the orientations to be compatible with the inclusions of faces $\mathring{\frakC}_{m_2} \times \mathring{\frakC}_{m_1} \hookrightarrow \partial\frakC_{m_1+m_2}$, (note the ordering of components is from $s = +\infty$ to $s = -\infty$, which agrees with our convention for Fulton-MacPherson spaces). The same holds for faces $\mathring{\frakF}_{m_2} \times \mathring{\frakC}_{m_1} \hookrightarrow \partial\frakC_{m_1+m_2-1}$.

\item 
For angle-decorated cylinders, we group each $(\sigma_k,\theta_k)$ into a complex number, and use the sequence parallel to \eqref{eq:right-moving},
\begin{equation}
0 \rightarrow \bR \longrightarrow \bC^r \longrightarrow T_{[\sigma_1,\theta_1,\dots,\sigma_r,\theta_r]} \mathring{\frakA}_r \rightarrow 0.
\end{equation}
As in the previous cases, this is compatible with boundary faces. The same principle is used for angle-decorated configurations.

\item 
For type (A) Cartan homotopy moduli spaces (Section \ref{section:cartan-A}), we orient the line on which the first marked point lies towards the right (away from $z_0 = -\infty$). The other marked points and angle decorations are treated as before, and we orient these moduli spaces by the resulting exact sequence:  
\begin{equation}
0 \rightarrow \bR \longrightarrow \bR \oplus \bC^{m+r-1} \longrightarrow T_{[(\sigma_i,\theta_i),z_1,\dots,z_m]} \mathring{\AC}_{m,r,w}^{(A)} \rightarrow 0.
\end{equation}
In the simplest case $\mathring{\AC}_{1,0,0}^{(A)} = \{\mathit{point}\}$, this convention counts that point as $+1$. In general, this is the same as the orientation obtained by breaking the translational symmetry to set $z_1 = (0,\theta_{\geq w+1})$, and then treating the remaining $z_k$ and $(\sigma_k,\theta_k)$ as complex numbers. 

Yet another equivalent description would be to say that the map $\mathring{\AC}_{m,r,w}^{(A)} \times S^1 \rightarrow \mathring{\AC}_{m,r}$, where the $S^1$ factor rotates the second component of $z_1 = (s_1,t_1)$, is compatible with orientations. This point of view is particularly convenient for considering compatibility with boundary orientations, since it reduces that question to the previously discussed cases. It turns out that (again with factors ordered from $s = +\infty$ to $s = -\infty$)
\begin{align} \label{eq:type-a-orientation-1}
& \mathring{\AC}_{m_2,r_2,w-r_1}^{(A)} \times \mathring{\AC}_{m_1,r_1} \hookrightarrow
\partial \mathring{\AC}_{m_1+m_2,r_1+r_2,w}^{(A)} \text{ is orientation-reversing;} \\
\label{eq:type-a-orientation-2}
& \mathring{\AC}_{m_2,r_2} \times \mathring{\AC}_{m_1,r_1,w}^{(A)} \hookrightarrow
\partial \mathring{\AC}_{m_1+m_2,r_1+r_2,w}^{(A)} \text{ is orientation-preserving;} \\
& \mathring{\frakF}_{m_2} \times \mathring{\AC}_{m_1,r,w}^{(A)} \hookrightarrow \partial \mathring{\AC}_{m_1+m_2-1,r,w}^{(A)} \text{ is orientation-preserving.}
\label{eq:type-a-orientation-3}
\end{align}
(The orientation-reversal happens because one has to swap the extra $S^1$ factor with $\mathring{\AC}_{m_1,r_1}$ in the ordering, and both are odd-dimensional.)

\item For type (B) Cartan homotopy moduli spaces (Section \ref{section:cartan-B}) we orient the circle on which $z_1$ lies counterclockwise, meaning in direction of increasing $t_1$. This gives rise to an analogous exact sequence which we use to orient the moduli spaces:
\begin{equation} \label{eq:b-sequence}
0 \rightarrow \bR \longrightarrow \bR \oplus \bC^{m+r-1} \longrightarrow T_{[(\sigma_i,\theta_i),z_1,\dots,z_m]} \mathring{\AC}_{m,r,w}^{(B)} \rightarrow 0.
\end{equation}
Equivalently, one can say that the map $\mathring{\AC}_{m,r,w}^{(B)} \times \bR \rightarrow \mathring{\AC}_{m,r}$, where $\bR$ translates the first component of $z_1 = (s_1,t_1)$ to the left, is compatible with orientations. For the boundary faces parallel to those in \eqref{eq:type-a-orientation-1}--\eqref{eq:type-a-orientation-3}, we get the same orientation behaviour. 

There is an additional boundary face, which appears in the forgetful map \eqref{eq:0-forget}. Let's consider the simplest case $m = r = w = 1$. The space $\bR \oplus \bC$ in the middle of \eqref{eq:b-sequence} is oriented according to the coordinates $(t_1,\sigma_1,\theta_1^{\mathrm{lift}})$, and since $\theta_1^{\mathrm{lift}} = 0$ is the minimal value of that coordinate, the resulting boundary face is ordered opposite to $(t_1,\sigma_1)$, or equivalently agrees with the orientation from $(\sigma_1,t_1)$. The forgetful map corresponds to considering $(\sigma_1,t_1)$ as a complex number, and is accordingly orientation-preserving. The same is true in general.
\end{itemize}


\subsection{Closed string operations \label{section:closedstringoperations}} 

\subsubsection{$L_\infty$-structure\label{section:Linf}} We are now ready to construct the operations on symplectic cohomology, following the general framework laid out in Section \ref{subsubsection:Floerfamilies}. Take $\hat{N}$ as in \eqref{eq:liouville-completion}, \eqref{eq:calabi-yau}. We also fix $(H^{\floer}, J^{\floer})$ which sets up the Floer complex (Section \ref{subsubsec:gradings-and-orientations}).

We begin by introducing operations parameterized by the spaces $\frakF_m$ (Section \ref{section:FMd}) for each $m \geq 2$. We suppose that a universal family of cylindrical ends  \eqref{eq:punctured-plane-ends} on each $\mathring{\frakF}_m$ has been fixed, consistently with gluing. We also choose families of Hamiltonian perturbations and almost complex structures on the universal curves,
\begin{equation} \label{eq:f-pair}
K_{\frakF_m} \in \Omega^1_{U_{\frakF_{m}}/ \frakF_{m}}(U_{\frakF_{m}}, \scrH(\hat{N})) \text{ and } J_{\frakF_m} \in C^{\infty}(U_{\frakF_{m}}, \scrJ(\hat{N})),
\end{equation} 
which satisfy the conditions of Section \ref{subsubsec:operations} when restricted to each point in parameter space, are conformally consistent with respect to boundary strata in the sense of Section \ref{subsubsec:neck}, and are equivariant with respect to the $\mathit{Sym}(m)$-action.

\begin{remark}
The equivariance constraint is unproblematic for transversality, because the action on $\mathring{\frakF}_m$ is free. More concretely, any point in $\mathring{\frakF}_m$ has a neighbourhood $U$ which is disjoint from its image $\sigma(U)$ under any nontrivial $\sigma \in \mathit{Sym}(m)$. On such a $U$, equivariance does not restrict the choice of \eqref{eq:f-pair}, and since regularity of the parametrized moduli space is a local property over the parameter space, that suffices. Equivalently, one can think of regularity being achieved over the quotient $\mathring{\frakF}_m/\mathit{Sym}(m)$.
\end{remark}

Let $\mathring{\frakF}_m(\mathbf{x}) = \mathring{\frakF}_m(x_0,\dots,x_m)$ be the parametrized moduli space of maps satisfying the Cauchy-Riemann equations, and with fixed asymptotics (this follows the general notation from Section \ref{subsubsection:Floerfamilies}). For generic choice of \eqref{eq:f-pair}, this is a manifold of the expected dimension 
\begin{equation} \label{eq: vdimM} 
\operatorname{dim}\,\mathring{\frakF}_{m}(\mathbf{x}) = \deg(x_0)-\sum_{i=1}^{m}(\deg(x_i))+2m-3.
\end{equation} 
Following \eqref{eq:oru}, each point $(r,u)$ in this space gives rise to an isomorphism 
\begin{equation} \label{eq:oru-2}
o(r,u): \lambda^{\mathrm{top}}(T_{(r,u)}\mathring{\frakF}_m(\bfx)) \otimes \frako_{x_1} \otimes \cdots \otimes \frako_{x_m} \iso \lambda^{\mathrm{top}}(T_r\mathring{\frakF}_m) \otimes \frako_{x_0}.
\end{equation}
The spaces $\mathring{\frakF}_{m}(\mathbf{x})$ have Gromov compactifications 
\begin{equation} 
\frakF_{m}(\mathbf{x}) = \bigsqcup_T \frakF_T(\mathbf{x}) 
\end{equation} 
whose strata are indexed by trees $T$ similar to \eqref{eq:fm}, but allowing unstable vertices with $|v|_{\mathrm{in}} = 1$, which correspond to breaking off of positive energy Floer cylinders. When the expected dimension \eqref{eq: vdimM} is $0$, all strata except the interior are empty, hence $\mathring{\frakF}_m(\mathbf{x}) = \frakF_m(\mathbf{x})$ is a finite set. 
Via \eqref{eq:oru-2} and the orientation of $\frakF_m$ from Section \ref{subsec:or-con}, each point in that set gives an isomorphism $\frako_{x_1} \otimes \cdots \otimes \frako_{x_m} \iso \frako_{x_0}$. We add up the $\bK$-normalization of those isomorphisms to obtain the operation
\begin{equation} \label{eq:elld} 
\ell^m: \mathit{CF}^*(H^{\floer})^{\otimes m} \longrightarrow \mathit{CF}^{*+3-2m}(H^{\floer}).
\end{equation}
On $\scrG = \mathit{CF}^*(H^{\floer})[1]$, the degree of these operations becomes $2-m$.
We extend them to $m = 1$ by declaring $\ell^1$ to be the Floer differential. The fact that our perturbation data are equivariant implies that our operations satisfy the symmetry property \eqref{eq:gradedsymmetry}. (Because $\scrG$ is the shifted Floer complex, the signs from \eqref{eq:gradedsymmetry} are $(-1)^{\|x_k\|} = (-1)^{\mathrm{deg}(x_k)}$, which agrees with their geometric origin as Koszul signs associated to permuting the orientation operators $D_{x_k}$.)

\begin{prop} \label{lem:Linfinty} 
The operations \eqref{eq:elld} define an $L_\infty$-structure on $\mathit{CF}^*(H^{\floer})[1]$. 
\end{prop}

\begin{proof} 
To check that these operations satisfy the $L_\infty$-relations \eqref{eq:Linfty}, we follow the standard pattern of considering one-dimensional spaces $\frakF_{m}(\mathbf{x})$. In this case,  for a generic choice of perturbation data, all nonempty boundary strata $\frakF_T(\mathbf{x})$ correspond to trees with exactly two vertices (and one internal edge). Concretely, these boundary terms are of two kinds: 
\begin{itemize} \itemsep.5em
\item ($T$ unstable, which means that one of the vertices has $|v|_{\mathrm{in}} = 1$) This corresponds to cylindrical breaking-off of a Floer trajectory for $H^{\floer}$, meaning the strata are 
\begin{equation} 
\mathring{\frakC}(x_0,y_0) \times \mathring{\frakF}_{m}(y_0,x_1,\cdots,x_m) \text{ or } \mathring{\frakF}_{m}(x_0,x_1,\cdots,y_i,\cdots ,x_m) \times \mathring{\frakC}(y_i,x_i).
\end{equation} 
Here, as in Section \ref{subsubsec:floer}, $\mathring{\frakC}(x_0,y_0)$ and $\mathring{\frakC}(y_i,x_i)$ denote spaces of Floer trajectories up to $\bR$-translation. 

\item ($T$ stable) The other kind of boundary strata correspond to a degeneration to a codimension $1$ stratum in $\partial\frakF_m$. After applying the action of the symmetric group, these strata can be identified with 
\begin{equation} \label{eq:2-face}
\mathring{\frakF}_{m_1}(x_0,y, x_{\sigma(m_2+1)},\dots,x_{\sigma(m)}) \times \mathring{\frakF}_{m_2}(y,x_{\sigma(1)},\dots,x_{\sigma(m_2)}), \;\; m_1+m_2 = m+1,
\end{equation}
for some $m_2$-shuffle $\sigma$, see \eqref{eq:shuffle2}, and one-periodic orbit $y$.
\end{itemize}  
 
There is an obvious bijection between these boundary strata and terms in \eqref{eq:Linfty}. (The stable trees correspond to compositions of operations of arity $\geq 2$; and the remaining ones to terms involving the differential $\ell^1.$) Standard gluing constructions show that such points correspond bijectively to ends of $\mathring{\frakF}_m(\mathbf{x})$, with the correspondence given by Gromov convergence. This means that $\frakF_m(\mathbf{x})$ is at least topologically a compact one-manifold with boundary, allowing for the usual ``signed count of boundary points is zero'' argument to go through.

The signs in the argument deserve a short discussion. To make that more transparent, we temporarily take a moduli space $\frakF_m(\bfx)$ of arbitrary dimension, while still considering only boundary faces of the form \eqref{eq:2-face}; for simplicity we will assume that $\sigma$ is trivial. Start with a point on that boundary face, given by a pair $(p_1,u_1), (p_2,u_2)$, and let $(p,u)$ be the result of applying gluing, for some small value of the gluing parameter. We have isomorphisms
\begin{equation} \label{eq:linearized-glue-1}
\bR \oplus T_{p_2}\mathring{\frakF}_{m_2} \oplus T_{p_1}\mathring{\frakF}_{m_1} \iso T_p\mathring{\frakF}_m
\end{equation}
and
\begin{equation} \label{eq:linearized-glue-2}
\bR \oplus T_{(p_1,u_1)}\mathring{\frakF}_{m_1}(x_0,y,x_{m_2+1},\dots,x_m) \oplus T_{(p_2,u_2)}\mathring{\frakF}_{m_2}(y,x_1,\dots,x_{m_2}) \iso
T_{(p,u)}\mathring{\frakF}_m(\bfx),
\end{equation}
which are not quite unique, but well-defined enough to allow us to carry over orientations. Both are obtained from linearized gluing: on the level of parameter spaces, for \eqref{eq:linearized-glue-1} (where the ordering follows the conventions from Section \ref{subsec:or-con}, so that we can quote the results from there); and for parametrized moduli spaces, in the case of \eqref{eq:linearized-glue-2} (where the standard ordering for gluing of determinant lines is used, as in Section \ref{subsubsection:Floerfamilies}). In both cases, $\bR$ corresponds to the gluing parameter. The resulting isomorphisms of top exterior powers are compatible with each other via \eqref{eq:turn-off} and \eqref{eq:det-gluing}. By that, we mean that the diagram of one-dimensional vector spaces and isomorphisms shown in Figure \ref{fig:orientation-spaces} is commutative (up to multiplication with positive numbers, as usual). 

To simplify the discussion, let's suppose temporarily that identifications $\frako_{x_k} \iso \bR$ have been chosen. Together with the orientations of Fulton-MacPherson spaces, this determines orientations of $\mathring{\frakF}_m(\bfx)$, as an instance of \eqref{eq:turn-off}. Recall from Section \ref{subsec:or-con} that the orientations of Fulton-MacPherson spaces are compatible with $\mathring{\frakF}_{m_2} \times \mathring{\frakF}_{m_1} \subset \partial \frakF_m$. The commutativity of the diagram in Figure \ref{fig:orientation-spaces} then shows that the resulting orientations of the parametrized moduli spaces are compatible with $\mathring{\frakF}_{m_2} \times \mathring{\frakF}_{m_1} \subset \partial \frakF_m$, up to the Koszul sign arising from the ``swap''. In the case relevant to our argument, that sign is trivial, since one of the factors is zero-dimensional. This geometric result translates algebraically into the absence of signs in the term
\begin{equation}
\ell^{m_1}(\ell^{m_2}(x_1,\dots,x_{m_2}),x_{m_2+1},\dots,x_m)
\end{equation}
of the $L_\infty$-relation \eqref{eq:Linfty}. To get the general statement from this, one only needs to add Koszul signs corresponding to permutations of the $x_k$.
\end{proof}
\begin{figure}
\[
\xymatrix{
\lambda^{\mathrm{top}}(T_{(p,u)}\mathring{\frakF}_m(\bfx)) \otimes \frako_{x_1} \otimes \cdots \otimes \frako_{x_m} 
\ar[r]^-{o(r,u)}
&
\lambda^{\mathrm{top}}(T_p\mathring{\frakF}_m) \otimes \frako_{x_0} 
\ar@{=}[dddd]
\\ 
\ar[u]^-{\eqref{eq:linearized-glue-2}}
\parbox{22em}{\centering
$\bR \otimes \lambda^{\mathrm{top}}(T_{(p_1,u_1)}\mathring{\frakF}_{m_1}(x_0,y,x_{m_2+1},\dots,x_m)) \otimes \lambda^{\mathrm{top}}(T_{(p_2,u_2)}\mathring{\frakF}_{m_2}(y,x_1,\dots,x_{m_2})) \otimes \frako_{x_{1}} \otimes \cdots \otimes \frako_{x_{m}}$
}
\ar[d]_-{o(r_2,u_2)}
\\
\parbox{20em}{\centering
$\bR \otimes \lambda^{\mathrm{top}}(T_{(p_1,u_1)}\mathring{\frakF}_{m_1}(x_0,y,x_{m_2+1},\dots,x_m)) \otimes \lambda^{\mathrm{top}}(T_{p_2}\mathring{\frakF}_{m_2}) \otimes \frako_y \otimes\frako_{x_{m_2+1}}\otimes \cdots \otimes \frako_{x_{m}}$
}
\ar[d]_-{\text{swap}}
\\
\parbox{18em}{\centering
$\bR \otimes \lambda^{\mathrm{top}}(T_{p_2}\mathring{\frakF}_{m_2}) \otimes \lambda^{\mathrm{top}}(T_{(p_1,u_1)}\mathring{\frakF}_{m_1}(x_0,y,x_{m_2+1},\dots,x_m))
\otimes \frako_y \otimes\frako_{x_{m_2+1}}\otimes \cdots \otimes \frako_{x_{m}}$
}
\ar[d]_-{o(r_1,u_1)} 
\\  
\bR \otimes \lambda^{\mathrm{top}}(T_{p_2}\mathring{\frakF}_{m_2}) \otimes \lambda^{\mathrm{top}}(T_{p_1}\mathring{\frakF}_{m_1}) \otimes \frako_{x_{0}} 
\ar[r]_-{\eqref{eq:linearized-glue-1}} 
& 
\lambda^{\mathrm{top}}(T_r\mathring{\frakF}_{m}) \otimes \frako_{x_0}
}
\]
\caption{\label{fig:orientation-spaces}The diagram for the sign considerations in Proposition \ref{lem:Linfinty}.}
\end{figure}

\subsubsection{$L_\infty$-module structure\label{section:linftymod}}
We next consider the parameter spaces $\mathring{\frakC}_{m}$ from Section \ref{section:FM-cylinder}. Recall that this moduli space has been equipped with a universal choices of positive cylindrical ends at $z_{m+1} = +\infty$ and intermediate punctures $z_1,\dots,z_m$, and a negative end at $z_0=-\infty$, as in \eqref{eq:endoncylinder1}-\eqref{eq:endoncylinder2}. Again, we choose a family of perturbation data which is chosen within our allowed class from Section \ref{subsubsec:operations}, conformally consistent with respect to boundary strata, and $\mathit{Sym}(m)$-equivariant (note that part of consistency involves our previous choices for the $\scrF_m$ parameter spaces).
%

Let $\mathbf{x} = (x_0,x_1,\cdots,x_m,x_{m+1})$ be a collection of orbits. Following our usual notational habits, we write $\mathring{\frakC}_{m}(\mathbf{x})$ for the parameterized moduli space of solutions. These spaces generically have dimension 
\begin{equation} 
\operatorname{dim}(\mathring{\frakC}_{m}(\mathbf{x}))=\operatorname{deg}(x_0)-\sum_{i=1}^{m+1} \operatorname{deg}(x_i) + 2m-1.
\end{equation} 
Counting points in zero-dimensional moduli spaces then gives rise to operations of degree $1-m$,
\begin{equation} \label{eq:opsliemodule}
\ell^{m,1}:(\mathit{CF}^*(H^{\floer}[1])^{\otimes m})\otimes \mathit{CF}^*(H^{\floer}) \longrightarrow \mathit{CF}^*(H^{\floer})
\end{equation}

\begin{prop} \label{prop:Lmodule} The operations $\ell^{m,1}$ give $\mathit{CF}^*(H^{\floer})$ the structure of an $L_\infty$-module over $\mathit{CF}^*(H^{\floer})[1].$ 
\end{prop} 

\begin{proof} The proof again proceeds by examining the boundaries of one-dimensional moduli spaces $\frakC_{m}(\mathbf{x})$. There are two possible boundary types to consider:
\begin{itemize} \itemsep.5em
\item The first is where we have a degeneration to a codimension $1$ boundary of \eqref{eq:heart0} (meaning we have a Fulton-MacPherson configuration bubbling off at one of the interior marked points) or a breaking of a Floer trajectory at one of the points $z_1,\dots,z_m$. These correspond to the first sum on the right-hand side of \eqref{eq:LinftyM}. 

\item The second possible degeneration is a cylindrical breaking at $\pm \infty$. This means either an ordinary Floer breaking 
\begin{align} 
\frakC(x_0,y_0) \times \frakC_{m}(y_0,x_1,\dots,x_m,x_\infty) \text{ or } \frakC_{m}(x_0,x_1,\dots,x_m,y_{\infty}) \times \frakC(y_{\infty},x_\infty); 
\end{align} 
or else, that the domain degenerates to the boundary of $\frakC_m$, so that we have a limit
\begin{align} 
(p_1,u_1,p_2,u_2) \in \mathring{\frakC}_{m_1} \times \mathring{\frakC}_{m_2}, \quad m_1+m_2=m.
\end{align} 
This boundary type accounts for the second sum in \eqref{eq:LinftyM}. 
\end{itemize} 
The verification that these boundary strata contribute with the correct signs proceeds as in Lemma \ref{lem:Linfinty} (details omitted).  \end{proof}

\begin{remark} 
As noted in Example \ref{ex:diagonal}, one can also consider the ``diagonal" module $M = \mathit{CF}^*(H^{\floer})$ over $\scrG = \mathit{CF}^*(H^{\floer})[1]$ with operations $\ell^{m,1} \stackrel{\mathrm{def}}{=} \ell^{m+1}$. However, the module structure defined by the operations \eqref{eq:opsliemodule} is in general different from the diagonal one, due to the fact that we have chosen $S^1$-invariant markers over $\frakC_m$ as opposed to aligned markers over $\frakF_{m+1}$ (compare Remark \ref{th:compare-c-f}). This will later on lead to two versions of $q$-deformed Floer groups.
\end{remark}
 

\subsubsection{$S^1$-equivariant Floer groups\label{subsubsec:s1-equivariant}}
The definition of the $S^1$-equivariant Floer complex involves parametrized operations over the spaces $\mathring{\frakA}_r$ (Section \ref{section:angles}; a version of this approach is used in \cite{ganatra19}). Recall that the ends are chosen to be of the form \eqref{eq:rotated-end} (in particular, the negative puncture at $s=-\infty$ is rotated by $\theta^{\operatorname{tot}}$) and compatible with the map \eqref{eq:forget-circle} over the boundary $\partial_{\sigma_{i+1}=\sigma_i} \frakA_r$. We choose perturbation data for each fiber of the universal curve (still within the class from Section \ref{subsubsec:operations}) which over $\partial_{\sigma_{i+1}=\sigma_i} \frakA_r$: \begin{align} \label{eq:fakeboundary} \text{coincide with the pull-back of the data chosen over } \mathring{\mathfrak{A}}_{r-1}. \end{align} Of course, we also require these perturbations to be chosen conformally consistently over boundary strata. For any pair of orbits $x_0,x_{1}$, let $\mathring{\mathfrak{A}}_r(x_0,x_{1})$ be the moduli space of parameterized solutions for these choices. It has virtual dimension 
\begin{align} \operatorname{dim}(\mathring{\mathfrak{A}}_r(x_0,x_1))= \operatorname{deg}(x_0)-\operatorname{deg}(x_1)+2r-1; 
\end{align} 
hence, after choosing perturbation data generically, rigid moduli spaces give rise to operations
\begin{equation} 
\delta_{S^1}^r: \mathit{CF}^*(H^{\floer}) \longrightarrow \mathit{CF}^{*+1-2r}(H^{\floer}) 
\end{equation} 
for any $r \geq 1$. In the degenerate case $r=0$, we define $\delta_{S^1}^0 = \delta$ to be the Floer differential.  

\begin{lemma} \label{lem:S1diff} For any $r \geq 0$, 
\begin{equation} 
\label{eq: S1diff} 
\sum_{j=0}^{r}\delta_{S^1}^j\delta_{S^1}^{r-j}=0.
\end{equation} 
\end{lemma} 

\begin{proof} 
Examining the boundary of compactified one-dimensional moduli spaces (keeping the orientation computations from Section \ref{subsec:or-con} in mind) immediately gives \begin{align} \sum_{j=0}^{r}\delta_{S^1}^j\delta_{S^1}^{r-j}+ \sum_i \delta^r_{i,i+1}=0.\end{align} Here $\delta^r_{i,i+1}$ denotes the contributions from curves lying over the boundary strata $\partial_{\sigma_{i+1}=\sigma_i} \frakA_r.$ However, the consistency condition \eqref{eq:fakeboundary} implies that the Floer datum chosen for any element $C_r$ for $r \in \partial_{\sigma_{i+1}=\sigma_i} \frakA_r$ only depends on its image under the forgetful map to $\mathring{\mathfrak{A}}_{r-1}$. As this forgetful map has one dimensional fibers, there can be no rigid curves lying over those strata. Thus all of the $\delta^r_{i,i+1}$ vanish, yielding  \eqref{eq: S1diff}.   \end{proof}

The equivariant Floer complex is the $\bZ$-graded complex of $\bK[[u]]$-modules
\begin{equation}
\mathit{CF}_{S^1}^*(H^{\floer}) \stackrel{\mathrm{def}}{=} \mathit{CF}^*(H^{\floer})[[u]], 
\quad
\delta_{S^1}=\sum_{r=0}^{\infty} u^r\delta_{S^1}^r.  
\end{equation}

\subsubsection{$S^1$-equivariant $L_\infty$-module structure\label{sec:s1module}}
Here we consider the operations governed by parameterized Floer theory over $\mathring{\AC}_{m,r}$ from Section \ref{section:anglesFM}. The universal ends at $s=\pm\infty$ are chosen as in \eqref{eq:rotated-end} (again, this means that the negative puncture at $s=-\infty$ is rotated by $\theta^{\operatorname{tot}}$), while at interior punctures they are chosen as in \eqref{eq:log-end}. Over the boundary where two heights coincide, $\partial_{\sigma_{i+1}=\sigma_i}\mathring{\AC}_{m,r}$, the ends are chosen compatibly with \eqref{eq:ac-forget}. As usual, the perturbation data are chosen to be consistent with respect to boundary strata, and $Sym(m)$-equivariant. We additionally require that these data satisfy the analogue of \eqref{eq:fakeboundary}, meaning that over $\partial_{\sigma_{i+1}=\sigma_i}\mathring{\AC}_{m,r}$, they \begin{align} \label{eq:fakeboundary2} \text{coincide with the pull-back of the data chosen over } \mathring{\AC}_{m,r-1}. \end{align} Let $\mathbf{x} = (x_0,x_1,x_2, \cdots, x_m,x_{m+1})$ be a collection of orbits. Using our standard procedure, we define moduli spaces $\mathring{\AC}_{m,r}(\mathbf{x})$ where our curve is asymptotic to $x_0$ at $s=-\infty$, $x_{m+1}$ at $s=+\infty$ and to $x_1,\cdots,x_m$ at the intermediate punctures. These moduli spaces generically have dimension
\begin{align} \operatorname{dim}(\mathring{\AC}_{m,r}(\mathbf{x}))=\operatorname{deg}(x_0)-\sum_{i=1}^{m+1}\operatorname{deg}(x_i)+2(m+r)-1. \end{align} 
Counting rigid moduli spaces then gives rise to operations on Floer complexes 
\begin{equation}
\ell_{S^1}^{m,1,r}:(\mathit{CF}^*(H^{\floer})^{\otimes m})\otimes \mathit{CF}^*(H^{\floer}) \longrightarrow \mathit{CF}^*(H^{\floer})
\end{equation}
of degree $1-2(m+r)$.  We extend this definition to $m=r=0$ by setting $\ell_{S^1}^{0,1,0} = \delta$, the Floer differential. The sum $\ell_{S^1}^{m,1}=\sum u^r \ell_{S^1}^{m,1,r}$ can then be viewed as an operation 
\begin{equation} 
\ell_{S^1}^{m,1}: (\mathit{CF}^*(H^{\floer})[1]^{\otimes m})\otimes \mathit{CF}_{S^1}^*(H^{\floer}) \longrightarrow \mathit{CF}_{S^1}^*(H^{\floer}) 
\end{equation}
of degree $1-m$.

\begin{prop} \label{prop:LmoduleS1} 
The operations $\ell_{S^1}^{m,1}$ give $\mathit{CF}_{S^1}^*(H^{\floer})$ the structure of an $L_\infty$-module over $\mathit{CF}^*(H^{\floer})[1].$ 
\end{prop}

\begin{proof} 
This is an $S^1$-equivariant version of Proposition \ref{prop:Lmodule}. Let us consider the boundary of the one-dimensional moduli spaces. There are now three possible boundary-types to consider. The first two are exactly as in the non-equivariant case, meaning we have a degeneration at one of the interior marked points, or cylindrical breaking at $s=\pm \infty$. The third situation is where the parameter lies in one of the boundary strata $\partial_{\sigma_{i+1}=\sigma_i}\mathring{\AC}_{m,r}$. However, in view of the condition \eqref{eq:fakeboundary2} we have placed on our perturbations, we again have that Floer curves along this boundary necessarily come in one-dimensional families and can never be rigid. Hence these terms contribute zero. \end{proof}

\subsubsection{The $q$-deformed groups\label{sec:q-deformed}}
From this point onwards, we assume that a Maurer-Cartan element (in the sense of Section \ref{sec:mcelements}) for the Floer $L_\infty$-algebra has been fixed,
\begin{equation} \label{eq:floer-mc}
\alpha \in (q\mathit{CF}^*(H^{\floer})[[q]])^2.
\end{equation}

\begin{definition} \label{th:deformed-structures}
(i) Take the deformation of the Floer differential induced by $\alpha$, in the sense of \eqref{eq:deform-the-linfty-algebra}:
\begin{equation} \label{eq:diagonaldeformationCF}
\begin{aligned} 
& \mathit{CF}_q^{\operatorname{diag}}(H^{\floer}) \stackrel{\mathrm{def}}{=} 
\mathit{CF}^*(H^{\floer})[[q]], \\
& \delta_q^{\operatorname{diag}} = \ell^1_\alpha = \ell^1 + \ell^2(\alpha,\cdot) + \half \ell^3(\alpha,\alpha,\cdot) - \cdots 
\end{aligned}
\end{equation}
We have only written down the differential, but as a consequence of the general algebraic theory, there is an entire deformed $L_\infty$-algebra structure here. (The superscript is because one can think of this as deforming the diagonal bimodule.)

(ii) $\mathit{CF}^*(H^{\floer})$ is an $L_\infty$-module, following Section \ref{section:linftymod}. Again, we can use $\alpha$ to deform that structure as in \eqref{eq:deform-the-linfty-module}, which means that the outcome is an $L_\infty$-module over the deformed algebra from (i). The resulting $q$-deformed complex is 
\begin{equation} \label{eq:deformedfloermoduled}
\begin{aligned}
& \mathit{CF}_q^*(H^{\floer}) \stackrel{\mathrm{def}}{=} 
\mathit{CF}^*(H^{\floer})[[q]], \\
& \delta_q^{\operatorname{module}} = \ell^{0,1}_{\alpha} = l^{0,1} + \ell^{1,1}(\alpha,\cdot) + \half \ell^{2,1}(\alpha,\alpha,\cdot) + \cdots
\end{aligned}
\end{equation}

(iii) The argument from (ii) has an $S^1$-equivariant extension, following Section \ref{sec:s1module}. The $q$-deformed $S^1$-equivariant complex is
\begin{equation} \label{eq:qs1-complex}
\begin{aligned}
& \mathit{CF}_{S^1,q}^*(H^{\floer}) \stackrel{\mathrm{def}}{=} 
\mathit{CF}^*(H^{\floer})[[q,u]], \\
& \delta_{S^1,q} = \ell^{0,1}_{S^1,\alpha} = \ell^{0,1}_{S^1} + \ell^{1,1}_{S^1}(\alpha,\cdot) + \half \ell^{2,1}_{S^1}(\alpha,\alpha,\cdot) + \cdots
\end{aligned}
\end{equation}
For later use, we find it convenient to also spell out the next term of the $L_\infty$-module structure:
\begin{equation}
\begin{aligned} \label{eq:deformedS1liemodule} &
\ell^{1,1}_{S^1,q} : \mathit{CF}_q^{\operatorname{diag}}(H^{\floer})\otimes \mathit{CF}_{S^1,q}^{*}(H^{\floer}) \longrightarrow \mathit{CF}_{S^1,q}^{*}(H^{\floer}),
\\ 
&
\ell^{1,1}_{S^1,q}(\cdot,\cdot) \stackrel{\mathrm{def}}{=} \ell^{1,1}_{S^1}(\cdot,\cdot) + \ell^{2,1}_{S^1}(\alpha,\cdot,\cdot)+ \half \ell^{3,1}_{S^1}(\alpha^{\otimes 2},\cdot,\cdot)+\cdots.  \end{aligned}
\end{equation}
\end{definition}

\begin{remark} \label{rem:MCcancel} 
The fact that the deformed differentials in Definition \ref{th:deformed-structures} square to zero was deduced algebraically. Nevertheless, it is worth briefly re-interpreting the argument geometrically, because the same idea will occur again in the next section. Consider for instance \eqref{eq:deformedfloermoduled}. Deforming the differential introduces new terms given by counts of configurations on the cylinder where $\alpha$ is inserted (as a linear combination of periodic orbits) into the interior marked points $z_1,\cdots,z_{m}$. The count of codimension $1$ boundary strata gives the relation
\begin{equation}
\sum_{j_1+j_2-1=m} {\textstyle\frac{1}{j_1!j_2!}} \ell^{j_1,1}(\ell^{j_2}(\alpha^{\otimes j_{2}}),\alpha^{\otimes j_1-1}, x)  + \sum_{j_1+j_2=m} {\textstyle\frac{1}{j_1!j_2!}} \ell^{j_1,1}(\alpha^{\otimes j_{1}},\ell^{j_2,1}(\alpha^{\otimes j_2},x))=0.
\end{equation} 
As in the proof of Proposition \ref{prop:Lmodule}, the first sum corresponds to degenerating in $\frakC_{m}^{\heart}$ (or Floer differential when $j_2=1$) and the second one to breaking at infinity.  However, the fact that $\alpha$ satisfies the Maurer-Cartan equation means that after summing these relations over all $j_1,j_2$, the first sum contributes zero while the second sum is precisely the relation $(\delta_q^{\operatorname{module}})^2 = 0$. 
\end{remark}

\subsubsection{The closed string connection\label{section:connection}}
We finally turn to defining the connection $\nabla_{u\partial_{q}}$ on the cohomology of $\mathit{CF}_{S^1,q}^*(H^{\floer})$, the complex from Definition \ref{th:deformed-structures}. To do this, we consider the parameter spaces from Sections \ref{section:cartan-A}--\ref{section:cartan-B}. For a collection of Hamiltonian orbits $\mathbf{x} = (x_0,x_1 \dots,x_{m},x_{m+1})$, the corresponding moduli spaces $\mathring{\AC}_{m,r,w}^{(A)}(\mathbf{x})$, $\mathring{\AC}_{m,r,w}^{(B)}(\mathbf{x})$ have dimension 
\begin{equation} 
\operatorname{deg}(x_0)-\sum_{i=1}^{m+1} \operatorname{deg}(x_i) + 2(m+r-1). 
\end{equation} 
We will assume that all of the universal cylindrical ends and perturbation data used to construct these moduli spaces are pulled back from $\AC_{m,r}.$ This requires placing additional (generic) constraints on the data over $\AC_{m,r}$. Specifically, the following transversality conditions are required to hold for all $\mathbf{x}$ such that $\operatorname{dim}(\mathring{\AC}_{m,r,w}^{(A)}(\mathbf{x})) = \operatorname{dim}(\mathring{\AC}_{m,r,w}^{(B)}(\mathbf{x})) \leq 1$: 
\begin{itemize}
\itemsep.5em
\item The moduli spaces  $\mathring{\AC}_{m,r,w}^{(A)}(\mathbf{x})$, $\mathring{\AC}_{m,r,w}^{(B)}(\mathbf{x})$ are cut out transversally. 

\item Concerning the compactifications $\AC_{m,r,w}^{(A)}(\mathbf{x})$, $\AC_{m,r,w}^{(B)}(\mathbf{x})$, we require that the moduli spaces over all codimension-$k$ boundary strata of $\AC_{m,r,w}^{(A)}$ and $\AC_{m,r,w}^{(B)}$  are also cut out transversally. In particular, our spaces contain no curves lying over the codimension $\geq 2$ boundary strata of $\AC_{m,r,w}^{(A)}$, $\AC_{m,r,w}^{(B)}$.  
\end{itemize} 
These conditions imply that the zero-dimensional moduli spaces consist of a finite collection of points, and the one-dimensional moduli spaces are one-manifolds with boundary. For certain considerations involving type (B) moduli spaces, we will also need to impose that: 

\begin{itemize}
\item If any angle $\theta_i = 0$, the perturbation data are pulled back along the map $\mathring{\AC}_{m,r} \longrightarrow \mathring{\AC}_{m,r-1}$ which forgets the circle $s=\sigma_i$. (If several $\theta_i=0$ at once, this forces us to forget all of the corresponding circles simultaneously.)
\end{itemize} 
Let's start with $\mathring{\AC}_{m,r,w}^{(A)}(\mathbf{x})$. After summing over $w$, the zero-dimensional spaces give rise to maps
\begin{equation} \label{eq:iota1} 
\mathit{KH}_{(A)}^{m,r}: \mathit{CF}^*(H^{\floer})^{\otimes m} \otimes \mathit{CF}^*(H^{\floer}) \longrightarrow \mathit{CF}^{*-2(m+r-1)}(H^{\floer}).
\end{equation} 

\begin{example} \label{th:pair-of-pants}
In the special case where $m=1$ and $r=0$, $\AC_{1,0,0}^{(A)}$ is a point, and the operation $\mathit{KH}_{(A)}^{1,0}$ from \eqref{eq:iota1} is just the pair-of-pants product. 
\end{example}

Let $\mathit{KH}_{(A)}$ be the operation obtained by formally inserting the Maurer-Cartan element $\alpha$ into the marked points $z_1,\dots,z_m$, and $\partial_q\alpha$ into $z_1$:
\begin{equation}
\begin{aligned} \label{eq:iota3}   
& \mathit{KH}_{(A)}: \mathit{CF}^*_{S^1,q}(H^{\floer}) \longrightarrow \mathit{CF}_{S^1,q}^*(H^{\floer}),
\\
&
\mathit{KH}_{(A)}(\cdot) = \sum_{m,r} \textstyle{\frac{1}{(m-1)!}} u^r \,\mathit{KH}^{(A)}_{m,r}(\partial_q \alpha, \alpha^{\otimes m-1},\cdot). 
\end{aligned}
\end{equation}
Repeating the steps \eqref{eq:iota1}-\eqref{eq:iota3} for type (B) moduli spaces similarly yields
\begin{equation} \label{eq:iota4} 
\mathit{KH}_{(B)}: \mathit{CF}^*_{S^1,q}(H^{\floer}) \longrightarrow \mathit{CF}^*_{S^1,q}(H^{\floer}). 
\end{equation} 
Finally, we set 
\begin{equation}
\mathit{KH} = \mathit{KH}_{(A)} + \mathit{KH}_{(B)}.
\end{equation}

\begin{definition} 
The closed string connection is
\begin{equation}
\begin{aligned} 
& \nabla_{u\partial_{q}}: \mathit{CF}^*_{S^1,q}(H^{\floer}) \longrightarrow \mathit{CF}^*_{S^1,q}(H^{\floer}),
\\ &
\nabla_{u\partial_{q}} = u\partial_q + \mathit{KH}.
\end{aligned}
\end{equation}
\end{definition}

\begin{prop} 
We have $\delta_{S^1,q} \nabla_{u\partial_q} = \nabla_{u\partial_q} \delta_{S^1,q}$. As a consequence, $\nabla_{u\partial_{q}}$ induces a connection on $H^*(\mathit{CF}_{S^1,q}^*(H^{\floer}))$. 
\end{prop}

\begin{proof}
From \eqref{eq:commutationalphaeq} and the definition of the differential \eqref{eq:qs1-complex} we see that, in the notation from \eqref{eq:deformedS1liemodule},
\begin{equation} 
\delta_{S^1,q} \partial_q(\cdot) - \partial_q \delta_{S^1,q}(\cdot) = -\ell^{1,1}_{S^1,q}(\partial_q\alpha,\cdot).
\end{equation} 
The fact that $\nabla_{u\partial_{q}}$ is a cochain map is therefore an immediate consequence of the ``Cartan homotopy relation'' 
\begin{equation} 
\label{eq:cartan0} 
\mathit{KH}\circ\delta_{S^1,q}(\cdot)- \delta_{S^1,q}\circ KH (\cdot) +u \ell^{1,1}_{S^1,q}(\partial_q\alpha,\cdot)=0.  
\end{equation} 
 
To prove \eqref{eq:cartan0}, we look at how these moduli spaces can degenerate in codimension $1$, beginning with the type (B) moduli spaces: 
\begin{itemize}  \itemsep.5em
\item 
As in previous situations, the boundary strata where $\sigma_{i+1} = \sigma_i$ contribute nothing; 
the Floer data are then pulled back from a lower-dimensional parameter space.

\item We can have Floer breaking or marked points colliding at some of the $(z_1,\dots,z_m)$. There are two cases to consider, depending on whether these degenerations involve the distinguished marked point $z_{1}$ or not. The degenerations which do not involve the distinguished marked point $z_{1}$ contribute nothing, by the argument from Remark \ref{rem:MCcancel}. The degenerations which do involve $z_{1}$ a priori contribute a term with $\delta^{\operatorname{diag}}_q(\partial_q\alpha)$ inserted  into $z_1$,
\begin{equation} 
\sum_{m,r} {\textstyle \frac{1}{(m-1)!}} u^r \mathit{KH}_{(B)}^{m,r}(\delta_q^{\operatorname{diag}}(\partial_q\alpha), \alpha^{\otimes m-1},\cdot) \in \mathit{CF}^*_{S^1,q}(H^{\floer}). 
\end{equation}  
But \eqref{eq:ell1alpha} shows that $\delta_q^{\operatorname{diag}}(\partial_q\alpha)=0$, so this term vanishes as well.    

\item We can have cylindrical breaking at $\pm \infty$, which corresponds to the terms $\delta_{S^1} \circ \mathit{KH}_{(B)}$ and $\mathit{KH}_{(B)} \circ \delta_{S^1,q}$.
More precisely (taking into account the ordering conventions used in our discussion of orientations), the first case is the type (B) counterpart of \eqref{eq:type-a-orientation-1}, which explains why it appears with a $-1$ sign; and the second case is the counterpart of \eqref{eq:type-a-orientation-2}, hence the $+1$ sign.

\item Curves can degenerate to the boundary stratum $\partial_{t_1 = \theta_{\geq w}} \mathring{\AC}_{m,r,w}^{(B)}$ or $\partial_{t_1 = \theta_{\geq w+1}} \mathring{\AC}_{m,r,w}^{(B)}$.

\item Curves can degenerate to the boundary $\partial_{\theta_w^{\operatorname{lift}} = 0} \mathring{\AC}_{m,r,w}^{(B)}$. Recall that if we forget the circle $s = \sigma_w$, the image consists of configurations decorated with $(r-1)$ angles, such that $z_1$ lies in the annulus $[\sigma_{w-1},\sigma_w] \times S^1$. Moreover, because this stratum maps to the subset of $\mathring{\AC}_{m,r}$ where $\theta_w=0$, we have arranged for the perturbation data to be pulled back from $\mathring{\AC}_{m,r-1}.$ Therefore, summing over all $w$, we can identify the contribution of this stratum with $u\ell^{1,1}_{S^1,q}(\partial_q\alpha,\cdot)$. (See the discussion in Section \ref{subsec:or-con} for the comparison of orientations, which underlies the $+$ sign of this contribution.)
\end{itemize} 

From the boundaries of type (A) moduli spaces, we have the following contributions:
\begin{itemize} \itemsep.5em
\item 
Versions of the first three kinds of strata which appeared in type (B). More precisely, we can have $\sigma_{i+1}=\sigma_i$ (which contributes nothing); interior Floer breaking or marked points colliding (which also contributes nothing); and cylindrical breaking at $s= \pm \infty$ which contributes $\mathit{KH}_{(A)} \circ \delta_{S^1,q} - \delta_{S^1,q} \circ \mathit{KH}_{(A)}$.

\item Curves can degenerate to $\partial_{s_1 = \sigma_w} \mathring{\AC}_{m,r,w}^{(A)}$ or $\partial_{s_1 = \sigma_{w+1}} \mathring{\AC}_{m,r,w}^{(A)}$. The identifications \eqref{eq:cylinder-left-right},  \eqref{eq:ABisomorphism2} of boundary strata are orientation reversing. It follows that the contributions from $\partial_{s_1 = \sigma_w} \mathring{\AC}_{m,r,w-1}^{(A)}$ cancel those from $\partial_{t_1 = \theta_{\geq w}} \mathring{\AC}_{m,r,w}^{(B)}$ and the contributions from $\partial_{s_1 = \sigma_w} \mathring{\AC}_{m,r,w}^{(A)}$ cancel those from $\partial_{t_1 = \theta_{\geq w+1}} \mathring{\AC}_{m,r,w}^{(B)}.$ 
\end{itemize}

Summing up all of these different contributions therefore proves \eqref{eq:cartan0}.
\end{proof} 

As a final remark, take the $u=0$ reduction of the Cartan homotopy operator. It follows from the Cartan homotopy formula \eqref{eq:cartan0} that this defines a chain map 
\begin{equation} \label{eq:iotacf-2}
\begin{aligned}
& \iotaCF: \mathit{CF}_q^*(H^{\floer}) \longrightarrow \mathit{CF}_q^*(H^{\floer}), \\
& \iotaCF = \sum_m {\textstyle \frac{1}{(m-1)!}} \mathit{KH}_{(A)}^{m,0}(\partial_q \alpha, \alpha^{\otimes m-1},\cdot).
\end{aligned}
\end{equation} 
Geometrically, the underlying parameter spaces have $m$ marked points on the cylinder, where the first marked point (carrying $\partial_q\alpha$) is forced to lie on $\{t = 0\}$, while the others (carrying $\alpha$) are unconstrained. The induced operation on cohomology can be viewed as a $q$-deformation of the pair-of-pants product with $[\alpha^1]$, where $\alpha = q\alpha^1 + O(q^2)$ (compare Example \ref{th:pair-of-pants}). To summarize, we have a diagram
\begin{equation} \label{eq:iota-diagram}
\xymatrix{
H^*(\mathit{CF}^*_{S^1,q}(H^{\floer})) \ar[rrr]^-{\nabla_{u\partial q}} \ar[d]_-{u=0}
&&&
H^*(\mathit{CF}^*_{S^1,q}(H^{\floer})) \ar[d]^-{u=0}
\\
\ar[d]_-{q=0}
H^*(\mathit{CF}^*_{q}(H^{\floer})) \ar[rrr]^-{\iotaCF}
&&&
H^*(\mathit{CF}^*_{q}(H^{\floer}))
\ar[d]^-{q=0}
\\
H^*(\mathit{CF}^*(H)) \ar[rrr]^-{\text{pair-of-pants with $[\alpha^1]$}}
&&&
H^*(\mathit{CF}^*(H))
}
\end{equation}



\section{Open string constructions\label{sec:open-string}}

This section defines the relevant Fukaya-categorical structures. The setting is the same as in the previous section: we are given a Liouville manifold with vanishing first Chern class, and a Maurer-Cartan element in its symplectic cohomology $L_\infty$-algebra. The first step is to set up the resulting deformation of the wrapped Fukaya category. Next, we construct corresponding deformed open-closed and closed-open string maps (the choice of asymptotic markers means that the closed string Floer groups which appear in those two maps are not the same in general, even though they will ultimately agree in our specific geometric application). The open-closed map has a cyclic analogue which lands in deformed $S^1$-equivariant symplectic cohomology. The last part, Section \ref{subsec:intertwine}, focuses on a single issue, namely compatibility of the cyclic open-closed map with connections (on the geometric side, the connection on $S^1$-equivariant deformed symplectic Floer cohomology defined in Section \ref{section:connection}; and on the algebraic side, the Getzler-Gauss-Manin connection on the cyclic homology of the deformed Fukaya category). There are large overlaps with results elsewhere in the literature. The construction of the deformed Fukaya category is similar to the familiar relative Fukaya category \cite{sheridan16, perutz-sheridan22}. The cyclic open-closed map was constructed in \cite{ganatra19}, and we use a modified version of that approach. Finally, in the context of the relative Fukaya category, \cite{ganatra-sheridan25} (posted after the first version of the present paper) proves compatibility for connections, a result originally announced in \cite{ganatra-perutz-sheridan15}. In view of the existence of those other sources, and also because of the sheer amount of moduli spaces and formulae involved, the level of detail is lower than elsewhere in the paper. 

In particular, signs are often omitted throughout the main exposition. However, in the concluding Section \ref{subsec:open-string-signs}, we discuss the principles underlying the choice of signs (continuing on from the closed string case, in Sections \ref{subsec:or-con} and \ref{section:Linf}), and we spell out the signs in the following three instances.
\begin{itemize} \itemsep.5em
\item 
The $A_\infty$-associativity equation for the deformed Fukaya category (constructed in Section \ref{subsubsec:Fukaya}, based on spaces from Section \ref{subsubsec:with-interior}). Here, the signs are verified in Section \ref{subsubsec:signs-fukaya}.
\item
The chain map property of the open-closed map (from Section \ref{section:deformed-oc}, based on spaces from Section \ref{subsubsection:oc}). Sign issues that arise here are dealt with in Section \ref{subsubsec:signs-oc}.
\item
For the construction of the modified Getzler-Gauss-Manin connection (Section \ref{subsubsec:half-strip}), the key sign calculation occurs in Section \ref{subsubsec:signs-gm}.
\end{itemize}
Together, these arguments illustrate the main issues that arise in geometric sign arguments throughout this section (notably, when considering the open-closed map, the use of half-cylinders rather than thimbles). Again, there are other references: the proof of the $A_\infty$-associativity equation is a streamlined version of that from \cite[Chapter 12]{seidel04}; the construction of cyclic open-closed maps in \cite{ganatra19} includes signs; and \cite{ganatra-sheridan25} has a much more systematic discussion, in which signs are treated ``organically'' by including the possible orientations of parameter spaces into the definition of operations (instead of our more traditional approach, which picks one orientation and spells out the resulting signs explicitly).

\subsection{Parameter spaces}

\subsubsection{Pointed discs\label{sec:pointeddiscs}}
We start with a quick review of the Stasheff-Fukaya \cite{fukaya97} moduli space of discs with boundary marked points. Throughout, $\disc$ stands for the closed unit disc, and $\bH = \{\mathrm{re}(z) \leq 0\}$ for the closed left half-plane, in $\bC$. Fix $d \geq 2$, and let $(\zeta_0,\dots,\zeta_d)$ be a configuration of points in $\partial\bD$ which are numbered compatibly with their cyclic order on the circle. One considers surfaces
\begin{equation} \label{eq:pointed-disc}
S = \bD \setminus \{\zeta_0,\dots,\zeta_d\},
\end{equation}
and divides by the automorphism group of $\bD$ to get
\begin{equation} 
\oSd = \mathit{Conf}_{d+1}^{\operatorname{ord}}(\partial \disc)/\mathit{PU}(1,1).
\end{equation}
To represent this space in parallel with \ref{subsec:deligne-mumford}, one can go from the disc to the half-plane, thinking of $\zeta_0 = -\infty$. In that picture, $(\zeta_1,\dots,\zeta_d)$ are points on $\partial\bH = i\bR$ in upwards order, the surface is
\begin{equation} \label{eq:half-plane-0}
S = \bH \setminus \{\zeta_1,\dots,\zeta_d\},
\end{equation}
and the parameter space turns into
\begin{equation}
\oSd = \mathit{Conf}_d^{\operatorname{ord}}(\partial\bH)/(\bR \rtimes \bR^{>0}).
\end{equation}

The Fukaya-Stasheff spaces $\Sd$ are compactifications of the configuration spaces $\oSd.$ They arise as subsets of the real locus of Deligne-Mumford spaces (see e.g.\ \cite[Section 9f]{seidel04}), and from that inherit the structure of manifolds with corners.
\begin{itemize} \itemsep.5em
\item Take a tree $T$ with $(d+1)$ semi-infinite ends. We say that $T$ is planar if it comes with the topological datum of an embedding into $\bR^2$, such that the ordering of the semi-infinite edges as $\{0,\dots,d\}$ is compatible with the cyclic ordering inherited from the orientation of the plane.
\item Given a planar tree, the edges adjacent to a vertex $v$ inherit an ordering by $\{0,1,\dots,|v|_{\operatorname{\operatorname{in}}}\}$, which is such that $0$-th edge is the outgoing one.
\end{itemize}
With that, the stratification of the compactification is
\begin{equation} \label{eq:Tstratumdiscs}
\Sd = \bigsqcup_T \mathring{\frakR}_T, \quad \mathring{\frakR}_T = \prod_{v} \mathring{\frakR}_{|v|_{\operatorname{\operatorname{in}}}}
\end{equation}
where the disjoint union is over (isomorphism classes of) planar stable trees $T$.

\subsubsection{Strip-like ends\label{sec:pointeddiscsends}}
Strip-like ends for a surface \eqref{eq:pointed-disc} are defined as in Section \ref{subsec:open-string-operations}. They come with a gluing process, which is the open string analogue of \eqref{eq:glued-surface}, again with a gluing parameter \eqref{eq:log-gamma}. We prefer to make a more restricted choice of (rational) ends, which is particularly easy to write down in terms of \eqref{eq:half-plane-0}:
\begin{align} \label{eq:punctured-half-plane-ends-1}
&
\begin{aligned}
& [0,\infty) \times [0,1] \longrightarrow S, \; (s,t) \longmapsto \zeta_j - i\rho_j\exp(-\pi(s+it)) \\ & \qquad \qquad \qquad \qquad \qquad \text{ $\rho_j>0$, near $\zeta_j$ ($j = 1,\dots,m$);}
\end{aligned}
\\ &
\begin{aligned}
\label{eq:punctured-half-plane-ends-2}
& (-\infty,0] \times [0,1] \longrightarrow S, \; (s,t) \longmapsto \chi - i\rho_0\exp(-\pi(s+it)) 
\\ & \qquad \qquad \qquad \qquad \qquad \text{ $\rho_0>0$, $\chi \in \bR$, near $\zeta_0 = -\infty$.}
\end{aligned}
\end{align}
After choosing consistent strip-like ends over each $\oSd$, the gluing construction again gives rise to charts near each stratum of \eqref{eq:Tstratumdiscs}, \begin{equation} \label{eq:gluingLg} 
[0,1)^{\mathit{E}_{\operatorname{fin}}(T)} \times \mathring{\frakR}_T \longrightarrow \Sd.
\end{equation}

\subsubsection{Discs with boundary and interior punctures\label{subsubsec:with-interior}}
We now consider surfaces
\begin{align}
\label{eq:disc-picture}
& S = \bD \setminus \{\zeta_0,\dots,\zeta_d,z_1,\dots,z_m\} 
\;\; \text{or equivalently} \\ &
S = \bH \setminus \{\zeta_1,\dots,\zeta_d,z_1,\dots,z_m\}, \label{eq:half-plane}
\end{align}
where the $\zeta_j$ are boundary points as before, and the $z_k$ are interior points. For $m = 0$ this reduces to the previous discussion, so we assume $d \geq 2$ in that case; while for $m>0$, any $d \geq 0$ is allowed. Let 
\begin{equation}
\begin{aligned}
\oRkm & = \big(\mathit{Conf}_{d+1}^{\operatorname{ord}}(\partial \bD) \times \mathit{Conf}_m(\bD \setminus \partial \bD)\big)/\mathit{PU}(1,1) \\
& \iso \big(\mathit{Conf}_d^{\operatorname{ord}}(\partial\bH) \times
\mathit{Conf}_m(\bH \setminus \partial \bH)\big)/(\bR \rtimes \bR^{>0})
\end{aligned}
\end{equation}
be the resulting parameter space. To discuss its compactification, we again need to augment our tree terminology.
\begin{itemize} \itemsep.5em
\item Take a tree $T$. We ask that the $d+m$ incoming semi-infinite edges should be divided into $d$ {\em open} and $m$ {\em closed} (string) ones, which are numbered independently as $\{1,\dots,d\}$ and $\{1,\dots,m\}$. The single outgoing edge, which one can number as $0$, is always considered open. The vertices of $T$ should also come with a division into open and closed ones. An open semi-infinite edge can only be adjacent to an open vertex (so the root vertex is open). A finite edge can go between two vertices of the same kind, or from a closed vertex to an open vertex, but not the other way. 

\item Let's consider the sub-tree $T^{\operatorname{op}}$ consisting only of open vertices and edges connecting them. This should come with a planar embedding, compatible with the ordering of open vertices. If $v$ is an open vertex, we write $|v|_{\operatorname{\operatorname{in}}}^{\operatorname{op}}$ for the number of incoming edges in $T^{\operatorname{op}}$, and $|v|_{\operatorname{\operatorname{in}}}^{\operatorname{cl}} = |v|_{\operatorname{\operatorname{in}}} - |v|_{\operatorname{\operatorname{in}}}^{\operatorname{op}}$.
\end{itemize}
In the compactification, the open vertices correspond to Fukaya-Stasheff spaces and the closed vertices to Fulton-MacPherson spaces. More precisely,
\begin{equation} \label{eq:rdm}
\Rkm = \bigsqcup_T \mathring{\frakR}_T, \quad 
\mathring{\frakR}_T = \prod_{v \text{ open}} 
\mathring{\frakR}_{|v|_{\operatorname{in}}^{\operatorname{op}}, |v|_{\operatorname{in}}^{\operatorname{cl}}} 
\times \prod_{v \text{ closed}} 
\mathring{\frakF}_{|v|_{\operatorname{in}}},
\end{equation} 
where $T$ ranges over stable trees (with $(d+1)$ open, and $m$ closed, semi-infinite edges). Here, the stability condition is $|v|_{\operatorname{in}} \geq 2$ at a closed vertex, and $|v|_{\operatorname{in}}^{\operatorname{op}} + 2|v|_{\operatorname{in}}^{\operatorname{cl}} \geq 2$ at an open vertex.
As in previous situations we have encountered, the compactification is a smooth manifold with corners. The degenerations leading to the strata in \eqref{eq:rdm} are best understood in the half-plane picture \eqref{eq:half-plane}: as punctures collide at some point in the interior of $\bH$, we rescale to get a limiting configuration of $\bC$, which is unique up $\bC \rtimes \bR^{>0}$, hence gives a well-defined point in Fulton-MacPherson space. $\mathit{Sym}(m)$ acts freely on $\Rkm$, by permuting $\{z_1,\dots,z_m\}$ (thanks to Lemma \ref{lem:freeabs}, it's enough to show freeness on $\mathring{\frakR}_{d,m}$; but there, the isotropy group of a point corresponds to a finite subgroup of $\bR \rtimes \bR^{>0}$, and there are no nontrivial such groups).

We equip the interior marked points with asymptotic markers, so that in the half-plane picture they point left (towards $\zeta_0 = -\infty$). Equivalently, in terms of the hyperbolic metric on the interior, one can think of this as pointing along the geodesic that goes from $z_k$ to $\zeta_0$ (that point of view applies both to $\bD$ and $\bH$, and will be used again several times in the future); see Figure \ref{fig:punctured-half-plane}. Correspondingly, the tubular ends can be taken as in \eqref{eq:punctured-plane-ends}. For the boundary punctures, we choose strip-like ends as in \eqref{eq:punctured-half-plane-ends-1}, \eqref{eq:punctured-half-plane-ends-2}. Within this class, we can make a universal choice of cylindrical and strip-like ends, which is $\mathit{Sym}(m)$-invariant, and consistent with gluing operations (and our previous choices for Fulton-MacPherson spaces). 
\begin{figure}
\begin{centering}
\includegraphics{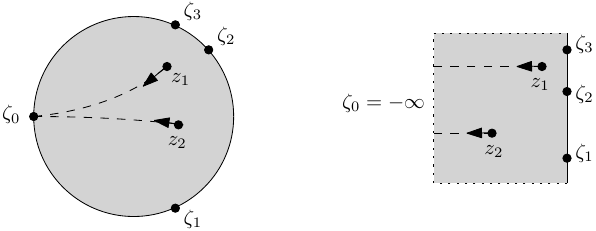}
\caption{\label{fig:punctured-half-plane}A summary of the conventions for $\mathring{\frakR}_{d,m}$. The left picture follows \eqref{eq:disc-picture}, and the equivalent one on the right follows \eqref{eq:half-plane}.}
\end{centering}
\end{figure}

\begin{remark}
Start with the space $\mathring{\frakM}_{d+2m}$, and look at the involution which reverses the complex structure of the Riemann surface, and simultaneously acts on the set of marked points $(\zeta_0,\dots,\zeta_d,z_1,\dots,z_m,\bar{z}_1,\dots,\bar{z}_m)$ by exchanging $z_k$ and $\bar{z}_k$. That involution extends to $\frakM_{d+2m}$, and we can lift it to the framed version. One can map $\frakR_{d,m}$ to the fixed point set (real locus) of $\frakM_{d+2m}^{\mathit{fr}}$, and derive its structure as manifold with corners from that. 
\end{remark}

\subsubsection{Punctured half-cylinders\label{subsubsection:oc}}
Here we will look at a variation of the previous parameter space. For $d, m \geq 0$, consider punctured half-cylinders
\begin{equation} \label{eq:half-cylinder}
\begin{aligned}
&
S = ((-\infty,0] \times S^1) \setminus \{z_1,\dots,z_m, \zeta_0,\dots,\zeta_d\},
\\ & 
z_1,\dots,z_m \in (-\infty,0) \times S^1, \;\; \zeta_0,\dots,\zeta_d \in \{0\} \times S^1 \text{ in cyclic order}.
\end{aligned}
\end{equation}
Denote the resulting parameter space by
\begin{equation}
\mathring{\frakH}_{d,m} = \mathit{Conf}_m((-\infty,0) \times S^1) \times
\mathit{Conf}_{d+1}^{\operatorname{ord}}(S^1).
\end{equation}
Note that here, as in the previous \eqref{eq:frakc}, we are not dividing by rotations in $S^1$-direction. This space admits a compactification to a manifold with corners $\frakH_{d,m}.$ As in the parallel situation of Section \ref{section:FM-cylinder}, it is convenient to describe the compactification in two steps.
\begin{itemize} \itemsep.5em
\item We first introduce a space capturing bubbling in the interior of the cylinder, and disc bubbling at the boundary:
\begin{equation} \label{eq:half-cylinder-heart}
\begin{aligned} &
\frakH_{d,m}^{\heart} = \bigsqcup_T \mathring{\frakH}_T^\heart, \\ &
\mathring{\frakH}_T^\heart = 
\mathring{\frakH}_{|v_{\operatorname{root}}|_{\operatorname{in}}^{\operatorname{op}},|v_{\operatorname{root}}|_{\operatorname{in}}^{\operatorname{cl}}}
\times \prod_{\substack{v \text{ open} \\ v \neq v_{\operatorname{root}}}}
\mathring{\frakR}_{|v|_{\operatorname{in}}^{\operatorname{op}},|v|_{\operatorname{in}}^{\operatorname{op}}} 
\times \prod_{v \text{ closed}} \mathring{\frakF}_{|v|_{\operatorname{in}}}.
\end{aligned}
\end{equation}
Here, $T$ ranges over colored trees as before, but the stability condition at the root vertex has been dropped.

\item The full compactification builds in breaking of the cylinder, as points go to $s=-\infty$. This can be formulated succinctly using the spaces from Section \ref{section:FM-cylinder}:
\begin{equation} \label{eq:cm-space}
\frakH_{d,m} = \bigsqcup_{m_1+m_2 = m, m_1> 0}\frakC_{m_1} \times \frakH_{d,m_2}^{\heart}.
\end{equation}
\end{itemize}
The compactification is a manifold with corners. It carries an action of $\bZ/(d+1)\bZ \times \mathit{Sym}(m)$ by permuting the punctures (cyclically for boundary punctures), and another application of Lemma \ref{lem:freeabs} shows this action to be free.

The asymptotic markers at interior punctures, including $-\infty$, are chosen as in Section \ref{section:FM-cylinder}. Correspondingly, the ends at interior punctures are as in \eqref{eq:endoncylinder1}, \eqref{eq:log-end}. At boundary punctures we pick them as in \eqref{eq:punctured-half-plane-ends-1} but this time applied to $\zeta_0$ as well. This means that all boundary ends are considered as inputs, and parametrized by $[0,\infty) \times [0,1]$. We can make a choice of those ends over the parameter spaces, which is compatible with gluing and also $\bZ/(m+1)\bZ \times \mathit{Sym}(m)$-equivariant.

\begin{remark} \label{th:should-we-rotate}
The space $\frakH_{d,m}$ carries a free $S^1$-action by rotating cylinders, and there is an isomorphism 
\begin{equation}
\frakH_{d,m}/S^1 \iso \frakR_{d,m+1}.
\end{equation}
If we take $\frakR_{d,m+1}$ but label the interior punctures as $(z_0,\dots,z_m)$, then adding an arbitrary asymptotic marker at $z_0$ yields an $S^1$-bundle which is canonically isomorphic to $\frakH_{d,m}$. Note however that this is incompatible with our choices of asymptotic markers at the other interior points: in the case of $\mathring{\frakH}_{d,m}$ those markers point towards $z_0$, whereas for $\mathring{\frakR}_{d,m+1}$ they point towards $\zeta_0$. For that reason, and in view of their differing applications, we have chosen different notation for those spaces.
\end{remark}
%

\subsubsection{Angle-decorated half-cylinders\label{subsubsection:angledecoratedhalf}}
An angle-decorated half-cylinder is a surface \eqref{eq:half-cylinder} together with 
\begin{equation} \label{eq:rangledechalf}
(\sigma_1,\theta_1),\dots,(\sigma_r,\theta_r) \in (-\infty,0] \times S^1,
\;\; \sigma_1 \leq \cdots \leq \sigma_r.
\end{equation}
Let $\mathring{\AH}_{d,m,r}$ be the resulting parameter space. Again, it admits a compactification to a manifold with corners $\AH_{d,m,r}$. Note that breaking off of cylinders at $-\infty$ is subject to the same angle-twisting conventions as in Section \ref{section:angles}. On the codimension $1$ boundary face where two of the $\sigma_i$ collide, there is a forgetful map 
\begin{equation} \label{eq:forget-circleopen}
\partial_{\sigma_{i+1}=\sigma_i} \AH_{d,m,r} \longrightarrow \AH_{d,m,r-1}
\end{equation}
which over $\mathring{\AH}_{d,m,r}$ is defined as in \eqref{eq:forget-circle}. On the boundary face where the rightmost $\sigma_r$ becomes $0$, we have a map
\begin{equation} \label{eq:sigmarboundary} 
\partial_{\sigma_r= 0}\AH_{d,m,r} \longrightarrow \AH_{d,m,r-1}
\end{equation}
which over $\mathring{\AH}_{d,m,r}$ is defined by forgetting $\sigma_r$ and then rotating the half-cylinder (with its punctures) by $-\theta_r$. 

The asymptotic marker and end at $z_0 = -\infty$ are chosen as in \eqref{eq:rotated-end} (meaning, they are rotated by $\theta_{\operatorname{tot}}$), and the remaining ones follow the same idea as for the previously discussed half-cylinder spaces. When making a choice of ends over the parameter spaces, we always want them to be consistent with gluing; invariant under the (free) action of $\bZ/(d+1)\bZ \times \mathit{Sym}(m)$; and compatible with both \eqref{eq:forget-circleopen} and \eqref{eq:sigmarboundary}.

\subsubsection{Open-closed parameter spaces (1)\label{subsubsection:ocs1}}
The open-closed map and its cyclic extension do not use the entirety of the spaces $\mathring\frakH_{d,m}$ and $\mathring{\AH}_{d,m,r}$, but only certain subspaces. The first of these is the subset of those angle-decorated half-cylinders where
\begin{equation} \label{eq:zeta0-is-0}
\zeta_0 = (0,0).
\end{equation}
We denote these by $\mathring\frakH_{d,m}^{(1)} \subset \mathring\frakH_{d,m}$ respectively $\mathring{\AH}_{d,m,r}^{(1)} \subset \mathring{\AH}_{d,m,r}$, and similarly for the closure. It is unproblematic to see that the closure is a submanifold with corners. For future reference, we list the codimension $1$ boundary strata (the first two already appear in $\mathring{\AH}_{d,m,r}^{(1)}$, the others only in the compactification).
\begin{itemize} \itemsep.5em
\item One can have $\sigma_i = \sigma_{i+1}$ for some $0 \leq i < r$, as in \eqref{eq:forget-circleopen}.
\item One can have $\sigma_r = 0$, as in \eqref{eq:sigmarboundary}.
\item A cylinder can break off at $-\infty$, which yields an open stratum of the form
\begin{equation}
\mathring{\AC}_{m_1,r_1} \times \mathring{\AH}_{d,m_2,r_2}^{(1)}, \;\;
m_1+m_2 = m, \; r_1+r_2 = r.
\end{equation}
\item Several punctures (of either kind) converge towards a point on $\{0\} \times S^1$, and bubble off into a disc. These strata have the form
\begin{equation}
\mathring{\AH}_{d_1,m_1,r_1}^{(1)} \times \mathring{\frakR}_{d_2,m_2}, \;\;
d_1+d_2 = d+1, \; m_1+m_2 = m.
\end{equation}
\item Fulton-Macpherson bubbling occurs at some interior point.
%
\end{itemize}

\begin{remark}
Take the case $r = 0$. In terms of Remark \ref{th:should-we-rotate}, the condition \eqref{eq:zeta0-is-0}  means that 
\begin{equation}
\frakH_{d,m}^{(1)} \iso \frakR_{d,m+1}.
\end{equation}
If one then also sets $m = 0$, the outcome is the space of discs with one interior puncture. By the relevant special case of \eqref{eq:rotated-end}, the asymptotic marker at the unique interior point $z_0 = -\infty$, points towards $\zeta_0$. This recovers the setup used in the classical construction of (non-equivariant, undeformed) open-closed maps, see e.g. \cite{abouzaid10}.
\end{remark}

\subsubsection{Open-closed parameter spaces (2)\label{subsubsection:ocs2}}
Our second space is again defined as a subset, or more precisely a codimension $0$ submanifold with boundary,
\begin{equation} \label{eq:type-2}
\mathring{\AH}_{d,m,r}^{(2)} \subset \mathring{\AH}_{d-1,m,r}. 
\end{equation}
The jump in the value of $d$ accomodates the convention that specifically for $\mathring{\AH}_{d,m,r}^{(2)}$, we number the boundary punctures by $\{\zeta_1,\dots,\zeta_d\}$ (see Remark \ref{th:plus-point} below for motivation). With that in mind, the subset is defined by asking that
\begin{equation} \label{eq:markerconstraints} 
(0,0) \text{ lies in the closed interval inside $\{0\} \times S^1$ which starts at $\zeta_d$ and ends at $\zeta_1$.}
\end{equation}
Here, ``starts and ends'' is with respect to the natural boundary orientation, which is the same as that given by the ordering $\zeta_1,\dots,\zeta_d$. (In the case $d = 1$, the condition is empty, meaning that $\zeta_1$ may lie anywhere on $\{0\} \times S^1$; for $d = 2$ it singles out one of the two closed intervals in $S^1$ with endpoints $\{\zeta_1,\zeta_2\}$.) We define $\AH_{d,m,r}^{(2)}$ to be the closure of $\mathring{\AH}_{d,m,r}^{(2)}$ in $\AH_{d-1,m,r}$. The codimension $1$ boundary faces are:
\begin{itemize} \itemsep.5em
\item counterparts of all five kinds previously encountered for type (1), which means the intersections of our space with the codimension $1$ boundary faces of $\AH_{d-1,m,r}$;
\item additionally, strata where $\zeta_1 = (0,0)$ or $\zeta_d = (0,0)$.
\end{itemize}
The space $\AH_{d,m,r}^{(2)}$ is not a manifold with corners. However, it is a manifold with boundary away from a set which is ``of codimension $\geq 2$'', meaning that it is the union of pieces that are locally closed submanifolds of codimension $\geq 2$ in $\AH_{d-1,m,r}$. This will be sufficient for our purpose, since all the data needed for Floer-theoretic constructions are first chosen over $\mathring{\AH}_{d-1,m,r}$, in a way which is compatible with the compactification, and then restricted to \eqref{eq:type-2}. This applies in particular to the strip-like ends.
\begin{figure}
\begin{centering}
\includegraphics{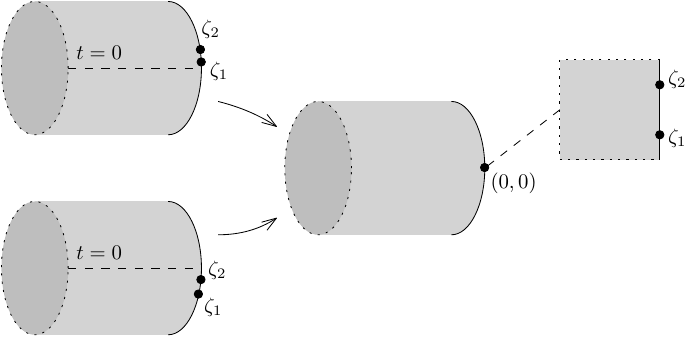}
\caption{\label{fig:glue-triangle-2}The limit $(\zeta_1^{\operatorname{lift}}, \zeta_2^{\operatorname{lift}}) \rightarrow (0,0)$ or $(1,1)$ in Example \ref{th:triangle}.}
\end{centering}
\end{figure}%
\begin{figure}
\begin{centering}
\includegraphics{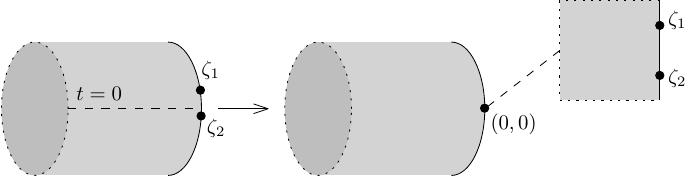}
\caption{\label{fig:glue-triangle-1}The limit $(\zeta_1^{\operatorname{lift}}, \zeta_2^{\operatorname{lift}}) \rightarrow (0,1)$ in Example \ref{th:triangle}.}
\end{centering}
\end{figure}%

\begin{example} \label{th:triangle}
Let's look at the two-dimensional space 
\begin{equation} \label{eq:triangle-space}
\mathring{\AH}_{2,0,0}^{(2)} \iso \{0 \leq \zeta_1^{\operatorname{lift}} < \zeta_2^{\operatorname{lift}} \leq 1\} \setminus \{(0,1)\}.
\end{equation}
where the position of the marked points has been lifted to $[0,1]$ by starting at $(0,0)$ and moving with the boundary orientation. There are two boundary faces already visible in $\mathring{\AH}_{2,0,0}$, namely $\zeta_1^{\operatorname{lift}} = 0$ and $\zeta_2^{\operatorname{lift}} = 1$. The limit where $\zeta_1^{\operatorname{lift}} = \zeta_2^{\operatorname{lift}} \in (0,1)$ is a standard bubbling process, giving rise to a smooth boundary side of the compactification. This takes care of the codimension $1$ faces. 

The two corners $(0,0)$ and $(1,1)$ in the closure of \eqref{eq:triangle-space} actually correspond to the same point in $\AH_{2,0,0}^{(2)}$. The corresponding degeneration is shown in Figure \ref{fig:glue-triangle-2}. If one starts with the half-plane bubble with marked points at $(0,\pm 1)$, and glues the two parts together using small parameters $t \in \bR$ (varying the position of the attaching point) and $\gamma>0$ (gluing parameter) inherited from the larger space $\AH_{1,0,0}$, the outcome (omitting some irrelevant constants) is a half-cylinder with $\zeta_1^{\operatorname{lift}} = t-\gamma$ and $\zeta_2^{\operatorname{lift}} = t+\gamma$. The condition \eqref{eq:triangle-space} becomes
\begin{equation} \label{eq:t-gamma}
t \leq -\gamma \text{ or }  t \geq \gamma.
\end{equation}
Hence, this is not even topologically a codimension $2$ corner. In contrast, the corner $(0,1)$ of \eqref{eq:triangle-space} poses no complications, since there the  analogue of \eqref{eq:t-gamma} is $-\gamma \leq t \leq \gamma$ (Figure \ref{fig:glue-triangle-1}).
\end{example}
\begin{figure}
\begin{centering}
\includegraphics{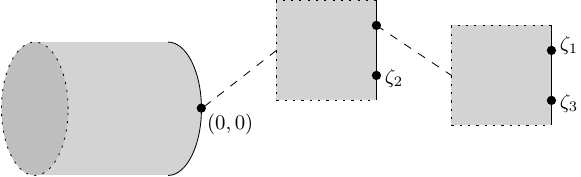}
\caption{\label{fig:sheel-glue}The point in $\AH_{3,0,0}^{(2)}$ from Example \ref{th:sheel}.}
\end{centering}
\end{figure}

\begin{example} \label{th:sheel}
A slightly more complicated case is 
\begin{equation} \label{eq:3-ring-circus}
\mathring{\AH}_{3,0,0}^{(2)} \iso \{0 \leq \zeta_1^{\operatorname{lift}} < \zeta_2^{\operatorname{lift}} < 
\zeta_3^{\operatorname{lift}} \leq 1\} \setminus (\{0\} \times (0,1) \times \{1\}).
\end{equation}
Instead of discussing the entire compactification, we'll just focus on the neighbourhood of a single point, shown in Figure \ref{fig:sheel-glue}. In the larger space $\AH_{3,0,0}$, a neighbourhood of this point is parametrized by $t \in \bR$ (varying the position of the attaching point on the boundary of the half-cylinder) and $\gamma_1,\gamma_2 \geq 0$ (gluing parameters). Assuming the marked points on both half-planes start out at $(0,\pm 1)$, gluing roughly yields 
\begin{equation}
\zeta_1 = (0,\gamma_1\gamma_2+t), \; \zeta_2 = (-\gamma_1+t), \;
\zeta_3 = (0,-\gamma_1\gamma_2+t).
\end{equation}
For this to lie in \eqref{eq:3-ring-circus} one needs $-\gamma_1\gamma_2 \leq t \leq \gamma_1\gamma_2$; which is not a corner in the $\smooth$-sense, since the graphs of $\pm \gamma_1\gamma_2$ become tangent at $(\gamma_1,\gamma_2) = (0,0)$.
\end{example}

\begin{remark} \label{th:plus-point}
It can make sense to consider the half-cylinder as coming with an additional decoration $\zeta_0^+ = (0,0)$, which can agree with either $\zeta_d$ or $\zeta_1$, and otherwise has to be such that $(\zeta_0^+,\zeta_1,\dots,\zeta_d)$ appear in cyclic order. This makes the situation more parallel to \eqref{eq:zeta0-is-0}. In particular, the choice of asymptotic marker in both cases is rotated by $\theta_{\operatorname{tot}}$ from the direction pointing towards $\zeta_0$ respectively $\zeta_0^+$. It is also intuitive in view of the definition of the cyclic complex \eqref{eq:plus-complex-cyclic}, where $\zeta_0^+$ can be thought of as standing in for the artificial unit $e^+$. Note that such a $\zeta_0^+$ would merely be a mnemonic for remembering \eqref{eq:markerconstraints}, and not a marked point or puncture, hence cannot bubble off. This is where our approach differs from that in \cite[Section 5.2]{ganatra19}: there, $\zeta_0^+$ (written as $z_f$) is treated as a marked point, leading to a larger compactification which is geometrically better-behaved but has additional codimension $1$ boundary strata \cite[Remark 45]{ganatra19}. For instance, take the situation from Example \ref{th:triangle}. If one treats $\zeta_0^+$ as a puncture, the resulting compactification is isomorphic to $\frakR_{1,2}$, which is a hexagon. In comparison with our compactification, not only have the limit points $(0,0)$ and $(1,1)$ in \eqref{eq:triangle-space} now become separate, each of them has actually been expanded into an entire boundary interval, and the same expansion has happened to $(0,1)$.
\end{remark}

\subsection{Open-string operations}

\subsubsection{The Fukaya category\label{subsubsec:Fukaya}}
We recall briefly the construction of the wrapped Fukaya category of $(\hat{N},\theta_{\hat{N}})$ using quadratic Hamiltonians, as in \cite{abouzaid10}. Fix a finite collection of Lagrangian submanifolds, which are conical \eqref{eq:l-cone}, carry brane structures \eqref{eq:brane}, and such that each pair $(L_0,L_1)$ satisfies Assumption \ref{th:no-p-chords} (with a fixed $P$, the same for all pairs). These will be the objects of our category. For every pair of objects we choose Floer data $(H_{L_{0},L_{1}}, J_{L_{0},L_{1}})$ as in Section \ref{section:LagrangianFloercomplex}; this yields the Floer complex $\scrA(L_0,L_1) = \mathit{CF}^*(L_0,L_1)$.

For each $d \geq 2$, we fix a consistent choice of strip-like ends $\varepsilon_0, \varepsilon_1, \cdots, \varepsilon_d$ over $\mathring{\frakR}_d$ (Section \ref{sec:pointeddiscs}). On the universal families $U_{\Sd} \longrightarrow \Sd$,
we choose perturbation data 
\begin{equation} \label{eq:u-data}
K_{\Sd} \in \Omega^1_{\mathcal{U}_{\Sd}/ \Sd} (U_{\Sd}, \scrH(\hat{N}))
\text{ and }
J_{\Sd} \in C^{\infty}(U_{\Sd}, \scrJ(\hat{N})). 
\end{equation}
These should satisfy the conditions of Section \ref{subsec:open-string-operations} when restricted to each fiber of the universal curve, and should be conformally consistent with respect to boundary strata (this is the open string analogue of the condition from Section \ref{subsubsec:neck}). Take chords $\mathbf{y} = (y_0,y_1,\cdots,y_d)$, where $y_0$ is  associated to the pair $(L_0,L_d)$, and each $y_i$, $i>0$, is associated to $(L_{i-1},L_i)$. We then consider the parametrized moduli space $\oSd(\mathbf{y})$ of pairs $(r,u)$, where $r \in \oSd$ and $u:S_r \rightarrow M$ is a solution to \eqref{eq:cauchy-riemann}, \eqref{eq:boundary-condition}, \eqref{eq:asymptoticslagrangian}. For a generic choice of perturbation data, these moduli spaces are smooth manifolds with
\begin{equation} 
\mathrm{dim}\, \oSd(\mathbf{y}) =
\operatorname{deg}(y_0)-\sum_{i=1}^d \operatorname{deg}(y_i)+d-2.
\end{equation} 
Moreover, they admit Gromov compactifications $\frakR_d(\mathbf{y})$, and in particular are finite sets if the dimension is $0$. As in the corresponding closed string context of Section \ref{subsubsection:Floerfamilies}, an isolated point $(r,u)$ gives rise to an isomorphism $o(r,u): \frako_{y_1} \otimes \cdots \otimes \frako_{y_m} \iso \frako_{y_0}$. The $\bK$-normalizations of those isomorphisms 
yield the $A_\infty$-operations 
\begin{equation}
\label{eq:mud} 
\mu^d: \mathit{CF}^*(L_{0},L_1) \otimes \cdots \otimes \mathit{CF}^*(L_{d-1},L_d) \longrightarrow \mathit{CF}^{*+2-d}(L_0,L_d)
\end{equation}
for $d \geq 2$; for $d = 1$ we use the Floer differential.

\subsubsection{The deformed Fukaya category\label{section:deform-fuk}}
This section explains how to deform the Fukaya category using a Maurer-Cartan element \eqref{eq:floer-mc}. The objects will be the same as before. The structure maps 
\begin{equation}
\label{eq:muqd} 
\mu_q^d: \mathit{CF}^*(L_{0},L_1) \otimes \cdots \otimes \mathit{CF}^*(L_{d-1},L_d) \to \mathit{CF}^*(L_0,L_d)[[q]]
\end{equation} 
reduce to the previous ones for $q = 0$ (which in particular means that $\mu^0_q$ has zero $q$-constant term). To define the $q$-deformed operations, we use the parameter spaces $\mathring{\frakR}_{d,m}$. We assume that ends have been chosen for those spaces, as in Section \ref{subsubsec:with-interior}, and we also choose perturbation data for the universal families over them, generalizing \eqref{eq:u-data}. Since the action of $\mathit{Sym}(m)$ on $\Rkm$ is free, we can achieve transversality while simultaneously asking that the perturbation data be $\mathit{Sym}(m)$-equivariant.
 
Having made these choices, given a collection of chords $\mathbf{y} = (y_0,y_1,\cdots,y_d)$ and periodic orbits $\mathbf{x} = (x_1, \cdots,x_{m})$, the resulting parametrized moduli space $\oRkm(\mathbf{y},\mathbf{x})$ satisfies
\begin{equation} 
\operatorname{dim}\, \oRkm(\mathbf{y},\mathbf{x}) = \operatorname{deg}(y_0)-\sum_{i=1}^{d}\operatorname{deg}(y_i)-\sum_{j=1}^{m}\operatorname{deg}(x_j)+d-2+2m. \end{equation}
Counting isolated solutions gives rise to operations
\begin{equation} \label{eq:iota11} 
\mu^{d,m}: \mathit{CF}^*(L_{0},L_1) \otimes \cdots \otimes \mathit{CF}^*(L_{d-1},L_d) \otimes \mathit{CF}^*(H)^{\otimes m} \longrightarrow \mathit{CF}^{*+2-d-2m}(L_0,L_d). 
\end{equation} 
We include the previously defined $A_\infty$-structure here as the special case $m = 0$.
Then, \eqref{eq:muqd} is defined by inserting the Maurer-Cartan element $\alpha$ in the way familiar from the closed string constructions in Section \ref{sec:q-deformed}:
\begin{equation} \label{eq:defnmuk} 
\mu_q^d(y_1,\cdots,y_d) \stackrel{\mathrm{def}}{=} \sum_{m \geq 0} {\textstyle\frac{1}{m!}}\,\mu^{d,m}(y_1,\cdots,y_d,\alpha^{\otimes m}) \in \mathit{CF}^*(L_0,L_d)[[q]].
\end{equation}

\begin{prop} \label{prop:deformedfukayawelldefined} 
The operations \eqref{eq:muqd} define a curved $A_\infty$-deformation $\scrA_q$ of $\scrA$.
\end{prop}

\begin{proof} 
The proof again proceeds by examining the boundaries of one-dimensional moduli spaces $\Rkm(\mathbf{y},\mathbf{x}))$. There are four different kinds of boundary points:
\begin{itemize} \itemsep.5em
\item 
Floer cylinder breaking at one of the interior marked points.  

\item Degeneration of the domain to the codimension one strata where $2 \leq m_1 \leq m$ of the interior marked points collide. On the parameter space $\frakR_{d,m}$, this means that we converge to a point in a stratum
\begin{equation} \label{eq:boundaryRkm1} 
\mathring{\frakR}_{d,m_2} \times \mathring{\mathfrak{F}}_{m_1} \subset \partial\Rkm,
\;\; m_1+m_2 = m+1.
\end{equation} 

\item Floer strip breaking at a boundary marked point. 

\item Boundary marked points collide, or interior marked points approach the boundary (or both). On the parameter space, this means convergence to a point in 
\begin{equation} \label{eq:boundaryRkm2} \mathring{\frakR}_{d_{2},m_{2}} \times \mathring{\frakR}_{d_{1},m_{1}} \subset \partial\Rkm, \;\;
m_1+m_2 = m, \; d_1+d_2 = d+1.
\end{equation}
\end{itemize}  
The fact that $\alpha$ satisfies the Maurer-Cartan equation means that the sum of the first two contributions is zero. The third and fourth kind of degeneration account for the terms in the $A_\infty$-equation.
\end{proof}

\subsubsection{The deformed closed-open map\label{subsubsec:deform-co}} 
The same moduli spaces as in the definition of $\scrA_q$ also give rise to a deformation of the closed-open string map. Given $w \in \mathit{CF}_q^{\operatorname{diag}}(H^{\floer})$ (see \eqref{eq:diagonaldeformationCF} for the notation), we insert $\alpha$ into the first $(m-1)$ closed string entries, and $w$ into the $m$-th entry, of the operation $\mu^{d,m}$ from \eqref{eq:iota11}; and write the outcome as
\begin{equation}
\begin{aligned} &
\mathit{CO}_q^d : \mathit{CF}_q^{\operatorname{diag}}(H^{\floer}) \longrightarrow \mathit{hom}(\scrA(L_{0},L_1,\dots,L_d),\scrA(L_0,L_d))[[q]], \\ &
 CO_q^d(w) = \sum_{m \geq 1} \pm \textstyle{\frac{1}{(m-1)!}} \, \mu^{d,m}(\dots,\alpha^{\otimes m-1},w).
\end{aligned}
\end{equation}
The collection of all such maps, for $d \geq 0$, yields
\begin{equation} \label{eq:COq3} 
\mathit{CO}_q: \mathit{CF}_q^{\operatorname{diag}}(H) \longrightarrow C^*(\scrA_q). \end{equation}
A variant of Proposition \ref{prop:deformedfukayawelldefined} shows that this is a chain map; compared to the argument given there, the boundary points which play a different role here are those where a Floer cylinder or Fulton-MacPherson configuration which contains the point $z_m$ bubbles off. The contribution from those boundary points is precisely $\mathit{CO}_q \circ \delta_q^{\operatorname{diag}}$, which is part of the chain map equation.

\subsubsection{The deformed cyclic open-closed map\label{section:deformed-oc}}
In this section, we construct the deformed cyclic open-closed map
\begin{equation} \label{eq:equi-oc}
\begin{aligned}
& \mathit{OC}_{S^1,q}: \mathit{CC}_*^+(\scrA_q) \longrightarrow CF_{S^{1},q}^{*+n}(H), \\
& (-1)^n \delta_{S^1,q} \circ \mathit{OC}_{S^1,q} = \mathit{OC}_{S^1,q} \circ d_{\mathit{CC}_*^+(\scrA_q)}.
\end{aligned}
\end{equation}

Recall that $\mathit{CC}_*^+(\scrA_q)$ is the direct sum of two kinds of pieces, $\scrA_q(L_d,L_0,\dots,L_d)$ and $e_{\scrA}^+ \otimes \scrA_q(L_0,L_1,\dots,L_{d-1},L_0)$, see \eqref{eq:categorical-nonunital-hochschild}. Correspondingly, the map \eqref{eq:equi-oc} has components 
\begin{align} &
\label{eq:OC1def} 
\mathit{OC}_{S^1,q,(1)}: \scrA_q(L_d,L_0,L_1,\dots,L_d) \longrightarrow \mathit{CF}_{S^{1},q}^{*+n-d}(H), 
\\ 
\label{eq:OC2def} & 
\mathit{OC}_{S^1,q,(2)}: e_\scrA^+ \otimes \scrA_q(L_0,L_1,\dots,L_{d-1},L_0) \longrightarrow \mathit{CF}_{S^{1},q}^{*+n-d}(H). 
\end{align}
In addition to those, we will introduce another operation $\mathit{AH}_{S^1,q,(0)}$ which plays a purely expository role (since it can ultimately be written in terms of $\mathit{OC}_{S^1,q,(2)}$; see Lemma \ref{lem:OChatversusOC2} below).

\begin{remark}
After setting $q = 0$ (or, geometrically, eliminating the extra interior punctures), our construction reduces to that from \cite{ganatra19}.
\end{remark}

To set up the construction, we choose perturbation data over the moduli spaces of angle-decorated half-cylinders $\AH_{d,m,r}$. These must lie in the class of perturbations allowed in Section \ref{subsubsec:closedopen}, and should be invariant with respect to $\bZ/(d+1)\bZ \times \mathit{Sym}(m)$ (which acts freely on the parameter space, see the discussion in Section \ref{subsubsection:angledecoratedhalf}). In addition to the usual conformal consistency conditions with respect to degenerations of the surfaces, we require that along certain boundary strata, the Floer data are pulled back from lower-dimensional parameter spaces: 
\begin{itemize} \itemsep.5em
\item over $\partial_{\sigma_{i+1}=\sigma_i} \mathring\AH_{d,m,r}$, the perturbation data are pulled back along \eqref{eq:forget-circleopen}. 
\item Along $\partial_{\sigma_{r}=0} \mathring{\AH}_{d,m,r}$, the perturbation data are pulled back along \eqref{eq:sigmarboundary}. 
\end{itemize}
Further conditions will be added throughout the subsequent discussion (one imposes them as part of the initial choice, but from an expository viewpoint it makes more sense to discuss them at the point where they are needed).

Let's first consider the moduli spaces $\mathring{\AH}_{d,m,r}(\bfy,\bfx)$ parametrized over $\mathring\AH_{d,m,r}$. Here, $\mathbf{x} = (x_0,x_1,\cdots,x_m)$ are periodic orbits, where $x_0$ is placed at the output $s = -\infty$, and $\bfy = (y_0,\dots,y_d)$ are chords. The zero-dimensional spaces give rise to operations
\begin{equation}
\mathit{AH}^{d,m,r}: \mathit{CF}^*(H)^{\otimes m} \otimes \scrA(L_d,L_0,L_1,\dots,L_d) 
\longrightarrow \mathit{CF}^{*+n-d-2r-2m-1}(H).
\end{equation}
We insert $\alpha$ at the $m$ closed string inputs, add up over all $m$, and also and add up over all $r$ with powers $u^r$; the outcome being maps
\begin{equation} \label{eq:ah}
\mathit{AH}_{S^1,q}: \scrA_q(X_d,X_0,\dots,X_d) \longrightarrow \mathit{CF}^{*+n-d-1}_{S^1,q}(H).
\end{equation}
Next take the subspace $\mathring{\AH}_{d,m,r}^{(1)}$ from Section \ref{subsubsection:ocs1}. The perturbation data that we use on these moduli spaces will be pulled back from $\mathring{\AH}_{d,m,r}$, with the following condition on the original choice:
\begin{itemize} \item
the parametrized moduli spaces $\mathring{\AH}_{d,m,r}^{(1)}(\bfy,\bfx)$ are regular; or equivalently, the map $\mathring{\AH}_{d,m,r}(\bfy,\bfx) \rightarrow \mathring{\AH}_{d,m,r}$ is transverse to the submanifold $\mathring{\AH}_{d,m,r}^{(1)}$.
\end{itemize}
The zero-dimensional moduli spaces spaces define operations
\begin{equation} \label{eq:oc-1-operations}
\mathit{OC}^{d,m,r}_{(1)}: \mathit{CF}^*(H)^{\otimes m} \otimes \scrA(L_d,L_0,L_1,\dots,L_d) 
\longrightarrow \mathit{CF}^{*+n-d-2r-2m}(H),
\end{equation}
which we manipulate algebraically as before to get \eqref{eq:OC1def}. Restricting to the moduli spaces without angle-decorations ($r = 0$), or equivalently setting the equivariant parameter $u$ to zero, yields the deformed version of the ordinary (non-equivariant) open-closed map,
\begin{equation} \label{eq:ordinary-open-closed}
\mathit{OC}_q: C_*(\scrA_q) \longrightarrow \mathit{CF}^{*+n}_q(H).
\end{equation}

As in our discussion of parameter spaces in Section \ref{subsubsection:ocs2}, we will now change conventions slightly: the number of boundary punctures is written as $d>0$, and they will be numbered by $\{\zeta_1,\dots,\zeta_d\}$. The Lagrangian submanifolds involved are correspondingly $(L_0,\dots,L_{d-1})$. We use the subspace $\mathring{\AH}_{d,m,r}^{(2)} \subset \mathring{\AH}_{d-1,m,r}$ and the associated parametrized moduli spaces $\mathring{\AH}_{d,m,r}^{(2)}(\bfy,\bfx)$. The additional transversality requirement is:
\begin{itemize} \item The moduli spaces $\AH_{d,m,r}^{(2)}(\mathbf{y},\mathbf{x})$ of (expected) dimension $\leq 1$ contain no curves lying over the codimension two or higher boundary strata of $\AH_{d,m,r}^{(2)}$. 
\end{itemize}
Using those spaces, we define maps
\begin{equation}
\mathit{OC}^{d,m,r}_{(2)}: \mathit{CF}^*(H)^{\otimes m} \otimes \big(e^+_{\scrA} \otimes \scrA(L_{0},L_1,\dots,L_d,L_0)\big)
\longrightarrow \mathit{CF}^{*+n-d-2r-2m}(H),
\end{equation}
which then lead to \eqref{eq:OC2def}.

Take the Connes operator which is part of \eqref{eq:c-unital-cyclic}, 
\begin{equation} 
B(a_0(a_1|\dots|a_d)) = -\sum_j (-1)^{(\|a_0\|+\cdots+\|a_j\|)(\|a_{j+1}\|+\cdots+\|a_d\|)} e_\scrA^+ (a_{j+1}|\dots|a_d|a_0|\dots|a_j).
\end{equation} 

\begin{lemma} \label{lem:OChatversusOC2} 
The operations \eqref{eq:ah} and \eqref{eq:OC2def} are related by  
\begin{equation} \label{eq:OChatOC2} 
\mathit{AH}_{S^1,q}(a_0(a_1|\dots|a_d))= OC_{S^1,q,(2)}\circ B(a_0(a_1|\dots|a_d)).
\end{equation}
\end{lemma}

\begin{proof}[Sketch of proof]
In the moduli spaces $\mathring{\AH}_{d,m,r}(\mathbf{y},\mathbf{x})$, the point $(0,0)$ can lie between any two consecutive  boundary marked points. (The situation where $(0,0)$ agrees with one of the marked points $\zeta_j$ is codimension $1$ and can be ignored.) Suppose $(0,0)$ lies between $\zeta_j$ and $\zeta_{j+1}$. Because the Floer data have been chosen equivariantly under cyclic permutation of the boundary marked points, we can view this as contributing to $\pm OC_{S^1,q,(2)}\circ e_\scrA^+(a_{j+1}|\dots|a_d|a_0|\dots|a_j).$ The boundary of the disc is naturally decomposed into intervals lying between consecutive marked points, and so summing over all $j$ proves \eqref{eq:OChatOC2}.  \end{proof}

\begin{prop} \label{prop:OCS1eq1} 
The following equation holds: 
\begin{equation} 
\begin{aligned}
&
\mathit{OC}_{S^1,q,(1)} \circ d_{\mathit{C}_*^+(\scrA_{q})}(a_0(a_1|\dots|a_d)) + u\mathit{OC}_{S^1,q,(2)} \circ B(a_0(a_1|\dots|a_d)) 
\\ & 
= (-1)^n \delta_{S^1,q} \circ \mathit{OC}_{S^1,q,(1)}(a_0(a_1|\dots|a_d)). 
\end{aligned}
\end{equation} 
\end{prop}

\begin{proof} The contributions from codimension one boundary strata of $\mathring{\AH}_{d,m,r}^{(1)}(\mathbf{y},\mathbf{x})$ are as follows:
\begin{itemize} \itemsep.5em 
\item We have the boundary strata where $\sigma_{i+1}=\sigma_i$. These do not contribute because the Floer data is pulled back from a lower-dimensional space along \eqref{eq:forget-circleopen}. \item We can have Floer breaking or Fulton-MacPherson bubbling at some interior point. As in prior situations, the Maurer-Cartan equation ensures this contributes nothing. 

\item Consider the stratum where $\sigma_{r}=0$. Away from higher codimension subsets, the map \eqref{eq:sigmarboundary} defines an isomorphism from this boundary stratum to $\mathring{\AH}_{d,m,r-1}$. As the Floer data is pulled back from there, the argument from Lemma \ref{lem:OChatversusOC2} shows that this stratum contributes the term $u\mathit{OC}_{S^1,q,(2)} \circ B(a_0(a_1|\dots|a_d))$. 

\item We can have Floer cylindrical breaking or degeneration as $s \to -\infty$, which yields the term $(-1)^n\delta_{S^1,q} \circ OC_{S^1,q,(1)}$. 

\item Floer strip breaking or disc bubbling at the boundary contributes the term $OC_{S^1,q,(1)} \circ d_{\mathit{C}_*^+(\scrA_{q})}(a_0(a_1|\dots|a_d))$.
\end{itemize} 
\end{proof}

\begin{prop} \label{prop:OCS1eq2} 
The following equation is also satisfied: 
\begin{equation} \label{eq:also}
\begin{aligned}
& -\mathit{OC}_{S^1,q,(2)} \circ e^+_{\scrA}\, d_{D_*(\scrA_{q})}(a_1|\dots|a_d))
+ \mathit{OC}_{S^1,q,(1)}(a_1(a_2|\dots|a_d)) 
\\ & \qquad
-(-1)^{\|a_d\|(\|a_1\|+\cdots+\|a_{d-1}\|)}\mathit{OC}_{S^1,q,(1)}(a_d(a_1|\dots|a_{d-1}))
\\ & = 
(-1)^n\delta_{S^1,q} \circ OC_{S^1,q,(2)}(e_\mathcal{A}^{+}(a_1|\dots|a_d)).
\end{aligned}
\end{equation}
Here, $D_*(\scrA_{q})$ denotes the deformed version of the bar complex from \eqref{eq:bar-complex} with differential \eqref{eq:dbar-differential}.
\end{prop}

\begin{proof} 
Consider the contributions from codimension one boundary points of $\mathring{\AH}_{d,m,r}^{(2)}(\mathbf{y},\mathbf{x})$:
\begin{itemize} 
\itemsep.5em
\item As usual, the boundary strata where $\sigma_{i+1}=\sigma_i$ contribute nothing, because the Floer data are pulled back from a lower-dimensional space. 

\item We can have Floer breaking or Fulton-MacPherson bubbling at some interior point. Again, the Maurer-Cartan equation ensures this contributes nothing. 

\item This time, the stratum where $\sigma_{r}=0$ also contributes nothing because the perturbation data is pulled back from $\mathring{\AH}_{d-1,m,r-1}$, which is a lower-dimensional parameter space, along \eqref{eq:sigmarboundary}.

\item As in the previous Proposition, we can have a cylinder breaking off at $s \to -\infty$, which now yields the term $(-1)^n \delta_{S^1,q} \circ OC_{S^1,q,(2)}$. 

\item Floer strip breaking or disc bubbling at the boundary gives rise to $e^+_{\scrA}\, d_{D_*(\scrA_{q})}(a_1|\dots|a_d)$. 

\item The remaining boundary strata are where $\zeta_1 = (0,0)$ or $\zeta_d = (0,0)$. These boundaries contribute the last two terms on the left hand side of \eqref{eq:also}.
\end{itemize} 
\end{proof}

The two Propositions above combine to show that $\mathit{OC}_{S^1,q}$ is a chain map.

\subsection{The cyclic open-closed map and connections\label{subsec:intertwine}}

\subsubsection{The modified Getzler-Gauss-Manin connection\label{subsubsec:half-strip}}
Start with the operations $\mu^{d,m}$ from \eqref{eq:iota11} and $\mu^d_q$ from \eqref{eq:defnmuk} and consider 
\begin{equation} \label{eq:partial-q-mu}
\begin{aligned} &
\partial_q\mu^d_q: \scrA_q(L_0,L_1,\dots,L_d)
\longrightarrow \scrA_q(L_0,L_d)[-d], \\
& (\partial_q \mu^d_q)(a_1,\dots,a_d) = 
\sum_{m \geq 1} \textstyle{\frac{1}{(m-1)!}} \mu^{d,m}(a_1,\dots,a_d,\partial_q\alpha,
\alpha^{\otimes m-1}).
\end{aligned}
\end{equation}
In constructing additional operations, we will always proceed as follows:
\begin{itemize}
\itemsep.5em
\item Geometrically, we will use subsets of the space $\mathring{\frakR}_{d,m}$ (where we always have $m>0$) of discs with boundary and interior punctures (Section \ref{subsubsec:with-interior}). We also recycle the choices of ends and of additional data used in the definition of the deformed Fukaya category (Section \ref{section:deform-fuk}), which means that the compactifications of the moduli spaces of pseudo-holomorphic maps are inherited from there. An additional transversality condition will be imposed on those data, but that can be achieved as part of the initial generic choice.

\item Algebraically, the resulting operations will always be defined in the same fashion as the right hand side of \eqref{eq:partial-q-mu}, which means: inserting $\partial_q\alpha$ at $z_1$, and $\alpha$ at all other interior punctures; and then summing with weight $\frac{1}{(m-1)!}$, which one can think of as accounting for permutations of $(z_2,\dots,z_m)$.
\end{itemize}

Let's split $\mathring{\frakR}_{d,m}$ into halves, based on the position of $z_1$.
Say that our discs $S$ have $j+k+2$ boundary punctures, for some $j,k \geq 0$ ($d = j+k+1 \geq 1$). We fix $\bar{S} \iso \bD$ so that
\begin{equation} \label{eq:z-surface}
\zeta_0 = -1, \, \zeta_{j+1} = 1 \in \partial \bD.
\end{equation}
With this convention, define $\mathring{\frakR}^{\pm}_{j,1,k,m} \subset \mathring{\frakR}_{j+k+1,m}$ by requiring that (Figure \ref{fig:plusminus-space})
\begin{equation} \label{eq:z1-constraints}
\pm\!\mathrm{im}(z_1) \geq 0.
\end{equation}
\begin{figure}
\begin{centering}
\includegraphics{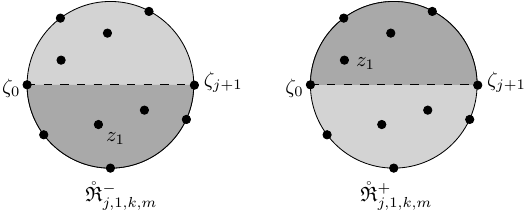}
\caption{\label{fig:plusminus-space}The position constraint from \eqref{eq:z1-constraints} which defines $\mathring{\frakR}^{\pm}_{j,1,k,m}$ (as shown, $j=3$, $k = 2$, $m = 4$).}
\end{centering}
\end{figure}%
Let's start by looking at the common boundary 
\begin{equation} \label{eq:common}
\mathring{\frakR}^0_{j,1,k,m} = \partial \mathring{\frakR}^+_{j,1,k,m} = \partial \mathring{\frakR}^-_{j,1,k,m} = \{\mathrm{im}(z_1) = 0\}
\end{equation}
(on that boundary, one could break the remaining symmetry by requiring that $z_1 = 0 \in \bD$). The additional transversality condition mentioned at the start of the discussion is precisely that the moduli space of pseudo-holomorphic maps parametrized by $\mathring{\frakR}_{j,1,k,m}^0$ should be regular (equivalently, if we take all of $\mathring{\frakR}_{d,m}$ as a parameter space, then the projection from the space of maps to the parameter space should be transverse to $\mathring{\frakR}^0_{j,1,k,m}$ for each of the finitely many $(j,k)$ that can occur). The operations obtained from the space of pseudo-holomorphic maps parametrized by $\mathring{\frakR}^0_{j,1,k,m}$, which we call
\begin{equation}
\mathit{GM}^{j,1,k}_0: \scrA_q(L_0,L_1,\dots,L_{j+k+1})
\longrightarrow \scrA_q(L_0,L_{j+k+1})[-j-k],
\end{equation}
form an $\scrA_q$-bimodule homomorphism (from the diagonal to itself). In the list of codimension $1$ degenerations from Figure \ref{fig:z-degeneration}, the top four are terms in the bimodule map equation; and the bottom two contribute zero (the first by the Maurer-Cartan equation, the second by \eqref{eq:ell1alpha}; this is another argument we've seen before, in Section \ref{section:connection}). 
\begin{figure}
\begin{centering}
\includegraphics{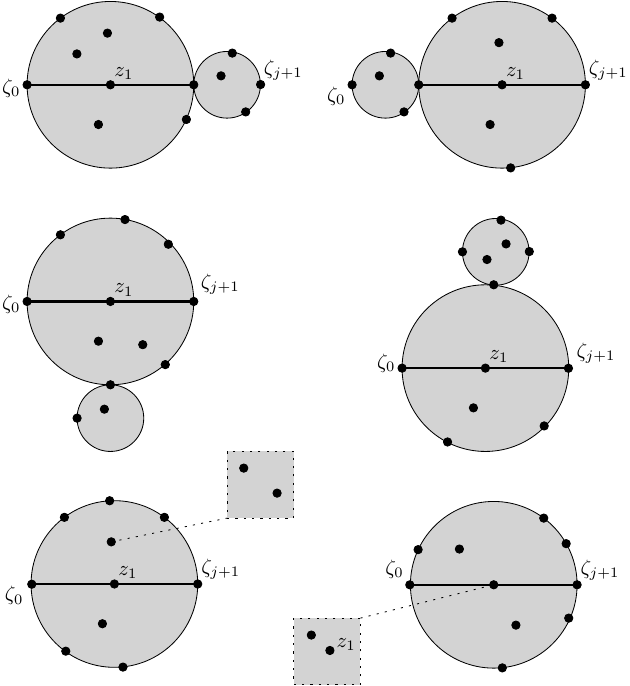}
\caption{\label{fig:z-degeneration}Codimension $1$ degenerations for $\frakR_{j,1,k,m}^0$.}
\end{centering}
\end{figure}

If we now pass to the spaces $\mathring{\frakR}^{\pm}_{j,1,k,m}$ themselves, the same kinds of degenerations as before appear; and there is an additional degeneration, where $z_1$ approaches the lower (for $\frakR^-$) or upper (for $\frakR^+$) half of $\partial \bD$. Finally, the parameter space itself has boundary \eqref{eq:common}, and that also contributes to the boundaries of one-dimensional moduli spaces. The outcome is that the associated operations
\begin{align}
& \mathit{GM}^{j,1,k}_\pm: \scrA_q(L_0,L_1,\dots,L_{j+k+1})
\longrightarrow \scrA_q(L_0,L_{j+k+1})[-j-k-1], \\
& \label{eq:both-xi}
\mathit{GM}_-^{j,1,k} + \mathit{GM}_+^{j,1,k} = \pm \partial_q\mu_q^{j+k+1},
\end{align}
satisfy the following equations, with respect to the differential $d_{(\scrA_q,\scrA_q)}$ in the dg category of $\scrA_q$-bimodules:
\begin{align} \label{eq:d-of-gm-minus}
&
\begin{aligned}
&
(d_{(\scrA_q,\scrA_q)}\mathit{GM}_-)^{j,1,k}(a_1,\dots,\underline{a_{j+1}},\dots,a_{j+k+1}) = 
\pm \mathit{GM}_0^{j,1,k}(a_1,\dots,a_{j+k+1}) \\ & \quad +
\sum_{i l} \pm \mu_q^{j+k-l+2}(a_1,\dots,a_i,
\partial_q\mu_q^l(a_{i+1},\dots,a_{i+l}),\dots,\underline{a_{j+1}},\dots,a_{j+k+1}),
\end{aligned}
\\ &
\begin{aligned}
&
(d_{(\scrA_q,\scrA_q)}\mathit{GM}_+)^{j,1,k}(a_1,\dots,\underline{a_{j+1}},\dots,a_{j+k+1}) = 
\mp \mathit{GM}_0^{j,1,k}(a_1,\dots,a_{j+k+1}) \\\ & \quad +
\sum_{i l} \pm \mu_q^{j+k-l+2}(a_1,\dots,\underline{a_{j+1}}, \dots
\partial_q\mu_q^l(a_{i+1},\dots,a_{i+l}),\dots,a_{j+k+1}).
\end{aligned}
\end{align}
\begin{figure}
\begin{centering}
\includegraphics{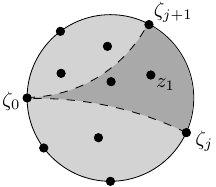}
\caption{\label{fig:wedge}The position constraint from \eqref{eq:wedge}, which defines $\mathring{\frakR}_{j,h,m}^\prec$ (as shown here, $j = 3$, $h = 2$, $m = 5$).}
\end{centering}
\end{figure}%
We use $\mathit{GM}_-$ to modify the Getzler-Gauss-Manin connection, in the general algebraic way indicated in \eqref{eq:modified-gauss-manin}--\eqref{eq:modified-gauss-manin-3}. The last-mentioned equation can be given a more geometric meaning, as follows. In \eqref{eq:z-surface} we could work with $\zeta_0$, $\zeta_{j+1}$ in arbitrary position, dividing by the entire automorphism group of the disc; then the role of the line $\{\mathrm{im}(z) = 0\}$ would be played by the hyperbolic geodesic connecting those two boundary points. In that perspective, the first term in \eqref{eq:modified-gauss-manin-3} concerns surfaces where $z_1$ lies in a region bounded by the geodesic connecting $\zeta_0$ to $\zeta_{j+1}$, whereas the second term is the same for the region connecting $\zeta_0$ to $\zeta_j$. Subtracting those two contributions from each other leads us to consider surfaces with the following modified condition:
\begin{equation} \label{eq:wedge}
\parbox{36em}{$z_1$ lies in the closed region of $S$ bounded by: the hyperbolic geodesic connecting $\zeta_0$ to $\zeta_j$; the hyperbolic geodesic connecting $\zeta_0$ to $\zeta_{j+1}$; and the part of $\partial S$ between $\zeta_j$ and $\zeta_{j+1}$ (Figure \ref{fig:wedge}).} 
\end{equation}
Write $\mathring{\frakR}_{j,h,m}^{\prec}$ for the space of such surfaces, where $h = d-j$. Following our usual principle, this gives rise to operations
\begin{equation}
\begin{aligned}
& \mathit{GM}_{\prec}^{j,h}: \scrA_q(L_0,L_1,\dots,L_{j+h})
\longrightarrow \scrA_q(L_0,L_{j+h})[-j-h], \\
& \mathit{GM}_{\prec}^{j,h} = \mathit{GM}_-^{j,1,h-1} - \mathit{GM}_-^{j-1,1,h}.
\end{aligned}
\end{equation}
As given, the definition only makes sense when $j,h>0$, but we extend it by setting (Figure \ref{fig:wedge-2})
\begin{equation} \label{eq:wedge-2}
\begin{aligned}
& \mathring{\frakR}^{\prec}_{0,h,m} = \mathring{\frakR}_{0,1,h-1,m}^- \;\; \Longrightarrow \;\;
\mathit{GM}_{\prec}^{0,h} = \mathit{GM}_-^{0,1,h-1} \text{ for $h>0$}, \\
& \mathring{\frakR}^{\prec}_{j,0,m} = \mathring{\frakR}_{j-1,1,0,m}^+ \;\; \Longrightarrow \;\;
\mathit{GM}_{\prec}^{j,0} = \mathit{GM}_+^{j-1,1,0} \text{ for $j>0$}, \\
& \mathring{\frakR}^{\prec}_{0,0,m} = \mathring{\frakR}_{0,m} \;\; \Longrightarrow
\mathit{GM}_{\prec}^{0,0} = \partial_q \mu_q^0.
\end{aligned}
\end{equation}
Combining this with the relation between $\mathit{GM}_{\pm}$ from \eqref{eq:both-xi}, the modified connection is:
\begin{equation} \label{eq:modified-gauss-manin-4}
\begin{aligned}
& 
\tilde{\nabla}_{u\partial_{q}}(a_0(a_1|\dots|a_l)) = u(\partial_q a_0)(a_1|\dots|a_l) + u
\sum_i a_0(a_1|\dots|\partial_q a_i|\dots|a_l) \\
& \qquad
+ \sum_{jk} \pm \mathit{GM}_0^{j,1,k}(a_{l-j+1},\dots,\underline{a_0},\dots,a_k)(a_{k+1}|\dots|a_{l-j}) \\
& \qquad
+ u \sum_{ijk} \pm e_{\scrA}^+ (a_{i+1}|\dots|\partial_q\mu_q^k(a_{j+1},\dots,a_{j+k})|\dots|a_0|\dots|a_i) \\
& \qquad
+ u\sum_{ijk} \pm e^+_{\scrA} (a_{i+1}|\dots|\mathit{GM}_-^{l-j,1,k}(a_{j+1},\dots,a_l,\underline{a_0},a_1,\dots,a_k)|\dots|a_i) \\
& \qquad \qquad \qquad \qquad \qquad \text{for } a_0(a_1|\dots|a_l) \in C_*(\scrA_q)[[u]],
\\ &
\tilde\nabla_{u\partial_{q}}(e_{\scrA}^+(a_1|\dots|a_l)) = u \sum_i e_{\scrA}^+(a_1|\dots|\partial_q a_i|\dots|a_l) 
\\ & \qquad + 
\sum_{jh} \pm \mathit{GM}_{\prec}^{j,h}(a_{l-j+1},\dots,a_l,a_1,\dots,a_h) (a_{j+1}|\dots|a_{l-j}) \\
& \qquad \qquad \qquad \qquad \qquad
\text{for } e_{\scrA}^+(a_1|\dots|a_l) \in e_\scrA^+ \otimes D_*(\scrA_q)[[u]].
\end{aligned}
\end{equation}
Here, the $D_*$ notation is as in \eqref{eq:bar-complex}; and the last line of our formula includes terms \eqref{eq:wedge-2}.
\begin{figure}
\begin{centering}
\includegraphics{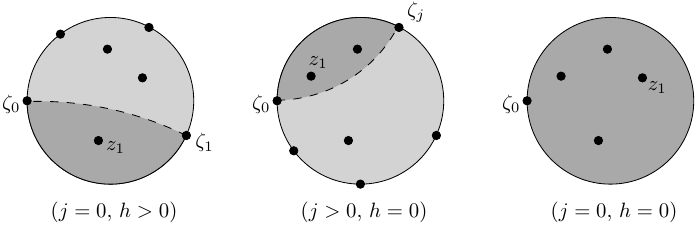}
\caption{\label{fig:wedge-2}The additional terms from \eqref{eq:wedge-2}.}
\end{centering}
\end{figure}

\subsubsection{The structure of the argument}
To prove compatibility of the cyclic open-closed map and connections, one needs to construct a suitable chain homotopy, which we call an intertwiner. This is a degree $n-1$ map, from the cyclic complex of the $q$-deformed Fukaya category, to the $q$-deformed $S^1$-equivariant Hamiltonian Floer complex, which satisfies the following equation:
\begin{equation}
\begin{aligned}
0 & = (\text{$q$-deformed equivariant Floer differential}) \circ (\text{intertwiner}) 
&& \text{(i)} \\ 
& + (\text{intertwiner}) \circ (\text{cyclic homology differential}) 
&& \text{(ii)} \\
& + u\partial_q(\text{cyclic open-closed map})
&& \text{(iii)} \\
& + (\text{closed string connection}-u\partial_q) \circ (\text{cyclic open-closed map}) 
&& \text{(iv)} \\
& - (\text{cyclic open-closed map}) \circ 
(\text{modified Getzler-Gauss-Manin}-u\partial_q).
&& \text{(v)}
\end{aligned}
\end{equation}
In (iii), what we are doing is taking the open-closed map, and differentiating all its coefficients with respect to $q$; as in \eqref{eq:partial-q-mu} this basically means replacing the Maurer-Cartan element at one interior puncture with its derivative $\partial_q\alpha$. In (iv) and (v), subtracting $u\partial_q$ simply means that we are using all the nontrivial terms in those connections, just ignoring the straightforward differentiation.
We have used text since that may be easier to understand, but for actually keeping track of all the terms, we need to replace this by symbols. Recall that the deformed $S^1$-equivariant Floer complex (Section \ref{sec:q-deformed}) is written as $(\mathit{CF}^*_{S^1,q}(H),\delta_{S^1,q})$. The equation for the intertwiner $\mathit{IT}: \mathit{CC}_*^+(\scrA_q) \rightarrow \mathit{CF}^*_{S^1,q}(H)$ is
\begin{equation} \label{eq:1-to-5}
\begin{aligned}
0 & = (-1)^n \delta_{S^1,q} \circ \mathit{IT}
\qquad \qquad \qquad \qquad \qquad \qquad \quad { } && \text{(i)} \\ 
& + \mathit{IT} \circ d_{\mathit{CC}_*^+(\scrA_q)} 
&& \text{(ii)} \\
& + u\partial_q(\mathit{OC}_{S^1,q})
&& \text{(iii)} \\
& + \mathit{KH} \circ \mathit{OC}_{S^1,q}
&& \text{(iv)} \\
& - \mathit{OC}_{S^1,q} \circ 
(\tilde{\nabla}_{u\partial_{q}}-u\partial_q).
&& \text{(v)}
\end{aligned}
\end{equation}
(As a consequence of previous notational conventions, certain operations carry a subscript $q$ while others don't; however, in fact all these structures belong to the $q$-deformed context.) 

Following the standard paradigm, the intertwiner is based on moduli spaces whose codimension $1$ degenerations correspond to the terms in the equation above. Recalling that the definition of the closed string connection involved two kinds of moduli spaces labeled (A) and (B), there will be corresponding versions for the intertwiner. Moreover, the open-closed map has two terms labeled (1) and (2), corresponding to the two parts of the cyclic complex, and the intertwiner will also follow the same pattern. This gives a total of four moduli spaces to be defined. The occurrence of boundary terms (i) and (iv) will be geometrically obvious; that of (iii) follows the same idea as in the construction of the closed string connection; and that of (ii) follows the same idea as in the construction of the cyclic open-closed map. Most of the thinking goes into (v), where we need to see how the rather ad-hoc-looking terms in \eqref{eq:modified-gauss-manin-4} arise from suitable degenerations. Because of its length, it makes sense to list here how the construction will be set up and organized:
\begin{itemize}
\itemsep.5em
\item Geometrically, we will use subsets of the spaces $\mathring{\AH}_{d,m,r}$ (Section \ref{subsubsection:angledecoratedhalf}), and the same ends and additional data as in the construction of the cyclic open-closed map (Section \ref{section:deformed-oc}). For this to work, the data are subject to a finite number of additional transversality conditions, which will no longer be stated explicitly.

\item For each parameter space, we will list the codimension $1$ boundary strata of the compactification (including ones that were already present as boundary faces of the uncompactified parameter space). In each case, we give an informal description of how points in the interior converge to that stratum, and then if necessary describe the structure of the stratum itself in terms of other moduli spaces (that discussion is not strictly in order; occasionally a space will appear as a boundary stratum before it's been formally introduced). For those boundary strata that contribute to the terms listed above, we will include (i)--(v) in the notation; those that don't contribute (because of cancellations) get Arabic numerals instead.

\item Algebraically, we treat those spaces as in Section \ref{subsubsec:half-strip}, but also sum over $r$ (the number of angle decorations) with powers $u^r$.
\end{itemize}

\subsubsection{Intertwining spaces (A1)\label{sec:A1}}
\begin{figure}
\begin{centering}
\includegraphics{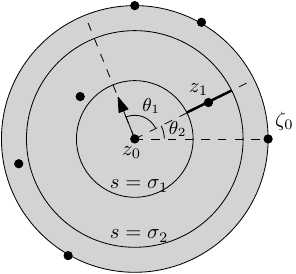}
\caption{A summary of the definition of $\mathring{\AH}_{d,m,r,w}^{(A1)}$ (as drawn, $d = 3$, $m = 3$, $r = 2$, $w = 1$). Due to the authors' limited drawing skills, the half-cylinder is shown as a disc with the output puncture $z_0$ in the center; this will be true of most images in this section.\label{fig:define-A1}}
\end{centering}
\end{figure}%
We take the space from Section \ref{section:cartan-A} and marry it to that of Section \ref{subsubsection:ocs1}. Fix $d \geq 0$, $m > 0$, $r \geq 0$, and $0 \leq w \leq r$. Consider half-cylinders \eqref{eq:half-cylinder} with angle-decorations \eqref{eq:rangledechalf}, and impose a version of \eqref{eq:s-moves} for $z_1$, as well as \eqref{eq:zeta0-is-0} for $\zeta_0$ (Figure \ref{fig:define-A1}):
\begin{equation} \label{eq:s-moves-2}
z_1 = (s_1,t_1 = \theta_{\geq w+1}), \;\; 
s_1 \in \begin{cases}
(-\infty,0) & r = 0, \\
(-\infty,\sigma_1] & w = 0,\, r>0, \\
[\sigma_w,\sigma_{w+1}] & w = 1,\dots,r-1, \\
[\sigma_r,0) & w = r,\,r>0,
\end{cases} \qquad
\zeta_0 = (0,0).
\end{equation}
The resulting parameter space is written as $\mathring{\AH}_{d,m,r,w}^{(A1)}$. The codimension $1$ strata of the compactification are as follows (see Figure \ref{fig:A1-boundary} for the simplest example):
\begin{itemize} 
\itemsep.5em
\item (A1.1) We can have $\sigma_i = \sigma_{i+1}$, for $i \neq w$.

\item (A1.2) $s_1$ can reach its extremal values: $z_1 = (\sigma_w,\theta_{\geq w+1})$ (for $w>0$) or $z_1 = (\sigma_{w+1},\theta_{\geq w+1})$ (for $w<r$).

\item (A1.3) Several interior punctures could collide (at an interior point of the half-cylinder) and form a Fulton-MacPherson bubble. Strictly speaking there are two sub-cases here, depending on whether $z_1$ ends up on the bubble or not.

\item (A1.i) A cylinder can bubble off at $-\infty$, but where $z_1$ remains in the original half-cylinder. The bubble can contain interior punctures as well as angle-decorations, so this stratum is the union of
\begin{equation}
\mathring{\AC}_{m_1,r_1} \times \mathring{\AH}_{d,m_2,r_2,w-r_1}^{(A1)}, \;\;
m_1+m_2 = m, \, r_1+r_2 = r.
\end{equation}

\item (A1.ii) While $z_1$ and all angle-decorations remain in the half-cylinder, several punctures (of both kinds) approach a boundary point and bubble off into a punctured disc. This stratum is a union of
\begin{equation}
\mathring{\AH}_{d_1,m_1,r,w}^{(A1)} \times \mathring{\frakR}_{d_2,m_2}, \;\; 
d_1+d_2 = d+1, \, m_1+m_2 = m.
\end{equation}

\item (A1.ii') (Only for $w<r$) Suppose that $\sigma_r \rightarrow 0$. The principle at work here has already appeared in the construction of the cyclic open-closed map, but we repeat it for convenience. Let $\theta_i^*$ be the limiting values of $\theta_i$ in our degeneration. Take the limiting half-cylinder and rotate it by $-\theta_r^*$, so that its asymptotic marker points in direction $\theta_1^* + \cdots + \theta_{r-1}^*$. Remove the pieces of parameter space where, after rotation, one of the boundary punctures in the half-cylinder lies at $(0,0)$. The rest can be thought of the disjoint union of components where $(0,0)$ lies between two specific boundary punctures; these are copies of $\mathring{\AH}^{(A2)}_{d,m,r-1,w}$ where the boundary punctures have been cyclically permuted.

\item (A1.iv) (Only for $w = 0$) $s_1$ can go to $-\infty$, which means that $z_1$ bubbles off into a cylinder. This stratum is the union of 
\begin{equation}
\mathring{\AC}_{m_1,r_1,0}^{(A)} \times \mathring{\AH}_{d,m_2,r_2}^{(1)}, \;\;
m_1+m_2 = m, \, r_1+r_2 = r.
\end{equation}

\item (A1.v) (Only for $w = r$, where $\theta_{\geq r+1} = 0$) $s_1$ can go to $0$, which means that $z_1$ approaches the boundary puncture $\zeta_0 = (0,0)$ along the $\{t = 0\}$ axis (perpendicularly to the boundary). The bubble is a punctured disc, where the position of the limit of $z_1$ on that disc is constrained to a specific hyperbolic geodesic. One gets the union of
\begin{equation}
\mathring{\AH}^{(1)}_{d_1,m_1,r} \times \mathring{\frakR}^0_{j,1,k,m_2}, \;\;
m_1 + m_2 = m, \, d_1 + j + k = d.
\end{equation}
\end{itemize}
\begin{figure}
\begin{centering}
\includegraphics{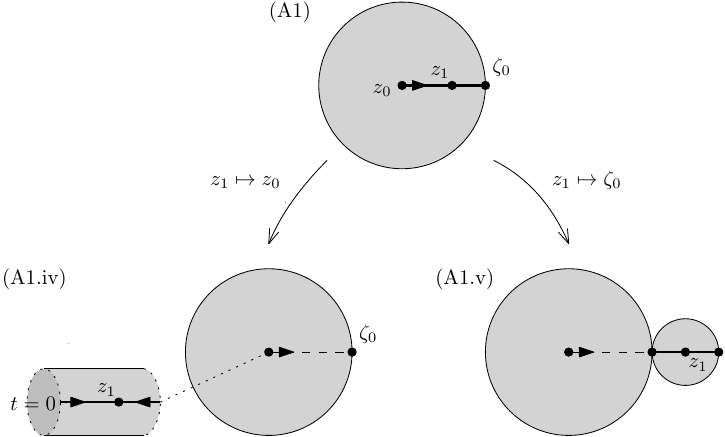}
\caption{For $\mathring{\AH}_{0,1,0,0}^{(A1)} \iso (0,1)$, only two degenerations can appear.\label{fig:A1-boundary}}
\end{centering}
\end{figure}%

\begin{remark} \label{th:A1-not}
The codimension of strata in the compactification depends on both the number of bubbles and on constraints in the parameter space. Careful consideration of that excludes contributions from many degenerations. For instance, consider the situation where $w = r-1$ and $s_1 \rightarrow 0$, which necessarily means that $\sigma_r \rightarrow 0$ as well. The interior puncture $z_1$ approaches the point $(0,\theta_r)$ and will bubble off into a disc. The remaining half-cylinder carries: the extra boundary puncture at $(0,\theta_r)$; angle decorations $(\sigma_1,\theta_1),\dots,(\sigma_{r-1},\theta_{r-1})$; and an asymptotic marker at $z_0$ rotated by $\theta_1 + \cdots + \theta_r$. Note that these three pieces of data are not independent, and therefore the limit has codimension $2$ (Figure \ref{fig:A1-not}). In a full compactification, this stratum belongs to the closure of codimension $1$ strata of type (A1.2) and (A1.ii').
\end{remark}
\begin{figure}
\begin{centering}
\includegraphics{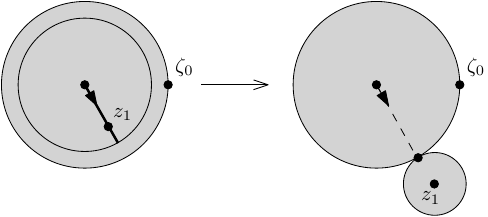}
\caption{The simplest example of the degeneration from Remark \ref{th:A1-not}. The dimension drops from $3$ (left) to $1$ (right).\label{fig:A1-not}}
\end{centering}
\end{figure}%

Algebraically, the outcome is a map $\mathit{IT}_{(A1)}: C_*(\scrA_q)[[u]] \rightarrow \mathit{CF}^{*+n-1}_{S^1,q}(H)$ which satisfies
\begin{equation} \label{eq:A1-operation}
\begin{aligned}
0 & = \text{(term from $s_1$ reaching its extremal values)} && \text{(A1.2)} \\
& \pm \delta_{S^1,q}\big(\mathit{IT}_{(A1)}(a_0(a_1|\dots|a_d))\big) && \text{(A1.i)} \\
& \pm \mathit{IT}_{(A1)}\big(d_{C_*(\scrA_q)}(a_0(a_1|\dots|a_d))\big) && \text{(A1.ii)} \\
& + u \sum_j \pm \mathit{IT}_{(A2)}(e^+_{\scrA}(a_j|\dots|a_{j-1})) && \text{(A1.ii')} \\
& \pm \mathit{KH}_{(A)}\big(\mathit{OC}_{S^1,q,(1)}(a_0(a_1|\dots|a_d))\big) &&
\text{(A1.iv)} \\
& + \sum_{ij} \pm \mathit{OC}_{S^1,q,(1)}\big(\mathit{GM}_0^{j,1,k}(a_{d-j+1},\dots,\underline{a_0},\dots,a_k)|a_{k+1}|\dots|a_{d-j}\big) && \text{(A1.v)}
\end{aligned}
\end{equation}
where $d_{C_*(\scrA_q)}$ is given by the formula \eqref{eq:unnormalized-c}. Here, we have tacitly added terms from breaking off of Floer cylinders to (A1.i) and (A1.ii). As usual, the contributions from (A1.3) are zero due to the Maurer-Cartan equation, and those from (A1.1) are zero because that space projects to a lower-dimensional one. We have not given a name to the (A1.2) term, because it will presently turn out to cancel with another term, just as in the definition of the closed string connection (Section \ref{section:connection}).

\subsubsection{Intertwining spaces (B1)}
\begin{figure}
\begin{centering}
\includegraphics{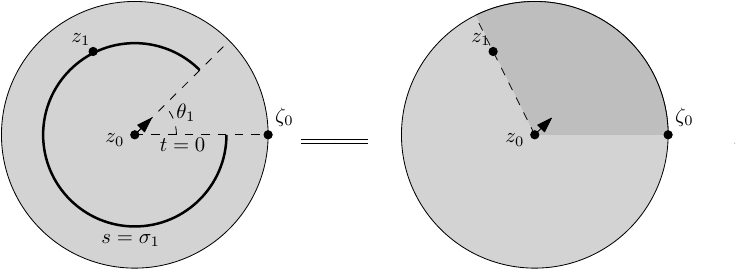}
\caption{Two ways of thinking of the condition \eqref{eq:t-moves-2}, in the simple case of $\mathring{\AH}_{0,1,1,1}^{(B1)}$. On the left, the marked point $z_1$ can lie anywhere on the thickened line. On the right, the asymptotic marker at $z_0$ can point into any direction in the darker shaded part of the disc.\label{fig:intertwining-B1}}
\end{centering}
\end{figure}%
We now apply the same idea to the space from Section \ref{section:cartan-B} (see Figure \ref{fig:intertwining-B1}). This means that the first part of \eqref{eq:s-moves-2} is replaced by \eqref{eq:t-moves}, while the second part remains the same. We summarize these conditions for convenience:
\begin{equation} \label{eq:t-moves-2}
\theta_w^{\operatorname{lift}} \in [0,1],\;\; \xi \in [\theta_w^{\operatorname{lift}},1], \quad
z_1 = (s_1 = \sigma_w,\, t_1 = \theta_{\geq w+1} + \xi),
\quad \zeta_0 = (0,0).
\end{equation}
This time, examination of the codimension $1$ degenerations in the resulting space $\mathring{\AH}_{d,m,r,w}^{(B1)}$ is more interesting, see (B1.v) and the even more surprising (B1.v') below. The full list is:
\begin{itemize}
\itemsep.5em
\item (B1.1) We can have $\sigma_i = \sigma_{i+1}$ for some $i$.

\item (B1.2) We can have $\xi = \theta_w^{\operatorname{lift}}$, which means $z_1 = (\sigma_w, \theta_{\geq w})$; or $\xi = 1$, which means $z_1 = (\sigma_w, \theta_{\geq w+1})$.

\item (B1.3) We can have bubbling at interior points as in (A1.3).

\item (B1.i), (B1.ii), (B1.ii') as in (A1.i), (A1.ii), (A1.ii'); in the last-mentioned case, the spaces that appear are $\mathring{\AH}_{d,m,r-1,w}^{(B2)}$.

\item (B1.iii) We can have $\theta^{\operatorname{lift}}_w = 0$, which means that $\xi \in [0,1]$ is arbitrary. By moving $\sigma_w$ around, one can then achieve that $z_1$ lies anywhere in $[\sigma_{w-1},\sigma_{w+1}] \times S^1$ (of course, $t_1 = \theta_{\geq w+1}$ can be achieved by either setting $\xi = 0$ or $\xi = 1$, but that is a higher codimension phenomenon and hence irrelevant). One can think of this as part of $\mathring{\AH}_{d,m,r-1}^{(1)}$, and by taking the union over all $w$ one gets all of that space (again, up to higher codimension differences). Looking slightly head, note that while $\mathring{\AH}_{d,m,r-1}^{(1)}$ is the parameter space underlying the term $\mathit{OC}_{S^1,q,(1)}$ of the open-closed map, here we are inserting $\partial_q\alpha$ at $z_1$, so the algebraic contribution will in fact be $\partial_q \mathit{OC}_{(1)}$.
%

\item (B1.iv) as in (A1.iv).

\item \parindent0em \parskip.5em
(B1.v) (This only applies to $w = r$, in which case $\theta_{\geq r+1} = 0$.) Suppose we have $\sigma_r \rightarrow 0$, and that $\xi$ also converges to some $\xi^* \in (0,1)$; more geometrically, $z_1$ approaches the boundary point $(0,\xi^*) \neq \zeta_0 = (0,0)$, and will bubble off into a punctured disc. Assume moreover that $\theta_r^{\operatorname{lift}}$ approaches some limit
\begin{equation}
\theta_r^{\operatorname{lift},*} \in (0,\xi^*). 
\end{equation}
Note that on the boundary of the half-cylinder, the points $(0,0)$, $(0,\theta_r^{\operatorname{lift},*})$, $(0,\xi^*)$ appear in that cyclic order.

In the limiting half-cylinder, $(0,\xi^*)$ becomes a boundary puncture. Between $(0,0)$ and that can lie other boundary punctures, but removing higher codimension pieces, we can assume that none of those punctures equals $(0,\theta_r^{\operatorname{lift},*})$; then the outcome can be subdivided into cases, depending on the boundary punctures that are on either side of the point $(0,\theta_r^{\operatorname{lift},*})$. Let's mentally rotate the half-cylinder by $-\theta_r^{\operatorname{lift},*}$, so that those punctures end up on either side of $(0,0)$, and cyclically permute the labels so that the two punctures now appear as the last and first one in our order. Altogether, this yields boundary faces of the form (Figure \ref{fig:B1-degenerations})
\begin{equation}
\mathring{\AH}_{d_1,m_1,r-1}^{(2)} \times \mathring{\frakR}_{d_2,m_2}, \;\; d_1+d_2 = d+1,
m_1+m_2 = m.
\end{equation}
Let's remind ourselves that one of the $m_2$ interior marked points on the disc component is the limit of $z_1$. Algebraically, this carries $\partial_q\alpha$, which means that the contribution from the disc component will be $\partial_q\mu_q$. Note that in the new ordering of boundary punctures on the half-cylinder, the puncture $(0,\xi^*-\theta_r^{\operatorname{lift},*})$ where the disc bubble is attached appears before $(0,-\theta_r^{\operatorname{lift},*})$. This is a consequence of the previous observation concerning cyclic order, and will effect an asymmetry in the algebraic contribution, where the $\partial_q\mu_q$ term appears before the $0$-th entry of the Hochschild chain.
\begin{figure}
\begin{centering}
\includegraphics{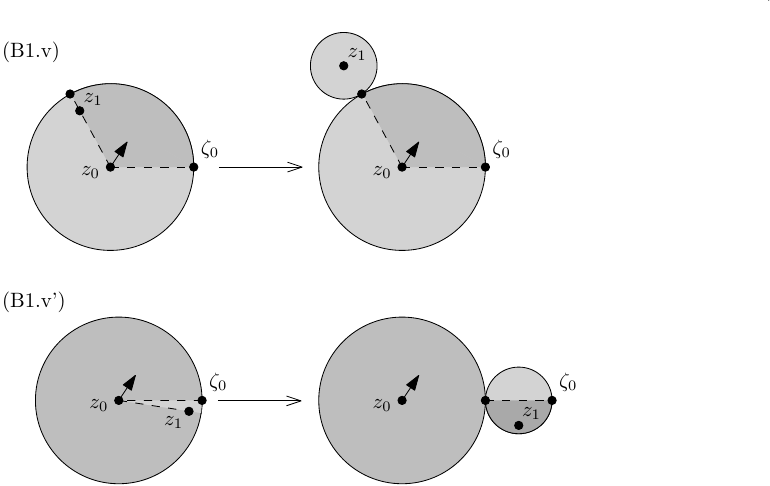}
\caption{The degenerations of type (B1.v) and (B1.v'), drawn in the same way as the right part of Figure \ref{fig:intertwining-B1}.\label{fig:B1-degenerations}}
\end{centering}
\end{figure}

\item \parskip0.5em \parindent0em
(B1.v') (This also only applies to $w = r$.) Suppose that $\sigma_r \rightarrow 0$ and $\xi \rightarrow 1$, but in such a way that $\sigma_r/(1-\xi)$ approaches a nonzero (and then necessarily negative) limit. This means that $z_1 \rightarrow (0,0)$, but the approach is from below the line $\{t = 0\}$ and at some nonzero slope. It is instructive to think of the resulting disc bubble as the left half-plane obtained by rescaling the situation near $(0,0)$. Then, the limit of $z_1$ lies in the left lower quarter-plane (Figure \ref{fig:B1-rescaling}). In more intrinsic terms, it lies below the hyperbolic geodesic connecting two specific boundary punctures in the bubble.
\begin{figure}
\begin{centering}
\includegraphics{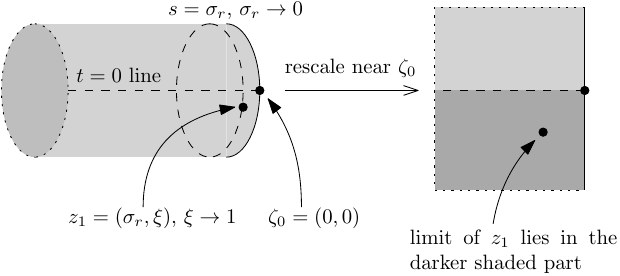}
\caption{An alternative picture of the bubbling from (B1.v'), showing it as a rescaling process.\label{fig:B1-rescaling}}
\end{centering}
\end{figure}%

Note that in this situation, $\theta_r^{\mathrm{lift}}$ can converge to an arbitrary limit, since any value is compatible with $\xi \rightarrow 1$. The half-cylinder component of the limit is then treated as in (A1.ii'). Ignoring higher codimension pieces, one can think of the outcome as copies of 
\begin{equation}
\mathring{\AH}_{d_1,m_1,r-1}^{(2)} \times \mathring{\frakR}^-_{j,1,k,m_2}, \;\;
j+k+d_1 = d, \; m_1+m_2 = m,
\end{equation}
summing over those cyclic permutations of the boundary punctures such that the limit of the $0$-th input remains on bubble ($\frakR^-$) component.
\end{itemize}

\begin{remark} \label{th:B1-not}
Again, there are other strata in which only one bubble component exists, but which are of higher codimension due to restrictions on the parameter. For instance, consider $w = r$ and the limit where $\sigma_r \rightarrow 0$ and $\xi \rightarrow 0$ (which necessarily means that $\theta_r^{\mathrm{lift}} \rightarrow 0$ as well). In that limit, the marked point $z_1$ approaches $(0,0)$ from above the axis $t = 0$, so we get a limit with a position constraint in the bubble, as in Figure \ref{fig:B1-not}. However, simple dimension-counting shows that this stratum is of codimension $2$.
\end{remark}
\begin{figure}
\begin{centering}
\includegraphics{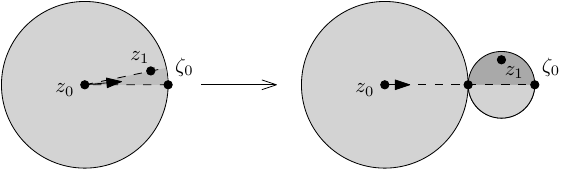}
\caption{\label{fig:B1-not}The limit from Remark \ref{th:B1-not}, in the simplest case of $\mathring{\AH}_{0,1,1,1}^{(B1)}$. We again follow the drawing conventions from the right part of Figure \ref{fig:intertwining-B1}.}
\end{centering}
\end{figure}%

The outcome is a map $\mathit{IT}_{(B1)}: C_*(\scrA_q)[[u]] \rightarrow \mathit{CF}^{*+n-1}_{S^1,_q}(H)$ which satisfies
\begin{equation} \label{eq:B1-operation}
\begin{aligned}
0 & = \text{(term from $\xi$ reaching its extremal values)} && \text{(B1.2)} \\
& \pm \delta_{S^1,q}\big(\mathit{IT}_{(B1)}(a_0(a_1|\dots|a_d))\big) && \text{(B1.i)} \\
& \pm \mathit{IT}_{(B1)}\big(d_{C_*(\scrA_q)}(a_0(a_1|\dots|a_d))\big) && \text{(B1.ii)} \\
& + u \sum_j \pm \mathit{IT}_{(B2)}(e^+_{\scrA}(a_j|\dots|a_{j-1})) && \text{(B1.ii')} \\
& \pm (u\partial_q \mathit{OC}_{S^1,q,(1)}) (a_0(a_1|\dots|a_d))
&& \text{(B1.iii)} \\
& \pm \mathit{KH}_{(B)}\big(\mathit{OC}_{S^1,q,(1)}(a_0(a_1|\dots|a_d))\big) &&
\text{(B1.iv)} \\
& + 
u\sum_{ijk} \pm \mathit{OC}_{S^1,q,(2)}\big(e_{\scrA}^+(a_{i+1}|\dots|
\\[-1em] & \qquad\qquad\qquad
\partial_q\mu_q^k(a_{j+1},\dots,a_{j+k})|\dots|a_0|\dots|a_i)\big) &&
\text{(B1.v)} \\ &
+ u\sum_{ijk} \pm \mathit{OC}_{S^1,q,(2)}\big( e^+_{\scrA} (a_{i+1}|\dots|
\\[-1em] & \qquad\qquad\qquad
\mathit{GM}_-^{d-j,1,k}(a_{j+1},\dots,\underline{a_0},\dots,a_k)|\dots|a_i) \big)
&& \text{(B1.v')}.
\end{aligned}
\end{equation}
The (A1.2) and (B1.2) terms in \eqref{eq:A1-operation} and \eqref{eq:B1-operation} concern the same surfaces. Hence, if we consider the sum $\mathit{IT}_{(A1)} + \mathit{IT}_{(B1)}$, they will cancel, in parallel with the construction of the closed string connection (Section \ref{section:connection}).

\subsubsection{Intertwining spaces (A2)}
\begin{figure}
\begin{centering}
\includegraphics{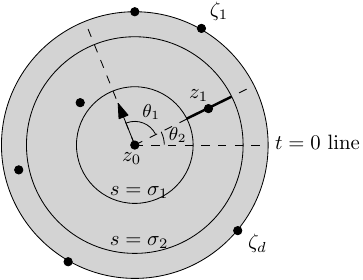}
\caption{The counterpart of Figure \ref{fig:define-A1} for $\mathring{\AH}_{d,m,r,w}^{(A2)}$. We remind the reader that in this context, $(0,0)$ is not a puncture.\label{fig:define-A2}}
\end{centering}
\end{figure}%
We now apply the same treatment as in Section \ref{sec:A1} to the second part of the cyclic open-closed map. The constraints (Figure \ref{fig:define-A2}) are:
\begin{equation} \label{eq:s-moves-3}
\begin{aligned}
& z_1 = (s_1,t_1 = \theta_{\geq w+1}), \quad
s_1 \in \begin{cases}
(-\infty,0) & r = 0, \\
(-\infty,\sigma_1] & w = 0,\, r>0, \\
[\sigma_w,\sigma_{w+1}] & w = 1,\dots,r-1, \\
[\sigma_r,0) & w = r,\,r>0,
\end{cases}
\\ &
(0,0) \text{ lies in the closed interval in $\{0\} \times S^1$ starting at $\zeta_d$ and ending at $\zeta_1$.}
\end{aligned}
\end{equation}
Here are the codimension $1$ degenerations of the resulting space $\mathring{\AH}^{(A2)}_{d,m,r,w}$:
\begin{itemize}
\itemsep.5em
\item (A2.1--3), (A2.i) as in (A1.1--3), (A1.i).

\item (A2.ii) While $z_1$ and all angle-decorations remain in the half-cylinder, several punctures (of both kinds) approach a boundary point and bubble off into a punctured disc; at most one of $(\zeta_1,\zeta_d)$ can belong to this group of colliding points. This stratum is a union of
\begin{equation}
\mathring{\AH}_{d_1,m_1,r,w}^{(A2)} \times \mathring{\frakR}_{d_2,m_2}, \;\; 
d_1+d_2 = d+1, \, m_1+m_2 = m.
\end{equation}

\item (A2.ii') The condition on $(\zeta_1,\zeta_d)$ in \eqref{eq:s-moves-3} achieves its extrema, which means that either $\zeta_1 = (0,0)$ or $\zeta_d = (0,0)$. In either case, we have arrived at \eqref{eq:s-moves-2} where the role of $\zeta_0$ is played by either $\zeta_1$ or $\zeta_d$, and with the remaining boundary punctures cyclically reordered; so this boundary stratum consists of two copies of $\mathring{\AH}_{d-1,m,r,w}^{(A1)}$; the shift in the $d$ subscript is by notational convention, see \eqref{eq:type-2}.

\item (A2.iv) as in (A1.iv).

\item (A2.v) (Only for $w = r$) we can have $s_1 \rightarrow 0$, which means $z_1 \rightarrow (0,0)$, but where both boundary punctures $\zeta_1$ and $\zeta_d$ stay away from the limiting point $(0,0)$. This creates a disc bubble with one boundary puncture, and with the limit of $z_1$ (as well as potentially other points) as interior puncture.
The outcome is
\begin{equation}
\mathring{\AH}_{d,m_1,r}^{(1)} \times \mathring{\frakR}_{0,m_2}, \;\;
m_1+m_2 = m, \, m_2 > 0.
\end{equation}
(For this and the following degenerations, see Figure \ref{fig:A2-degenerations.pdf}.)
\begin{figure}
\begin{centering}
\includegraphics{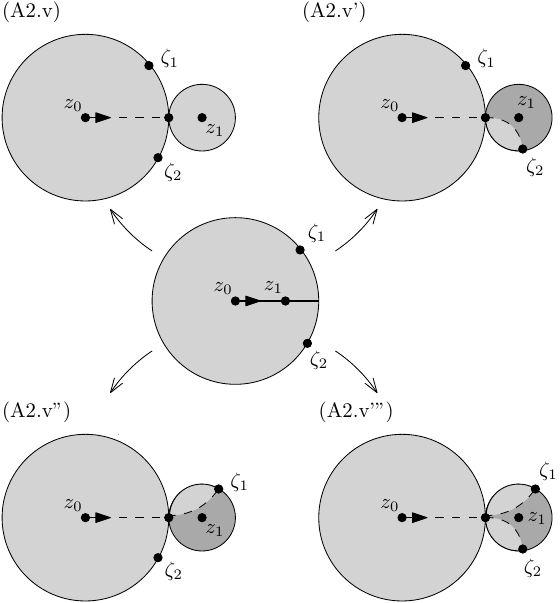}
\caption{\label{fig:A2-degenerations.pdf}Four related (A2) degenerations. In each case, the possible positions of $z_1$ on the bubble are in the darker shaded region.}
\end{centering}
\end{figure}%

\item (A2.v') (Also only for $w = r$) We have $s_1 \rightarrow 0$ as before, but this time 
some collection of points $(\zeta_{d-j+1},\dots,\zeta_d)$ also approaches $(0,0)$, in such a way that their limits as well as the limit of $z_1$ lie on the resulting bubble. As a consequence of the original condition on the positions of $z_1$ and $\zeta_d$, the points on the bubble inherit a constraint, which makes the stratum
\begin{equation}
\mathring{\AH}_{d-j,m_1,r}^{(1)} \times \mathring{\frakR}_{j-1,1,0,m_2}^+, \;\;
m_1+m_2 = m.
\end{equation}

\item (A2.v'') (Also only for $w = r$) Similar, but with $(\zeta_1,\dots,\zeta_h)$ which approach $(0,0)$. The resulting stratum has the form
\begin{equation}
\mathring{\AH}_{d-h,m_1,r}^{(1)} \times \mathring{\frakR}_{0,1,h-1,m_2}^-, \;\;
m_1+m_2 = m.
\end{equation}

\item (A2.v'{''}) (Also only for $w = r$) Finally, as $s_1 \rightarrow 0$, we can have $(\zeta_{d-j+1},\dots,\zeta_d,\zeta_1,\dots,\zeta_h)$ all approaching $(0,0)$, with $h,j>0$. This leads to boundary strata
\begin{equation}
\mathring{\AH}_{d-h-j,m_1,r}^{(1)} \times \mathring{\frakR}_{j,h,m_2}^{\prec}, \;\; m_1+m_2 = m.
\end{equation}
It is maybe better to think of the bubble as appearing through rescaling (Figure \ref{fig:A2-rescaling}). On the original half-cylinder, $z_1$ lies between two lines bounded by $\zeta_d$ and $\zeta_1$. After rescaling, we see the same phenomenon happening on the bubble half-plane, where the lines are now hyperbolic geodesics.
\end{itemize}
\begin{figure}
\begin{centering}
\includegraphics{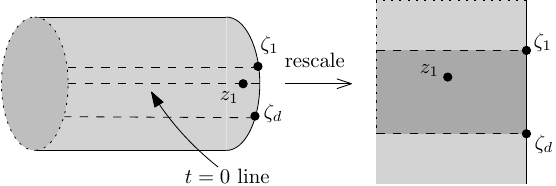}
\caption{\label{fig:A2-rescaling}An alternative picture of the (A2.v'{"}) degeneration.}
\end{centering}
\end{figure}%

\begin{remark}
In (A2.ii) we have not considered the case where $\zeta_1$ and $\zeta_d$ collide; such a collision can only occur at $(0,0)$, making the resulting stratum codimension $2$.
\end{remark}

The outcome is a map $\mathit{IT}_{(A2)}: e^+_{\scrA} \otimes B(\scrA_q)[[u]] \rightarrow \mathit{CF}^{*+n-2}_{S^1,q}(H)$ which satisfies
\begin{equation} \label{eq:A2-operation}
\begin{aligned}
0 & = \text{(term from $s_1$ reaching its extremal values)} && \text{(A2.2)} \\
& \pm \delta_{S^1,q}\big(\mathit{IT}_{(A2)}(e^+_{\scrA}(a_1|\dots|a_d))\big) && \text{(A2.i)} \\
& \pm \mathit{IT}_{(A2)}\big( e^+_{\scrA}\,d_{D_*(\scrA_q)}(a_1|\dots|a_d)\big) && \text{(A2.ii)} \\
& \pm \mathit{IT}_{(A1)}\big(a_1(a_2|\dots|a_d)\big) \pm \mathit{IT}_{(A1)}\big(a_d(a_1|\dots|a_{d-1})\big) && \text{(A2.ii')} \\
& \pm \mathit{KH}_{(A)}\big(\mathit{OC}_{S^1,q,(2)}(e^+_{\scrA}(a_1|\dots|a_d))\big)
&& \text{(A2.iv)} \\
& \pm \mathit{OC}_{S^1,q,(1)}\big(\mathit{GM}^{0,0}_{\prec}\,(a_1|\dots|a_d)\big) 
&& \text{(A2.v)} \\
& + \sum_j \pm \mathit{OC}_{S^1,q,(1)}\big(\mathit{GM}^{j,0}_{\prec}(a_{d-j+1},\dots,a_d)(a_1|\dots|a_{d-j})\big)
&& \text{(A2.v')} \\
& + \sum_h \pm \mathit{OC}_{S^1,q,(1)}\big(\mathit{GM}^{0,h}_{\prec}(a_1,\dots,a_h)(a_{h+1}|\dots|a_d)\big)
&& \text{(A2.v'')} \\
& + \sum_{j,h>0} \pm \mathit{OC}_{S^1,q,(1)}\big( \mathit{GM}^{j,h}_{\prec}(a_{d-j+1},\dots,a_d,a_1,\dots,a_h)(a_{h+1}|\dots|a_{d-j}) \big)
&& \text{(A2.v'{''})}
\end{aligned}
\end{equation}
where $d_{D_*(\scrA_q)}$ is given by the formula from \eqref{eq:dbar-differential} (applied to the deformed $A_\infty$-structure of $\scrA_q$, of course). The last four terms in \eqref{eq:A2-operation} appear algebraically the same, but that's because we have applied \eqref{eq:wedge-2} to slightly different geometric situations.

\subsubsection{Intertwining spaces (B2)\label{sec:B2}}
Our final space $\mathring{\AH}_{d,m,r,w}^{(B2)}$ imposes the following conditions on $z_1$ and the boundary punctures:
\begin{equation} \label{eq:t-moves-3}
\begin{aligned} &
\theta_w^{\operatorname{lift}} \in [0,1],\;\; \xi \in [\theta_w^{\operatorname{lift}},1], \quad
z_1 = (s_1 = \sigma_w,\, t_1 = \theta_{\geq w+1} + \xi),
\\ &
(0,0) \text{ lies in the closed interval in $\{0\} \times S^1$ starting at $\zeta_d$ and ending at $\zeta_1$.}
\end{aligned}
\end{equation}
There is nothing fundamentally new in the codimension $1$ degenerations:
\begin{itemize}
\item (B2.1-3), (B2.i) as in (B1.1-3), (B1.i).
\item (B2.ii-ii') as in (A2.ii-ii').
\item (B2.iii) as in (B1.iii).
\item (B2.iv) as in (A1.iv).
\end{itemize}

\begin{remark} \label{th:codim-2-again}
As in the construction of $\mathit{OC}_{(2)}$, the limit $s_1 \rightarrow 0$ does not contribute anything. The case not covered by previous considerations is $w = r$, where $s_1 \rightarrow 0$ means that $z_1$ approaches the boundary. However, this is again just a dimension count. Consider the simplest example $d = 2$, $m = 1$, $r = 1$. The space $\AH_{2,1,1,1}^{(B2)}$ has dimension $5$. (The degrees of freedom are: the positions of the two boundary marked points; the angle-decoration $(\sigma_1,\theta_1)$; and the position of $z_1$ on the circle $s = \sigma_1$. Note that we are not dividing by the $S^1$-action which rotates the half-cylinder, since that would break the condition that the asymptotic marker at $z_0 = -\infty$ is rotated by $\theta_1$ with respect to the $t = 0$ line.) The limit (Figure \ref{fig:intertwining-B2-not}) is a half-cylinder with three boundary punctures, together with a bubble component that is a disc with one boundary puncture and one interior puncture. At this point, we rotate the half-cylinder by $-\theta_1$ to bring the marker back into standard position, and then the outcome belongs to a $3$-dimensional parameter space. (Other limit configurations are possible, even in this simple situation, but the codimension argument still holds for those.)
\end{remark}
\begin{figure}
\begin{centering}
\includegraphics{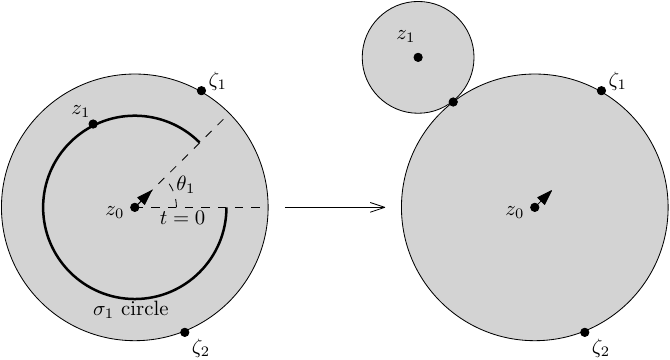}
\caption{\label{fig:intertwining-B2-not}The degeneration in $\mathring{\AH}_{2,1,1,1}$ discussed in Remark \ref{th:codim-2-again}.}
\end{centering}
\end{figure}%

Given that, the operation $\mathit{IT}_{(B2)} : e^+_{\scrA} \otimes B(\scrA_q)[[u]] \longrightarrow \mathit{CF}^{*+n-2}_{S^1,q}(H)$ satisfies
\begin{equation} \label{eq:B2-operation}
\begin{aligned}
0 & = \text{(term from $\xi$ reaching its extremal values)} && \text{(B2.2)} \\
& \pm \delta_{S^1,q}\big(\mathit{IT}_{(B2)}(e^+_{\scrA}(a_1|\dots|a_d))\big) && \text{(B2.i)} \\
& \pm \mathit{IT}_{(B2)}\big( e^+_{\scrA}\,d_{D_*(\scrA_q)}(a_1|\dots|a_d)\big) && \text{(B2.ii)} \\
& \pm \mathit{IT}_{(B1)}\big(a_1(a_2|\dots|a_d) \pm a_d(a_1|\dots|a_{d-1})\big) && \text{(B2.ii')} \\
& \pm (u\partial_q \mathit{OC}_{S^1,q,(2)})(e_{\scrA}^+(a_1|\dots|a_d)) &&
\text{(B2.iii)} \\
& \pm \mathit{KH}_{(B)}\big(\mathit{OC}_{S^1,q,(2)}(e_{\scrA}^+(a_1|\dots|a_d))\big)
&& \text{(B2.iv)} \\
\end{aligned}
\end{equation}

\subsubsection{Conclusion}
Write $\mathit{IT}_{(1)} = \mathit{IT}_{(A1)} \pm \mathit{IT}_{(B1)}$ and $\mathit{IT}_{(2)} = \mathit{IT}_{(A2)} \pm \mathit{IT}_{(B2)}$. By adding up \eqref{eq:A1-operation} and \eqref{eq:B1-operation} one sees that
\begin{equation} \label{eq:1-operation}
\begin{aligned}
0 & = \pm \delta_{S^1,q}(\mathit{IT}_{(1)}(a_0(a_1|\dots|a_d)))
&& \text{(A1+B1.i)} \\
& \pm \mathit{IT}_{(1)}(d_{C_*(\scrA_q)}(a_0(a_1|\dots|a_d))) 
&& \text{(A1+B1.ii)} \\
& + u \sum_j \pm \mathit{IT}_{(2)}(e^+_{\scrA}(a_j|\dots|a_{j-1})) 
&& \text{(A1+B1.ii')} \\
& \pm (u\partial_q\, \mathit{OC}_{S^1,q,(1)}) (a_0(a_1|\dots|a_d)) 
&& \text{(B1.iii)} \\
& \pm \mathit{KH}\big(\mathit{OC}_{S^1,q,(1)}(a_0(a_1|\dots|a_d))\big) 
&& \text{(A1+B1.iv)} \\
& + \sum_{ij} \pm \mathit{OC}_{S^1,q,(1)}(\mathit{GM}_0^{j,1,k}(a_{d-j+1},\dots,\underline{a_0},\dots,a_k)|
\\[-1em] & \qquad\qquad \qquad
a_{k+1}|\dots|a_{d-j}) 
&& \text{(A1.v)} \\
& + 
u\sum_{ijk} \pm \mathit{OC}_{S^1,q,(2)}\big(e_{\scrA}^+(a_{i+1}|\dots|
\\[-1em] & \qquad \qquad \qquad
\partial_q\mu_q^k(a_{j+1},\dots,a_{j+k})|\dots|a_0|\dots|a_i)
&& \text{(B1.v)} \\
& + u\sum_{ijk} \pm \mathit{OC}_{S^1,q,(2)}\big( e^+_{\scrA} (a_{i+1}|\dots|
\\[-1em] & \qquad \qquad \qquad
\mathit{GM}_-^{d-j,1,k}(a_{j+1},\dots,\underline{a_0},\dots,a_k)|\dots|a_i) \big)
&& \text{(B1.v')}
\end{aligned}
\end{equation}
Similarly, adding up \eqref{eq:A2-operation} and \eqref{eq:B2-operation} yields
\begin{equation} \label{eq:2-operation}
\begin{aligned}
0 & = \pm \delta_{S^1,q}\big(\mathit{IT}_{(2)}(e_+(a_1|\dots|a_d))\big)
&& \text{(A2+B2.i)}
\\
& \pm \mathit{IT}_{(2)}\big(e^+_{\scrA}\, d_{D_*(\scrA_q)}(a_1|\dots|a_d)\big)
&& \text{(A2+B2.ii)}
\\
& \pm \mathit{IT}_{(1)}\big(a_1(a_2|\dots|a_d)\big) \pm \mathit{IT}_{(1)}\big(a_d(a_2|\dots|a_{d-1})\big) 
&& \text{(A2+B2.ii')}
\\
& \pm (u\partial_q\,\mathit{OC}_{S^1,q,(2)})(e^+_{\scrA}(a_1|\dots|a_d))
&& \text{(B2.iii)}
\\ 
& \pm \mathit{KH}\big(\mathit{OC}_{S^1,q,(2)}(e_{\scrA}^+(a_1|\dots|a_d))\big)
&& \text{(A2+B2.iv)}
\\
& + \sum_{jh} \pm \mathit{OC}_{S^1,q,(1)}\big(\mathit{GM}_{\prec}^{j,h}(a_{d-j+1},\dots,a_d,a_1,\dots,a_h)(a_{h+1}|\dots|a_{d-j})\big) && \text{(A2.v-v'{''})}
\end{aligned}
\end{equation}
The equations \eqref{eq:1-operation}, \eqref{eq:2-operation} are exactly what one gets by spelling out \eqref{eq:1-to-5} using the formulae for the differential on the cyclic complex \eqref{eq:c-unital-cyclic} and the modified Getzler-Gauss-Manin connection \eqref{eq:modified-gauss-manin-4}. We have now shown that \eqref{eq:1-to-5} holds. Since $\tilde{\nabla}_{u\partial_{q}}$ is chain homotopic to $\nabla_{u\partial_{q}}$ by definition, the outcome is:

\begin{theorem}
For any Maurer-Cartan element \eqref{eq:floer-mc} and the associated deformation $\scrA_q$ of $\scrA$, the following diagram commutes up to chain homotopy:
\begin{equation}
\xymatrix{
\mathit{CC}_*^+(\scrA_q) \ar[rr]^-{\mathit{OC}_{S^1,q}} \ar[d]_-{\nabla_{u\partial_{q}}}
&&
\mathit{CF}^{*+n}_{S^1,q}(H) \ar[d]^-{\nabla_{u\partial_{q}}}
\\
\mathit{CC}_*^+(\scrA_q) \ar[rr]^-{\mathit{OC}_{S^1,q}}
&&
\mathit{CF}^{*+n}_{S^1,q}(H)
}
\end{equation}
\end{theorem}

\subsection{Sample sign considerations\label{subsec:open-string-signs}}

\subsubsection{The $A_\infty$-structure\label{subsubsec:signs-fukaya}}
To orient $\mathring{\frakR}_{d,m}$ for $d \geq 2$, we put the boundary marked points $\zeta_0,\zeta_1,\zeta_2$ in fixed position, and the use the boundary orientation for the remaining $(d-2)$ points $\zeta_3,\dots,\zeta_d$, as well as the complex orientation for the interior marked points $z_1,\dots,z_m$. Another formulation, for $m>0$, is to keep $\zeta_0$ and $z_1$ in fixed position, then use the boundary orientation of $\zeta_1,\dots,\zeta_d$ as well as the complex orientation of $z_2,\dots,z_m$. Those two choices coincide on their common domain of validity, and together cover all spaces $\mathring{\frakR}_{d,m}$. We follow the convention from \cite{seidel04} in which the orientation lines associated to the boundary marked points are put in decreasing order. This means that an isolated regular point of $\mathring{\frakR}_d(y_0,\dots,y_d)$ gives rise to an isomorphism
\begin{equation} \label{eq:reverse-y}
\frako_{y_d} \otimes \cdots \otimes \frako_{y_1} \stackrel{\iso}{\longrightarrow} \lambda^{\mathrm{top}}(T\mathring{\frakR}_{d,m}) \otimes \frako_{y_0}.
\end{equation}
When defining the $A_\infty$-operations $\mu^d(a_1,\dots,a_d)$, one takes the standard index-theoretic construction based on \eqref{eq:reverse-y}, and then adds an ad hoc sign $(-1)^\dag$,
\begin{equation} \label{eq:ad-hoc-sign-1}
\dag = |a_1| + 2|a_2| + \cdots + d|a_d|.
\end{equation}
The same applies to the deformed operations $\mu_q^d$ (note that in the deformation, the generators of $\mathit{CF}^*(H)$ that are inserted at the interior marked points are of even degree). 

%

Take the faces $\mathring{\frakR}_{j,m_2} \times \mathring{\frakR}_{d-j+1,m_1}  \subset \partial\frakR_{d,m}$ (in our usual branch-to-root ordering convention) responsible for the term 
\begin{equation} \label{eq:sign-term-1}
\mu^{d-j+1}_q(a_1,\dots,a_i,\mu^j_q(a_{i+1},\cdots,a_{i+j}),\dots,a_d)
\end{equation}
in the $A_\infty$-associativity equation \eqref{eq:associativity}. These boundary faces have orientation discrepancy $(-1)^\S$, 
\begin{equation} \label{eq:orientation-discrepancy-1}
\S = ij+i+j.
\end{equation}
Let's consider the effect on \eqref{eq:sign-term-1}. The ad hoc signs \eqref{eq:ad-hoc-sign-1} in the two $\mu$ operations contribute with the parity of, respectively,
\begin{align} \label{eq:ad-hoc-sign-1b}
& |a_{i+1}| + 2|a_{i+2}| + \cdots + j|a_{i+j}|,\\
& 
\begin{aligned} &
\big( |a_1| + 2|a_2| + \cdots + i|a_i| \big) 
\\ & \qquad 
+ (i+1) (|a_{i+1}| + \cdots + |a_{i+j}| + j) + \big( (i+2)|a_{i+j+1}| + \cdots + (d-j+1)|a_d| \big).
\end{aligned}
\label{eq:ad-hoc-sign-1c}
\end{align}
There is an additional Koszul sign, which comes from switching parameter spaces and inputs. Specifically, this occurs when moving $\lambda^{\mathrm{top}}T\mathring{\frakR}_{j,m_2}$, which underlies $\mu_q^j$, to the left of $a_d,\dots,a_{i+j+1}$ (remember the opposite ordering convention); hence the sign is given by the parity of 
\begin{equation} \label{eq:koszul-sign-1}
j(|a_{i+j+1}| + \cdots + |a_d|). 
\end{equation}
The sum of \eqref{eq:orientation-discrepancy-1}--\eqref{eq:koszul-sign-1} is, taken modulo $2$ of course,
\begin{equation} \label{eq:sum-sign-1}
\big(\dag + |a_1| + \cdots + |a_d|) + \big(\|a_1\| + \|a_2\| + \cdots + \|a_i\| \big);
\end{equation}
see \cite[Equation (12.25)]{seidel04}. The first part of \eqref{eq:sum-sign-1} is irrelevant, since it applies equally to all boundary faces; the second one gives the sign from \eqref{eq:associativity}.

\subsubsection{The open-closed map\label{subsubsec:signs-oc}}
The next step of our discussion requires some preliminary considerations. Recall that every generator $x$ of $\mathit{CF}^*(H)$ comes with a one-dimensional vector space $\frako_x$, which is defined \eqref{eq:orientation-line} as the determinant line of a Cauchy-Riemann operator $D_x$ on the thimble. Now suppose that instead we have an operator on a half-cylinder $(-\infty,0] \times S^1$, with the same asymptotic behaviour over the ends, and having totally real boundary conditions along $\{0\} \times S^1$. Additionally, we require that the boundary values carry a grading and {\em Spin} structure. This gives rise to an operator $D_x^{\mathrm{cyl}}$ and a well-defined index and determinant line, which we write as
\begin{align}
& \mathrm{deg}^{\mathrm{cyl}}(x) = \mathrm{index}(D_x^{\mathrm{cyl}}), \\
& \frako^{\mathrm{cyl}}(x) = \mathit{det}(D_x^{\mathrm{cyl}}).
\end{align}
To relate that to the previous notions, we can degenerate the half-cylinder to a nodal surface, whose components are a thimble and a disc. The operator on the disc can be deformed, in an essentially unique way, to a standard model operator (with a trivial vector bundle and constant boundary conditions). Such a model operator has index $n$, and its determinant line can be identified with the top exterior power of the totally real boundary condition at any point \cite[Lemma 11.12]{seidel04}; moreover, as part of the {\em Spin} structure, that top exterior power carries a canonical orientation. From this and index-theoretic gluing (with the convention that the disc part is placed on the left and the thimble part on the right), we get
\begin{align}
\label{eq:thimble-index}
& 
\mathrm{deg}^{\mathrm{cyl}}(x) = \mathrm{deg}(x) - n, \\
\label{eq:thimble-det}
& \frako_x^{\mathrm{cyl}} \iso \bR \otimes \frako_x.
\end{align}
In \eqref{eq:thimble-det} we mean canonical isomorphism up to a positive rescaling; the first factor on the right is trivial, but shifted downwards by $n$, which affects Koszul signs. When constructing the open-closed string map, it is natural to define the Floer complex using $D_x^{\mathrm{cyl}}$. The outcome, compared to the standard approach, is that the complex is shifted downwards by $n$, and the differential changed to $(-1)^n\delta$. The same will apply to the $q$-deformed differential.

The space $\mathring{\frakH}_{d,m}^{(1)}$ is oriented by using $(\zeta_1,\dots,\zeta_d)$ and then the complex orientation of $(z_1,\dots,z_m)$. When using these spaces to define $\mathit{CO}_q^d(a_0(a_1|\cdots|a_d))$, see \eqref{eq:ordinary-open-closed}, one uses \eqref{eq:ad-hoc-sign-1} as before, as well as an extra $(-1)^{\ddag}$ with
\begin{equation} \label{eq:ad-hoc-sign-2}
\ddag = \half(\|a_0\| + \cdots + \|a_d\|)(\|a_0\| + \cdots + \|a_d\|-1).
\end{equation}
We will now investigate the geometry underlying the fact that $\mathit{OC}_q$ is a chain map (of degree $n$), with respect to the Hochschild differential from \eqref{eq:unnormalized-c}.

\begin{itemize} \itemsep1em
\item
Take the boundary faces $\mathring{\frakR}_{j,m_2} \times \mathring{\frakH}_{d-j+1,m_1}^{(1)}$ responsible for the term
\begin{equation} 
\mathit{OC}_q^{d-j+1}\big(a_0(a_1|\cdots|\mu_q^j(a_{i+1},\dots,a_{i+j})|\cdots|a_d)\big),
\end{equation}
which occurs when spelling out \eqref{eq:oc-chain}. The orientation discrepancy is as in \eqref{eq:orientation-discrepancy-1}, and essentially the same computation can be carried out.
The outcome, taking \eqref{eq:ad-hoc-sign-2} into account, is the mod $2$ quantity
\begin{equation} \label{eq:hochschild-sign-1}
(\dag + \ddag + d) + (\|a_0\|  + \cdots + \|a_i\|).
\end{equation}

\item
There are other faces $\mathring{\frakR}_{d-j+i+1,m_2} \times \mathring{\frakH}_{j-i,m_1}^{(1)}$, responsible for terms in \eqref{eq:oc-chain} of the form
\begin{equation} \label{eq:rotated-term}
\mathit{OC}_q^{j-i}\big(\mu_q^{d-j+i+1}(a_{j+1},\dots,a_d,a_0,a_1,\dots,a_i) (a_{i+1}|\cdots|a_j)\big).
\end{equation}
Those have orientation discrepancy
\begin{equation} \label{eq:orientation-discrepancy-3}
\S = d(j+1) + 1.
\end{equation}
Consider the effect on \eqref{eq:rotated-term}. The ad hoc signs \eqref{eq:ad-hoc-sign-1}contribute, respectively, 
\begin{align}
\label{eq:ad-hoc-sign-3a}
& |a_{j+1}| + 2|a_{j+2}| + \cdots + (d-j)|a_d| + (d-j+1)|a_0| + \cdots + (d-j+i+1)|a_i|, 
\\
\label{eq:ad-hoc-sign-3b}
& 
|a_{i+1}| + 2|a_{i+2}| + \cdots + (j-i)|a_{i+j}|
\end{align}
There are two kinds of Koszul signs. The first of them comes from the cyclic permutation of inputs, and the second one has the same origin as in our previous arguments:
\begin{align} 
& \label{eq:koszul-sign-3a}
\big(|a_{j+1}| + \cdots + |a_d|\big)\big(|a_0| + \cdots + |a_j|\big), \\
& \label{eq:koszul-sign-3b}
(d-j+i+1)\big(|a_{i+1}| + \cdots + |a_j|\big).
\end{align}
The mod $2$ sum of the expressions \eqref{eq:orientation-discrepancy-3}--\eqref{eq:koszul-sign-3b}, plus the contribution from \eqref{eq:ad-hoc-sign-2}, yields
\begin{equation} \label{eq:hochschild-sign-2}
(\dag + \ddag + d) + \big(\|a_{j+1}\| + \cdots + \|a_d\|\big)\big(\|a_0\| + \cdots + \|a_j\|).
\end{equation}

\item 
The final type of boundary face is where a cylinder breaks off at the negative ($-\infty$) end. These faces are of the form $\mathring{\frakH}_{d,m_2}^{(1)} \times \mathring{\frakC}_{m_1}$, and are responsible for the term $\delta_q \circ \mathit{OC}_q$ in the formula. The orientation discrepancy here is simply $d+1$. It helps to look at the simplest case $m_1 = 1$, $d=0$, $m_2 = 0$. There, $\mathring{\frakH}_{1,0}^{(1)} \iso (-\infty,0) \times S^1$ just parametrizes the position of the unique interior marked point $z_1 = (s_1,t_1)$. A neighbourhood of the boundary face carries the orientation given by: as a first basis vector, the one that moves $s_1$ to the left (which, in the limit, yields the outwards point normal vector to the boundary); and as a second vector, increasing $t_1$ (which is the orientation of $\mathring{\frakC}_1$, see Section \ref{subsec:or-con}). 

Additionally we have the ad hoc signs \eqref{eq:ad-hoc-sign-1} and \eqref{eq:ad-hoc-sign-2}, which apply to the operation obtained from the $\mathring{\frakH}_{d,m_2}^{(1)}$ factor. 
Altogether, this means that the sign is given by
\begin{equation} \label{eq:hochschild-sign-3}
(\dag + \ddag + d) + 1.
\end{equation}
\end{itemize}
Comparing the formulae \eqref{eq:hochschild-sign-1}, \eqref{eq:hochschild-sign-2}, \eqref{eq:hochschild-sign-3} yields the desired equation
\begin{equation} \label{eq:oc-chain}
-(-1)^n \delta_q \circ \mathit{OC}_q + \mathit{OC}_q \circ d_{\bar{C}_*(\scrA_q)} = 0.
\end{equation}
%
%

\subsubsection{The modified Getzler-Gauss-Manin connection\label{subsubsec:signs-gm}}
The following discussion is much more algebraic than the previous ones. We refer to \cite{seidel08} for sign conventions concerning $A_\infty$-bimodules. In particular, the diagonal bimodule is related to the $A_\infty$-algebra structure by
\begin{equation} \label{diagonal-bimodule-sign}
\mu_{\mathrm{bimodule}\,\scrA_q}^{j,1,k}(a_1,\dots,a_j,\underline{a_{j+1}},\dots,a_{j+k+1}) = (-1)^{\triangle+1} \mu_{\mathrm{algebra}\, \scrA_q}^{j+1+k}(a_1,\dots,a_{j+k+1}),
\end{equation}
where
\begin{equation} \label{eq:triangle-sign}
\triangle = \|a_1\| + \cdots + \|a_j\|.
\end{equation}

We use the orientation of $\mathring{\frakR}^{\pm}_{j,1,k,m}$ as an open subspace of $\mathring{\frakR}_{j+k+1,m}$. For $\mathfrak{\frakR}^0_{j,1,k,m}$ we use the orientation as boundary of $\mathring{\frakR}^-_{j,1,k,m}$. When defining $\mathit{GM}_{\pm}^{j,1,k}(a_1,\dots,a_j,\underline{a_{j+1}},\dots,a_{j+k+1})$, we use the extra sign $(-1)^{\dag + \triangle}$, combining \eqref{eq:ad-hoc-sign-1} and \eqref{eq:triangle-sign}. As a consequence, \eqref{eq:both-xi} says that
\begin{equation} \label{eq:triangle-dmu}
\mathit{GM}_+^{j,1,k} + \mathit{GM}_-^{j,1,k} = (-1)^\triangle \partial_q \mu_q^{j+1+k}.
\end{equation}
The additional sign on the right hand side of \eqref{eq:triangle-dmu} makes sense algebraically, because the right hand side is a degree $-1$ morphism from the diagonal bimodule $\scrA_q$ to itself. The same signs are used for the $\mathit{GM}_0$ operation. Equation \eqref{eq:d-of-gm-minus} with signs is
\begin{equation} \label{eq:dgm-minus-signs}
\begin{aligned}
& (d_{(\scrA_q,\scrA_q)}\mathit{GM}_-)^{j,1,k}(a_1,\dots,\underline{a_{j+1}},\dots,a_{j+k+1}) 
\\ & 
= (-1)^{|a_1|+\cdots+|a_{j+k+1}|} \mathit{GM}_0^{j,1,k}(a_1,\dots,\underline{a_{j+1}},\dots,a_{j+k+1})
\\
& + \sum_{ir} (-1)^{\|a_{i+1}\|+\cdots+\|a_j\|} \mu^{j+k-r+2}_q(a_1,\dots,\partial_q\mu^r_q(a_{i+1},\dots,a_{i+r}),\dots,\underline{a_{j+1}},\dots, a_{j+k+1}).
\end{aligned}
\end{equation}
The orientation argument underlying this is the same as for \eqref{eq:associativity} (after all, we're using an open subset of the previous parameter space), except for the added signs from the two occurrences of \eqref{eq:triangle-sign} (in the definition of the diagonal bimodule and of $\mathit{GM}_-$ respectively $\mathit{GM}_0$).

\begin{example}
In the case $(j,k) = (1,0)$, \eqref{eq:dgm-minus-signs} decodes into the sum of the following terms being zero: 
\begin{equation} \label{eq:tedious}
\begin{aligned}
(a)\;\;\; & 
\mu_q^1 \mathit{GM}_-^{0,1,1}(a_1,\underline{a_2})
\\
(b) \;\;\; &
\mu_q^2(a_1,\mathit{GM}_-^{0,1,0}(\underline{a_2})) 
\\
(c)\;\;\; & 
(-1)^{\|a_1\|} \mathit{GM}_-^{0,1,1}(a_1,\mu_q^1(\underline{a_2})) 
\\
(d)\;\;\; & 
-\mathit{GM}_-^{1,1,0}(\mu_{\scrA_q}^1(a_1), \underline{a_2}) 
\\
(e)\;\;\; & 
(-1)^{\|a_1\|} \mathit{GM}_-^{0,1,0}(\mu_q^2(a_1,\underline{a_2})) \\
(f)\;\;\; & -\mathit{GM}_-^{2,1,0}(\mu^0_q,a_1,\underline{a_2}) \\
(g)\;\;\; & -(-1)^{\|a_1\|} \mathit{GM}_-^{2,1,0}(a_1,\mu^0_q, \underline{a_2}) \\
(h)\;\;\; & -(-1)^{\|a_1\|+|a_2|} \mathit{GM}_-^{1,1,1}(a_1,\underline{a_2},\mu^0_q) \\
(i)\;\;\; & +(-1)^{|a_1|+|a_2|} \mathit{GM}_0^{1,1,0}(a_1,\underline{a_2}) \\
(j)\;\;\; & + (-1)^{\|a_1\|} \mu^2_q( \partial_q\mu^1_q(a_1),\underline{a_2}) \\
(k)\;\;\; & +(-1)^{\|a_1\|} \mu^3_q(\partial_q\mu^0_q,a_1,\underline{a_2}) \\
(l)\;\;\; & + \mu^3_q(a_1,\partial_q\mu^0_q,\underline{a_2})
\end{aligned}
\end{equation}
Geometrically, the signs are obtained from the mod $2$ quantities below. The first column are the orientation discrepancies, as in \eqref{eq:orientation-discrepancy-1}; the second one are the ad hoc signs involved in defining $\mu_q$; the third one the ad hoc signs in $\mathit{GM}$; and the fourth one, Koszul signs as in \eqref{eq:koszul-sign-1} (empty entries mean that the sign doesn't occur).
\begin{equation} \label{eq:tedious-2}
\begin{array}{lllll}
& \text{orientations} & \mu_q & \mathit{GM} & \text{Koszul} \\
\hline
(a) \;\; & 0\;\;\;\; & |a_1|+|a_2| \;\;\;\; & 1 & 0
\\
(b) & 1 & |a_1| & |a_2| & 0 
\\
(c) & 1 & |a_2| & 1 & 0
\\
(d) & 1 & |a_1| & 1 & |a_2|
\\
(e) & 0 & |a_1| & |a_1|+|a_2| & 0 
\\
(f) & 0 & 0 &  |a_2|+|a_1|\;\;\;\; & 0
\\
(g) & 1 & 0 & |a_2| + 1 & 0
\\
(h) & 0 & 0 & 1 & 0
\\
(i) & 0 & & 1 &  
\\
(j) & 1 & 1 & & |a_2|   
\\
(k) & 0 & |a_2| & & 0
\\
(l) & 1 & |a_1|+|a_2| & & 0 
\end{array}
\end{equation}
In each line, taking all columns together yields the desired sign from \eqref{eq:tedious} times a fixed contribution $(-1)^{|a_1|+|a_2|+1}$.
\end{example}

We modify the Getzler-Gauss-Manin connection using 
\begin{equation}
\xi^{j,1,k}(a_1,\dots,\underline{a_{j+1}},\dots,a_{j+k+1}) = (-1)^{\triangle+1} \mathit{GM}_-^{j,1,k}(a_1,\dots,\underline{a_{j+1}},\dots,a_{k+1}), 
\end{equation}
as in \eqref{eq:modified-gauss-manin-2}. The last line of \eqref{eq:dgm-minus-signs} gives rise to the following terms in the modified connection:
\begin{equation}
\begin{aligned}
& (d_{\mathit{CC}_*^+(\scrA_q)}\Gamma(\xi) + \Gamma(\xi)d_{\mathit{CC}_*^+(\scrA_q)}(a_0(a_1|\cdots|a_l)) = 
\\ & \quad
\sum_{ijrs}  
(-1)^{(\|a_0\|+\cdots+\|a_j\|)(\|a_{j+1}\|+\cdots+\|a_l\|)
+ \|a_{j+1}\|+\cdots+\|a_i\|+1}
\mu^{l-j+s-r+2}_q(a_{j+1},\dots,a_i,
\\[-1em] & \qquad \qquad \qquad
\partial_q\mu_q^r(a_{i+1},\dots,a_{i+r}),\dots,\underline{a_0},\dots,a_s)(a_{s+1}|\cdots|a_j) 
\\ & \quad + \cdots 
\end{aligned}
\end{equation}
The signs here exactly cancel out the corresponding terms in the original definition of the Getzler-Gauss-Manin connection \eqref{eq:recap-gauss-manin}, which is the key ingredient that makes the construction of the modified connection work.

\section{The complement of a smooth anticanonical divisor}

Whereas before we worked with general Liouville manifolds, here we use specific geometric properties of the complement of a smooth anticanonical divisor, and their implications for Floer cohomology. We describe the overall algebraic argument first, since that brings together tools from all the different parts of the paper, leading to the proof of our main results (in Section \ref{subsubsec:proof}). After that, we go back to geometry for the necessary ingredients, collected from various parts of the literature. The most important geometric input is the description of deformed symplectic cohomology from \cite{pomerleano-seidel24}.


\subsection{The main argument\label{subsec:finiteness-results}}

\subsubsection{Symplectic cohomology\label{subsubsec:manipulate-h}}
Let $(M,D)$ be as in Assumption \ref{th:anticanonical-divisor}. Let $K_M^{-1} = \lambda^{\mathrm{top}}_{\bC}(TM)$ be the anticanonical line bundle, for some compatible almost complex structure. The assumption on $c_1(M)$ says that we can find a smooth section of $K_M^{-1}$ which vanishes exactly along $D$. From that, $M \setminus D$ inherits a trivialization of its anticanonical bundle. Next, because $\mathit{PD}([D]) = [\omega_M]$, one can turn the complement of a suitable tubular neighbourhood of $D$ into a Liouville domain $N$, and then complete that to $\hat{N}$ (see Section \ref{subsubsec:liouville2} for an exposition).

\begin{notation}
Write $C$ for the Hamiltonian Floer complex on $\hat{N}$, and $H$ for its cohomology. (In the notation of Section \ref{subsubsec:gradings-and-orientations}, $C = \mathit{CF}^*(H)$; and in standard symplectic topology notation, see e.g.\ \cite{seidel-biased}, $H = \mathit{SH}^*(\hat{N})$.) 
\end{notation}

We now begin to build our understanding of symplectic cohomology and related theories, starting with general structural properties (boundedness and finite generation results).

\begin{lemma} \label{th:boundedness} (see Section \ref{subsubsection:nonnegativity})
By a suitable choice of Hamiltonian, one can assume that $C$ is concentrated in degrees $[0,\operatorname{dim}_\mathbb{R}(N)-1]$.
\end{lemma}

For us it will only be important that $C$ is concentrated in degrees $\geq 0$, and that its cohomology is bounded above. Recall that $H$ is a graded commutative ring. From the compactification $M$, we inherit a distinguished class in $H^0$, called here the Borman-Sheridan class $b$. 

\begin{lemma} \label{th:h-finite} (see Section \ref{subsubsec:bs})
$H$ is a finitely generated $\bK[t]$-module, where $t$ acts by multiplication with $b$.
\end{lemma}

In view of Lemma \ref{th:boundedness}, $b$ has a unique cocycle representative, denoted by $\beta \in C^0$. The element
\begin{equation} \label{eq:gamma-element}
\alpha = q\beta \in qC[q]
\end{equation}
is automatically Maurer-Cartan for the $L_\infty$-structure on $C[1]$ from Section \ref{section:Linf}, since all the nonlinear terms in \eqref{eq:lie-mc} vanish for degree reasons: if $m \geq 2$, $\ell^m(\alpha,\dots,\alpha) = q^m \ell^m(\beta,\dots,\beta) \in q^m C^{3-2m} = 0$.

\begin{notation}
Write $C_q$ for the space $C[q]$ equipped with the differential deformed by $\alpha$ (in the notation of Definition \ref{th:deformed-structures}(ii), $C_q = \mathit{CF}^*_q(H)$; note there's no distinction between polynomials and power series in $q$, because $C$ is bounded below). Write $H_q$ for its cohomology. $H_q$ has the structure of a $\bK[q,t]$-module, where the degree $0$ variable $t$ acts as $\iota_q$, as defined in \eqref{eq:iotacf-2}. 
\end{notation}

The map $H_q \rightarrow H$ obtained by setting $q = 0$ is $t$-linear, essentially by definition. Namely, the only term in \eqref{eq:iotacf-2} not containing any $q$ is the one with $m = 0$, and $\mathit{KH}^{(A)}_{1,0}(\partial_q\alpha,\cdot) = \mathit{KH}^{(A)}_{1,0}(\beta,\cdot)$ is the pair-of-pants product with the Borman-Sheridan class; see Example \ref{th:pair-of-pants} and \eqref{eq:iota-diagram}.

\begin{lemma} \label{th:hq-finite}
$H_q$ is a finitely generated $\bK[q,t]$-module.
\end{lemma}

\begin{proof}
The long exact sequence
\begin{equation}
\cdots \rightarrow
H_q^{*-2} \stackrel{q}{\longrightarrow}
H_q \stackrel{q = 0}{\longrightarrow}
H \rightarrow \cdots
\end{equation}
shows that $H_q/qH_q$ injects into $H$. Hence, by Lemma \ref{th:h-finite}, $H_q/qH_q$ is finitely generated over the Noetherian ring $\bK[t]$. Pick generators $\bar{h}_1,\dots,\bar{h}_k \in H_q/qH_q$, and lifts $h_1,\dots,h_k$ of these generators to $H_q$. Without loss of generality, we can assume that those are homogeneous (each lies in a fixed degree). Given some $x \in H_q^d$ with image $\bar{x} \in H_q^d/qH_q^{d-2}$, one can write
\begin{equation}
\bar{x} = \sum_{\substack{j \in \{1,\dots,k\} \\ \text{such that } |\bar{h}_j| = d}} f_j(t)\bar{h}_j, \quad \text{with $f_j(t) \in \bK[t]$.}
\end{equation}
This means that
\begin{equation} \label{eq:tilde-x-difference}
x -\!\!\!\!\!\!\!\!\! \sum_{\substack{j \in \{1,\dots,k\} \\ \text{such that } |h_j| = d}} f_j(t)h_j(t) \in qH_q^{d-2}.
\end{equation}
We can write the difference \eqref{eq:tilde-x-difference} as $q\hat{x}$, apply the same argument to $\hat{x}$, and then iterate to get an expression for $x$ as a $\bK[q,t]$-linear combination of the $h_j$ (because $H_q$ is bounded below, the process terminates after finitely many iterations).
\end{proof}

The argument above, and its cousins yet to follow, could be formulated more concisely in spectral sequence terms; but we prefer to avoid that language, in order to make (the absence of) convergence issues as clear as possible. 

\begin{lemma} \label{th:weak-1} (see Section \ref{subsubsec:deformed})
$\bK[q^{\pm 1}] \otimes_{\bK[q]} H_q$ is a finitely generated $\bK[q^{\pm 1}]$-module (equivalently, it is of finite dimension over $\bK$ in each degree).
\end{lemma}


\begin{lemma} \label{th:hq-support}
There is a $0 \neq p(t) \in \bK[t]$ and an $r \in \bN$ such that $p(t)q^r$ acts trivially on $H_q$.
\end{lemma}

\begin{proof}
Because $t$ is an endomorphism of the finite-dimensional $\bK$-vector space $H_{q,q^{-1}}^0 \oplus H_{q,q^{-1}}^1$, there must be some $p$ such that $p(t)$ acts trivially there. By multiplication with powers of $q$, one sees that $p(t)$ acts trivially on all of $H_{q,q^{-1}}$. Let $h_1,\dots,h_k$ be generators of $H_q$ over $\bK[q,t]$ (Lemma \ref{th:hq-finite}). The previous argument shows that $p(t)h_j$ is $q$-torsion for each $j$. Therefore, there is some $r$ such that $q^r p(t) h_j = 0$ for all $j$. 
\end{proof}

\begin{lemma} \label{th:generic-t}
$\bK(t) \otimes_{\bK[t]} H_q$ is finite-dimensional over $\bK(t)$ in each degree, and vanishes in sufficiently high degrees.
\end{lemma}

\begin{proof}
The finite generation statement from Lemma \ref{th:hq-finite} clearly implies that in each degree, $H_q$ is a finitely generated $\bK[t]$-module, which is stronger than the corresponding $\bK(t)$ statement. For the second part, let's use Lemma \ref{th:hq-finite} more explicitly, choosing homogeneous generators $h_1,\dots,h_k$ of $H_q$ over $\bK[t,q]$. Given any $x \in H_q^d$, one can write 
\begin{equation} \label{eq:geqzeroandeven}
x = \!\!\!\!\!\! \sum_{\substack{j \in \{1,\dots,k\} \\ \text{such that } d-|h_j| \\ \text{is $\geq 0$ and even}}} q^{(d-|h_j|)/2} f_j(t) h_j
\quad \text{with $f_j(t) \in \bK[t]$}.
\end{equation}
As $d$ gets larger, it follows that $x$ is divisible by higher and higher powers of $q$. For sufficiently large $d$, Lemma \ref{th:hq-support} then shows that $p(t) x = 0$, which means that $x$ becomes zero in $\bK(t) \otimes_{\bK[t]} H_q$.
\end{proof}

\begin{notation} \label{th:c-q-u}
Write $C_u$ for the $S^1$-equivariant Floer complex, which means that it's $C[u]$ with a $u$-deformation of the previous Floer differential ($C_u = \mathit{CF}^*_{S^1}(H)$ in the notation of Section \ref{subsubsec:s1-equivariant}, where this is constructed; again, the difference between polynomials and power series is irrelevant for grading reasons). Let $H_u$ be its cohomology ($H_u = \mathit{SH}^*_{S^1}(\hat{N})$ in standard symplectic topology notation).

As before, there is a deformation induced by \eqref{eq:gamma-element}, denoted by $C_{q,u}$, with underlying $\bK[q,u]$-module $C[q,u]$ (this is $C_{q,u} = \mathit{CF}^*_{S^1,q}(H)$ from Section \ref{sec:q-deformed}). We write $H_{q,u}$ for the cohomology of $C_{q,u}$. The action of $u$, $q$, and the closed string connection $t = \nabla_{u\partial_q}$ (constructed in Section \ref{section:connection}) make $C_{q,u}$ into a complete and $u$-torsionfree $W_{q,u}$-module, in the sense of Section \ref{subsec:u-theory}.
\end{notation}

The map $H_{q,u} \rightarrow H_q$ obtained by setting $u = 0$ relates the $W_{q,u}$-module structure on $H_{q,u}$ with the $\bK[q,t]$-module structure on $H_q$, by \eqref{eq:iota-diagram}.

\begin{lemma} \label{th:hqu-finite}
$H_{q,u}$ is finitely generated over $W_{q,u}$.
\end{lemma}

\begin{proof} 
The long exact sequence
\begin{equation}
\cdots \rightarrow H_{q,u}[-2] \stackrel{u}{\longrightarrow} H_{q,u} \stackrel{u=0}{\longrightarrow} H_q \rightarrow \cdots
\end{equation}
shows that $H_{q,u}/uH_{q,u}$ injects into $H_q$, hence is finitely generated over $\bK[q,t]$ by Lemma \ref{th:hq-finite}. Pick homogeneous generators $\bar{h}_1,\dots,\bar{h}_k \in H_{q,u}/uH_{q,u}$, and lifts $h_1,\dots,h_k$ to $H_{q,u}$. Given some $x \in H_{q,u}^d$ with image $\bar{x} \in H_{q,u}^d/uH_{q,u}^{d-2}$, one proceeds as in \eqref{eq:geqzeroandeven}:
\begin{equation}
\bar{x} = \!\!\!\!\!\!\!\sum_{\substack{j \in \{1,\dots,k\} \\ \text{such that } d-|h_j| \\ \text{is $\geq 0$ and even}}} \!\!\!\!\!\!\! q^{(d-|h_j|)/2} \bar{f}_j(t) \bar{h}_j, \quad \text{with $\bar{f}_j(q,t) \in \bK[t]$.}
\end{equation}
Take $f_j = q^{(d-|h_j|)/2} \bar{f}_j(t) \in W_{q,u}$ (one could also choose any other element $f_j$ of $W_{q,u}$ of the same degree and with the same $u = 0$ reduction). Then
\begin{equation}
x - \!\!\!\!\!\!\!\sum_{\substack{j \in \{1,\dots,k\} \\ \text{such that } d-|h_j| \\ \text{is $\geq 0$ and even}}} \!\!\!\!\!\!\! f_j(t) h_j \in uH_{q,u}^{d-2}.
\end{equation}
Writing this as $u\hat{x}$, one then iterates as in the proof of Lemma \ref{th:hq-finite}.
\end{proof}

The next two statements yield an analogue (in the $W_{q,u}$-context) of the classical idea that finitely generated modules over the Weyl algebra are holonomic if and only if they have one-dimensional singular support.

\begin{lemma} \label{th:bounded-dimension}
The dimension of $\bK(t) \otimes_{\bK[t]} H_{q,u}^i$, as a $\bK(t)$-vector space, is uniformly bounded for all degrees $i$.
\end{lemma}

\begin{proof}
The long exact sequence
\begin{equation}
\cdots \rightarrow \bK(t) \otimes_{\bK[t]} H_{q,u}^{i-2} \stackrel{u}{\longrightarrow} \bK(t) \otimes_{\bK[t]} H_{q,u}^i 
\stackrel{u=0}{\longrightarrow} \bK(t) \otimes_{\bK[t]} H_q^i \rightarrow \cdots.
\end{equation}
implies that 
\begin{equation}
\mathrm{dim}_{\bK(t)}( \bK(t) \otimes_{\bK[t]} H_{q,u}^i )
\leq \mathrm{dim}_{\bK(t)} (\bK(t) \otimes_{\bK[t]} H_{q,u}^{i-2} ) +
\mathrm{dim}_{\bK(t)} (\bK(t) \otimes_{\bK[t]} H_q^i).
\end{equation}
Because everything is bounded below, and because of the finite-dimensionality statement from Lemma \ref{th:generic-t}, it follows by induction on $i$ that $\bK(t) \otimes_{\bK[t]} H_{q,u}^i$ is finite-dimensional. From the same long exact sequence and the other part of Lemma \ref{th:generic-t}, we see that $\bK(t) \otimes_{\bK[t]} H_{q,u}^i$ eventually becomes $2$-periodic in $i$.
\end{proof}

\begin{lemma} \label{th:u-holonomic}
For every homogeneous $x \in H_{q,u}$ there is a nonzero homogeneous $w \in W_{q,u}$ such that $wx = 0$. 
\end{lemma}

\begin{proof}
Take some $m$ and consider the $(m+1)$ elements $q^m x$, $q^{m-1}u x$, \dots, $u^m x$, all of which have degree $2m$ more than $x$. From Lemma \ref{th:bounded-dimension}, we know that if $m$ is sufficiently large, there must be a relation
\begin{equation} \label{eq:localizedintermediateeq}
(f_0(t) q^m + f_1(t) q^{m-1} u + \cdots + f_m(t) u^m) x = 0 \in \bK(t) \otimes_{\bK[t]} H^{|x|+2m}_{q,u},
\end{equation}
where $f_j(t) \in \bK(t)$ are not all zero. After clearing denominators in \eqref{eq:localizedintermediateeq} (i.e. multiplying by a suitable polynomial $g(t)$), one gets a relation in $H^{|x|+2m}_{q,u}$. By setting $t = u\partial_q$ as in \eqref{eq:Fourier-Laplacenc}, this relation in $H^{|x|+2m}_{q,u}$ can be interpreted as a formula for a nonzero element of $W_{q,u}^{2m}$ which annihilates $x$.
\end{proof}

\begin{notation} \label{th:bar-q-module}
In each degree, $\bK[u^{\pm 1}] \otimes_{\bK[u]} H_{q,u}$ becomes a module over the classical Weyl algebra $W_{\bar{q}}$ in the variable $\bar{q} = q/u$, and where $\partial_{\bar{q}}$ acts by $t = \nabla_{u\partial_q}$.
\end{notation}

\begin{lemma} \label{th:holonomic}
In each degree, $\bK[u^{\pm 1}] \otimes_{\bK[u]} H_{q,u}$ is a holonomic $D$-module in the classical sense.
\end{lemma}

\begin{proof}
Lemma \ref{th:hqu-finite} implies that $\bK[u^{\pm 1}] \otimes_{\bK[u]} H_{q,u}$ is finitely generated over $\bK[u^{\pm 1}] \otimes_{\bK[u]} W_{q,u}$. Take generators for the even degree part, assumed to be homogeneous without loss of generality, and multiply each by a suitable power of $u$ so that they all lie in degree $0$. Note that $W_{\bar{q}}$ can be identified with the degree $0$ part of $\bK[u^{\pm 1}] \otimes_{\bK[u]} W_{q,u}$. Hence, it follows that our generators will generate the degree $0$ part of $\bK[u^{\pm 1}] \otimes_{\bK[u]} H_{q,u}$ over $W_{\bar{q}}$. Similarly, we know from Lemma \ref{th:u-holonomic} that for every degree $0$ element $x \in \bK[u,u^{-1}] \otimes_{\bK[u]} H_{q,u}$, there is a homogeneous $w \in \bK[u,u^{-1}] \otimes W_{q,u}$ such that $wx = 0$. After multiplying by the appropriate power of $u$, one can achieve that $w$ has degree $0$, which shows the required properties in the case of degree $0$. By multiplying with powers of $u$, it follows that the same holds in any even degree. The argument in odd degrees is parallel.
\end{proof}

We will now introduce some modified versions of $C_{q,u}$, following the general algebraic formalism from Section \ref{subsubsec:uq-weyl}. First of all, one can invert $q$, as in \eqref{eq:invert-q} and \eqref{eq:quotient-q}.

\begin{notation} \label{th:modify-c}
Starting with $C_{q,u}$ as in Notation \ref{th:c-q-u}, consider
\begin{align} 
\label{eq:q-plusminus}
& C_{q^{\pm 1},u} = \bK[q^{\pm 1}] \hat\otimes_{\bK[q]} C_{q,u}, \\
& \label{eq:cross-into-closed}
q^{-1}C_{q^{-1},u} = q^{-1}\bK[q^{-1}] \hat\otimes_{\bK[q]} C_{q,u},
\end{align}
where $\hat\otimes$ denotes $u$-adic completion (spelled out in Example \ref{th:example-modules}).
\end{notation}

These groups are not yet our actual target: in a subsequent step we invert a polynomial in $t = \nabla_{u\partial_q}$, following \eqref{eq:quotient-t}.

\begin{notation}
For nonzero $p(t) \in \bK[t]$, define
\begin{align}
\label{eq:invert-p-1}
& C_{q,1/p,u} = \bK[t,1/p] \otimes_{\bK[t]} C_{q,u}, \\
\label{eq:invert-p-2}
& C_{q^{\pm 1},1/p,u} = \bK[t,1/p] \hat\otimes_{\bK[t]} C_{q^{\pm 1},u}, \\
\label{eq:invert-p-3}
& q^{-1}C_{q^{-1},1/p,u} = \bK[t,1/p] \hat\otimes_{\bK[t]} q^{-1}C_{q^{-1},u}.
\end{align}
In \eqref{eq:invert-p-1}, we could have inserted a $u$-adic completion, but that would be redundant for grading reasons; while it is necessary for \eqref{eq:invert-p-2}. 
As usual, we write $H_{q,1/p,u}$, $H_{q^{\pm 1},1/p,u}$ and $q^{-1}H_{q^{-1},1/p,u}$ for the cohomology groups. The first of these is straightforward to describe, because of the absence of completion and the exactness of the localization functor:
\begin{equation}
H_{q,1/p,u} = \bK[t,1/p] \otimes_{\bK[t]} H_{q,u}.
\end{equation}
\end{notation}

\begin{lemma}
In the derived category $D(W_{q,u})$ from Section \ref{subsubsec:derived}, there is an exact triangle
\begin{equation} \label{eq:desired-triangle}
\xymatrix{
C_{q,1/p,u} \ar[r] & C_{q^{\pm 1},1/p,u} \ar[r] & q^{-1}C_{q^{-1},1/p,u}
\ar@/^1pc/[ll]^-{[1]}
}
\end{equation}
\end{lemma}

\begin{proof}
The $u = 0$ specialization of $C_{q,u}$ is $C_q$, which by construction is $q$-torsionfree. Hence, as discussed in \eqref{eq:aq-triangle}, there is an exact triangle
\begin{equation}
\xymatrix{
C_{q,u} \ar[r] & C_{q^{\pm 1},u} \ar[r] & q^{-1}C_{q^{-1},u}
\ar@/^1pc/[ll]^-{[1]}
}
\end{equation}
As mentioned in Section \ref{subsubsec:derived}, inverting $p$ and completing is an exact functor; applying that gives \eqref{eq:desired-triangle}.
\end{proof}

\begin{lemma} \label{th:generically-acyclic}
There is a $p(t)$ and an isomorphism of $W_{q,u}$-modules,
\begin{equation}
\bK[t,1/p] \otimes_{\bK[t]} H_{q,u} \iso q^{-1}H_{q^{-1},1/p,u}[-1].
\end{equation}
\end{lemma}

\begin{proof}
By a combination of Lemmas \ref{th:invert-q} and \ref{th:invert-t}, we know that the $u = 0$ reduction of $C_{q^{\pm 1},1/p,u}$ is
\begin{equation} \label{eq:qtp}
\bK[q^{\pm 1}, t,1/p] \otimes_{\bK[t]} C_q.
\end{equation}
Lemma \ref{th:hq-support} says that for a suitable choice of $p$, this space is acyclic. In that case, $C_{q^{\pm 1},1/p,u}$ is filtered acyclic, hence isomorphic to the zero object in $D(W_{q,u})$, which means that the remaining nontrivial morphism in \eqref{eq:desired-triangle} is an isomorphism in that category, hence induces an isomorphism on cohomology.
\end{proof}

\subsubsection{The wrapped Fukaya category}
We choose some Weinstein structure on $\hat{N}$.

\begin{notation}
Let $\scrA$ be the full subcategory of the wrapped category of $\hat{N}$, whose objects are co-cores for the Weinstein structure. We think of this as an $A_\infty$-algebra, by taking the direct sum of all morphism spaces.
\end{notation}


\begin{lemma} \label{th:closed-open-iso}
(i) $\scrA$ is smooth.

(ii) The closed-open and open-closed maps
\begin{align}
\label{eq:clop} & H^* \longrightarrow \mathit{HH}^*(\scrA), \\
\label{eq:opcl} & \mathit{HH}_*(\scrA) \longrightarrow H^{*+n}
\end{align}
are isomorphisms.
\end{lemma}

These are general properties of Weinstein manifolds (as explained in \cite{nonproper}, results in \cite{many, gps2, gao} establish that all Weinstein manifolds are nondegenerate in the sense of \cite{ganatra13}, so one can apply \cite[Theorem 1.1 and Theorem 1.3]{ganatra13}; the part concerning \eqref{eq:opcl} is also explicitly stated in \cite[Theorem 1.4]{many}).

\begin{lemma} \label{th:a-proper} (see Section \ref{subsubsec:bs})
$H^*(\scrA)$ is a finitely generated $\bK[t]$-module, where $t$ acts by multiplication with the image of the Borman-Sheridan class under \eqref{eq:clop}, followed by the forgetful map $\mathit{HH}^*(\scrA) \rightarrow H^*(\scrA)$.
\end{lemma}

\begin{notation}
We write $\scrA_q$ for the curved deformation of $\scrA$ associated to $\alpha$, as in Section \ref{section:deform-fuk}. On the closed string side, write $C_q^{\operatorname{diag}}$ for the space $C[q]$ (or equivalently $C[[q]]$, because $C$ is bounded below) equipped with the differential deformed by $\alpha$, in the sense of Definition \ref{th:deformed-structures}(i); and $H_q^{\operatorname{diag}}$ for its cohomology.
\end{notation}

\begin{lemma} \label{th:weak-2} (see Section \ref{subsubsec:deformed})
$\bK[q^{\pm 1}] \otimes_{\bK[q]} H_q^{\operatorname{diag}}$ is a finitely generated $\bK[q^{\pm 1}]$-module (equivalently, it is of finite dimension over $\bK$ in each degree).
\end{lemma}

\begin{lemma} \label{th:exploit-a}
$\scrA$ and $\scrA_q$ satisfy the conditions (i)--(v) from Corollary \ref{th:end-of-algebra}, where the constant $d$ in (v) is $\mathrm{dim}_{\bC}(M)-1$.
\end{lemma}

\begin{proof}
(i) is Lemma \ref{th:closed-open-iso}(i). 

Concerning (ii), by a $q$-filtration argument and Lemma \ref{th:closed-open-iso}(ii), the deformed closed-open map 
\begin{equation}
H_q^{\operatorname{diag}} \longrightarrow \mathit{HH}^*(\scrA_q)
\end{equation}
defined in Section \ref{subsubsec:deform-co}, is an isomorphism. From that and Lemma \ref{th:weak-2}, one sees that $\bK[q^{\pm 1}] \otimes_{\bK[q]} \mathit{HH}^*(\scrA_q)$ is finite-dimensional over $\bK$ in each degree. As a consequence, for any $x \in \mathit{HH}^0(\scrA_q)$ there is a nonzero polynomial $p$ such that $p(x) = 0 \in \bK[q^{\pm 1}] \otimes_{\bK[q]} \mathit{HH}^*(\scrA_q)$. This means that $q^r p(x) = 0 \in \mathit{HH}^{2r}(\scrA_q)$ for some $r>0$. In particular, that applies to $x = [\kappa_{\scrA_q}]$.

(iii) By definition, $[\mu_{\scrA_q}^{0,(1)}]$ is the image of $b$ under $H \rightarrow \mathit{HH}^*(\scrA) \rightarrow H^*(\scrA)$. Hence, the desired result follows from Lemma \ref{th:a-proper}.

(iv) follows from the fact that \eqref{eq:clop} is an isomorphism and that $H$ vanishes in negative degrees. Similarly, (v) holds because \eqref{eq:opcl} is an isomorphism and $H$ vanishes in degrees $\geq \mathrm{dim}_{\bR}(M)$.
\end{proof}


\subsubsection{Open-closed comparison\label{subsubsec:proof}}
We now adopt the notation from Section \ref{subsubsec:manipulate-a}, starting with the cyclic complex $A_{q,u}$ from \eqref{eq:final-cyclic}, and constructing $q^{-1}A_{q^{-1},1/p,u}$ by applying \eqref{eq:invert-q} and then \eqref{eq:quotient-t}. This is parallel to the closed strings constructions from Notation \ref{th:modify-c}, and in fact we have:

\begin{lemma} \label{th:open-closed-w}
There is an isomorphism of $W_{q,u}$-modules, 
\begin{equation}
q^{-1}H(A_{q^{-1},1/p,u}) \iso q^{-1}H_{q^{-1},1/p,u}.
\end{equation}
\end{lemma}

\begin{proof}
Take the cyclic open-closed map $A_{q,u} \longrightarrow C_{q,u}$, from Section \ref{section:deformed-oc}. By Lemma \ref{th:closed-open-iso}(ii), this is a filtered quasi-isomorphism, and it satisfies the assumptions of Lemma \ref{th:homotopy-dq} (it is strictly $q$-linear, and commutes with connections up to chain homotopy). Hence, it is an isomorphism in the category $D(W_{q,u})$. Both modifications we have applied (passing to negative powers of $q$, and inverting $p$) are endofunctors of that category, hence the outcome of carrying them out on each side inherits the isomorphism.
\end{proof}

\begin{lemma} \label{th:ft-h}
Set $\bK = \bC$. Consider $\bC[t,1/p,u^{\pm 1}] \otimes_{\bC[t,u]} H_{q,u}$ as a module over $W_t$, with the connection $\nabla_{\partial_t} = -q/u$. Then, that connection has regular singularities (including at $\infty$); quasi-unipotent monodromies around each singularity; and each such monodromy has Jordan blocks of size $\leq \mathrm{dim}_{\bC}(M)$.
\end{lemma}

\begin{proof}
From Lemma \ref{th:exploit-a} and Corollary \ref{th:end-of-algebra}, we get corresponding properties for $\bC[u^{\pm 1}] \otimes_{\bC[u]} H(q^{-1}A_{q^{-1},1/p,u})$. Those are carried over to $\bC[u^{\pm 1}] \otimes_{\bC[u]} q^{-1}H_{q^{-1},1/p,u}$ by Lemma \ref{th:open-closed-w}, and then to $\bC[u^{\pm 1},t,1/p] \otimes_{\bC[t,u]} H_{q,u}$ by Lemma \ref{th:generically-acyclic}.
\end{proof}

The last missing puzzle pieces are: the classical Fourier-Laplace transform; and a final appeal to \cite{pomerleano-seidel24}, which here enters in a much more substantive way than before.

\begin{proof}[Proof of Theorems \ref{th:main} and \ref{th:main-2}]
By Lemma \ref{th:holonomic}, $\bC[u^{\pm 1}] \otimes_{\bC[u]} H_{q,u}$ is a holonomic $D$-module in each degree. Lemma \ref{th:ft-h} described the connection associated to the Fourier-Laplace transform of that $D$-module. Proposition \ref{th:fourier} and Corollary \ref{th:quasiunipotent} translate that into properties of the original $D$-module, or more precisely of the connection $\nabla_{u\partial_q}$ on $\bC[q^{\pm 1},u^{\pm 1}] \otimes_{\bC[u]} H_{q,u}$. By Theorem \ref{th:bsv-3}, this is the quantum connection in the form \eqref{eq:quantum-connection-with-u}.
\end{proof}

\begin{remark} \label{th:summary}
Let's summarize the argument, allowing for some expository simplifications: we use $\bK = \bC$ throughout; omit the notational details that distinguish various versions of the cyclic complex; and leave out the final algebraic step of inverting $u$. Roughly speaking, we'll be working our way up the left column of Figure \ref{fig:diagram}, from bottom (cyclic homology) to top (quantum cohomology).

The general algebraic material from Section \ref{subsec:fiber} associates to $\scrA_q$ an $A_\infty$-algebra $\scrA_t$ over $\bC[t]$. If $\mathit{CC}_*(\scrA_t)$ is its cyclic complex over $\bC[t]$, then the $u$-completed tensor product $\bC[t,1/p] \hat\otimes_{\bC[t]} \mathit{CC}_*(\scrA_t)$ is the cyclic complex of the category $\scrA_{t,1/p}$ over $\bC[t,1/p]$ obtained by removing finitely many values of $t$. The ``generic smoothness'' argument from Section \ref{subsubsec:smoothness} allows us to apply the theorem from \cite{petrov-vaintrob-vologodsky18} and obtain information about (periodic) cyclic homology over $\bC[t,1/p]$. The ``categorical Fourier-Laplace transform'' from Theorem \ref{th:noncommutative-fourier-transform} relates $\mathit{CC}_*(\scrA_t)$ to $q^{-1}\bC[q] \hat\otimes_{\bC[q]} \mathit{CC}_*(\scrA_q)$, and we can apply $\bC[t,1/p] \hat\otimes_{\bC[t]} -$ to both sides. At this point, we have obtained some understanding of the connection $\nabla_{\partial_t}$ on $\bC[t,1/p] \hat\otimes_{\bC[t]} (q^{-1}\bC[q] \hat\otimes_{\bC[q]} \mathit{CC}_*(\scrA_q))$, which is where things stand in Corollary \ref{th:end-of-algebra}. Now, we use the cyclic open-closed quasi-isomorphism to carry over the information to $q^{-1}C_{q^{-1},1/p,u} = \bC[t,1/p] \hat\otimes_{\bC[t]} (q^{-1}\bC[q^{-1}] \hat\otimes_{\bC[q]} C_{q,u})$. An acyclicity result and the exact triangle \eqref{eq:desired-triangle} show that this is quasi-isomorphic to $C_{q,1/p,u} = \bC[t,1/p] \hat\otimes_{\bC[t]} C_{q,u}$. At this point, for grading reasons, we can finally dispense with the completion $\hat\otimes_{\bC[t]}$ and consider it as an ordinary tensor product, which means that its cohomology is $\bC[t,1/p] \otimes_{\bC[t]} H_{q,u}$. Section \ref{subsec:finiteness-results} shows that we are in a context of holonomic $D$-modules, which allows us to apply classical results on Fourier-Laplace transforms (Proposition \ref{th:fourier}); the outcome being information about the connection $\nabla_{u\partial_q}$ on $\bC[q^{\pm 1}] \otimes_{\bC[q]} H_{q,u}$. Finally, Theorem \ref{th:bsv-3} identifies that with the quantum connection.
\end{remark}  

\subsection{Geometric ingredients\label{subsec:dan}}

\subsubsection{The Liouville domain\label{subsubsec:liouville2}}
Our first task is to recall how the pair $(M,D)$ gives rise to a Liouville domain $(N, \theta_N)$. We define $\omega_D$ to be the restriction of the symplectic form $\omega_M$ to the symplectic hypersurface $D$. Let 
\begin{equation} 
\pi: \scrL \longrightarrow D 
\end{equation} 
be the normal bundle of $D$, and $\scrL_0 \subset \scrL$ the zero section.  Choose a Hermitian metric $||\cdot||$ on $\scrL$, and set $\mu(v)=\frac{1}{2}||v||^2$. For any $\mu_0>0$, define $\scrL_{\mu < \mu_0} = \{v \in \scrL\;:\;\mu(v) < \mu_0\}$. We also choose a Hermitian connection $\nabla$ on $\scrL$ whose associated connection one-form $\alpha^{\nabla} \in \Omega^1(\scrL\setminus \scrL_0)$ satisfies 
\begin{equation} 
d\alpha^{\nabla} = -\pi^*(\omega_D).
\end{equation}  
The closed two-form
\begin{equation} 
\omega_{(||\cdot||, \nabla)} = d(\mu\cdot \alpha^{\nabla}) + \pi^*\omega_D
\end{equation} 
extends smoothly over $\scrL_0$, and is symplectic on $\scrL_{\mu < 1}$. Rotation of the fibers of $\scrL$ defines a Hamiltonian $S^1$-action on $(\scrL_{\mu < 1}, \omega_{(||\cdot||, \nabla)})$ with moment map $\mu$. We write the infinitesimal generator of this action as $\partial_\phi$.

For $\epsilon$ sufficiently small, the symplectic tubular neighborhood theorem shows that there is a symplectic embedding 
\begin{equation} 
\psi: \scrL_{\mu < 2\epsilon} \hookrightarrow M,  \quad  
\psi^*(\omega_M)=\omega_{(||\cdot||, \nabla)}, 
\end{equation} 
sending the zero section to $D$. We fix such an embedding, and let $UD$ be its image. The identification $\psi$ equips $(UD, \omega_{UD}=(\omega_M)_{|UD})$ with a Hamiltonian $S^1$-action. In a slight abuse of notation, we will let $\partial_\phi$ denote the infinitesimal generator of this circle action, and $\mu$ the moment map on $UD$ (as before, normalized so that it vanishes on $D$).

Our hypothesis that $D$ is Poincar\'e dual to the symplectic class implies that $\omega_M|(M \setminus D)$ is an exact symplectic form. 

\begin{lemma}
For a suitable choice of connection $\nabla$, there exists a primitive $\theta_{M \setminus D} \in \Omega^1(M \setminus D)$ of $\omega_M$, such that (after making $\epsilon$ smaller)
\begin{equation} \label{eq:theta-alpha}
\psi^*(\theta_{M\setminus D}) = (\mu-1)\alpha^{\nabla}.
\end{equation} 
\end{lemma}

\begin{proof}
Given any $\nabla$, the associated embedding $\psi$, and any primitive $\theta_{M \setminus D}$, consider the closed one-form
\begin{equation} \label{eq:diogolisi}
\psi^*(\theta_{M \setminus D}) - (\mu-1)\alpha^{\nabla} \in \Omega^1(\scrL_{\mu<2\epsilon} \setminus \scrL_0).
\end{equation}
If this is exact, then by using a suitable cutoff function to modify $\theta_{M \setminus D}$ one can achieve that \eqref{eq:theta-alpha} holds (on a smaller neighbourhood).

Let's look at what happens for a one-parameter family $\nabla_r$ of connections, all having the same curvature, so that $\partial_r \alpha^{\nabla_r} = \pi^*\beta_r$ for some closed one-forms $\beta_r \in \Omega^1(D)$. Differentiating the analogue of \eqref{eq:diogolisi} with respect to $r$ yields
\begin{equation} \label{eq:r-derivative}
\partial_r (\psi_r^*(\theta_{M \setminus D}) - (\mu-1)\alpha^{\nabla_r}) =
d(i_{X_r}\psi_r^*\theta_{M \setminus D}) - i_{X_r}\omega_{(||\cdot||, \nabla^r)} - (\mu-1) \pi^*\beta_r;
\end{equation}
here $X_r$ is the vector field such that $D\psi_r(X_r) = \partial_r \psi_r$, which vanishes along the zero-section. The first term on the right hand side of \eqref{eq:r-derivative} is exact. The other two terms combine to a closed one-form which extends over the zero-section, and which on the zero-section equals $\pi^*\beta_r$. Hence,
\begin{equation}
\partial_r [\psi_r^*(\theta_{M \setminus D}) - (\mu-1)\alpha^{\nabla_r}] = \pi^*[\beta_r] \in H^1(\scrL_{\mu<2\epsilon} \setminus \scrL_0).
\end{equation}
As already observed in \cite[Lemma 2.2]{diogolisi19}, $\pi^*: H^1(D) \rightarrow H^1(\scrL_{\mu<2\epsilon} \setminus \scrL_0)$ is an isomorphism. Hence, by starting with an arbitrary $\nabla_0$ and choosing $\beta_r$ appropriately, one can achieve that $\psi_1^*(\theta_{M \setminus D}) - (\mu-1)\alpha^{\nabla_1}$ is exact; which means that $\nabla_1$ can be used to obtain \eqref{eq:theta-alpha}.
\end{proof}

The Liouville domain $(N,\theta_N)$ is defined by 
\begin{equation} 
N\stackrel{\mathrm{def}}{=} M\setminus UD_{\mu<\epsilon}, \quad \theta_N\stackrel{\mathrm{def}}{=}(\theta_{M\setminus D})_{|N}. 
\end{equation} 
Here $UD_{\mu<\epsilon}$ denotes the locus of $UD$ where $\mu<\epsilon$. The boundary $\partial N = \{\mu=\epsilon\}$ is a circle bundle $\partial N \to D$. If we decompose the pull-back $\psi^*(Z_N)$ of the Liouville vector field into a base and fiber component (using the connection), the fiber component is a suitable negative multiple of the radial vector field. Therefore, the Liouville vector field points strictly outwards along $\partial N$. The Reeb field on $\partial N$ is
\begin{equation} 
R_{\partial N} = \textstyle\frac{1}{\epsilon-1}\partial_\phi. 
\end{equation} 
Thus, the Reeb flow is tangent to the fibers of $\partial N \to D$, and the set of periods of its orbits is
\begin{equation} 
\label{eq:setofperiods} \lbrace 1-\epsilon, 2(1-\epsilon), \cdots \rbrace. 
\end{equation}

\subsubsection{Floer complex of $N$} \label{subsubsection:nonnegativity}
Let $\hat{N}$ be the Liouville manifold associated to $N$. Recall from Section \ref{subsubsec:liouville} that $\rho$ denotes the radial coordinate on the cone $[1,\infty) \times \partial N \subset \hat{N}$. Since the Liouville flow exists for all (positive and negative) time, the embedding of the cone into $\hat{N}$ extends to 
\begin{equation}
(0,\infty) \times \partial N \to \hat{N}.
\end{equation}

To simplify the discussion which follows, we will assume that the constant $\epsilon$ involved in the construction of the Liouville domain $N$ has been taken to be less than $\frac{1}{4}$.  
Consider a time-independent quadratic Hamiltonian $h: \hat{N} \to \mathbb{R}$ which satisfies the following conditions: 
\begin{itemize} 
\itemsep.5em
\item On $[\frac{1}{2},\infty)\times \partial N$, $h=\frac{1}{2}\rho^2.$ 
\item Over $N \setminus (\frac{1}{4},1] \times \partial N$, $h$ is a $C^2$-small Morse function such that $dh(Z_N)>0$ along $\{ \frac{1}{4} \} \times \partial N$. We further require that the critical points of $h$ over $N \setminus (\frac{1}{4},1] \times \partial N$ have Morse index concentrated in $[0,\operatorname{dim}_\mathbb{R}(N)-1]$.
\item The Hamiltonian flow of $h$ has no periodic orbits in the shell $[\frac{1}{4},\frac{1}{2}] \times \partial N$.
\end{itemize}
The condition that the critical points of $h$ have Morse index concentrated in $[0,\operatorname{dim}_\mathbb{R}(N)-1]$ can always be achieved, by \cite[Theorem 8.1]{milnor65}. (In the notation there, set $W = N \setminus (\frac{1}{4},1] \times \partial N$, $V = \{ \frac{1}{4} \} \times \partial N$, and also reverse the sign of the Morse function.) These conditions imply that the non-constant periodic points of $h$ are precisely those in the level sets 
\begin{equation} 
\mathcal{Q}_d \stackrel{\mathrm{def}}{=}\{d(1-\epsilon)\} \times \partial N, \quad d\geq 1. 
\end{equation} 
The periodic flow along $\mathcal{Q}_d$ generates an $S^1$-action \begin{align} \gamma_d: S^1 \times \mathcal{Q}_d \to \mathcal{Q}_d. \end{align} These orbit sets are transversally non-degenerate and we perturb $h$ to a (time-dependent) nondegenerate Hamiltonian $H$ whose one-periodic orbits are explicitly determined. To ensure that this is done compatibly with the analysis in Section \ref{subsubsec:floer}, we set the constant $P$ from  Assumption \ref{th:no-p-orbits} to be: \begin{align} P=\sqrt{2}(1-\epsilon). \end{align}

For each critical submanifold $\mathcal{Q}_d$, we choose a small constant $\tau_d \in (0,1-\epsilon)$ which satisfies: \begin{itemize} \item $iP \notin [d(1-\epsilon)-\tau_d,d(1-\epsilon)+\tau_d]$ for any $i$. \end{itemize} We then let $\mathcal{UQ}_d$ denote the isolating shell:  
\begin{equation}
\mathcal{UQ}_d \stackrel{\mathrm{def}}{=}[d(1-\epsilon)-\tau_d,d(1-\epsilon)+\tau_d] \times \partial N. 
\end{equation}  
By construction, the isolating shell $\mathcal{UQ}_d$ does not contain any critical submanifolds besides $\mathcal{Q}_d$ and also \begin{align} \label{eq:shellsforHD} \text{does not intersect any of the level sets } \{ iP \} \times \partial N. \end{align}  

Over $\hat{N} \setminus \bigcup_d \mathcal{UQ}_d$, $H$ will be unperturbed, i.e. we will have $H=h$. Over each shell $\mathcal{UQ}_d$, we use the standard Morse-Bott perturbation procedure (see \cite[Section 2]{cfh96}, \cite[Appendix B]{kwonvankoert16}, \cite[Appendix C]{ritterzivanovic23}). Let  $\pi_d: \mathcal{UQ}_d \to \mathcal{Q}_d$ denote the natural projection map and choose a $C^2$-small Morse function  \begin{align} f_d:\mathcal{Q}_d \to \mathbb{R}. \end{align} We also choose a suitable cutoff function $\kappa_d:\mathcal{UQ}_d \to [0,1]$ which is $1$ near the critical level set $\mathcal{Q}_d$ and $0$ near the boundary of $\mathcal{UQ}_d$. Over $\mathcal{UQ}_d$, we then perturb $h$ as follows:  
\begin{equation}
\begin{aligned} 
& H_{|\mathcal{UQ}_d}: S^1 \times \mathcal{UQ}_d \to \mathbb{R} \\ 
& H_{|\mathcal{UQ}_d} =  h+\kappa_d \cdot(f_d \circ \gamma_d \circ (-\mathrm{id}_{S^1} \times \pi_d)).
\end{aligned}
\end{equation}

Provided the perturbing function $f_d$ is chosen sufficiently small (which we can achieve separately for each $d$), the periodic orbits of $H_{|\mathcal{UQ}_d}$ are in bijection with the critical points of $f_d.$ Moreover, if $c_d \in \operatorname{crit}(f_d)$ is a critical point of Morse index $\operatorname{deg}(c_d)$, then the corresponding periodic orbit $x_{c_{d}}$ is nondegenerate and has degree (see e.g.\ \cite[Section 3.1]{diogolisi19b}, which uses homological rather than cohomological grading conventions; generally, the calculation of Floer indices when one breaks Morse--Bott degeneracies goes back to \cite{Pozniak1994}): \begin{align} \label{eq:degreecomp} \operatorname{deg}(x_{c_{d}})= \operatorname{deg}(c_d).\end{align} 
It is clear that we can arrange that the perturbation $\tilde{H}=H-h$ is bounded and has bounded derivative $\partial_\rho\tilde{H}$. In particular, we assume that it is of the form \eqref{eq:perturbed-hamiltonian} along the cone. In view of \eqref{eq:shellsforHD}, we have $H=\frac{1}{2}\rho^2$ along $iP$-shells for $i\geq 1$. The nondegenerate Hamiltonian $H$ therefore satisfies all of the conditions needed for the analysis of \S \ref{subsubsec:floer}. The other important property of $H$ is the following:

\begin{proposition} \label{bounded-2} For the Hamiltonian $H$ constructed above, the complex $CF^*(H)$ is concentrated in degrees $[0,\operatorname{dim}_\mathbb{R}(N)-1].$
\end{proposition}

\begin{proof}[Proof of Lemma \ref{th:boundedness}] We use the Hamiltonian $H$ constructed above. Note that in view of \eqref{eq:degreecomp}, the gradings of non-constant orbits are concentrated in degrees $[0,\operatorname{dim}_\mathbb{R}(N)-1]$. By assumption, the constant orbits which arise as critical points of $h$ have degrees concentrated in $[0,\operatorname{dim}_\mathbb{R}(N)-1]$ as well and the result follows. \end{proof}


\subsubsection{The Borman-Sheridan class\label{subsubsec:bs}}
Let $H_\lambda:\hat{N} \to \mathbb{R}$ be a linear Hamiltonian of slope $\lambda>1$. The Borman-Sheridan class \begin{align} b_\lambda \in HF^0(H_\lambda) \end{align} was defined in \cite{ganatrapomerleano21,tonkonog19}. There is an acceleration map
\begin{equation} 
\mathfrak{ac}:  HF^0(H_\lambda) \to HF^0(H),
\end{equation} 
where $H$ is the quadratic Hamiltonian from Section \ref{subsubsection:nonnegativity}. We then set 
\begin{align} 
b \stackrel{\mathrm{def}}{=} \mathfrak{ac}(b_\lambda). 
\end{align}
As noted in Section \ref{subsec:finiteness-results}, the class $b$ admits a unique cochain level representative $\beta$, which gives rise to our Maurer-Cartan element. Our remaining task is to explain the cohomological finiteness properties of $b$ (Lemma \ref{th:h-finite}, Lemma \ref{th:a-proper}).

\begin{proof}[Proof of Lemma \ref{th:h-finite}]
This is very similar to a special case of \cite[Theorem 5.30]{ganatra-pomerleano20}; we will summarize the proof in a form suitable for our purpose. (Just like the definition of the Borman-Sheridan class, the argument in \cite{ganatra-pomerleano20} uses direct limits of Floer cohomologies of linear Hamiltonians as a model for symplectic cohomology, as opposed to the quadratic Hamiltonians used here; however, the isomorphism between the two models, given by acceleration maps, is compatible with pair-of-pants products, so all results about multiplicative structures carry over.) The argument from \cite{ganatra-pomerleano20} is based on the following properties of symplectic cohomology and the Borman-Sheridan class.
\begin{itemize} \itemsep.5em
\item
There is a multiplicative spectral sequence converging to symplectic cohomology, with
\begin{equation} \label{eq:gp-spectral-sequence} 
E_1^{pq} = \begin{cases}
H^q(N;\bK) & p = 0, \\
H^{p+q}(\partial N;\bK) z^p & p>0, \\
0 & p<0.
\end{cases}
\end{equation}
The powers of $z$ are just notation, which roughly speaking keeps track of winding numbers of orbits around $D$ (the labeling of the columns corresponds to an increasing filtration, and therefore the differentials are $d_r^{pq}: E_r^{pq} \rightarrow E_r^{p-r,q+r+1}$). Given two classes $\alpha_1, \alpha_2 \in H^*(\partial N)$, the product of $\alpha_1 z^{p_1}, \alpha_2 z^{p_2}$ on the $E_1$ page is the ordinary cup product:
\begin{equation} \label{eq:e1product} 
(\alpha_1 z^{p_1}) (\alpha_2 z^{p_2}) = (\alpha_1 \smile \alpha_2)z^{p_1+p_2}. 
\end{equation} 

\item
Let 
\begin{equation} \label{eq:f-filtration}
F_0 \subset F_1 \subset \cdots \subset \mathit{SH}^*(\hat{N})
\end{equation}
be the (bounded below exhaustive) filtration of symplectic cohomology whose associated graded is the $E_\infty$ page of our spectral sequence. By definition we have $b \in F_1$; and 
\begin{equation}
b^\mathrm{gr} \stackrel{\mathrm{def}}{=} 1_{\partial N}\,z \in H^0(\partial N;\bK)z = E_1^{1,-1}
\end{equation}
survives to $E_\infty$, where it yields the image of $b$ in $F_1/F_0$.  
\end{itemize}

Multiplicativity means that every page of the spectral sequence is a module over $\bK[b^{\mathrm{gr}}]$, compatibly with differentials. From \eqref{eq:e1product}, the $E_1$ page is finitely generated over $\bK[b^{\mathrm{gr}}]$. Define $Z_\infty^{**} \subset E_1^{**}$ to be the subspace of elements that survive to the $E_\infty$ page. This is a $\bK[b^{\mathrm{gr}}]$-submodule, hence finitely generated because $\bK[b^{\mathrm{gr}}]$ is a Noetherian ring. The $E_\infty$ page is by definition a quotient of $Z_\infty^{**}$, and therefore finitely generated over $\bK[b^{\mathrm{gr}}]$ as well. Finally, 
finite generation of the associated graded of \eqref{eq:f-filtration} over $\bK[b^{\mathrm{gr}}]$ implies that of $\mathit{SH}^*(\hat{N})$ over $\bK[b]$.
\end{proof}

\begin{proof}[Proof of Lemma \ref{th:a-proper}] By \cite[Theorem 1.1(c)]{pomerleano21}, $\mathcal{A}$ is proper over $SH^0(\hat{N})$; and \cite[Lemma 5.38]{ganatra-pomerleano20} shows that $\mathit{SH}^0(\hat{N})$ is a one-variable polynomial ring generated by $b$. These two statements together imply the desired result. 
\end{proof}

\subsubsection{Deformed symplectic cohomology\label{subsubsec:deformed}}
Finally, we summarize the results from \cite{pomerleano-seidel24} (partly overlapping with \cite{el-alami-sheridan24b}) which will be used in our argument.

\begin{theorem} \label{th:bsv-1}
There are isomorphisms of graded $\bK[q]$-modules
\begin{equation} \label{eq:bsv-1}
H^*(M;\bK)[q] \oplus \bigoplus_{w \geq 1} H^*(D;\bK)z^w \iso H_q \iso H_q^{\mathrm{diag}},
\end{equation}
where $z^w$ are formal symbols of degree $0$. The $q$-action on the left hand side has the following properties: it is the standard on $H^*(M;\bK)[q]$; and any element in $H^*(D;\bK)z^w$ will be mapped to $H^*(M;\bK)[q]$ by a sufficiently high power of $q$.
\end{theorem}

Note that the theorem above concerns both $H_q$ and $H_q^{\mathrm{diag}}$. In general (meaning, for an abstract choice of Maurer-Cartan element) those could be different theories, but in the Borman-Sheridan case they coincide. In fact, the main part \cite{pomerleano-seidel24} works with a different definition of deformed symplectic cohomology, in which there is no distinction between those two groups; the translation into the framework used here is explained in \cite[Section 10]{pomerleano-seidel24} (in particular, see \cite[Section 10.6]{pomerleano-seidel24} for $H_q^{\mathrm{diag}}$). After inverting $q$, we get a simpler statement which is sufficient for our purpose:

\begin{corollary} \label{th:bsv-1b}
There are isomorphisms of graded $\bK[q^{\pm 1}$-modules
\begin{equation} 
H^*(M;\bK)[q^{\pm 1}] \iso \bK[q^{\pm 1}] \otimes_{\bK[q]} H_q \iso
\bK[q^{\pm 1}] \otimes_{\bK[q]} H_q^{\mathrm{diag}}.
\end{equation}
\end{corollary}

For $H_q^{\mathrm{diag}}$, an alternative proof of Corollary \ref{th:bsv-1b} is given in \cite{el-alami-sheridan24b}; and one should be able to adapt the argument there to cover $H_q$ as well.

\begin{proof}[Proofs of Lemma \ref{th:weak-1} \and \ref{th:weak-2}] These follow immediately from the statement above (in fact, they are substantially weaker).
\end{proof}

The equivariant versions of these results are as follows.

\begin{theorem} \label{th:bsv-2}
There is an isomorphism of graded $\bK[q,u]$-modules,
\begin{equation} \label{eq:bsv-2}
H^*(M;\bK)[q,u] \oplus \bigoplus_{w \geq 1} H^*(D;\bK)[u]z^w \iso H_{q,u}.
\end{equation}
The $u$-action on the left hand side is the standard one, and the $q$-action has the following properties: it is standard on $H^*(M;\bK)[q,u]$; and any element in $H^*(D;\bK)[u]z^w$ will be mapped to $H^*(M;\bK)[q,u]$ by a sufficiently high power of $q$.
\end{theorem}

\begin{corollary} \label{th:bsv-2b}
There is an isomorphism of graded $\bK[u,q^{\pm 1}]$-modules,
\begin{equation} \label{eq:equivariant-bsv}
H^*(M;\bK)[q^{\pm 1},u] \iso \bK[q^{\pm 1}] \otimes_{\bK[q]} H_{q,u}.
\end{equation}
\end{corollary}

The final result we need concerns connections. For simplicity, we'll state it only in $q$-inverted form (readers interested in what it looks like without inverting $q$ are referred to \cite[Section 9]{pomerleano-seidel24}).

\begin{theorem} \label{th:bsv-3}
The isomorphism \eqref{eq:equivariant-bsv} identifies the quantum connection \eqref{eq:quantum-connection-with-u} with the canonical connection on $H_{q,u}$.
\end{theorem}


\begin{thebibliography}{10}

\bibitem{stacks}
Stacks project.
\newblock \url{https://stacks.math.columbia.edu}.

\bibitem{abouzaid10}
M.~Abouzaid.
\newblock A geometric criterion for generating the {F}ukaya category.
\newblock {\em Publ. Math. IHES}, 112:191--240, 2010.

\bibitem{abouzaid14}
M.~Abouzaid.
\newblock Symplectic cohomology and {Viterbo}'s theorem.
\newblock In \emph{Free Loop Spaces in Geometry and Topology}, IRMA Lectures in Mathematics and Theoretical Physics, vol. 24, pp. 271--485. European Mathematical Society, Z\"urich, 2015.

\bibitem{abouzaid-groman-varolgunes}
M.~Abouzaid, Y.~Groman, and U.~Varolgunes.
\newblock Framed {$E2$} structures in {F}loer theory.
\newblock Preprint arXiv:2210.11027.

\bibitem{abouzaid-seidel07}
M.~Abouzaid and P.~Seidel.
\newblock An open string analogue of {V}iterbo functoriality.
\newblock {\em Geom. Topol.}, 14:627--718, 2010.

\bibitem{amorim-tu22}
L.~Amorim and J.~Tu.
\newblock Categorical primitive forms of {C}alabi-{Y}au {$A_\infty$}-categories
  with semi-simple cohomology.
\newblock {\em Selecta Math.}, 28:Article no. 54, 2022.

\bibitem{babbitt-varadarajan83}
D.~Babbitt and V.~Varadarajan.
\newblock Formal reduction theory of meromorphic differential equations: a
  group theoretic view.
\newblock {\em Pacific J. Math.}, 109:1--80, 1983.

\bibitem{bernstein-lunts94}
J.~Bernstein and V.~Lunts.
\newblock {\em Equivariant sheaves and functors}.
\newblock Number 1578 in Lecture Notes in Math. Springer, 1994.

\bibitem{borman-el-alami-sheridan24}
S.~Borman, M.~El Alami, and N.~Sheridan.
\newblock An {$L_\infty$}-structure on symplectic cohomology.
\newblock Preprint arXiv:2408.09163.

\bibitem{el-alami-sheridan24b}
S.~Borman, M.~El Alami, and N.~Sheridan.
\newblock Maurer-Cartan elements in symplectic cohomology from compactifications.
\newblock Preprint arXiv:2408.09221.

\bibitem{barkatou97}
M.~Barkatou.
\newblock An algorithm to compute the exponential part of a formal fundamental matrix solution of a linear differential system.
\newblock {\em Appl. Algebra Engrg. Comm. Comput.}, 8:1--23, 1997.

\bibitem{borman-sheridan-varolgunes21}
S.~Borman, N.~Sheridan, and U.~Varolgunes.
\newblock Quantum cohomology as a deformation of symplectic cohomology.
\newblock Preprint arXiv:2108.08391.

\bibitem{many}
B.~Chantraine, G.~Dimitroglou Rizell, P.~Ghiggini, and R.~Golovko.
\newblock Geometric generation of the wrapped {F}ukaya category of {W}einstein
  manifolds and sectors.
\newblock Annales ENS, to appear.


\bibitem{cieliebak94}
K.~Cieliebak.
\newblock Pseudo-holomorphic curves and periodic orbits on cotangent bundles.
\newblock {\em J. Math. Pures Appl.}, 73:251--278, 1994.

\bibitem{cfh96}
K.~Cieliebak, A.~Floer, H.~Hofer, and K.~Wysocki.
\newblock Applications of symplectic homology. {II}. {S}tability of the action
  spectrum.
\newblock {\em Math. Z.}, 223(1):27--45, 1996.

\bibitem{coates-corti-galkin-golyshev-kasprzyk13}
T.~Coates, A.~Corti, S.~Galkin, V.~Golyshev, and A.~Kasprzyk.
\newblock Mirror symmetry and {F}ano manifolds.
\newblock In {\em European {C}ongress of {M}athematics}, pages 285--300.
  European. Math. Soc., 2013.

\bibitem{coates-kasprzyk-prince19}
T.~Coates, A.~Kasprzyk, and T.~Prince.
\newblock Laurent inversion.
\newblock {\em Pure Appl. Math. Q.}, 15:1135--1179, 2019.

\bibitem{crauder-miranda94}
B.~Crauder and R.~Miranda.
\newblock Quantum cohomology of rational surfaces.
\newblock In R.~Dijkgraaf, C.~Faber, and G.~van~der Geer, editors, {\em The
  moduli space of curves}, volume 129 of {\em Progress in Mathematics}, pages
  33--80. Birkh{\"a}user, 1995.

\bibitem{cox-katz}
S.~Katz and D.~Cox.
\newblock {\em Mirror symmetry and algebraic geometry}.
\newblock Amer. Math. Soc., 1999.

\bibitem{dagnolo}
A.~D'Agnolo, M.~Hien, G.~Morando, C.~Sabbah.
\newblock Topological computation of some {S}tokes phenomena on the affine line.
\newblock Ann. Inst. Fourier, 70:739--808, 2020.

\bibitem{difrancesco-itzykson94}
P.~DiFrancesco and C.~Itzykson.
\newblock Quantum intersection rings.
\newblock In R.~Dijkgraaf, C.~Faber, and G.~van~der Geer, editors, {\em The
  moduli space of curves}, volume 129 of {\em Progress in Mathematics}, pages
  81--148. Birkh{\"a}user, 1995.

\bibitem{diogolisi19b}
L.~Diogo and S.~Lisi.
\newblock Morse-{B}ott split symplectic homology.
\newblock {\em J. Fixed Point Theory Appl.}, 21(3):Paper No. 77, 77, 2019.

\bibitem{diogolisi19}
L.~Diogo and S.~Lisi.
\newblock Symplectic homology of complements of smooth divisors.
\newblock {\em J. Topol.}, 12(3):967--1030, 2019.

\bibitem{drinfeld04}
B. ~Drinfeld. 
\newblock DG quotients of DG categories. 
\newblock {\em J. Algebra}, 272(2), 643--691, 2004.

\bibitem{dubrovin99}
B.~Dubrovin.
\newblock {P}ainlev{\'e} transcendents and two-dimensional topological field
  theory.
\newblock In {\em The {P}ainlev{\'e} property}, pages 287--412. Springer, 1999.

\bibitem{evans11}
J.~Evans.
\newblock Quantum cohomology of twistor spaces and their {L}agrangian
  submanifolds.
\newblock {\em J. Differential Geom.}, 96:353--397, 2014.

\bibitem{faltings}
G.~Faltings.
\newblock Crystalline cohomology and $p$-adic {G}alois representations.
\newblock In {\em Algebraic analysis, geometry, and number theory (Baltimore, MD, 1988)}, pages 25--80. Johns Hopkins Press, 1989.


\bibitem{fine-panov10}
J.~Fine and D.~Panov.
\newblock Hyperbolic geometry and non-{K}\"{a}hler manifolds with trivial
  canonical bundle.
\newblock {\em Geom. Topol.}, 14:1723--1763, 2010.

\bibitem{flanders51}
H.~Flanders.
\newblock Elementary divisors of {$AB$} and {$BA$}.
\newblock {\em Proc. Amer. Math. Soc.}, 2:871--874, 1951.

\bibitem{floer89}
A.~Floer.
\newblock Symplectic fixed points and holomorphic spheres.
\newblock {\em Comm. Math. Phys.}, 120:575--611, 1989.

\bibitem{floer-hofer94}
A.~Floer and H.~Hofer.
\newblock Symplectic homology. {I}. {O}pen sets in {$\mathbb{C}^n$}.
\newblock {\em Math. Z.}, 215:37--88, 1994.

\bibitem{floer-hofer-salamon94}
A.~Floer, H.~Hofer, and D.~Salamon.
\newblock Transversality in elliptic {M}orse theory for the symplectic action.
\newblock {\em Duke Math. J.}, 80:251--292, 1995.

\bibitem{fukaya97}
K.~Fukaya.
\newblock Morse homotopy and its quantization.
\newblock In {\em Geometric topology ({A}thens, {GA}, 1993)}, volume 2.1 of
  {\em AMS/IP Stud. Adv. Math.}, pages 409--440. Amer. Math. Soc., Providence,
  RI, 1997.

\bibitem{fukaya02}
K.~Fukaya.
\newblock Floer homology and mirror symmetry. {II}.
\newblock In {\em Minimal surfaces, geometric analysis and symplectic geometry
  ({B}altimore, 1999)}, volume~34 of {\em Adv. Stud. Pure Math.}, pages
  31--127. Math. Soc. Japan, 2002.

\bibitem{fulton-macpherson94}
W.~Fulton and R.~MacPherson.
\newblock A compactification of configuration Spaces.
\newblock {\em Annals of Math.}, 139:183--225, 1994.

\bibitem{galkin-golyshev-iritani16}
S.~Galkin, V.~Golyshev, and H.~Iritani.
\newblock Gamma classes and quantum cohomology of {F}ano manifolds: gamma
  conjectures.
\newblock {\em Duke Math. J.}, 165:2005--2077, 2016.

\bibitem{galkin-iritani19}
S.~Galkin and H.~Iritani.
\newblock Gamma conjecture via mirror symmetry.
\newblock In {\em Primitive forms and related subjects}, volume~83 of {\em Adv.
  Stud. Pure Math.}, pages 55--115. Math. Soc. Japan, 2019.

\bibitem{ganatra17}
S.~Ganatra.
\newblock Automatically generating {F}ukaya categories and computing quantum
  cohomology.
\newblock Preprint arXiv:1605.07702.

\bibitem{ganatra-sheridan25}
S.~Ganatra and N.~Sheridan.
\newblock The cyclic open-closed map and variations of {H}odge structures.
\newblock Preprint arXiv:2511.04498.

\bibitem{nonproper}
S.~Ganatra.
\newblock Categorical non-properness in wrapped {F}loer theory.
\newblock Preprint arXiv:2104.06516.

\bibitem{ganatra19}
S.~Ganatra.
\newblock Cyclic homology, {$S^1$}-equivariant {F}loer cohomology, and
  {C}alabi-{Y}au structures.
\newblock Geom. Topol., to appear.

\bibitem{ganatra13}
S.~Ganatra.
\newblock Symplectic cohomology and duality for the wrapped {F}ukaya category.
\newblock Preprint arXiv:1304.7312.

\bibitem{gps2}
S.~Ganatra, J.~Pardon, and V.~Shende.
\newblock Sectorial descent for wrapped {F}ukaya categories.
\newblock Preprint arXiv:1809.03427.

\bibitem{ganatra-perutz-sheridan15}
S.~Ganatra, T.~Perutz, and N.~Sheridan.
\newblock Mirror symmetry: from categories to curve-counts.
\newblock Preprint arXiv 1510:03839.

\bibitem{ganatra-pomerleano20}
S.~Ganatra and D.~Pomerleano.
\newblock Symplectic cohomology rings of affine varieties in the topological
  limit.
\newblock {\em Geom. Funct. Anal.}, 30:334--456, 2020.

\bibitem{ganatrapomerleano21}
S.~Ganatra and D.~Pomerleano.
\newblock A log {PSS} morphism with applications to {L}agrangian embeddings.
\newblock {\em J. Topol.}, 14(1):291--368, 2021.

\bibitem{gao}
Y.~Gao.
\newblock Functors of wrapped {F}ukaya categories from {L}agrangian
  correspondences.
\newblock Preprint arXiv:1712.00225.

\bibitem{getzler95}
E.~Getzler.
\newblock Cartan homotopy formulas and the {G}auss-{M}anin connection in cyclic
  homology.
\newblock In {\em Quantum deformations of algebras and their representations},
  pages 65--78. Bar-Ilan Univ., 1993.

\bibitem{getzler-jones94}
E.~Getzler and J.~Jones.
\newblock Operads, homotopy algebra, and iterated integrals for double loop spaces.
\newblock Preprint arXiv:hep-th/9403055.

\bibitem{giansiracusa-salvatore}
J.~Giansiracusa and P.~Salvatore.
\newblock Cyclic operad formality for compactified moduli spaces of genus zero
  surfaces.
\newblock {\em Trans. Amer. Math. Soc.}, 364:5881--5911, 2012.

\bibitem{givental95}
A.~Givental.
\newblock Homological geometry and mirror symmetry.
\newblock In {\em Proceedings of the {I}nternational {C}ongress of
  {M}athematicians, ({Z}\"{u}rich, 1994)}, pages 472--480. Birkh\"{a}user,
  1995.

\bibitem{gross-siebert18}
M.~Gross and B.~Siebert.
\newblock Intrinsic mirror symmetry and punctured {G}romov-{W}itten invariants.
\newblock In {\em Algebraic geometry: {S}alt {L}ake {C}ity 2015}, volume~97 of
  {\em Proc. Sympos. Pure Math.}, pages 199--230. Amer. Math. Soc., 2018.

\bibitem{hori-katz-klemm-pandharipande-thomas-vafa-vakil-zaslow}
K.~Hori, S.~Katz, A.~Klemm, R.~Panharipande, R.~Thomas, C.~Vafa, R.~Vakil, and
  E.~Zaslow.
\newblock {\em Mirror Symmetry}, volume~1 of {\em Clay Math. Monographs}.
\newblock Amer. Math. Soc., 2003.

\bibitem{hoering-smiech20}
A.~H\"{o}ring and R.~\'{S}miech.
\newblock Anticanonical system of {F}ano fivefolds.
\newblock {\em Math. Nachr.}, 293:115--119, 2020.

\bibitem{hoering-voisin11}
A.~H\"{o}ring and C.~Voisin.
\newblock Anticanonical divisors and curve classes on {F}ano manifolds.
\newblock {\em Pure Appl. Math. Q. (special issue in memory of E. Viehweg)},
  7:1371--1393, 2011.

\bibitem{hotta}
R.~Hotta, K.~Takeuchi, and T.~Tanisaki.
\newblock {\em {$D$}-modules, perverse sheaves, and representation theory},
  volume 236 of {\em Progress in Mathematics}.
\newblock Birkh\"{a}user, 2008.

\bibitem{hugtenburg22}
K.~Hugtenburg.
\newblock The cyclic open-closed map, {$u$}-connections and {$R$}-matrices.
\newblock Preprint arXiv:2205.13436.

\bibitem{hugtenburg24}
K.~Hugtenburg.
\newblock On the quantum differential equations for a family of non-Kähler monotone symplectic manifolds.
\newblock Preprint arxiv:2402.10867.

\bibitem{iritani-mann-mignon16}
H.~Iritani, E.~Mann, and Th. Mignon.
\newblock Quantum {S}erre theorem as a duality between quantum {$D$}-modules.
\newblock {\em Int. Math. Res. Not. IMRN}, pages 2828--2888, 2016.

\bibitem{katz}
N.~Katz.
\newblock Nilpotent connections and the monodromy theorem: {A}pplications of a
  result of {T}urrittin.
\newblock {\em Inst. Hautes \'{E}tudes Sci. Publ. Math.}, (39):175--232, 1970.

\bibitem{katzarkov-kontsevich-pantev08}
L.~Katzarkov, M.~Kontsevich, and T.~Pantev.
\newblock Hodge theoretic aspects of mirror symmetry.
\newblock In {\em From {H}odge theory to integrability and {TQFT}
  tt*-geometry}, volume~78 of {\em Proc. Sympos. Pure Math.}, pages 87--174.
  Amer. Math. Soc., 2008.

\bibitem{kawamata00}
Y.~Kawamata.
\newblock On effective non-vanishing and base-point-freeness.
\newblock {\em Asian J. Math. (special issue in honor of K. Kodaira)},
  4:173--181, 2000.

\bibitem{keller94}
B.~ Keller
\newblock Deriving DG categories. 
\newblock {\em Ann. Sci. École Norm. Sup}, 
27(4): 63--102, 1994.

\bibitem{KSV}
T.~Kimura, J.~Stasheff, and A.~Voronov.
\newblock On operad structures of moduli spaces and string theory.
\newblock {\em Comm. Math. Phys.}, 171(1):1--25, 1995.

\bibitem{kontsevich-soibelman00}
M.~Kontsevich and Y.~Soibelman.
\newblock Homological mirror symmetry and torus fibrations.
\newblock In {\em Symplectic geometry and mirror symmetry}, pages 203--263.
  World Scientific, 2001.

\bibitem{kwonvankoert16}
M.~Kwon and O.~van Koert.
\newblock Brieskorn manifolds in contact topology.
\newblock {\em Bull. Lond. Math. Soc.}, 48:173--241, 2016.

\bibitem{malgrange83}
B.~Malgrange.
\newblock La classification des connexions irr\'{e}guli\`eres \`a une variable.
\newblock In {\em Mathematics and physics ({P}aris, 1979/1982)}, volume~37 of
  {\em Progr. Math.}, pages 381--399. Birkh\"{a}user, 1983.

\bibitem{malgrange}
B.~Malgrange.
\newblock {\em Equations diff{\'e}rentielles {\`a} coefficients polynomiaux},
  volume~96 of {\em Progress in Math.}
\newblock Birkh{\"a}user, 1991.

\bibitem{markl-shnider-stasheff}
M.~Markl, S.~Shnider, J.~Stasheff.
\newblock {\em Operads in algebra, topology and physics.}
\newblock Amer. Math. Soc., 2002.

\bibitem{milnor65}
John Milnor.
\newblock {\em Lectures on the {$h$}-cobordism theorem}.
\newblock Princeton University Press, 1965.

\bibitem{miranda-persson86}
R.~Miranda and U.~Persson.
\newblock On extremal rational elliptic surfaces.
\newblock {\em Math. Z.}, 193:537--558, 1986.

\bibitem{oancea08}
A.~Oancea.
\newblock Fibered symplectic cohomology and the {L}eray-{S}erre spectral
  sequence.
\newblock {\em J. Symplectic Geom.}, 6:267--351, 2008.

\bibitem{orlov}
D.~Orlov.
\newblock Formal completions and idempotent completions of triangulated categories of singularities.
\newblock {\em Adv. Math.}, 226:206--217, 2011.

\bibitem{ostrover-tyomkin08}
Y.~Ostrover and I.~Tyomkin.
\newblock On the quantum homology algebra of toric {F}ano manifolds.
\newblock {\em Selecta Math.}, 15:121--149, 2009.

\bibitem{petrov-vaintrob-vologodsky18}
A.~Petrov, D.~Vaintrob, and V.~Vologodsky.
\newblock The {G}auss-{M}anin connection on the periodic cyclic homology.
\newblock {\em Selecta Math.}, 24:531--561, 2018.

\bibitem{perutz-sheridan22}
T.~Perutz and N.~Sheridan.
\newblock Constructing the relative {F}ukaya category.
\newblock Preprint arXiv:2203.15482.

\bibitem{pham79}
F.~Pham.
\newblock {\em Singularit\'{e}s des syst\`emes diff\'{e}rentiels de
  {G}auss-{M}anin}, volume~2 of {\em Progress in Mathematics}.
\newblock Birkh\"{a}user, 1979.
\newblock With contributions by L. K. Chan, Ph. Maisonobe and J.-\'{E}.
  Rombaldi.

\bibitem{pomerleano21}
D.~Pomerleano.
\newblock Intrinsic mirror symmetry and categorical crepant resolutions.
\newblock Preprint arxiv:2103.01200.

\bibitem{pomerleano-seidel24}
D.~Pomerleano and P.~Seidel.
\newblock Symplectic cohomology relative to a smooth anticanonical divisor.
\newblock Preprint arXiv:2408.09039.

\bibitem{Pozniak1994}
M. ~Po\'zniak.
\newblock Floer homology, Novikov rings and clean intersections.
\newblock University of Warwick Thesis, 1994.

\bibitem{preygel11}
A.~Preygel.
\newblock {T}hom-{S}ebastiani duality for matrix factorizations.
\newblock Preprint arXiv:1101.5834.

\bibitem{reichelt-sevenheck17}
T.~Reichelt and Ch. Sevenheck.
\newblock Non-affine {L}andau-{G}inzburg models and intersection cohomology.
\newblock {\em Ann. Sci. \'{E}c. Norm. Sup\'{e}r.}, 50:665--753, 2017.

\bibitem{reznikov93}
A.~Reznikov.
\newblock Symplectic twistor spaces.
\newblock {\em Ann. Global Anal. Geom.}, 11:109--118, 1993.

\bibitem{ritterzivanovic23}
A.~Ritter and F.~Zivanovic.
\newblock Symplectic $\mathbb{C}^*$-manifolds {II}: Morse-bott-floer spectral
  sequences.
\newblock Preprint arXiv:2304.14384.

\bibitem{sabbah-isomonodromic}
C.~Sabbah.
\newblock {\em Isomonodromic deformations and {F}robenius manifolds. An
  introduction}.
\newblock Springer, 2007.

\bibitem{sabbah10}
C.~Sabbah.
\newblock Fourier-{L}aplace transform of a variation of polarized complex
  {H}odge structure, {II}.
\newblock In {\em New developments in algebraic geometry, integrable systems
  and mirror symmetry ({RIMS}, {K}yoto, 2008)}, volume~59 of {\em Adv. Stud.
  Pure Math.}, pages 289--347. Math. Soc. Japan, 2010.

\bibitem{sabbah11}
C.~Sabbah.
\newblock Non-commutative {H}odge structures.
\newblock {\em Ann. Inst. Fourier}, 61:2681--2717, 2011.

\bibitem{sabbah-stokes}
C.~Sabbah.
\newblock {\em Introduction to {S}tokes structures}, volume 2010 of {\em
  Lecture notes in math.}
\newblock Springer, 2013.

\bibitem{sabbah16}
C.~Sabbah.
\newblock Introduction to pure non-commutative {H}odge structures.
\newblock Lecture notes (available on the author's homepage), 2016.

\bibitem{salvatore21}
P.~Salvatore.
\newblock The {F}ulton-{M}ac{P}herson operad and the {$W$}-construction.
\newblock {Homology Homotopy Appl.}, 23:1--8, 2021.


\bibitem{sanda-shamoto19}
F.~Sanda and Y.~Shamoto.
\newblock An analogue of {D}ubrovin's conjecture.
\newblock {\em Ann. Inst. Fourier}, 70:621--682, 2020.

\bibitem{sano14}
T.~Sano.
\newblock Unobstructedness of deformations of weak {F}ano manifolds.
\newblock {\em IMRN}, pages 5124--5133, 2014.

\bibitem{seidel-code}
P.~Seidel.
\newblock \url{math.mit.edu/~seidel/}, under ``Code''.

\bibitem{seidel-biased}
P.~Seidel.
\newblock A biased view of symplectic cohomology.
\newblock In {\em Current developments in mathematics, 2006}, pages 211--253.
  Intl. Press, 2008.

\bibitem{seidel04}
P.~Seidel.
\newblock {\em {F}ukaya categories and {P}icard-{L}efschetz theory}.
\newblock European Math. Soc., 2008.

\bibitem{seidel03b}
P.~Seidel.
\newblock {\em Homological mirror symmetry for the quartic surface}, volume
  1116 of {\em Mem. Amer. Math. Soc.}
\newblock Amer. Math. Soc., 2015.

\bibitem{seidel08}
P.~Seidel.
\newblock {$A\sb \infty$}-subalgebras and natural transformations.
\newblock {\em Homology Homotopy Appl.}, 10:83--114, 2008.

\bibitem{seidel18}
P.~Seidel.
\newblock Connections on equivariant {H}amiltonian {F}loer cohomology.
\newblock {\em Comment. Math. Helv.}, 93:587--644, 2018.

\bibitem{sheridan16}
N.~Sheridan.
\newblock On the {F}ukaya category of a {F}ano hypersurface in projective
  space.
\newblock {\em Publ. Math. IHES}, 124:165--317, 2016.

\bibitem{sheridan20}
N.~Sheridan.
\newblock Formulae in noncommutative {H}odge theory.
\newblock {\em J. Homotopy Relat. Struct.}, 15:249--299, 2020.

\bibitem{shklyarov07}
D.~Shklyarov.
\newblock On {S}erre duality for compact homologically smooth {DG} algebras.
\newblock Preprint arXiv:math/0702590.

\bibitem{shklyarov13}
D.~Shklyarov.
\newblock Hirzebruch-{R}iemann-{R}och-type formula for {DG} algebras.
\newblock {\em Proc. Lond. Math. Soc. (3)}, 106:1--32, 2013.

\bibitem{shklyarov14}
D.~Shklyarov.
\newblock Non-commutative {H}odge structures: towards matching categorical and
  geometric examples.
\newblock {\em Trans. Amer. Math. Soc.}, 366:2923--2974, 2014.

\bibitem{shokurov79}
V.~Shokurov.
\newblock Smoothness of a general anticanonical divisor on a fano variety.
\newblock {\em Math. USSR Izvestiya}, 14:395--405, 1980.

\bibitem{spaltenstein88}
N.~Spaltenstein.
\newblock Resolutions of unbounded complexes.
\newblock {\em Compositio Math.}, 65:121--154, 1988.

\bibitem{tonkonog19}
D.~Tonkonog.
\newblock From symplectic cohomology to {L}agrangian enumerative geometry.
\newblock {\em Adv. Math.}, 352:717--776, 2019.

\bibitem{vanderput-singer}
M.~van~der Put and M.~Singer.
\newblock {\em Galois theory of linear differential equations}.
\newblock Springer, 2003.

\bibitem{wasow65}
W.~Wasow.
\newblock {\em Asymptotic expansions for ordinary differential equations}.
\newblock Interscience, 1965.

\bibitem{weibel}
C.~Weibel.
\newblock {\em An introduction to homological algebra}.
\newblock Cambridge Univ. Press, 1994.

\end{thebibliography}
\end{document}